\documentclass[11pt,a4paper]{article}

\usepackage{soul}
\usepackage{tikz}
\usepackage{amsthm}
\usepackage{amsmath}
\usepackage{amssymb}
\usepackage{mathrsfs}

\usepackage{geometry}
\usepackage{graphicx}
\usepackage{amsfonts}
\usepackage{epstopdf}
\usepackage{placeins}
\usepackage{enumerate}
\usepackage{enumitem}
\usepackage{subcaption}
\usepackage{microtype}
\usepackage[hidelinks]{hyperref} 
\usepackage[normalem]{ulem} 

\usetikzlibrary{calc,arrows.meta, positioning}

\numberwithin{equation}{section}

\allowdisplaybreaks

\newcommand{\esup}{\operatorname*{ess\,sup}}
\newcommand{\einf}{\operatorname*{ess\,inf}}
\newcommand{\eosc}{\operatorname*{ess\,osc}}

\newcommand{\Rmnum}[1]{\uppercase\expandafter{\romannumeral#1}} 

\def\Xint#1{\mathchoice
{\XXint\displaystyle\textstyle{#1}}%
{\XXint\textstyle\scriptstyle{#1}}%
{\XXint\scriptstyle\scriptscriptstyle{#1}}%
{\XXint\scriptscriptstyle\scriptscriptstyle{#1}}%
\!\int}
\def\XXint#1#2#3{{\setbox0=\hbox{$#1{#2#3}{\int}$ }
\vcenter{\hbox{$#2#3$ }}\kern-.6\wd0}}

\def\dashint{\Xint-}

\theoremstyle{plain}
\newtheorem{theorem}{Theorem}[section]
\newtheorem{proposition}[theorem]{Proposition}
\newtheorem{lemma}[theorem]{Lemma}
\newtheorem{corollary}[theorem]{Corollary}

\theoremstyle{definition}
\newtheorem{remark}[theorem]{Remark}
\newtheorem{question}[theorem]{Question}

\renewcommand{\thefootnote}{}

\AtEndDocument{
\bigskip{\footnotesize%
   \textsc{Department of Mathematics, Aarhus University, 8000 Aarhus C, Denmark} \par
  \textit{E-mail address}: \texttt{yangmengqh@gmail.com}\par
  \textit{E-mail address}: \texttt{yang@math.au.dk}
}}

\begin{document}

\title{{Local and non-local $p$-energies on metric measure spaces}}
\author{Meng Yang}
\date{}

\maketitle

\abstract{For $p>1$, we study subordination phenomena for local and non-local regular $p$-energies on metric measure spaces. Under suitable geometric assumptions, we show that if a local regular $p$-energy satisfies a Poincar\'e inequality and a cutoff Sobolev inequality with scaling function $\Psi$, then any non-local $p$-form induced by a jumping kernel with scaling function $\Upsilon$, where $\Upsilon$ lies \emph{strictly above $\Psi$ at small scales}, defines a regular $p$-energy satisfying a non-local Poincar\'e inequality and a non-local cutoff Sobolev inequality. The corresponding scaling function $\Xi$ is explicitly determined by $\Psi$ and $\Upsilon$. Our results also cover examples whose jumping kernels have light polynomial tails at infinity. These results provide a nonlinear extension of the classical subordination principle beyond the Dirichlet form framework.}

\footnote{\textsl{Date}: \today}
\footnote{The author is very grateful to Jiaxin Hu for raising the question of comparing cutoff functions for local and non-local Dirichlet forms, which initiated this work, to Eryan Hu for stimulating discussions on insightful techniques for non-local Dirichlet forms, and to Josh Kline for discussions on the BBM-type Poincar\'e inequalities.}

\renewcommand{\thefootnote}{\arabic{footnote}}
\setcounter{footnote}{0}

\section{Introduction}\label{sec_intro}

Let us recall several classical results. First, let $(X,d,m)$ be a $d_h$-Ahlfors regular metric measure space. Assume that there exists a Brownian motion, or equivalently, a strongly local regular Dirichlet form, whose associated heat kernel satisfies the following two-sided sub-Gaussian estimates:
\begin{equation*}\label{eq_HKL}\tag*{$\text{HK}^{(L)}(d_w)$}
\frac{C_1}{t^{d_h/d_w}}\exp\left(-C_2\left(\frac{d(x,y)}{t^{1/d_w}}\right)^{\frac{d_w}{d_w-1}}\right)\le p_t(x,y)\le\frac{C_3}{t^{d_h/d_w}}\exp\left(-C_4\left(\frac{d(x,y)}{t^{1/d_w}}\right)^{\frac{d_w}{d_w-1}}\right),
\end{equation*}
where $d_w$ is a new parameter called the walk dimension, which is typically strictly greater than 2 on fractals. By the classical subordination technique (see, for example, \cite{Kum03, Pie08}), for any $\beta\in(0,d_w)$, the stable-like quadratic form
\begin{align*}
&\mathcal{E}^{(\beta)}(u)=\int_X\int_X \frac{\lvert u(x)-u(y)\rvert^2}{d(x,y)^{d_h+\beta}}m(\mathrm{d}x) m(\mathrm{d}y),\\
&\mathcal{F}^{(\beta)}=\left\{u\in L^2(X;m):\int_X\int_X \frac{\lvert u(x)-u(y)\rvert^2}{d(x,y)^{d_h+\beta}}m(\mathrm{d}x) m(\mathrm{d}y)<+\infty\right\},
\end{align*}
defines a non-local regular Dirichlet form whose associated heat kernel satisfies the polynomial-type estimates
\begin{equation*}\label{eq_HKJ}\tag*{$\text{HK}^{(J)}(\beta)$}
\frac{C_1}{t^{d_h/\beta}}\left(1+\frac{d(x,y)}{t^{1/\beta}}\right)^{-(d_h+\beta)}\le p_t(x,y)\le\frac{C_2}{t^{d_h/\beta}}\left(1+\frac{d(x,y)}{t^{1/\beta}}\right)^{-(d_h+\beta)}.
\end{equation*}

In 2004, Barlow and Bass \cite{BB04} introduced the well-known cutoff Sobolev inequality $\text{CS}^{(L)}$ in the local setting. This condition provides energy inequalities for cutoff functions associated with Dirichlet forms. They proved that \ref{eq_HKL} is equivalent to the conjunction of the local Poincar\'e inequality $\text{PI}^{(L)}(d_w)$ and $\text{CS}^{(L)}(d_w)$, both with parameter $d_w$. In later work, Chen, Kumagai, Wang \cite{CKW20,CKW21}, and independently Grigor'yan, Hu, Hu \cite{GHH18}, introduced the non-local condition $\text{CS}^{(J)}(\beta)$ and proved that \ref{eq_HKJ} is equivalent to $\text{CS}^{(J)}(\beta)$.

The original formulation of $\text{CS}^{(L)}$ is technically involved, even after several simplifications in \cite{AB15,Mur24a}, and was long believed to be difficult to verify on concrete fractals, even in classical examples such as the Sierpi\'nski gasket. In the strongly recurrent setting, roughly speaking when $d_w>d_h$, the author \cite{Yan25c} provided a direct and simple proof of $\text{CS}^{(L)}(d_w)$; see also \cite{Yan25b} for the case of the Sierpi\'nski gasket, where the main idea is illustrated.

It had remained a long-standing conjecture to replace $\text{CS}$ by the substantially simpler capacity upper bound condition $\text{cap}_\le$, known as the resistance conjecture (see \cite{GHL14} and \cite{GHH18}), until it was very recently resolved by Eriksson-Bique \cite{Eri26} in the local setting and by Murugan \cite{Mur26} in the non-local setting.

At the level of functional inequalities alone, a natural question is as follows:

\begin{quote}
Can one obtain $\text{CS}^{(J)}(\beta)$ \emph{directly} from $\text{PI}^{(L)}(d_w)$ and $\text{CS}^{(L)}(d_w)$, without using heat kernel estimates?
\end{quote}

Second, let $\Upsilon:[0,+\infty)\to[0,+\infty)$ be a piecewise power function given by
$$\Upsilon(r)=r^{\beta_1}1_{(0,1)}(r)+r^{\beta_2}1_{[1,+\infty)}(r),$$
where $\beta_1\in(0,d_w)$ and $\beta_2\in(0,+\infty)$, then by the heat kernel estimate theory (see \cite{BKKL19,BKKL25,CKW22,LM23}), the quadratic form
\begin{equation}\label{eq_EUpsilon}
\begin{aligned}
&\mathcal{E}^{(\Upsilon)}(u)=\int_X\int_X \frac{\lvert u(x)-u(y)\rvert^2}{d(x,y)^{d_h}\Upsilon(d(x,y))}m(\mathrm{d}x) m(\mathrm{d}y),\\
&\mathcal{F}^{(\Upsilon)}=\left\{u\in L^2(X;m):\int_X\int_X \frac{\lvert u(x)-u(y)\rvert^2}{d(x,y)^{d_h}\Upsilon(d(x,y))}m(\mathrm{d}x) m(\mathrm{d}y)<+\infty\right\},
\end{aligned}
\end{equation}
defines a non-local regular Dirichlet form satisfying the non-local Poincar\'e inequality $\text{PI}^{(J)}(\Xi)$ and the non-local cutoff Sobolev inequality $\text{CS}^{(J)}(\Xi)$, where the scaling function $\Xi$ may be different from $\Upsilon$ and is given by $\Xi(r)=r^{\beta_1}$ for $r\in(0,1)$, while for $r\ge1$,
\begin{equation*}
\Xi(r)=
\begin{cases}
r^{\beta_2},&\text{if }\beta_2<d_w,\\
\frac{r^{d_w}}{\log(e-1+r)},&\text{if }\beta_2=d_w,\\
r^{d_w},&\text{if }\beta_2>d_w.
\end{cases}
\end{equation*}

For this example, where the resulting scaling function $\Xi$ may be different from $\Upsilon$, one may also ask the following question:

\begin{quote}
Can one derive $\text{PI}^{(J)}(\Xi)$ and $\text{CS}^{(J)}(\Xi)$ \emph{directly} from $\text{PI}^{(L)}(d_w)$ and $\text{CS}^{(L)}(d_w)$, without using heat kernel estimates?
\end{quote}

Third, for $p>1$, the classical Bourgain-Brezis-Mironescu (BBM) limiting theory provides a fundamental bridge between non-local and local $p$-energies. One important manifestation of this phenomenon is the so-called BBM-type Poincar\'e inequality: there exist $C>0$, $A>1$ such that, for any $\theta\in(0,1)$, any $u\in L^p(X;m)$, and any ball $B(x,r)$, we have
\begin{align*}
&\int_{B(x,r)}\lvert u-u_{B(x,r)}\rvert^p \mathrm{d}m\\
&\le C(1-\theta)r^{\theta p}\int_{B(x,Ar)}\int_{B(x,Ar)}\frac{\lvert u(y)-u(z)\rvert^p}{m(B(y,d(y,z)))d(y,z)^{\theta p}}m(\mathrm{d}y)m(\mathrm{d}z).
\end{align*}
Here the highly interesting extra factor $(1-\theta)$ appears in front of the fractional term and balances the limiting behavior of the right-hand side as $\theta\uparrow 1$, so that the renormalized quantity recovers the corresponding local energy in the BBM limit. In Euclidean spaces, this point of view goes back to \cite{BBM01,BBM02}; see also \cite{MS02,Pon04,HSMPPV23,MPW24}. More recently, BBM-type Poincar\'e inequalities have also been investigated on metric measure spaces with an upper gradient structure; see \cite{KLLZ25}.

One motivation for this work is the following:

\begin{quote}
To establish BBM-type Poincar\'e inequalities on general metric measure spaces, in particular, on fractals where no gradient or upper gradient structure is available.
\end{quote}

The first two results are formulated within the Dirichlet form framework: strongly local Dirichlet forms generalize the classical Dirichlet integral $\int_{\mathbb{R}^d}\lvert \nabla f(x)\rvert^2\mathrm{d}x$ on $\mathbb{R}^d$. By contrast, the local $p$-energies in the third result are formulated using upper gradients. For general $p>1$, extending the classical $p$-energy $\int_{\mathbb{R}^d}|\nabla f(x)|^p\mathrm{d} x$ in $\mathbb{R}^d$, as initiated by \cite{HPS04}, the study of local $p$-energies on fractals and general metric measure spaces has been recently advanced considerably, see \cite{CGQ22,Shi24,BC23,MS25,Kig23,CGYZ26,AB25,AES25a}. In this setting, a new parameter $d_{w,p}$, called the $p$-walk dimension, naturally arises; notably, $d_{w,2}$ coincides with $d_w$ in \ref{eq_HKL}. One may therefore expect that suitable variants of the cutoff Sobolev inequality $\text{CS}^{(L)}$, as formulated in \cite{Yan25a}, play a similarly fundamental role in the study of local $p$-energies. Some applications, such as the elliptic Harnack inequality, Wolff potential estimates, the singular behavior of energy measures, and the dichotomy of $p$-walk dimensions, were investigated in \cite{Yan25c,Yan25d,Yan25e}.

In (\ref{eq_EUpsilon}), replacing $2$ by $p$ and replacing the kernel $\frac{1}{d(x,y)^{d_h}\Upsilon(d(x,y))}$ by a jumping kernel $K^{(\Upsilon)}(x,y)$ satisfying
$$\frac{C_1}{m(B(x,d(x,y)))\Upsilon(d(x,y))}\le K^{(\Upsilon)}(x,y)\le\frac{C_2}{m(B(x,d(x,y)))\Upsilon(d(x,y))},$$
where $\Upsilon$ is a suitable general scaling function, one can define the following non-local $p$-form:
\begin{equation}\label{eq_EUpsilonp}
\begin{aligned}
&\mathcal{E}^{(\Upsilon,p)}(u)=\int_X\int_X\lvert u(x)-u(y)\rvert^p K^{(\Upsilon)}(x,y)m(\mathrm{d}x)m(\mathrm{d}y),\\
&\mathcal{F}^{(\Upsilon,p)}=\left\{u\in L^p(X;m):\int_X\int_X\lvert u(x)-u(y)\rvert^p K^{(\Upsilon)}(x,y)m(\mathrm{d}x)m(\mathrm{d}y)<+\infty\right\}.
\end{aligned}
\end{equation}
These forms can be formulated for any choice of $\Upsilon$. However, if the growth of $\Upsilon$ is not properly controlled, the space $\mathcal{F}^{(\Upsilon,p)}$ may be trivial, consisting only of constant functions. As in the Dirichlet form setting, we are therefore interested in the \emph{regular} regime, where $\mathcal{F}^{(\Upsilon,p)}$ contains sufficiently many continuous functions to support potential-theoretic and PDE methods. Moreover, one can formulate natural $p$-analogues of $\text{PI}^{(J)}$, $\text{CS}^{(J)}$, and $\text{cap}^{(J)}_{\le}$, which likewise play a central role in the study of non-local $p$-energies.

Motivated by the above results, we pose the following natural questions in the setting of $p$-energies.

\begin{question}\label{que_sub}
Let $(\mathcal{E}^{(L,p)},\mathcal{F}^{(L,p)})$ be a strongly local regular $p$-energy satisfying $\text{PI}^{(L)}(\Psi)$ and $\text{CS}^{(L)}(\Psi)$ for a suitable scaling function $\Psi$.
\begin{itemize}
\item For which scaling functions $\Upsilon$ is the $p$-form $(\mathcal{E}^{(\Upsilon,p)},\mathcal{F}^{(\Upsilon,p)})$, defined in (\ref{eq_EUpsilonp}), a regular $p$-energy?
\item In this case, does $(\mathcal{E}^{(\Upsilon,p)},\mathcal{F}^{(\Upsilon,p)})$ satisfy the non-local Poincar\'e inequality $\text{PI}^{(J)}$ and the non-local cutoff Sobolev inequality $\text{CS}^{(J)}$? If so, what are the corresponding scaling functions?
\end{itemize}
\end{question}

The main result of this paper makes progress towards answering this question; see Theorem \ref{thm_sub}. The key observation is that, under suitable geometric assumptions, if a regular $p$-energy, either local or non-local as in (\ref{eq_EUpsilonp}), satisfies a Poincar\'e inequality and a cutoff Sobolev inequality, then the cutoff functions appearing in the latter can be chosen to be H\"older continuous. These cutoff functions are closely related to the condition $\text{CE}$ recently introduced in \cite{Ant25a}, as well as to the original formulation of $\text{CS}$ in \cite{BB04}. In fact, we prove that several versions of $\text{CE}$ and $\text{CS}$ are equivalent: the strong versions $\text{CE}_{\text{strong}}$ and $\text{CS}_{\text{strong}}$, where the cutoff functions are H\"older continuous; the continuous versions $\text{CE}_{\text{cont}}$ and $\text{CS}_{\text{cont}}$, where the cutoff functions are continuous; and the weak versions $\text{CE}_{\text{weak}}$ and $\text{CS}_{\text{weak}}$, where no continuity condition on the cutoff functions is assumed. See Theorems \ref{thm_equiv_L} and \ref{thm_equiv_J}.

Assume that $(\mathcal{E}^{(L,p)},\mathcal{F}^{(L,p)})$ is a strongly local regular $p$-energy satisfying $\text{PI}^{(L)}(\Psi)$ and $\text{CE}^{(L)}_{\text{strong}}(\Psi)$. Let $\Upsilon$ be a scaling function satisfying the following strict upper growth condition at small scales:
\begin{equation}\label{eq_SUG0int}
\sum_{n=0}^{+\infty}\frac{\Psi(2^{-n})}{\Upsilon(2^{-n})}<+\infty.
\end{equation}
We prove that the $p$-form $(\mathcal{E}^{(\Upsilon,p)},\mathcal{F}^{(\Upsilon,p)})$ defined in (\ref{eq_EUpsilonp}) satisfies $\text{PI}^{(J)}(\Xi)$ and $\text{CE}^{(J)}_{\text{strong}}(\Xi)$, where $\Xi$ is the scaling function explicitly determined by $\Psi$ and $\Upsilon$ as follows:
$$
\Xi(r)=
\begin{cases}
0, & r=0,\\
\frac{\Psi(r)}{\int_0^r \frac{\mathrm{d}\Psi(t)}{\Upsilon(t)}}, & r>0.
\end{cases}
$$
These functional inequalities imply the \emph{regularity} of $(\mathcal{E}^{(\Upsilon,p)},\mathcal{F}^{(\Upsilon,p)})$; in this implication, the continuity of cutoff functions plays a crucial role, as shown in Proposition \ref{prop_regular}. This was also previously observed in \cite{CGHL25}. Moreover, we prove that (\ref{eq_SUG0int}) is necessary not only for the $p$-form $(\mathcal{E}^{(\Upsilon,p)},\mathcal{F}^{(\Upsilon,p)})$ to be regular, but even for $\mathcal{F}^{(\Upsilon,p)}$ to contain a non-constant function.

Finally, to prove $\text{CE}^{(J)}_{\text{strong}}(\Xi)$, we use the \emph{same} cutoff functions as those in $\text{CE}^{(L)}_{\text{strong}}(\Psi)$. To prove $\text{PI}^{(J)}(\Xi)$, we construct a partition of unity with controlled energy and apply a discrete convolution argument. To the best of our knowledge, this analytic argument to obtain non-local functional inequalities from local functional inequalities is new, even in the linear ($p=2$) Dirichlet form setting.

Throughout this paper, $p\in(1,+\infty)$ is fixed. The letters $C,C_1,C_2,C_A, C_B$ will always refer to some positive constants and may change at each occurrence. The sign $\asymp$ means that the ratio of the two sides is bounded from above and below by positive constants. The sign $\lesssim$ ($\gtrsim$) means that the LHS is bounded by positive constant times the RHS from above (below). For $x,y\in \mathbb{R}$, denote $x\vee y=\max\{x,y\}$, $x\wedge y=\min\{x,y\}$, $x_+=\max\{x,0\}$, and $x_-=\max\{-x,0\}$. We use $\# A$ to denote the cardinality of a set $A$.

\section{Statement of main results}\label{sec_result}

We say that a function $\Psi:[0,+\infty)\to[0,+\infty)$ is doubling if $\Psi$ is a homeomorphism, which implies that $\Psi$ is strictly increasing continuous and $\Psi(0)=0$, and there exists $C_\Psi>1$, called a doubling constant of $\Psi$, such that $\Psi(2r)\le C_\Psi\Psi(r)$ for any $r>0$. Moreover, we assume that there exist $\beta^{(1)}_\Psi,\beta^{(2)}_\Psi>0$ with $\beta^{(1)}_\Psi\le \beta^{(2)}_\Psi$ such that
\begin{equation}\label{eq_beta12}
\frac{1}{C_\Psi}\left(\frac{R}{r}\right)^{\beta_\Psi^{(1)}}\le\frac{\Psi(R)}{\Psi(r)}\le {C_\Psi}\left(\frac{R}{r}\right)^{\beta_\Psi^{(2)}}\text{ for any }r\le R.
\end{equation}
Indeed, we can take $\beta_\Psi^{(2)}=\log_2C_\Psi$. Throughout this paper, we always assume that $\Psi,\Upsilon$ are doubling functions with doubling constants $C_\Psi, C_\Upsilon$, and each of them satisfies (\ref{eq_beta12}) with constants $\beta^{(1)}_\Psi,\beta^{(2)}_\Psi,\beta^{(1)}_\Upsilon,\beta^{(2)}_\Upsilon$.

Let $(X,d,m)$ be a complete unbounded metric measure space, that is, $(X,d)$ is a complete unbounded locally compact separable metric space and $m$ is a positive Radon measure on $X$ with full support. Throughout this paper, we always assume that all metric balls are relatively compact. For any $x\in X$ and any $r>0$, denote $B(x,r)=\{y\in X:d(x,y)<r\}$ and $V(x,r)=m(B(x,r))$. If $B=B(x,r)$, then denote $\delta B=B(x,\delta r)$ for any $\delta>0$. Let $\mathcal{B}(X)$ be the family of all Borel measurable subsets of $X$. Let $C(X)$ be the family of all continuous functions on $X$. Let $C_c(X)$ be the family of all continuous functions on $X$ with compact support. Denote $\dashint_A=\frac{1}{m(A)}\int_A$ and $u_A=\dashint_Au\mathrm{d} m$ for any measurable set $A$ with $m(A)\in(0,+\infty)$ and any function $u$ such that the integral $\int_Au\mathrm{d} m$ is well-defined. We use the convention that $\int_A\int_B\ldots m(\mathrm{d}x)m(\mathrm{d}y)$ denotes integration over $x\in A$ and $y\in B$. For a measurable function $u$ on a measure set $A$, denote by $\esup_{A}u$ and $\einf_{A}u$ the essential supremum and essential infimum of $u$ on $A$, respectively, and define $\eosc_A=\esup_Au-\einf_Au$ whenever the difference is well-defined.

Let $\varepsilon>0$. We say that $V\subseteq X$ is an $\varepsilon$-net if for any distinct $x,y\in V$, we have $d(x,y)\ge \varepsilon$, and for any $z\in X$, there exists $x\in V$ such that $d(x,z)<\varepsilon$. Since $(X,d)$ is separable, all $\varepsilon$-nets are countable.

Let $U,V$ be two open subsets of $X$ satisfying $U\subseteq\overline{U}\subseteq V$. We say that $\phi$ is a cutoff function for $U\subseteq V$ if $0\le\phi\le1$ $m$-a.e. in $X$, $\phi=1$ $m$-a.e. in an open neighborhood of $\overline{U}$, and $\mathrm{supp}(\phi)\subseteq V$, where $\mathrm{supp}(f)$ refers to the support of the measure of $|f|\mathrm{d} m$ for any given function $f$. In our definition, we do \emph{not} require cutoff functions to be continuous.

We say that the chain condition \ref{eq_CC} holds if there exists $C_{cc}>0$ such that for any $x,y\in X$ and any integer $n\ge1$, there exists a sequence $\{x_k:0\le k\le n\}$ of points in $X$ with $x_0=x$ and $x_n=y$ such that
\begin{equation*}\label{eq_CC}\tag*{CC}
d(x_k,x_{k-1})\le C_{cc} \frac{d(x,y)}{n}\text{ for any }k=1,\ldots,n.
\end{equation*}
Throughout this paper, we always assume \ref{eq_CC}. Since we will consider the boundary regularity of solutions to PDEs, the chain condition will play an important role in establishing certain regularity results; see Lemma \ref{lem_corkscrew}.

We say that the volume doubling condition \ref{eq_VD} holds if there exists $C_{VD}>0$ such that
\begin{equation*}\label{eq_VD}\tag*{VD}
V(x,2r)\le C_{VD}V(x,r)\text{ for any }x\in X,r>0.
\end{equation*}

We say that $(\mathcal{E},\mathcal{F})$ is a $p$-form if $\mathcal{F}$ is a linear subspace of $L^p(X;m)$ and $\mathcal{E}:\mathcal{F}\to[0,+\infty)$ satisfies that $\mathcal{E}^{1/p}$ is a semi-norm on $\mathcal{F}$. In our definition, we allow a $p$-form $(\mathcal{E},\mathcal{F})$ to be trivial, in the sense that the domain $\mathcal{F}$ consists only of constant functions.

For any $\lambda>0$, let $\mathcal{E}_\lambda:\mathcal{F}\to[0,+\infty)$ be given by $\mathcal{E}_\lambda(\cdot)=\mathcal{E}(\cdot)+\lVert \cdot\rVert_{L^p(X;m)}^p$.

We say that a $p$-form $(\mathcal{E},\mathcal{F})$ is a $p$-energy if $\mathcal{F}$ is dense in $L^p(X;m)$ and satisfies the following conditions.
\begin{itemize}
\item (Closed property) $(\mathcal{F},\mathcal{E}_1^{1/p})$ is a Banach space.
\item (Markovian property) For any $\varphi\in C(\mathbb{R})$ with $\varphi(0)=0$ and $|\varphi(t)-\varphi(s)|\le|t-s|$ for any $t,s\in\mathbb{R}$, for any $f\in\mathcal{F}$, we have $\varphi(f)\in\mathcal{F}$ and $\mathcal{E}(\varphi(f))\le\mathcal{E}(f)$.
\item ($p$-Clarkson's inequality) For any $f,g\in\mathcal{F}$, we have
\begin{equation*}
\begin{cases}
\mathcal{E}(f+g)+\mathcal{E}(f-g)\ge2 \left(\mathcal{E}(f)^{\frac{1}{p-1}}+\mathcal{E}(g)^{\frac{1}{p-1}}\right)^{p-1}&\text{if }p\in(1,2],\\
\mathcal{E}(f+g)+\mathcal{E}(f-g)\le2 \left(\mathcal{E}(f)^{\frac{1}{p-1}}+\mathcal{E}(g)^{\frac{1}{p-1}}\right)^{p-1}&\text{if }p\in[2,+\infty).\\
\end{cases}
\end{equation*} 
\end{itemize}
By the $p$-Clarkson's inequality, for any $f,g\in \mathcal{F}$, the derivative
$$\mathcal{E}(f;g)=\frac{1}{p}\frac{\mathrm{d}}{\mathrm{d} t}\mathcal{E}(f+tg)|_{t=0}\in\mathbb{R}$$
exists, the map $\mathcal{E}(f;\cdot):\mathcal{F}\to\mathbb{R}$ is linear, $\mathcal{E}(f;f)=\mathcal{E}(f)$. Moreover, for any $f,g\in\mathcal{F}$, for any $a\in\mathbb{R}$, we have
$$\mathcal{E}(af;g)=\mathrm{sgn}(a)|a|^{p-1}\mathcal{E}(f;g),$$
$$|\mathcal{E}(f;g)|\le\mathcal{E}(f)^{(p-1)/p}\mathcal{E}(g)^{1/p}.$$
Moreover, for any $\lambda>0$, all of the above results remain valid with $\mathcal{E}$ replaced by $\mathcal{E}_\lambda$, and for any $f,g\in\mathcal{F}$, we have
$$\mathcal{E}_\lambda(f;g)=\mathcal{E}(f;g)+\lambda\int_X\mathrm{sgn}(f)|f|^{p-1}g\mathrm{d} m.$$
See \cite[Theorem 3.7, Corollary 3.25]{KS24a} for the proofs of these results.

We say that a $p$-form $(\mathcal{E},\mathcal{F})$ is
\begin{itemize}
\item strongly local if for any $f,g\in\mathcal{F}$ with compact support and $g$ constant in an open neighborhood of $\mathrm{supp}(f)$, we have $\mathcal{E}(f+g)=\mathcal{E}(f)+\mathcal{E}(g)$.
\item regular if $\mathcal{F}\cap C_c(X)$ is uniformly dense in $C_c(X)$ and $\mathcal{E}_1^{1/p}$-dense in $\mathcal{F}$.
\end{itemize}

Let $K^{(J)}:X\times X\backslash \mathrm{diag}\to[0,+\infty)$ satisfy $K^{(J)}(x,y)=K^{(J)}(y,x)$ for any distinct $x,y\in X$. We define the non-local $p$-form $(\mathcal{E}^{(J)},\mathcal{F}^{(J)})$ by
\begin{align*}
&\mathcal{E}^{(J)}(u)=\int_X\int_X \lvert u(x)-u(y)\rvert^p K^{(J)}(x,y)m(\mathrm{d}x)m(\mathrm{d}y),\\
&\mathcal{F}^{(J)}=\left\{u\in L^p(X;m):\int_X\int_X \lvert u(x)-u(y)\rvert^p K^{(J)}(x,y)m(\mathrm{d}x)m(\mathrm{d}y)<+\infty\right\}.
\end{align*}
Let
$$\widehat{\mathcal{F}}^{(J)}=\left\{u+a:u\in \mathcal{F}^{(J)},a\in \mathbb{R}\right\}.$$
For any $x_0\in X$ and any $R>0$, for any measurable function $u$, define
$$\mathrm{Tail}^{(J)}(u;x_0,R)=\left(\int_{X\backslash B(x_0,R)}\lvert u(y)\rvert^{p-1}K^{(J)}(x_0,y)m(\mathrm{d}y)\right)^{\frac{1}{p-1}}.$$

For any measurable subset $\Omega\subseteq X$, for any $u\in \widehat{\mathcal{F}}^{(J)}$, let $\Gamma^{(J)}_\Omega(u)$ be the measure on $X$ given by
$$\int_X f \mathrm{d}\Gamma^{(J)}_\Omega(u)=\int_X\left(\int_\Omega f(x)\lvert u(x)-u(y)\rvert^p K^{(J)}(x,y)m(\mathrm{d}y)\right)m(\mathrm{d}x)\text{ for any }f\in C_c(X).$$


We say that the upper polynomial radial regularity condition \ref{eq_UPR} holds if there exist $C_{U}>0$, $\theta>0$ such that for any $x_0,x,y\in X$ with $x_0\ne x$ and $x_0\ne y$, we have
\begin{equation*}\label{eq_UPR}\tag*{$\text{UPR}^{(J)}$}
\frac{K^{(J)}(x_0,x)}{K^{(J)}(x_0,y)}\le C_{U}\left(\frac{d(x_0,y)}{d(x_0,x)}\vee 1\right)^\theta.
\end{equation*}

Let $\Upsilon$ be a doubling function satisfying (\ref{eq_beta12}).

We say that the tail estimate \ref{eq_KJ_tail} holds if there exists $C_T>0$ such that for any $x\in X$ and any $r>0$, we have
\begin{equation*}\label{eq_KJ_tail}\tag*{$\text{T}^{{(J)}}(\Upsilon)$}
\int_{X\backslash B(x,r)}K^{(J)}(x,y)m(\mathrm{d}y)\le \frac{C_T}{\Upsilon(r)}.
\end{equation*}

We say that the non-local Poincar\'e inequality \ref{eq_PIJ} holds if there exist $C_{PI}>0$, $A_{PI}\ge1$ such that for any ball $B(x,r)$, for any $f\in \widehat{\mathcal{F}}^{(J)}$, we have
\begin{equation*}\label{eq_PIJ}\tag*{$\text{PI}^{{(J)}}(\Upsilon)$}
\int_{B(x,r)}\lvert f-f_{B(x,r)}\rvert^p \mathrm{d}m\le C_{PI}\Upsilon(r)\int_{B(x,A_{PI}r)}\mathrm{d}\Gamma^{(J)}_{B(x,A_{PI}r)}(f).
\end{equation*}

We say that \ref{eq_KJ} holds if there exist $C_{1},C_{2}>0$ such that for any distinct $x,y\in X$, we have
\begin{equation*}\label{eq_KJ}\tag*{$\text{K}^{{(J)}}(\Upsilon)$}
\frac{C_{1}}{V(x,d(x,y))\Upsilon(d(x,y))}\le K^{(J)}(x,y)\le \frac{C_{2}}{V(x,d(x,y))\Upsilon(d(x,y))}.
\end{equation*}
We say that \hypertarget{eq_uKJ}{$\text{K}^{(J)}_\le(\Upsilon)$} (resp. \hypertarget{eq_lKJ}{$\text{K}^{(J)}_\ge(\Upsilon)$}) holds if the upper bound (resp. the lower bound) in \ref{eq_KJ} holds. Assuming \ref{eq_VD}, \ref{eq_KJ}, we have \ref{eq_UPR}, \ref{eq_KJ_tail}, \ref{eq_PIJ}; see Lemma \ref{lem_KJ_ele} for the proofs.

It is clear that \ref{eq_PIJ} implies \hyperref[eq_PIJ]{$\text{PI}^{(J)}(\Xi)$} for any doubling function $\Xi$ satisfying $\Upsilon\lesssim \Xi$. However, as we will see in Theorem \ref{thm_sub}, it may happen that \ref{eq_KJ} holds for the doubling function $\Upsilon$, and hence implies \ref{eq_PIJ}, whereas \hyperref[eq_PIJ]{$\text{PI}^{(J)}(\Xi)$} also holds for another doubling function $\Xi$ satisfying $\Xi\lesssim \Upsilon$ but $\Xi\not\asymp \Upsilon$.

If $(\mathcal{E}^{(J)},\mathcal{F}^{(J)})$ is a non-local $p$-form given by a kernel $K^{(J)}$ satisfying \ref{eq_KJ}, then we also write $\mathcal{E}^{(J,\Upsilon)}$, $K^{(J,\Upsilon)}$, $\mathrm{Tail}^{(J,\Upsilon)}$, and so on, to emphasize the dependence on $\Upsilon$.

We say that the non-local cutoff Sobolev inequality \hypertarget{eq_CSJ_strong}{$\text{CS}^{(J)}_{\text{strong}}(\Upsilon)$} holds if there exist $\delta\in(0,1)$, $C,C_1,C_2>0$, $A_1,A_2>1$ with $A_1<A_2$ such that for any ball $B(x_0,r)$, there exists a cutoff function $\phi\in \mathcal{F}^{(J)}$ for $B(x_0,r)\subseteq B(x_0,A_1r)$ such that for any $x,y\in X$ and any $f\in \widehat{\mathcal{F}}^{(J)}$, we have
\begin{subequations}
\begin{equation*}\label{eq_CSJ_Holder}\tag*{$\text{CS}^{(J)}(\Upsilon)\text{-}1$}
\lvert \phi(x)-\phi(y)\rvert\le C \left(\frac{d(x,y)}{r}\wedge 1\right)^\delta,
\end{equation*}
\begin{equation*}\label{eq_CSJ_energy}\tag*{$\text{CS}^{(J)}(\Upsilon)\text{-}2$}
\int_{B(x_0,A_2r)}\lvert f\rvert^p \mathrm{d}\Gamma^{(J)}_{B(x_0,A_2r)}(\phi)\le C_1\int_{B(x_0,A_2r)}\mathrm{d}\Gamma^{(J)}_{B(x_0,A_2r)}(f)+\frac{C_2}{\Upsilon(r)}\int_{B(x_0,A_2r)}\lvert f\rvert^p \mathrm{d}m.
\end{equation*}
\end{subequations}
We say that \hypertarget{eq_CSJ_cont}{$\text{CS}^{(J)}_{\text{cont}}(\Upsilon)$} holds if \ref{eq_CSJ_energy} holds and the cutoff function $\phi$ is continuous. We say that \hypertarget{eq_CSJ_weak}{$\text{CS}^{(J)}_{\text{weak}}(\Upsilon)$} holds if only \ref{eq_CSJ_energy} holds.

We note that the condition \ref{eq_CSJ_energy} simplifies both the earlier AB condition introduced in \cite[Definition 2.1]{GHH18} and the CSJ condition introduced in \cite[DEFINITION 1.5]{CKW21}, which are formulated in the setting of Dirichlet forms ($p=2$), in the following aspects:
\begin{enumerate}[label=(\alph*)]
\item We only consider cutoff functions for $B(x,r)\subseteq B(x,A_1r)$ for any $r>0$, instead of cutoff functions for $B(x,R)\subseteq B(x,R+r)$ for any $R,r>0$, as required in the AB condition, or cutoff functions for $B(x,R)\subseteq B(x,R+r)$ for any $R\ge r>0$, as required in the CSJ condition.
\item For the term on the LHS and the second term on the RHS, we only consider integrals over $B(x,A_2r)$ for any $r>0$, instead of integrals over $B(x,R')$ for any $R'>R+r$, as required in the AB condition, or over $B(x,R+(1+C_0)r)$ for some $C_0\in(0,1]$, as required in the CSJ condition.
\item For the first term on the RHS, we only consider integrals over the ``largest" balls $B(x,A_2r)$, instead of integrals over balls $B(x,R+r)$, as required in the AB condition, or over annuli, as required in the CSJ condition.
\end{enumerate}
By a standard covering argument and a self-improvement argument, we will show that \ref{eq_CSJ_energy} implies a stronger condition, which corresponds to the $p$-energy version of the AB condition or the CSJ condition, see Proposition \ref{prop_CSJ_self}. See also \cite[Proposition 3.1]{Yan25c} for local $p$-energies and \cite[Lemma 6.2]{Mur24a} for local Dirichlet forms.

Assume \ref{eq_KJ_tail}. Then \ref{eq_CSJ_energy} is equivalent to
$$\int_{B(x_0,A_2r)}\lvert f\rvert^p \mathrm{d}\Gamma^{(J)}_{X}(\phi)\le C_1\int_{B(x_0,A_2r)}\mathrm{d}\Gamma^{(J)}_{B(x_0,A_2r)}(f)+\frac{C_2}{\Upsilon(r)}\int_{B(x_0,A_2r)}\lvert f\rvert^p \mathrm{d}m,$$
up to suitable modifications of the constants $C_1,C_2$; see Lemma \ref{lem_CSJ_energy} for the proof.

Let $A_1,A_2\in\mathcal{B}(X)$. We define the non-local capacity between $A_1,A_2$ as
\begin{align*}
&\mathrm{cap}^{(J)}(A_1,A_2)=\inf\left\{\mathcal{E}^{(J)}(\phi):\phi\in\mathcal{F}^{(J)},
\begin{array}{l}
\phi=1\text{ in an open neighborhood of }A_1,\\
\phi=0\text{ in an open neighborhood of }A_2
\end{array}
\right\},
\end{align*}
here we use the convention that $\inf\emptyset=+\infty$.

We say that the non-local capacity upper bound \ref{eq_ucapJ} holds if there exist $C_{cap}>0$, $A_{cap}>1$ such that for any ball $B(x,r)$, we have
\begin{align*}
\mathrm{cap}^{(J)}\left(B(x,r),X\backslash B(x,A_{cap}r)\right)&\le C_{cap} \frac{V(x,r)}{\Upsilon(r)}.\label{eq_ucapJ}\tag*{$\text{cap}^{{(J)}}(\Upsilon)_{\le}$}
\end{align*}
Assume \ref{eq_VD}, \ref{eq_KJ_tail}. Then \hyperlink{eq_CSJ_weak}{$\text{CS}^{(J)}_{\text{weak}}(\Upsilon)$} implies \ref{eq_ucapJ}; see Lemma \ref{lem_CSJ2ucapJ} for the proof.

Note that the above notions are always well-defined. Thus it is meaningful to ask whether the above conditions hold, even in the degenerate case where $\mathcal{F}^{(J)}$ or $\widehat{\mathcal{F}}^{(J)}$ is trivial and consists only of constant functions.

Let $(\mathcal{E}^{(L)},\mathcal{F}^{(L)})$ be a strongly local regular $p$-energy. By \cite[Theorem 2.4]{Sas26}, there exists a (canonical) $p$-energy measure $\Gamma^{(L)}:\mathcal{F}^{(L)}\times\mathcal{B}(X)\to[0,+\infty)$, $(f,A)\mapsto\Gamma^{(L)}(f)(A)$ satisfying the following conditions.
\begin{enumerate}[label=(\alph*),ref=(\alph*)]
\item\label{item_meas1} For any $f\in\mathcal{F}^{(L)}$, $\Gamma^{(L)}(f)(\cdot)$ is a positive Radon measure on $X$ with $\Gamma^{(L)}(f)(X)=\mathcal{E}^{(L)}(f)$.
\item\label{item_meas2} For any $A\in\mathcal{B}(X)$, $\Gamma^{(L)}(\cdot)(A)^{1/p}$ is a semi-norm on $\mathcal{F}^{(L)}$.
\item\label{item_meas3} For any $f,g\in\mathcal{F}^{(L)}\cap C_c(X)$, $A\in\mathcal{B}(X)$, if $f-g$ is constant on $A$, then $\Gamma^{(L)}(f)(A)=\Gamma^{(L)}(g)(A)$.
\item\label{item_meas4} ($p$-Clarkson's inequality) For any $f,g\in\mathcal{F}^{(L)}$, for any $A\in\mathcal{B}(X)$, we have
\begin{align*}
&\Gamma^{(L)}(f+g)(A)+\Gamma^{(L)}(f-g)(A)\\
&\begin{cases}
\ge2 \left(\Gamma^{(L)}(f)(A)^{\frac{1}{p-1}}+\Gamma^{(L)}(g)(A)^{\frac{1}{p-1}}\right)^{p-1}&\text{if }p\in(1,2],\\
\le2 \left(\Gamma^{(L)}(f)(A)^{\frac{1}{p-1}}+\Gamma^{(L)}(g)(A)^{\frac{1}{p-1}}\right)^{p-1}&\text{if }p\in[2,+\infty).\\
\end{cases}
\end{align*}
\item\label{item_meas5} (Chain rule) For any $f,g\in\mathcal{F}^{(L)}\cap C_c(X)$, for any piecewise $C^1$ functions $\varphi,\psi:\mathbb{R}\to\mathbb{R}$ with $\varphi(0)=\psi(0)=0$, we have
$$\mathrm{d}\Gamma^{(L)}\left(\varphi(f);\psi(g)\right)=\mathrm{sgn}(\varphi'(f))|\varphi'(f)|^{p-1}\psi'(g)\mathrm{d}\Gamma^{(L)}(f;g),$$
where $\Gamma^{(L)}(f;g)$ is a signed measure given by $\Gamma^{(L)}(f;g)=\frac{1}{p}\frac{\mathrm{d}}{\mathrm{d}t}\Gamma^{(L)}(f+tg)|_{t=0}$.
\item\label{item_meas6} (Leibniz rule) For any $f,g,h\in \mathcal{F}^{(L)}\cap C_c(X)$, we have $\mathrm{d}\Gamma^{(L)}(f;gh)=g \mathrm{d}\Gamma^{(L)}(f;h)+h \mathrm{d}\Gamma^{(L)}(f;g)$.
\end{enumerate}

We say that the local Poincar\'e inequality \ref{eq_PIL} holds if there exist $C_{PI}>0$, $A_{PI}\ge1$ such that for any ball $B(x,r)$, for any $f\in\mathcal{F}^{(L)}$, we have
\begin{equation*}\label{eq_PIL}\tag*{$\text{PI}^{({L})}(\Psi)$}
\int_{B(x,r)}\lvert f-f_{B(x,r)}\rvert^p\mathrm{d} m\le C_{PI}\Psi(r)\int_{{B(x,A_{PI}r)}}\mathrm{d}\Gamma^{(L)}(f).
\end{equation*}

We say that the local cutoff Sobolev inequality \hypertarget{eq_CSL_strong}{$\text{CS}^{(L)}_{\text{strong}}(\Psi)$} holds if there exist $\delta\in(0,1)$, $C,C_{1},C_{2}>0$, $A_{S}>1$ such that for any ball $B(x_0,r)$, there exists a cutoff function $\phi\in\mathcal{F}^{(L)}$ for $B(x_0,r)\subseteq B(x_0,A_Sr)$ such that for any $x,y\in X$ and any $f\in\mathcal{F}^{(L)}$, we have
\begin{subequations}
\begin{equation*}\label{eq_CSL_Holder}\tag*{$\text{CS}^{(L)}(\Psi)\text{-}1$}
\lvert \phi(x)-\phi(y)\rvert\le C \left(\frac{d(x,y)}{r}\wedge 1\right)^\delta,
\end{equation*}
\begin{equation*}\label{eq_CSL_energy}\tag*{$\text{CS}^{(L)}(\Psi)\text{-}2$}
\int_{B(x_0,A_{S}r)}|\widetilde{f}|^p\mathrm{d}\Gamma^{(L)}(\phi)\le C_{1}\int_{B(x_0,A_{S}r)}\mathrm{d}\Gamma^{(L)}(f)+\frac{C_{2}}{\Psi(r)}\int_{B(x_0,A_{S}r)}|f|^p\mathrm{d} m,
\end{equation*}
\end{subequations}
where $\widetilde{f}$ is a quasi-continuous modification of $f$, such that $\widetilde{f}$ is uniquely determined $\Gamma^{(L)}(\phi)$-a.e. in $X$, see \cite[Section 8]{Yan25a} for more details. We say that \hypertarget{eq_CSL_cont}{$\text{CS}^{(L)}_{\text{cont}}(\Psi)$} holds if \ref{eq_CSL_energy} holds and the cutoff function $\phi$ is continuous. We say that \hypertarget{eq_CSL_weak}{$\text{CS}^{(L)}_{\text{weak}}(\Psi)$} holds if only \ref{eq_CSL_energy} holds. 

Let $A_1,A_2\in\mathcal{B}(X)$. We define the local capacity between $A_1,A_2$ as
\begin{align*}
&\mathrm{cap}^{(L)}(A_1,A_2)=\inf\left\{\mathcal{E}^{(L)}(\phi):\phi\in\mathcal{F}^{(L)},
\begin{array}{l}
\phi=1\text{ in an open neighborhood of }A_1,\\
\phi=0\text{ in an open neighborhood of }A_2
\end{array}
\right\},
\end{align*}
here we use the convention that $\inf\emptyset=+\infty$.

We say that the local capacity upper bound \ref{eq_ucapL} holds if there exist $C_{cap}>0$, $A_{cap}>1$ such that for any ball $B(x,r)$, we have
\begin{equation*}\label{eq_ucapL}\tag*{$\text{cap}^{{(L)}}(\Psi)_{\le}$}
\mathrm{cap}^{(L)}\left(B(x,r),X\backslash B(x,A_{cap}r)\right)\le C_{cap} \frac{V(x,r)}{\Psi(r)}.
\end{equation*}
Under \ref{eq_VD}, by taking $f\equiv1$ in $B(x,A_Sr)$, it is easy to see that \hyperlink{eq_CSL_weak}{$\text{CS}^{(L)}_{\text{weak}}(\Psi)$} implies \ref{eq_ucapL}.

Assuming \ref{eq_VD}, \ref{eq_PIL}, \ref{eq_ucapL}, by \cite[Proposition 2.1, Remark 2.2]{Yan25a}, we have $\beta^{(1)}_\Psi\ge p$ in (\ref{eq_beta12}), hence
\begin{equation}\label{eq_Psi}
\frac{\Psi(R)}{\Psi(r)}\ge \frac{1}{C_\Psi}\left(\frac{R}{r}\right)^p\text{ for any }r\le R.
\end{equation}

For $\bullet\in\{L,J\}$, let $(\mathcal{E}^{(\bullet)},\mathcal{F}^{(\bullet)})$ be a strongly local regular $p$-energy for $\bullet=L$ or a non-local $p$-form for $\bullet=J$. We say that the cutoff energy inequality \hypertarget{eq_CE_strong}{$\text{CE}^{(\bullet)}_{\text{strong}}(\Psi)$} holds if there exist $\delta\in(0,1)$, $A_E>1$, $C>0$ such that for any ball $B(x_0,r)$, there exists a cutoff function $\phi\in \mathcal{F}^{(\bullet)}$ for $B(x_0,r)\subseteq B(x_0,A_Er)$ such that for any $x,y\in X$ and any $s>0$ we have
\begin{subequations}
\begin{equation*}\label{eq_CE_Holder}\tag*{$\text{CE}^{(\bullet)}(\Psi)\text{-}1$}
\lvert \phi(x)-\phi(y)\rvert\le C \left(\frac{d(x,y)}{r}\wedge 1\right)^\delta,
\end{equation*}
\begin{equation*}\label{eq_CE_energy}\tag*{$\text{CE}^{(\bullet)}(\Psi)\text{-}2$}
\int_{B(x,s)}\mathrm{d}\Gamma^{(\bullet)}_X(\phi)\le C \left(\frac{s}{r}\wedge1\right)^\delta \frac{V(x,s)}{\Psi(s\wedge r)},
\end{equation*}
\end{subequations}
where $\Gamma^{(\bullet)}_X=\Gamma^{(L)}$ for $\bullet=L$ and $\Gamma^{(\bullet)}_X=\Gamma^{(J)}_X$ for $\bullet=J$. We say that \hypertarget{eq_CE_cont}{$\text{CE}^{(\bullet)}_{\text{cont}}(\Psi)$} holds if \ref{eq_CE_energy} holds and the cutoff function $\phi$ is continuous. We say that \hypertarget{eq_CE_weak}{$\text{CE}^{(\bullet)}_{\text{weak}}(\Psi)$} holds if only \ref{eq_CE_energy} holds.

For $\square\in\{\text{strong},\text{cont},\text{weak}\}$, if \hyperlink{eq_CE_strong}{$\text{CE}^{(\bullet)}_{\square}(\Psi)$} holds with some $\delta_1\in(0,1)$, then it also holds for any $\delta_2\in(0,\delta_1]$. Hence, if necessary, we may take $\delta\in(0,1)$ sufficiently small. The precise value of $\delta$ is unimportant; what matters is the existence of such a $\delta$.

For $\beta>0$, we say that $\text{PI}^{(\bullet)}(\beta)$, $\text{cap}^{(\bullet)}(\beta)_\le$, $\text{CS}^{(\bullet)}_{\square}(\beta)$, etc. hold if $\text{PI}^{(\bullet)}(\Psi)$, $\text{cap}^{(\bullet)}(\Psi)_\le$, $\text{CS}^{(\bullet)}_{\square}(\Psi)$, etc. hold with $\Psi:r\mapsto r^{\beta}$.

The first main results of this paper establish equivalences between various CE and CS conditions in both the local and non-local settings.

\begin{theorem}\label{thm_equiv_L}
Assume \ref{eq_VD}. Let $(\mathcal{E}^{(L)},\mathcal{F}^{(L)})$ be a strongly local regular $p$-energy satisfying \ref{eq_PIL}. Then
\begin{center}
\hspace{12pt}\hyperlink{eq_CE_strong}{$\text{CE}^{(L)}_{\text{strong}}(\Psi)$} $\Leftrightarrow$ \hyperlink{eq_CE_cont}{$\text{CE}^{(L)}_{\text{cont}}(\Psi)$} $\Leftrightarrow$ \hyperlink{eq_CE_weak}{$\text{CE}^{(L)}_{\text{weak}}(\Psi)$}\\
\noindent$\Leftrightarrow$\hyperlink{eq_CSL_strong}{$\text{CS}^{(L)}_{\text{strong}}(\Psi)$} $\Leftrightarrow$ \hyperlink{eq_CSL_cont}{$\text{CS}^{(L)}_{\text{cont}}(\Psi)$} $\Leftrightarrow$ \hyperlink{eq_CSL_weak}{$\text{CS}^{(L)}_{\text{weak}}(\Psi)$}.
\end{center}
\end{theorem}

\begin{theorem}\label{thm_equiv_J}
Assume \ref{eq_VD}. Let $(\mathcal{E}^{(J)},\mathcal{F}^{(J)})$ be a non-local regular $p$-energy given by a kernel $K^{(J)}$ satisfying \ref{eq_UPR}, \ref{eq_KJ_tail}, \ref{eq_PIJ}. Then
\begin{center}
\hspace{12pt}\hyperlink{eq_CE_strong}{$\text{CE}^{(J)}_{\text{strong}}(\Upsilon)$} $\Leftrightarrow$ \hyperlink{eq_CE_cont}{$\text{CE}^{(J)}_{\text{cont}}(\Upsilon)$} $\Leftrightarrow$ \hyperlink{eq_CE_weak}{$\text{CE}^{(J)}_{\text{weak}}(\Upsilon)$}\\
\noindent$\Leftrightarrow$\hyperlink{eq_CSJ_strong}{$\text{CS}^{(J)}_{\text{strong}}(\Upsilon)$} $\Leftrightarrow$ \hyperlink{eq_CSJ_cont}{$\text{CS}^{(J)}_{\text{cont}}(\Upsilon)$} $\Leftrightarrow$ \hyperlink{eq_CSJ_weak}{$\text{CS}^{(J)}_{\text{weak}}(\Upsilon)$}.
\end{center}
\end{theorem}

\begin{remark}\label{rmk_equiv}
By the very recent results of \cite{Eri26} and \cite{Mur26}, the condition \ref{eq_ucapL} can now be added to the list of equivalent conditions in Theorem \ref{thm_equiv_L}. Likewise, the condition \ref{eq_ucapJ} can now be added to the equivalences in Theorem \ref{thm_equiv_J}.
\end{remark}

\begin{remark}\label{rmk_strategy}
We will prove the diagrams of implications as in Figure \ref{fig_strategy}.

\begin{figure}[ht]
\centering
\begin{tikzpicture}[
  node distance=1cm,
  every node/.style={font=\normalsize},
  imp/.style={double, double distance=2pt, -{Implies[length=2.5mm]}}
]

\node (A) {$\text{CE}^{(\bullet)}_{\text{strong}}$};
\node (B) [right=of A] {$\text{CE}^{(\bullet)}_{\text{cont}}$};
\node (C) [right=of B] {$\text{CE}^{(\bullet)}_{\text{weak}}$};

\node (D) [below=of A]{$\text{CS}^{(\bullet)}_{\text{strong}}$};
\node (E) [below=of B] {$\text{CS}^{(\bullet)}_{\text{cont}}$};
\node (F) [below=of C] {$\text{CS}^{(\bullet)}_{\text{weak}}$};

\draw[imp] (A) -- (B);
\draw[imp] (B) -- (C);
\draw[imp] (D) -- (E);
\draw[imp] (E) -- (F);

\draw[imp] (A) -- (D);
\draw[imp] (B) -- (E);
\draw[imp] (C) -- (F);

\end{tikzpicture}

\medskip

\centering
\begin{tikzpicture}[
  node distance=0.5cm,
  every node/.style={font=\normalsize},
  imp/.style={double, double distance=2pt, -{Implies[length=2.5mm]}}
]

\node (A) {$\text{CS}^{(\bullet)}_{\text{weak}}$};
\node (B) [right=of A] {$\text{CE}^{(\bullet)}_{\text{strong}}$};

\draw[imp] (A) -- (B);
\end{tikzpicture}
\caption{Strategy of the proof of ``$\text{CE}^{(\bullet)}_{\square}\Leftrightarrow\text{CS}^{(\bullet)}_{\square}$"}\label{fig_strategy}
\end{figure}

The implications ``$\text{CE}^{(\bullet)}_{\text{strong}}\Rightarrow \text{CE}^{(\bullet)}_{\text{cont}}\Rightarrow \text{CE}^{(\bullet)}_{\text{weak}}$" and ``$\text{CS}^{(\bullet)}_{\text{strong}}\Rightarrow \text{CS}^{(\bullet)}_{\text{cont}}\Rightarrow \text{CS}^{(\bullet)}_{\text{weak}}$" follow trivially from the definitions. For $\square\in\{\text{strong},\text{cont},\text{weak}\}$, we will prove ``$\text{CE}^{(\bullet)}_\square\Rightarrow\text{CS}^{(\bullet)}_\square$" with the cutoff functions provided by $\text{CE}^{(\bullet)}_\square$ and following the idea of \cite[Proposition 4.10]{Ant25a}. Moreover, the implication ``$\text{CE}^{(J)}_\square\Rightarrow\text{CS}^{(J)}_\square$" will be proved under the sole assumption that $(\mathcal{E}^{{(J)}},\mathcal{F}^{(J)})$ is a non-local $p$-form. The key novelty of our result lies in the proof of ``$\text{CS}^{(\bullet)}_{\text{weak}}\Rightarrow\text{CE}^{(\bullet)}_{\text{strong}}$". We will reduce the proof to the interior and boundary regularity of solutions to certain PDEs. \ref{eq_UPR} is used only in the proof of this implication.
\end{remark}

\begin{remark}
For $p=2$, the conjunction of \ref{eq_PIL} and \hyperlink{eq_CSL_weak}{$\text{CS}^{(L)}_{\text{weak}}(\Psi)$} is equivalent to the following two-sided heat kernel estimates
\begin{align}
&\frac{C_1}{V\left(x,\Psi^{-1}(t)\right)}\exp\left(-\Lambda\left(C_2d(x,y),t\right)\right)\nonumber\\
&\le p_t(x,y)\le\frac{C_3}{V\left(x,\Psi^{-1}(t)\right)}\exp\left(-\Lambda\left(C_4d(x,y),t\right)\right),\label{eq_HK}\tag*{\text{HK}($\Psi$)}
\end{align}
where
$$\Lambda(R,t)=\sup_{s\in(0,+\infty)}\left(\frac{R}{s}-\frac{t}{\Psi(s)}\right),$$
see \cite[Theorem 1.2]{GHL15}, which yield H\"older continuity of $p_t(x,y)$; see, for example, \cite[COROLLARY 4.2]{BGK12}, where this is obtained via the parabolic Harnack inequality, which in turn intrinsically implies the existence of some $\delta>0$ in \hyperlink{eq_CE_strong}{$\text{CE}_{\text{strong}}^{{(L)}}(\Psi)$}. However, for general $p>1$, due to non-linearity and the absence of a heat kernel, we will need to establish H\"older regularity using a different approach.
\end{remark}

The second main result of this paper shows that a non-local $p$-form satisfying $\text{CS}^{(J)}_{\text{cont}}$ is indeed a \emph{regular} $p$-energy, as stated below; see \cite{CGHL25} for the corresponding result in the linear case $p=2$.

\begin{proposition}\label{prop_regular}
Assume \ref{eq_VD}. Let $(\mathcal{E}^{(J)}, \mathcal{F}^{(J)})$ be a non-local $p$-form given by a kernel $K^{(J)}$ satisfying \ref{eq_KJ_tail}, \ref{eq_PIJ}, \hyperlink{eq_CSJ_cont}{$\text{CS}^{(J)}_{\text{cont}}(\Upsilon)$}. Then $(\mathcal{E}^{(J)}, \mathcal{F}^{(J)})$ is a regular $p$-energy.
\end{proposition}

\begin{remark}
In a forthcoming paper, we will relax \hyperlink{eq_CSJ_cont}{$\text{CS}^{(J)}_{\text{cont}}(\Upsilon)$} to a version of \ref{eq_ucapJ}, where the cutoff functions are also required to be \emph{continuous}.
\end{remark}

By Remark \ref{rmk_strategy}, we have the following direct consequence for non-local $p$-forms from Theorem \ref{thm_equiv_J} and Proposition \ref{prop_regular}.

\begin{corollary}\label{cor_pform}
Assume \ref{eq_VD}. Let $(\mathcal{E}^{(J)},\mathcal{F}^{(J)})$ be a non-local $p$-form given by a kernel $K^{(J)}$ satisfying \ref{eq_UPR}, \ref{eq_KJ_tail}, \ref{eq_PIJ}. Then
\begin{center}
\hyperlink{eq_CE_strong}{$\text{CE}^{(J)}_{\text{strong}}(\Upsilon)$} $\Leftrightarrow$ \hyperlink{eq_CE_cont}{$\text{CE}^{(J)}_{\text{cont}}(\Upsilon)$} $\Leftrightarrow$ \hyperlink{eq_CSJ_strong}{$\text{CS}^{(J)}_{\text{strong}}(\Upsilon)$} $\Leftrightarrow$ \hyperlink{eq_CSJ_cont}{$\text{CS}^{(J)}_{\text{cont}}(\Upsilon)$}.
\end{center}
Moreover, under any of these equivalent conditions, $(\mathcal{E}^{(J)}, \mathcal{F}^{(J)})$ is a regular $p$-energy.
\end{corollary}

The third and central main result of this paper is the following ``subordination" result.

\begin{theorem}\label{thm_sub}
Let $\Psi,\Upsilon$ be two doubling functions satisfying (\ref{eq_beta12}). Assume \ref{eq_VD}. Let $(\mathcal{E}^{(L)},\mathcal{F}^{(L)})$ be a strongly local regular $p$-energy satisfying \ref{eq_PIL}, \hyperlink{eq_CE_strong}{$\text{CE}^{(L)}_{\text{strong}}(\Psi)$}. Let $(\mathcal{E}^{(J,\Upsilon)}, \mathcal{F}^{(J,\Upsilon)})$ be a non-local $p$-form given by a kernel $K^{(J)}$ satisfying \ref{eq_KJ}. Then the followings are equivalent.
\begin{enumerate}[label=(\alph*),ref=(\alph*)]
\item\label{item_sub1} $(\mathcal{E}^{(J,\Upsilon)},\mathcal{F}^{(J,\Upsilon)})$ is a regular $p$-energy.
\item\label{item_sub2} $\mathcal{F}^{(J,\Upsilon)}$ contains a non-constant function.
\item\label{item_sub3} The pair $(\Upsilon,\Psi)$ satisfies the strictly upper growth condition at small scales \ref{eq_SUG0}, that is,
\begin{equation*}\label{eq_SUG0}\tag*{$\text{SUG}_0(\Upsilon,\Psi)$}
\sum_{n=0}^{+\infty}\frac{\Psi(\frac{1}{2^n})}{\Upsilon(\frac{1}{2^n})}<+\infty.
\end{equation*}
\end{enumerate}
Moreover, in this case, define $\Xi:[0,+\infty)\to[0,+\infty)$ by
\begin{equation}\label{eq_Xi}
\Xi(r)=
\begin{cases}
0, & r=0,\\
\frac{\Psi(r)}{\int_0^r \frac{\mathrm{d}\Psi(t)}{\Upsilon(t)}},&r>0.
\end{cases}
\end{equation}
Then $\Xi$ is well-defined, is a doubling function, and satisfies (\ref{eq_beta12}). Furthermore, $(\mathcal{E}^{(J,\Upsilon)}, \mathcal{F}^{(J,\Upsilon)})$ is a regular $p$-energy satisfying \hyperref[eq_PIJ]{$\text{PI}^{(J)}(\Xi)$}, \hyperlink{eq_CE_strong}{$\text{CE}^{(J)}_{\text{strong}}(\Xi)$}.
\end{theorem}

\begin{remark}
\begin{enumerate}[label=(\arabic*)]
\item The implication ``\ref{item_sub1}$\Rightarrow$\ref{item_sub2}" is trivial. We prove ``\ref{item_sub2}$\Rightarrow$\ref{item_sub3}" by contradiction. It remains to prove the conclusions under the assumption \ref{item_sub3}. More precisely, assuming that \ref{eq_SUG0} holds, we prove that $(\mathcal{E}^{(J,\Upsilon)}, \mathcal{F}^{(J,\Upsilon)})$ is a $p$-form satisfying \hyperref[eq_PIJ]{$\text{PI}^{(J)}(\Xi)$}, \hyperlink{eq_CE_strong}{$\text{CE}^{(J)}_{\text{strong}}(\Xi)$}, and also \ref{eq_UPR}, \hyperref[eq_KJ_tail]{$\text{T}^{(J)}(\Xi)$}, then \ref{item_sub1} follows from Corollary \ref{cor_pform}.
\item By Theorem \ref{thm_equiv_L}, the assumption \hyperlink{eq_CE_strong}{$\text{CE}^{(L)}_{\text{strong}}(\Psi)$} can be replaced by \hyperlink{eq_CSL_weak}{$\text{CS}^{(L)}_{\text{weak}}(\Psi)$}, or by any other equivalent condition stated there and in Remark \ref{rmk_equiv}.
\end{enumerate}
\end{remark}

As direct consequences, we obtain the following results, whose proofs are straightforward computations of $\Xi$ from $\Psi$ and $\Upsilon$ as in (\ref{eq_Xi}).

\begin{corollary}\label{cor_light}
Assume \ref{eq_VD}. Let $d_{w,p}>0$ and let $(\mathcal{E}^{(L)},\mathcal{F}^{(L)})$ be a strongly local regular $p$-energy satisfying \hyperref[eq_PIL]{$\text{PI}^{(L)}(d_{w,p})$}, \hyperlink{eq_CE_strong}{$\text{CE}^{(L)}_{\text{strong}}(d_{w,p})$}. For any $\beta_1\in(0,d_{w,p})$ and $\beta_2\in(0,+\infty)$, let
$$\Upsilon(r)=r^{\beta_1}1_{(0,1)}+r^{\beta_2}1_{[1,+\infty)}.$$
Let $(\mathcal{E}^{(J,\Upsilon)}, \mathcal{F}^{(J,\Upsilon)})$ be a non-local $p$-form given by a kernel $K^{(J)}$ satisfying \ref{eq_KJ}. Then $(\mathcal{E}^{(J,\Upsilon)}, \mathcal{F}^{(J,\Upsilon)})$ is a regular $p$-energy satisfying \hyperref[eq_PIJ]{$\text{PI}^{(J)}(\Xi)$}, \hyperlink{eq_CE_strong}{$\text{CE}^{(J)}_{\text{strong}}(\Xi)$}, where for $r\in(0,1)$, $\Xi(r)=r^{\beta_1}$, while for $r\ge1$,
$$
\Xi(r)=
\begin{cases}
r^{\beta_2},&\text{if }\beta_2<d_{w,p},\\
\frac{r^{d_{w,p}}}{\log(e-1+ r)},&\text{if }\beta_2=d_{w,p},\\
r^{d_{w,p}},&\text{if }\beta_2>d_{w,p}.
\end{cases}
$$
\end{corollary}

\begin{corollary}[BBM-type Poincar\'e inequality]\label{cor_BBM}
Assume \ref{eq_VD}. Let $d_{w,p}>0$ and let $(\mathcal{E}^{(L)},\mathcal{F}^{(L)})$ be a strongly local regular $p$-energy satisfying \hyperref[eq_PIL]{$\text{PI}^{(L)}(d_{w,p})$}, \hyperlink{eq_CE_strong}{$\text{CE}^{(L)}_{\text{strong}}(d_{w,p})$}. For any $\beta\in(0,d_{w,p})$, let $(\mathcal{E}^{(J,\beta)}, \mathcal{F}^{(J,\beta)})$ be a non-local $p$-form given by a kernel $K^{(J)}$ satisfying \hyperref[eq_KJ]{$\text{K}^{(J)}(\beta)$}. Then $(\mathcal{E}^{(J,\beta)}, \mathcal{F}^{(J,\beta)})$ is a regular $p$-energy satisfying the following BBM-type Poincar\'e inequality: there exist $C>0$, $A>1$ \emph{independent} of $\beta\in(0,d_{w,p})$ such that for any ball $B(x,r)$, for any $f\in \widehat{\mathcal{F}}^{(J,\beta)}$,
\begin{align*}
&\int_{B(x,r)}\lvert f-f_{B(x,r)}\rvert^p \mathrm{d}m\\
&\le C\left(d_{w,p}-\beta\right)r^{\beta}\int_{B(x,Ar)}\int_{B(x,Ar)}\frac{\lvert f(y)-f(z)\rvert^p}{V(y,d(y,z))d(y,z)^\beta}m(\mathrm{d}y)m(\mathrm{d}z).
\end{align*}
\end{corollary}

This paper is organized as follows. In Section \ref{sec_pre}, we collect some preliminary results for later use. In Section \ref{sec_sub}, we give the proof of Theorem \ref{thm_sub}. In Section \ref{sec_CSJ_self}, we prove the self-improvement property of {$\text{CS}^{(J)}_{\text{weak}}$} and {$\text{CS}^{(J)}_{\text{cont}}$}. In Section \ref{sec_regular}, we give the proof of Proposition \ref{prop_regular}. In Section \ref{sec_CE2CS}, we give the proof of ``$\text{CE}^{(\bullet)}_\square\Rightarrow\text{CS}^{(\bullet)}_\square$" in Theorems \ref{thm_equiv_L} and \ref{thm_equiv_J}. In Section \ref{sec_Moser_L}, we give the proof of ``\hyperlink{eq_CSL_weak}{$\text{CS}^{(L)}_{\text{weak}}$}$\Rightarrow$\hyperlink{eq_CE_strong}{$\text{CE}^{(L)}_{\text{strong}}$}" in Theorem \ref{thm_equiv_L}. In Section \ref{sec_Moser_J}, we give two oscillation inequalities. In Section \ref{sec_CS2CE_J}, we give the proof of ``\hyperlink{eq_CSJ_weak}{$\text{CS}^{(J)}_{\text{weak}}$}$\Rightarrow$\hyperlink{eq_CE_strong}{$\text{CE}^{(J)}_{\text{strong}}$}" in Theorem \ref{thm_equiv_J}. In Section \ref{sec_ele}, we give some elementary results used in the previous sections.

\section{Preliminary}\label{sec_pre}

Let $\Omega\subseteq X$ be an open subset and $x_0\in \partial\Omega$. Let $c,r_0>0$. We say that $X\backslash \Omega$ has a $(c,r_0)$-corkscrew at $x_0$ if for any $r\in(0,r_0)$, $B(x_0,r)\backslash \Omega$ contains a ball with radius $cr$. Obviously, if a $(c,r_0)$-corkscrew exists, then for any $A\ge1$, a $(\frac{1}{A}c,Ar_0)$-corkscrew and a $(\frac{1}{A}c,r_0)$-corkscrew also exist.

\begin{lemma}[{\cite[Lemma 14.2]{BB11}, \cite[Theorem 1.1]{Raj21}}]\label{lem_corkscrew}
Assume \ref{eq_CC}. Then for any $\delta_1,\delta_2>0$ with $\delta_1<\delta_2$, there exists $c>0$ depending only on $C_{cc},\delta_1,\delta_2$ such that for any ball $B=B(x_0,r)$, there exists an open set $\Omega$ with $\delta_1B\subseteq \Omega\subseteq \delta_2B$ such that, for any $x\in \partial \Omega$, we have $X\backslash \Omega$ has a $(c,r)$-corkscrew at $x$, and hence a $(\frac{1}{A}c ,Ar)$-corkscrew at $x$ for any $A\ge1$.
\end{lemma}

We need the following results concerning BMO spaces. Let $U$ be an open set and $u$ a locally integrable function in $U$. We define the semi-norm $\lVert {u}\rVert_{\mathrm{BMO}(U)}$ as
$$\lVert {u}\rVert_{\mathrm{BMO}(U)}=\sup \left\{\dashint_B|u-u_B|\mathrm{d} m:B\subseteq U\text{ is a ball}\right\}.$$
Let $\mathrm{BMO}(U)$ be the family of all locally integrable functions $u$ in $U$ with $\lVert {u}\rVert_{\mathrm{BMO}(U)}<+\infty$.

\begin{lemma}[John-Nirenberg inequality; {\cite[THEOREM 5.2]{ABKY11}}]\label{lem_JohnNirenberg}
Assume \ref{eq_VD}. Then there exist $C_1,C_2>0$ such that for any open set $U$ and any $u\in \mathrm{BMO}(U)$, for any ball $B$ with $12B\subseteq U$ and any $\lambda>0$, we have
$$m \left(B\cap\{|u-u_B|>\lambda\}\right)\le C_1m(B)\exp \left(-C_2 \frac{\lambda}{\lVert {u}\rVert_{\mathrm{BMO}(U)}}\right).$$
\end{lemma}

\begin{lemma}[{\cite[COROLLARY 5.6]{BM95AMPA}}]\label{lem_crossover}
Assume \ref{eq_VD}. Then for any ball $B(x_0,R)$ and any $u\in\mathrm{BMO}(B(x_0,R))$, for any ball $B$ with $12B\subseteq B(x_0,R)$, for any $b\ge \lVert {u}\rVert_{\mathrm{BMO}(B(x_0,R))}$, we have
$$\left\{\dashint_{B}\exp \left(\frac{C_2}{2b}u\right)\mathrm{d} m\right\}\left\{\dashint_{B}\exp \left(-\frac{C_2}{2b}u\right)\mathrm{d} m\right\}\le (C_1+1)^2,$$
where $C_1,C_2$ are the positive constants appearing in Lemma \ref{lem_JohnNirenberg}.
\end{lemma}

For $\bullet\in\{L,J\}$, let $(\mathcal{E}^{(\bullet)}, \mathcal{F}^{(\bullet)})$ be a \emph{regular} $p$-energy. Following \cite[Section 8]{Yan25a} and \cite[Section 6]{Yan25c}, we can define the capacity $\mathrm{cap}^{(\bullet)}_1$, and the notion of $\mathcal{E}^{(\bullet)}_1$-quasi-continuity, and obtain the corresponding results. Although the arguments in \cite{Yan25a} were presented in the local setting, the same line of reasoning also applies to the non-local setting. For simplicity, when no ambiguity arises, we write quasi-continuous and q.e. to mean $\mathcal{E}^{(\bullet)}_1$-quasi-continuous and $\mathcal{E}^{(\bullet)}_1$-quasi-everywhere (abbreviated $\mathcal{E}^{(\bullet)}_1$-q.e.), respectively.

\begin{lemma}[{\cite[Proposition 8.5]{Yan25a}}]
Each $u\in \mathcal{F}^{(\bullet)}$ admits a quasi-continuous modification $\widetilde{u}$, that is, $\widetilde{u}$ is quasi-continuous and $u=\widetilde{u}$ $m$-a.e. in $X$.
\end{lemma}

Let $\Omega\subseteq X$ be an open subset. Define
$$\mathcal{F}^{(\bullet)}(\Omega)=\text{the }\mathcal{E}^{(\bullet)}_1\text{-closure of }\mathcal{F}^{(\bullet)}\cap C_c(\Omega).$$
We say that $u\in \mathcal{F}^{(\bullet)}$ is harmonic in $\Omega$ if $\mathcal{E}^{(\bullet)}(u;\varphi)=0$ for any $\varphi\in \mathcal{F}^{(\bullet)}\cap C_c(\Omega)$, denoted by $-\Delta_p^{(\bullet)}u=0$ in $\Omega$. We say that $u\in \mathcal{F}^{(\bullet)}$ is superharmonic in $\Omega$ (resp. subharmonic in $\Omega$) if $\mathcal{E}^{(\bullet)}(u;\varphi)\ge0$ (resp. $\mathcal{E}^{(\bullet)}(u;\varphi)\le0$) for any non-negative $\varphi\in \mathcal{F}^{(\bullet)}\cap C_c(\Omega)$, denoted by $-\Delta_p^{(\bullet)}u\ge0$ in $\Omega$ (resp. $-\Delta_p^{(\bullet)}u\le0$ in $\Omega$). More generally, let $f\in L^\infty(\Omega;m)$, we write $-\Delta_p^{(\bullet)}u=f$ (resp. $-\Delta_p^{(\bullet)}u\ge f$) if $\mathcal{E}^{(\bullet)}(u;\varphi)=\int_X f\varphi \mathrm{d}m$ (resp. $\mathcal{E}^{(\bullet)}(u;v)\ge\int_X f\varphi \mathrm{d}m$) for any non-negative $\varphi\in \mathcal{F}^{(\bullet)}\cap C_c(\Omega)$. If $\Omega$ is additionally assumed to be bounded, then by the following result, the above (in)equalities also hold for any non-negative $v\in \mathcal{F}^{(\bullet)}(\Omega)$.

\begin{lemma}[{\cite[Lemma 6.2]{Yan25c}}]\label{lem_FOmega}
For any bounded open subset $\Omega\subseteq X$, we have
\begin{align*}
\mathcal{F}^{(\bullet)}(\Omega)=\left\{u\in \mathcal{F}^{(\bullet)}:\widetilde{u}=0\text{ q.e. on }X\backslash \Omega\right\}.
\end{align*}
\end{lemma}

We establish existence and uniqueness results for boundary value problems associated with equations involving $\Delta_p^{(\bullet)}$, as follows. The proof follows the same direct method in the calculus of variations as in \cite[Theorem 8]{LL17}, which is stated in the setting of $\mathbb{R}^n$. Owing to the $p$-Clarkson's inequality for $p>1$, the same argument applies in our setting, and the proof is therefore omitted.

\begin{proposition}\label{prop_exist}
Let $\Omega\subseteq X$ be a bounded open subset and let $\lambda\ge0$. Assume \ref{eq_VD}, $\text{PI}^{(\bullet)}(\Psi)$ for some $\Psi$ in the case $\lambda=0$. Then there exists a unique $u\in \mathcal{F}^{(\bullet)}(\Omega)$ such that $-\Delta_p^{(\bullet)}u+\lambda \lvert u\rvert^{p-2}u=1$ in $\Omega$, that is,
$$\mathcal{E}^{(\bullet)}(u;\varphi)+\lambda \int_X \lvert u\rvert^{p-2}u\varphi \mathrm{d}m=\int_X\varphi \mathrm{d}m,$$
for any $\varphi\in \mathcal{F}^{(\bullet)}(\Omega)$.
\end{proposition}

\begin{remark}
In the case $\lambda=0$, \ref{eq_VD}, $\text{PI}^{(\bullet)}(\Psi)$ ensure that $(\mathcal{F}^{(\bullet)}(\Omega),(\mathcal{E}^{(\bullet)})^{1/p})$ is a Banach space. In contrast, when $\lambda>0$, $(\mathcal{F}^{(\bullet)}(\Omega),(\mathcal{E}^{(\bullet)}_\lambda)^{1/p})$ is automatically a Banach space by definition.
\end{remark}

We have the following local and non-local comparison principles.

\begin{proposition}[{Comparison principle, local version; \cite[Proposition 3.5]{Yan25d}}]\label{prop_comparisonL}
Assume \ref{eq_VD}. Let $(\mathcal{E}^{(L)},\mathcal{F}^{(L)})$ be a strongly local regular $p$-energy satisfying \ref{eq_PIL}. Let $\lambda\ge0$. Let $\Omega\subseteq X$ be a bounded open subset and let $u,v\in\mathcal{F}^{(L)}$ satisfy
$$-\Delta^{(L)}_pu+\lambda|u|^{p-2}u\ge-\Delta^{(L)}_pv+\lambda|v|^{p-2}v\text{ in }\Omega,$$
that is,
$$\mathcal{E}^{(L)}(u;\varphi)+\lambda\int_X|u|^{p-2}u\varphi\mathrm{d} m\ge\mathcal{E}^{(L)}(v;\varphi)+\lambda\int_X|v|^{p-2}v\varphi\mathrm{d} m$$
for any non-negative $\varphi\in\mathcal{F}^{(L)}(\Omega)$, and $(v-u)_+\in \mathcal{F}^{(L)}(\Omega)$, or
$$\widetilde{u}\ge \widetilde{v}\text{ q.e. on }W\backslash \Omega\text{ for some open set }W\supseteq \overline{\Omega}\supseteq \Omega.$$
Then $u\ge v$ in $\Omega$.
\end{proposition}

\begin{proposition}[{Comparison principle, non-local version}]\label{prop_comparisonJ}
Let $(\mathcal{E}^{(J)}, \mathcal{F}^{(J)})$ be a non-local regular $p$-energy associated with $\Upsilon$. Let $\lambda\ge0$. Let $\Omega\subseteq X$ be a bounded open subset and let $u,v\in \mathcal{F}^{(J)}$ satisfy
$$-\Delta_p^{(J)}u+\lambda \lvert u\rvert^{p-2}u\ge -\Delta_p^{(J)}v+\lambda \lvert v\rvert^{p-2}v\text{ in }\Omega,$$
that is,
$$\mathcal{E}^{(J)}(u;\varphi)+\lambda\int_X \lvert u\rvert^{p-2} u\varphi \mathrm{d}m\ge\mathcal{E}^{(J)}(v;\varphi)+\lambda\int_X \lvert v\rvert^{p-2} v\varphi \mathrm{d}m$$
for any non-negative $\varphi\in \mathcal{F}^{(J)}(\Omega)$, and
$$\widetilde{u}\ge \widetilde{v}\text{ q.e. on }X\backslash \Omega.$$
Then $u\ge v$ in $\Omega$, that is, $u\ge v$ in $X$.
\end{proposition}

\begin{remark}
Proposition \ref{prop_comparisonJ} was proved in a different setting in \cite[Lemma 9]{LL14}. The same argument applies in our setting, and therefore the proof is omitted.
\end{remark}

We have the following local and non-local Sobolev inequalities.

\begin{proposition}[{Sobolev inequality, local version; \cite[Proposition 4.1]{Yan25d}}]\label{prop_SobolevL}
Assume \ref{eq_VD}. Let $(\mathcal{E}^{(L)},\mathcal{F}^{(L)})$ be a strongly local regular $p$-energy satisfying \ref{eq_PIL}. Then there exist $\kappa>1$, $C>0$ such that for any ball $B(x_0,R)$, for any $f\in\mathcal{F}^{(L)}(B(x_0,R))$, we have
$$\left(\int_{B(x_0,R)}|f|^{p\kappa}\mathrm{d} m\right)^{\frac{1}{\kappa}}\le C\frac{\Psi(R)}{V(x_0,R)^{\frac{\kappa-1}{\kappa}}}\mathcal{E}^{(L)}(f).$$
Indeed, we can take $\kappa=\frac{\nu}{\nu-\beta_\Psi^{(1)}}$, where $\nu=\max\{\beta_\Psi^{(1)}+1,\log_2C_{VD}\}$.
\end{proposition}

\begin{proposition}[{Sobolev inequality, non-local version; \cite[Proposition 3.5, Definition 3.6, Theorem 3.8]{CY25}}]\label{prop_SobolevJ}
Assume \ref{eq_VD}. Let $(\mathcal{E}^{(J)}, \mathcal{F}^{(J)})$ be a non-local regular $p$-energy given by a kernel $K^{(J)}$ satisfying \ref{eq_PIJ}. Then there exist $\kappa>1$, $C>0$ such that for any ball $B(x_0,R)$, for any $f\in\mathcal{F}^{(J)}(B(x_0,R))$, we have
$$\left(\int_{B(x_0,R)}|f|^{p\kappa}\mathrm{d} m\right)^{\frac{1}{\kappa}}\le C\frac{\Upsilon(R)}{V(x_0,R)^{\frac{\kappa-1}{\kappa}}}\mathcal{E}^{(J)}(f).$$
Indeed, we can take $\kappa=\frac{\nu}{\nu-\beta_\Upsilon^{(1)}}$, where $\nu=\max\{\beta_\Upsilon^{(1)}+1,\log_2C_{VD}\}$.
\end{proposition}

\section{Proof of Theorem \ref{thm_sub}}\label{sec_sub}

In this section, we give the proof of Theorem \ref{thm_sub}. Recall that $(X,d,m)$ is a metric measure space. For any doubling function $\Psi$ satisfying (\ref{eq_beta12}), any open subset $U\subseteq X$, any $u\in L^p(X;m)$, and any $r>0$, we define
$$E_{\Psi,U}(u,r)=\frac{1}{\Psi(r)}\int_U \left(\frac{1}{V(x,r)}\int_{B(x,r)}\lvert u(x)-u(y)\rvert^p m(\mathrm{d}y)\right)m(\mathrm{d}x),$$
and
$$U(r)=\bigcup_{x\in U}B(x,r).$$

The following characterization of ``$\iint\ldots$" in terms of ``$\sum E\ldots$" plays an important role in our proof.

\begin{lemma}\label{lem_nonlocal_E}
Assume \ref{eq_VD}. Let $\Upsilon$ be a doubling function satisfying (\ref{eq_beta12}) with constants $C_\Upsilon,\beta^{(1)}_\Upsilon,\beta^{(2)}_\Upsilon$. Then there exists $C>0$ such that for any $u\in L^p(X;m)$, for any non-empty bounded open subset $U\subseteq X$, for any $D\ge \mathrm{diam}(U)$ and any $R>0$, we have
\begin{equation}\label{eq_nonlocal_E1}
\int_U\int_U \frac{\lvert u(x)-u(y)\rvert^p}{V(x,d(x,y))\Upsilon(d(x,y))}m(\mathrm{d}x)m(\mathrm{d}y)\le C\sum_{n=0}^{+\infty}E_{\Upsilon,U}(u,\frac{1}{2^n}D ),
\end{equation}
\begin{equation}\label{eq_nonlocal_E2}
\sum_{n=0}^{+\infty}E_{\Upsilon,U}(u,\frac{1}{2^n}R)\le C\int_{U(R)}\int_{U(R)} \frac{\lvert u(x)-u(y)\rvert^p}{V(x,d(x,y))\Upsilon(d(x,y))}m(\mathrm{d}x)m(\mathrm{d}y),
\end{equation}
and
\begin{equation}\label{eq_nonlocal_E3}
\frac{1}{C}\sum_{n\in \mathbb{Z}}E_{\Upsilon,X}(u,\frac{1}{2^n})\le\int_X\int_X \frac{\lvert u(x)-u(y)\rvert^p}{V(x,d(x,y))\Upsilon(d(x,y))}m(\mathrm{d}x)m(\mathrm{d}y)\le {C}\sum_{n\in \mathbb{Z}}E_{\Upsilon,X}(u,\frac{1}{2^n}).
\end{equation}
\end{lemma}

\begin{proof}
The proof follows a standard annulus decomposition argument; see also \cite[Theorem 5.2]{GKS10}. Firstly, we have
\begin{align*}
&\int_U\int_U \frac{\lvert u(x)-u(y)\rvert^p}{V(x,d(x,y))\Upsilon(d(x,y))}m(\mathrm{d}x)m(\mathrm{d}y)\\
&\le\int_U \left(\sum_{n=0}^{+\infty}\int_{B(x,\frac{1}{2^n}D)\backslash B(x,\frac{1}{2^{n+1}}D)}\frac{\lvert u(x)-u(y)\rvert^p}{V(x,d(x,y))\Upsilon(d(x,y))}m(\mathrm{d}y)\right)m(\mathrm{d}x)\\
&\le\int_U \left(\sum_{n=0}^{+\infty}\frac{1}{V(x,\frac{1}{2^{n+1}}D)\Upsilon(\frac{1}{2^{n+1}}D)}\int_{B(x,\frac{1}{2^n}D)}\lvert u(x)-u(y)\rvert^p m(\mathrm{d}y)\right)m(\mathrm{d}x)\\
&\le C_1\sum_{n=0}^{+\infty}E_{\Upsilon,U}(u,\frac{1}{2^n}D),
\end{align*}
where $C_1=C_{\Upsilon}C_{VD}$, which gives (\ref{eq_nonlocal_E1}). Secondly, noting that
\begin{align*}
&\sum_{n=0}^{+\infty}E_{\Upsilon,U}(u,\frac{1}{2^n}R)\\
&=\sum_{n=0}^{+\infty}\frac{1}{\Upsilon(\frac{1}{2^n}R)}\int_U \left(\frac{1}{V(x,\frac{1}{2^n}R)}\sum_{k=n}^{+\infty}\int_{B(x,\frac{1}{2^k}R)\backslash B(x,\frac{1}{2^{k+1}}R)}\lvert u(x)-u(y)\rvert^p m(\mathrm{d}y)\right)m(\mathrm{d}x)\\
&=\sum_{k=0}^{+\infty}\int_U \left(\sum_{n=0}^{k}\frac{1}{V(x,\frac{1}{2^n}R)\Upsilon(\frac{1}{2^n}R)}\right)\left(\int_{B(x,\frac{1}{2^k}R)\backslash B(x,\frac{1}{2^{k+1}}R)}\lvert u(x)-u(y)\rvert^p m(\mathrm{d}y)\right)m(\mathrm{d}x),
\end{align*}
where by (\ref{eq_beta12}), we have
\begin{align*}
&\sum_{n=0}^{k}\frac{1}{V(x,\frac{1}{2^n}R)\Upsilon(\frac{1}{2^n}R)}\le \frac{1}{V(x,\frac{1}{2^k}R)}\sum_{n=0}^{k}\frac{1}{\Upsilon(\frac{1}{2^n}R)}\\
&\le C_\Upsilon\frac{1}{V(x,\frac{1}{2^k}R)\Upsilon(\frac{1}{2^k}R)}\sum_{n=0}^{k}2^{\beta^{(1)}_\Upsilon(n-k)}\le \frac{2^{\beta^{(1)}_\Upsilon}C_\Upsilon}{2^{\beta^{(1)}_\Upsilon}-1}\frac{1}{V(x,\frac{1}{2^k}R)\Upsilon(\frac{1}{2^k}R)}.
\end{align*}
Let $C_2=\frac{2^{\beta^{(1)}_\Upsilon}C_\Upsilon}{2^{\beta^{(1)}_\Upsilon}-1}$, then
\begin{align*}
&\sum_{n=0}^{+\infty}E_{\Upsilon,U}(u,\frac{1}{2^n}R)\\
&\le C_2\sum_{k=0}^{+\infty}\int_U\left(\frac{1}{V(x,\frac{1}{2^k}R)\Upsilon(\frac{1}{2^k}R)}\int_{B(x,\frac{1}{2^k}R)\backslash B(x,\frac{1}{2^{k+1}}R)}\lvert u(x)-u(y)\rvert^p m(\mathrm{d}y)\right)m(\mathrm{d}x)\\
&\le C_2\int_U\left(\sum_{k=0}^{+\infty}\int_{B(x,\frac{1}{2^k}R)\backslash B(x,\frac{1}{2^{k+1}}R)}\frac{\lvert u(x)-u(y)\rvert^p}{V(x,d(x,y))\Upsilon(d(x,y))} m(\mathrm{d}y)\right)m(\mathrm{d}x)\\
&=C_2\int_U\left(\int_{B(x,R)}\frac{\lvert u(x)-u(y)\rvert^p}{V(x,d(x,y))\Upsilon(d(x,y))} m(\mathrm{d}y)\right)m(\mathrm{d}x)\\
&\le C_2\int_{U(R)}\int_{U(R)}\frac{\lvert u(x)-u(y)\rvert^p}{V(x,d(x,y))\Upsilon(d(x,y))} m(\mathrm{d}x)m(\mathrm{d}y),
\end{align*}
which gives (\ref{eq_nonlocal_E2}). Finally, by replacing $U$ with $X$, $\sum_{n=0}^{+\infty}$ with $\sum_{n\in \mathbb{Z}}$, $\sum_{k=0}^{+\infty}$ with $\sum_{k\in \mathbb{Z}}$, $\sum_{n=0}^k$ with $\sum_{\mathbb{Z}\ni n\le k}$, and by letting $D=1$, $R=1$ in the above estimates, we have (\ref{eq_nonlocal_E3}).
\end{proof}

For a strongly local regular $p$-energy $(\mathcal{E}^{(L)},\mathcal{F}^{(L)})$, we have the relation between $\Gamma^{(L)}$ and $E$ as follows.

\begin{lemma}[{\cite[Proposition 3.5, Equation (1.18)]{Shi24a}}]\label{lem_local_E}
Assume \ref{eq_VD}. Let $(\mathcal{E}^{(L)},\mathcal{F}^{(L)})$ be a strongly local regular $p$-energy satisfying \ref{eq_PIL}. Then there exists $C>0$ such that for any open subset $U\subseteq X$, for any $u\in \mathcal{F}^{(L)}$, for any $r>0$, we have
$$E_{\Psi,U}(u,r)\le C\Gamma^{(L)}(u)\left(U(2A_{PI}r)\right),$$
where $A_{PI}\ge1$ is the constant in \ref{eq_PIL}.
\end{lemma}

We have the relation between the domains of $p$-forms as follows.

\begin{lemma}\label{lem_domain}
Assume \ref{eq_VD}. Let $(\mathcal{E}^{(L)},\mathcal{F}^{(L)})$ be a strongly local regular $p$-energy satisfying \ref{eq_PIL}. Let $\Upsilon$ be a doubling function satisfying (\ref{eq_beta12}) and \ref{eq_SUG0}. Let $(\mathcal{E}^{(J,\Upsilon)},\mathcal{F}^{(J,\Upsilon)})$ be a non-local $p$-form given by a kernel $K^{(J)}$ satisfying \ref{eq_KJ}. Then $\mathcal{F}^{(L)}\subseteq \mathcal{F}^{(J,\Upsilon)}$.
\end{lemma}

\begin{proof}
For any $u\in L^p(X;m)$, by (\ref{eq_nonlocal_E3}), we have
\begin{align*}
\mathcal{E}^{(J,\Upsilon)}(u)\lesssim\sum_{n\in \mathbb{Z}}E_{\Upsilon,X}(u,\frac{1}{2^n})=\left(\sum_{n\ge0}+\sum_{n<0}\right)E_{\Upsilon,X}(u,\frac{1}{2^n})=I_1+I_2.
\end{align*}
For any $r>0$, by \ref{eq_VD}, we have
$$E_{\Upsilon,X}(u,r)\le \frac{2^{p-1}(C_{VD}+1)}{\Upsilon(r)}\lVert u\rVert_{L^p(X;m)}^p,$$
see also \cite[Lemma 3.1]{Bau24}, hence
\begin{align*}
I_2\lesssim \lVert u\rVert_{L^p(X;m)}^p\sum_{n<0}\frac{1}{\Upsilon(\frac{1}{2^n})}\lesssim \frac{1}{\Upsilon(1)}\lVert u\rVert_{L^p(X;m)}^p\sum_{n<0}2^{\beta^{(1)}_\Upsilon n}<+\infty.
\end{align*}
For any $u\in \mathcal{F}^{(L)}$, for any $n\ge0$, by Lemma \ref{lem_local_E}, we have
$$E_{\Psi,X}(u,\frac{1}{2^n})\lesssim \Gamma^{(L)}(u)(X)=\mathcal{E}^{(L)}(u),$$
hence by \ref{eq_SUG0}, we have
$$I_1=\sum_{n\ge0}\frac{\Psi(\frac{1}{2^n})}{\Upsilon(\frac{1}{2^n})}E_{\Psi,X}(u,\frac{1}{2^n})\lesssim \mathcal{E}^{(L)}(u)\sum_{n\ge0}\frac{\Psi(\frac{1}{2^n})}{\Upsilon(\frac{1}{2^n})}<+\infty,$$
which gives $u\in \mathcal{F}^{(J,\Upsilon)}$.
\end{proof}

We give the proof of \hyperlink{eq_CE_strong}{$\text{CE}^{(J)}_{\text{strong}}(\Xi)$} under \ref{eq_SUG0} in Theorem \ref{thm_sub} as follows.

\begin{proof}[Proof of \hyperlink{eq_CE_strong}{$\text{CE}^{(J)}_{\text{strong}}(\Xi)$} under \ref{eq_SUG0} in Theorem \ref{thm_sub}]
Let $\delta\in(0,1)$, $A=A_E>1$ be the constants in {\hyperlink{eq_CE_strong}{$\text{CE}^{(L)}_{\text{strong}}(\Psi)$}}. By replacing $\delta$ with $\min\{\delta,\frac{\beta^{(1)}_\Upsilon}{2p}\}$, we may assume that $\delta\le\frac{\beta^{(1)}_\Upsilon}{2p}<\frac{\beta^{(1)}_\Upsilon}{p}$. For any ball $B(x_0,r)$, let $\phi\in \mathcal{F}^{(L)}$ be a cutoff function for $B(x_0,r)\subseteq B(x_0,Ar)$ given by {\hyperlink{eq_CE_strong}{$\text{CE}^{(L)}_{\text{strong}}(\Psi)$}}. By Lemma \ref{lem_domain}, we have $\phi\in \mathcal{F}^{(J,\Upsilon)}$. To prove that $\phi$ satisfies {\hyperlink{eq_CE_strong}{$\text{CE}^{(J)}_{\text{strong}}(\Xi)$}}, since \hyperref[eq_CE_Holder]{$\text{CE}^{(J)}(\Xi)\text{-}1$} follows trivially from \hyperref[eq_CE_Holder]{$\text{CE}^{(L)}(\Psi)\text{-}1$}, it remains to prove \hyperref[eq_CE_energy]{$\text{CE}^{(J)}(\Xi)\text{-}2$}.

For any $x\in X$, $s>0$, we have
\begin{align}
&\int_{B(x,s)}\mathrm{d}\Gamma^{(J,\Upsilon)}_{X}(\phi)\nonumber\\
&\lesssim\left(\int_{B(x,2s)}\int_{B(x,2s)}+\int_{B(x,s)}\int_{X\backslash B(x,2s)}\right)\frac{\lvert \phi(y)-\phi(z)\rvert^p}{V(y,d(y,z))\Upsilon(d(y,z))} m(\mathrm{d}y)m(\mathrm{d}z)\nonumber\\
&=I_1+I_2.\label{eq_CE_sub1}
\end{align}
By \hyperref[eq_CE_Holder]{$\text{CE}^{(L)}(\Psi)\text{-}1$}, we have
\begin{align*}
&I_2\lesssim\int_{B(x,s)}\int_{X\backslash B(y,s)}\left(\frac{d(y,z)}{r}\wedge 1\right)^{p\delta}\frac{1}{V(y,d(y,z))\Upsilon(d(y,z))}m(\mathrm{d}y)m(\mathrm{d}z)\\
&=\int_{B(x,s)}\left(\sum_{n=0}^{+\infty}\int_{B(y,2^{n+1}s)\backslash B(y,2^ns)}\left(\frac{d(y,z)}{r}\wedge 1\right)^{p\delta}\frac{1}{V(y,d(y,z))\Upsilon(d(y,z))}m(\mathrm{d}z)\right)m(\mathrm{d}y).
\end{align*}
Since
\begin{align*}
&\sum_{n=0}^{+\infty}\int_{B(y,2^{n+1}s)\backslash B(y,2^ns)}\left(\frac{d(y,z)}{r}\wedge 1\right)^{p\delta}\frac{1}{V(y,d(y,z))\Upsilon(d(y,z))}m(\mathrm{d}z)\\
&\le\sum_{n=0}^{+\infty} \left(\frac{2^{n+1}s}{r}\wedge 1\right)^{p\delta} \frac{V(y,2^{n+1}s)}{V(y,2^ns)\Upsilon(2^ns)}\lesssim \frac{1}{\Upsilon(s)}\sum_{n=0}^{+\infty}\left(\frac{2^{n+1}s}{r}\wedge 1\right)^{p\delta}\frac{1}{2^{\beta^{(1)}_\Upsilon n}}\\
&\overset{(\diamond)}{\scalebox{2}[1]{$\asymp$}}\frac{1}{\Upsilon(s)}\left(\frac{s}{r}\wedge1\right)^{p\delta}\le \frac{1}{\Upsilon(s)} \left(\frac{s}{r}\wedge1\right)^\delta,
\end{align*}
where $(\diamond)$ follows from the fact that $p\delta<\beta^{(1)}_\Upsilon$, we have
\begin{equation}\label{eq_CE_sub2}
I_2\lesssim \left(\frac{s}{r}\wedge1\right)^\delta \frac{V(x,s)}{\Upsilon(s)}.
\end{equation}
By (\ref{eq_nonlocal_E1}), we have
\begin{align*}
&I_1\lesssim\sum_{n=0}^{+\infty}E_{\Upsilon,B(x,2s)}(\phi,\frac{1}{2^n}(4s))=\sum_{n=0}^{+\infty}\frac{\Psi(\frac{1}{2^n}(4s))}{\Upsilon(\frac{1}{2^n}(4s))}E_{\Psi,B(x,2s)}(\phi,\frac{1}{2^n}(4s)).
\end{align*}

If $s\le Ar$, then for any $n\ge0$, by Lemma \ref{lem_local_E}, we have
\begin{align*}
&E_{\Psi,B(x,2s)}(\phi,\frac{1}{2^n}(4s))\lesssim \Gamma^{(L)}(\phi)\left(B(x,2s+2A_{PI}\frac{1}{2^n}(4s))\right)\le \Gamma^{(L)}(\phi)(B(x,16A_{PI}s))\\
&\overset{(\star)}{\scalebox{2}[1]{$\lesssim$}} \left(\frac{16A_{PI}s}{r}\wedge 1\right)^\delta \frac{V(x,16A_{PI}s)}{\Psi((16A_{PI}s)\wedge r)}\lesssim \left(\frac{s}{r}\right)^\delta \frac{V(x,s)}{\Psi(s)},
\end{align*}
where $(\star)$ follows from \hyperref[eq_CE_energy]{$\text{CE}^{(L)}(\Psi)\text{-}2$}. By \ref{eq_SUG0} and Lemma \ref{lem_Xi_basic}, we have
\begin{equation}\label{eq_CE_sub3}
I_1\lesssim \left(\frac{s}{r}\right)^\delta \frac{V(x,s)}{\Psi(s)}\sum_{n=0}^{+\infty}\frac{\Psi(\frac{1}{2^n}(4s))}{\Upsilon(\frac{1}{2^n}(4s))}\asymp \left(\frac{s}{r}\right)^\delta \frac{V(x,s)}{\Xi(s)}.
\end{equation}

If $s>Ar$, then
\begin{align*}
&I_1\le\int_{X}\int_{X}\frac{\lvert \phi(y)-\phi(z)\rvert^p}{V(y,d(y,z))\Upsilon(d(y,z))} m(\mathrm{d}y)m(\mathrm{d}z)\\
&= \left(\int_{B(x_0,2Ar)}\int_{B(x_0,2Ar)}+\int_{B(x_0,2Ar)}\int_{X\backslash B(x_0,2Ar)}+\int_{X\backslash B(x_0,2Ar)}\int_{B(x_0,2Ar)}\right)\ldots \mathrm{d}m \mathrm{d}m\\
&=I_{11}+I_{12}+I_{13}.
\end{align*}
Taking $x=x_0$, $s=Ar$ in (\ref{eq_CE_sub3}), we have
$$I_{11}\lesssim \frac{V(x_0,r)}{\Xi(r)}.$$
Since
\begin{align*}
&I_{12}=\int_{B(x_0,Ar)}\int_{X\backslash B(x_0,2Ar)}\frac{\phi(y)^p}{V(y,d(y,z))\Upsilon(d(y,z))}m(\mathrm{d}y)m(\mathrm{d}z)\\
&\le\int_{B(x_0,Ar)}\left(\int_{X\backslash B(y,Ar)}\frac{1}{V(y,d(y,z))\Upsilon(d(y,z))}m(\mathrm{d}z)\right)m(\mathrm{d}y)\lesssim \frac{V(x_0,r)}{\Upsilon(r)},
\end{align*}
and similarly
$$I_{13}\lesssim\frac{V(x_0,r)}{\Upsilon(r)},$$
using $\Xi\lesssim \Upsilon$ as in Lemma \ref{lem_Xi_basic}, we have
$$I_1\le I_{11}+I_{12}+I_{13}\lesssim\frac{V(x_0,r)}{\Xi(r)}.$$
If $d(x_0,x)<3s$, then
$$I_1\lesssim  \frac{V(x_0,r)}{\Xi(r)}\le \frac{V(x,d(x_0,x)+r)}{\Xi(r)}\le \frac{V(x,3s+\frac{1}{A}s)}{\Xi(r)}\lesssim \frac{V(x,s)}{\Xi(r)}.$$
If $d(x_0,x)\ge 3s$, then $B(x_0,Ar)\cap B(x,2s)=\emptyset$, hence
$$I_1=\int_{B(x,2s)}\int_{B(x,2s)}\frac{\lvert \phi(y)-\phi(z)\rvert^p}{V(y,d(y,z))\Upsilon(d(y,z))} m(\mathrm{d}y)m(\mathrm{d}z)=0\le  \frac{V(x,s)}{\Xi(r)}.$$
Therefore, if $s>Ar$, then
$$I_1\lesssim \frac{V(x,s)}{\Xi(r)}.$$
Combining this with (\ref{eq_CE_sub3}), we have
\begin{align*}
I_1\lesssim
\begin{cases}
\left(\frac{s}{r}\right)^\delta \frac{V(x,s)}{\Xi(s)}&\text{if }s\le Ar,\\
\frac{V(x,s)}{\Xi(r)}&\text{if }s>Ar,
\end{cases}
\lesssim\left(\frac{s}{r}\wedge 1\right)^\delta \frac{V(x,s)}{\Xi(s\wedge r)}.
\end{align*}
Combining this with (\ref{eq_CE_sub1}) and (\ref{eq_CE_sub2}), and using $\Xi \lesssim \Upsilon$ as in Lemma \ref{lem_Xi_basic}, we have
$$\int_{B(x,s)}\mathrm{d}\Gamma_X^{(J,\Upsilon)}(\phi)\lesssim I_1+I_2\lesssim \left(\frac{s}{r}\wedge 1\right)^\delta \frac{V(x,s)}{\Xi(s\wedge r)},$$
which is \hyperref[eq_CE_energy]{$\text{CE}^{(J)}(\Xi)\text{-}2$}. In summary, we have \hyperlink{eq_CE_strong}{$\text{CE}^{(J)}_{\text{strong}}(\Xi)$}.
\end{proof}

Let
\begin{align*}
&\mathcal{F}^{(L)}_{\mathrm{loc}}=\left\{u:
\begin{array}{l}
\text{for any relatively compact open set }U,\\
\text{there exists }u^\#\in\mathcal{F}^{(L)}\text{ such that }u=u^\#\text{ }m\text{-a.e. in }U
\end{array}
\right\}.
\end{align*}
For any $u\in\mathcal{F}_{\mathrm{loc}}^{(L)}$, let $\Gamma^{(L)}(u)|_U=\Gamma^{(L)}(u^\#)|_U$, where $u^\#$, $U$ are given as above, then $\Gamma^{(L)}(u)$ is a well-defined positive Radon measure on $X$, as follows from \ref{item_meas3} and \ref{item_meas6} together with an argument similar to that in \cite[Corollary 3.2.1]{FOT11}. Moreover, if \ref{eq_PIL} holds for any $u \in \mathcal{F}^{(L)}$, then it also holds for any $u \in \mathcal{F}^{(L)}_{\mathrm{loc}}$.

We have the following partition of unity with controlled energy.

\begin{lemma}[{\cite[Lemma 3.2]{Yan25a}}]\label{lem_partitionL}
Assume \ref{eq_VD}. Let $(\mathcal{E}^{(L)},\mathcal{F}^{(L)})$ be a strongly local regular $p$-energy satisfying \ref{eq_ucapL}. Then we have the following controlled cutoff condition. There exists $C_{cut}>0$ depending only on $p, C_{\Psi}, C_{VD}, C_{cap}, A_{cap}$ such that for any $\varepsilon>0$, for any $\varepsilon$-net $V$, there exists a family of functions $\{\psi_z\in\mathcal{F}^{(L)}:z\in V\}$ satisfying the following conditions.
\begin{enumerate}[label=(CO\arabic*),ref=(CO\arabic*)]
\item\label{item_COspt} For any $z\in V$, $0\le\psi_z\le 1$ in $X$, $\psi_z=1$ in $B(z,\frac{1}{4}\varepsilon)$, and $\psi_z=0$ on $X\backslash B(z,\frac{5}{4}\varepsilon)$.
\item\label{item_COunit} $\sum_{z\in V}\psi_z=1$.
\item\label{item_COenergy} For any $z\in V$, $\mathcal{E}^{(L)}(\psi_z)\le C_{cut}\frac{V(z,\varepsilon)}{\Psi(\varepsilon)}$.
\end{enumerate}
\end{lemma}

We have the following ``variant" of the Poincar\'e inequality.

\begin{proposition}\label{prop_PIL_E}
Assume \ref{eq_VD}. Let $(\mathcal{E}^{(L)},\mathcal{F}^{(L)})$ be a strongly local regular $p$-energy satisfying \ref{eq_PIL}, \ref{eq_ucapL}. Then there exist $C>0$, $A>1$ such that for any $x_0\in X$ and any $R,r>0$ with $r\le R$, for any $u\in L^p(X;m)$, we have
$$\int_{B(x_0,R)}\lvert u-u_{B(x_0,R)}\rvert^p \mathrm{d}m \le C\Psi(R)E_{\Psi,B(x_0,AR)}(u,r).$$
\end{proposition}

\begin{proof}
Let $\varepsilon=\frac{r}{6}$, and let $V\subseteq X$ be an $\varepsilon$-net. By Lemma \ref{lem_partitionL}, there exist $C_{cut}>0$ and $\{\psi_z\in\mathcal{F}^{(L)}:z\in V\}$ satisfying the conditions therein. For any $z\in V$, let
$$N_z=\left\{w\in V:d(z,w)<4\varepsilon\right\},$$
then by \ref{eq_VD}, there exists an integer $M\ge1$ depending only on $C_{VD}$ such that $\# N_z\le M$ for any $z\in V$. For $u\in L^p(X;m)$, let
$$u^{(\varepsilon)}=\sum_{z\in V}u_{B(z,\varepsilon)}\psi_z,$$
then $u^{(\varepsilon)}\in \mathcal{F}^{(L)}_{\mathrm{loc}}$. By \ref{eq_VD} and H\"older's inequality, we have
\begin{align*}
&\int_X \lvert u^{(\varepsilon)}\rvert^p \mathrm{d}m\le\sum_{z\in V}\int_{B(z,\varepsilon)}\lvert \sum_{w\in N_z}u_{B(w,\varepsilon)}\psi_w\rvert^p \mathrm{d}m\\
&\lesssim \sum_{z\in V}\int_{B(z,\varepsilon)}\sum_{w\in N_z}\lvert u_{B(w,\varepsilon)}\rvert^p \mathrm{d}m\lesssim\sum_{z\in V}\int_{B(z,5\varepsilon)}\lvert u\rvert^p \mathrm{d}m\lesssim \int_X \lvert u\rvert^p \mathrm{d}m,
\end{align*}
hence $u^{(\varepsilon)}\in L^p(X;m)$. By H\"older's inequality, we have
\begin{align}
&\int_{B(x_0,R)}\lvert u-u_{B(x_0,R)}\rvert^p \mathrm{d}m\nonumber\\
&\le 3^{p-1}\left(\int_{B(x_0,R)}\lvert u-u^{(\varepsilon)}\rvert^p \mathrm{d}m+\int_{B(x_0,R)}\lvert u^{(\varepsilon)}-(u^{(\varepsilon)})_{B(x_0,R)}\rvert^p \mathrm{d}m\right.\nonumber\\
&\hspace{40pt}\left.+\int_{B(x_0,R)}\lvert (u^{(\varepsilon)})_{B(x_0,R)}-u_{B(x_0,R)}\rvert^p \mathrm{d}m\right),\label{eq_PIL_E1}
\end{align}
where
\begin{equation}\label{eq_PIL_E2}
\int_{B(x_0,R)}\lvert (u^{(\varepsilon)})_{B(x_0,R)}-u_{B(x_0,R)}\rvert^p \mathrm{d}m\le \int_{B(x_0,R)}\lvert u-u^{(\varepsilon)}\rvert^p \mathrm{d}m.
\end{equation}

Firstly, by \ref{eq_VD} and H\"older's inequality, we have
\begin{align}
&\int_{B(x_0,R)}\lvert u-u^{(\varepsilon)}\rvert^p \mathrm{d}m\le\sum_{z\in V\cap B(x_0,2R)}\int_{B(z,\varepsilon)}\lvert u-u^{(\varepsilon)}\rvert^p \mathrm{d}m\nonumber\\
&=\sum_{z\in V\cap B(x_0,2R)}\int_{B(z,\varepsilon)}\lvert \sum_{w\in N_z}(u-u_{B(w,\varepsilon)})\psi_w \rvert^p \mathrm{d}m\nonumber\\
&\lesssim \sum_{z\in V\cap B(x_0,2R)}\int_{B(z,\varepsilon)} \sum_{w\in N_z}\lvert u-u_{B(w,\varepsilon)}\rvert^p \mathrm{d}m\nonumber\\
&\lesssim \sum_{z\in V\cap B(x_0,2R)}\int_{B(z,\varepsilon)} \left(\frac{1}{V(x,6\varepsilon)}\int_{B(x,6\varepsilon)}\lvert u(x)-u(y)\rvert^p m(\mathrm{d}y)\right)m(\mathrm{d}x)\nonumber\\
&\lesssim \int_{B(x_0,4R)}\left(\frac{1}{V(x,6\varepsilon)}\int_{B(x,6\varepsilon)}\lvert u(x)-u(y)\rvert^p m(\mathrm{d}y)\right)m(\mathrm{d}x).\label{eq_PIL_E3}
\end{align}

Secondly, by \ref{eq_PIL}, we have
\begin{equation}\label{eq_PIL_E4}
\int_{B(x_0,R)}\lvert u^{(\varepsilon)}-(u^{(\varepsilon)})_{B(x_0,R)}\rvert^p \mathrm{d}m\le C_{PI}\Psi(R)\int_{B(x_0,A_{PI}R)}\mathrm{d}\Gamma^{(L)}(u^{(\varepsilon)}),
\end{equation}
where by \ref{eq_VD}, \ref{item_COenergy}, and H\"older's inequality, we have
\begin{align}
&\int_{B(x_0,A_{PI}R)}\mathrm{d}\Gamma^{(L)}(u^{(\varepsilon)})\le\sum_{z\in V\cap B(x_0,2A_{PI}R)}\Gamma^{(L)}(u^{(\varepsilon)})\left(B(z,\varepsilon)\right)\nonumber\\
&=\sum_{z\in V\cap B(x_0,2A_{PI}R)}\Gamma^{(L)} \left(\sum_{w\in N_z}(u_{B(z,\varepsilon)}-u_{B(w,\varepsilon)})\psi_w\right)\left(B(z,\varepsilon)\right)\nonumber\\
&\lesssim\sum_{z\in V\cap B(x_0,2A_{PI}R)}\sum_{w\in N_z}\lvert u_{B(z,\varepsilon)}-u_{B(w,\varepsilon)} \rvert^p\Gamma^{(L)} \left(\psi_w\right)\left(B(z,\varepsilon)\right)\nonumber\\
&\lesssim \sum_{z\in V\cap B(x_0,2A_{PI}R)} \frac{1}{V(z,\varepsilon)}\int_{B(z,\varepsilon)} \left(\frac{1}{V(x,6\varepsilon)}\int_{B(x,6\varepsilon)}\lvert u(x)-u(y)\rvert^p m(\mathrm{d}y)\right)m(\mathrm{d}x) \cdot\frac{V(z,\varepsilon)}{\Psi(\varepsilon)}\nonumber\\
&\lesssim \frac{1}{\Psi(\varepsilon)}\int_{B(x_0,4A_{PI}R)} \left(\frac{1}{V(x,6\varepsilon)}\int_{B(x,6\varepsilon)}\lvert u(x)-u(y)\rvert^p m(\mathrm{d}y)\right)m(\mathrm{d}x).\label{eq_PIL_E5}
\end{align}

Combining (\ref{eq_PIL_E1})--(\ref{eq_PIL_E5}) and using $6\varepsilon=r\le R$, we have
\begin{align*}
&\int_{B(x_0,R)}\lvert u-u_{B(x_0,R)}\rvert^p \mathrm{d}m\\
&\lesssim\left(\int_{B(x_0,4R)}+\frac{\Psi(R)}{\Psi(\varepsilon)}\int_{B(x_0,4A_{PI}R)}\right)\left(\frac{1}{V(x,6\varepsilon)}\int_{B(x,6\varepsilon)}\lvert u(x)-u(y)\rvert^p m(\mathrm{d}y)\right)m(\mathrm{d}x)\\
&\lesssim \Psi(R)E_{\Psi,B(x_0,4A_{PI}R)}(u,r).
\end{align*}
\end{proof}

We give the proof of ``\ref{item_sub2}$\Rightarrow$\ref{item_sub3}" in Theorem \ref{thm_sub} as follows.

\begin{proof}[Proof of ``\ref{item_sub2}$\Rightarrow$\ref{item_sub3}" in Theorem \ref{thm_sub}]
Let $u\in \mathcal{F}^{(J,\Upsilon)}$ be a non-constant function, then there exists a ball $B(x_0,R)$ such that
$$\int_{B(x_0,R)}\lvert u-u_{B(x_0,R)}\rvert^p \mathrm{d}m\in(0,+\infty).$$
By Proposition \ref{prop_PIL_E}, there exist $C>0$, $A>1$ such that for any $r\in(0,R]$, we have
$$E_{\Psi,B(x_0,AR)}(u,r)\ge \frac{1}{C\Psi(R)}\int_{B(x_0,R)}\lvert u-u_{B(x_0,R)}\rvert^p \mathrm{d}m=:I\in(0,+\infty).$$

Suppose that \ref{eq_SUG0} does not hold. Take $N\in \mathbb{Z}$ such that $\frac{1}{2^N}\le R$, then
$$\sum_{n=N}^{+\infty}\frac{\Psi(\frac{1}{2^n})}{\Upsilon(\frac{1}{2^n})}=+\infty.$$
By above, for any $n\ge N$, we have $E_{\Psi,B(x_0,AR)}(u,\frac{1}{2^n})\ge I$. By Lemma \ref{lem_nonlocal_E}, we have
\begin{align*}
&+\infty>\mathcal{E}^{(J,\Upsilon)}(u)\gtrsim\sum_{n=N}^{+\infty}E_{\Upsilon,B(x_0,AR)}(u,\frac{1}{2^n})\\
&=\sum_{n=N}^{+\infty}\frac{\Psi(\frac{1}{2^n})}{\Upsilon(\frac{1}{2^n})}E_{\Psi,B(x_0,AR)}(u,\frac{1}{2^n})\ge I\sum_{n=N}^{+\infty}\frac{\Psi(\frac{1}{2^n})}{\Upsilon(\frac{1}{2^n})}=+\infty,
\end{align*}
contradiction. Hence \ref{eq_SUG0} holds.
\end{proof}

We give the proof of \hyperref[eq_PIJ]{$\text{PI}^{(J)}(\Xi)$} under \ref{eq_SUG0} in Theorem \ref{thm_sub} as follows.

\begin{proof}[{Proof of {\protect\hyperref[eq_PIJ]{$\text{PI}^{(J)}(\Xi)$}} under \ref{eq_SUG0} in Theorem \ref{thm_sub}}]
For any ball $B(x_0,R)$, for any $u\in L^p(X;m)$, for any $n\ge0$, by Proposition \ref{prop_PIL_E}, we have
\begin{align*}
&\int_{B(x_0,R)}\lvert u-u_{B(x_0,R)}\rvert^p \mathrm{d}m\\
&\lesssim \Psi(R)E_{\Psi,B(x_0,AR)}(u,\frac{1}{2^n}R)=\Psi(R)\frac{\Upsilon(\frac{1}{2^n}R)}{\Psi(\frac{1}{2^n}R)}E_{\Upsilon,B(x_0,AR)}(u,\frac{1}{2^n}R),
\end{align*}
where $A$ is a constant appearing therein, that is,
$$\frac{\Psi(\frac{1}{2^n}R)}{\Upsilon(\frac{1}{2^n}R)}\int_{B(x_0,R)}\lvert u-u_{B(x_0,R)}\rvert^p \mathrm{d}m\lesssim \Psi(R)E_{\Upsilon,B(x_0,AR)}(u,\frac{1}{2^n}R).$$
Taking the summation over $n \ge 0$ and using Lemma \ref{lem_nonlocal_E}, we have
\begin{align*}
&\sum_{n=0}^{+\infty}\frac{\Psi(\frac{1}{2^n}R)}{\Upsilon(\frac{1}{2^n}R)}\int_{B(x_0,R)}\lvert u-u_{B(x_0,R)}\rvert^p \mathrm{d}m\lesssim \Psi(R)\sum_{n=0}^{+\infty}E_{\Upsilon,B(x_0,AR)}(u,\frac{1}{2^n}R)\\
&\lesssim \Psi(R)\int_{B(x_0,2AR)}\int_{B(x_0,2AR)}\frac{\lvert u(x)-u(y)\rvert^p}{V(x,d(x,y))\Upsilon(d(x,y))}m(\mathrm{d}x)m(\mathrm{d}y)\\
&\lesssim \Psi(R)\int_{B(x_0,2AR)}\mathrm{d}\Gamma^{(J,\Upsilon)}_{B(x_0,2AR)}(u),
\end{align*}
where the last inequality follows from \ref{eq_KJ}. By Lemma \ref{lem_Xi_basic}, we have
\begin{align*}
&\int_{B(x_0,R)}\lvert u-u_{B(x_0,R)}\rvert^p \mathrm{d}m\\
&\lesssim \frac{\Psi(R)}{\sum_{n=0}^{+\infty}\frac{\Psi(\frac{1}{2^n}R)}{\Upsilon(\frac{1}{2^n}R)}}\int_{B(x_0,2AR)}\mathrm{d}\Gamma^{(J,\Upsilon)}_{B(x_0,2AR)}(u)\asymp \Xi(R)\int_{B(x_0,2AR)}\mathrm{d}\Gamma^{(J,\Upsilon)}_{B(x_0,2AR)}(u).
\end{align*}
Hence, we have \hyperref[eq_PIJ]{$\text{PI}^{(J)}(\Xi)$}.
\end{proof}

Finally, by \ref{eq_VD} and \ref{eq_KJ}, we have \ref{eq_UPR} and \ref{eq_KJ_tail} by Lemma \ref{lem_KJ_ele}. Assuming \ref{eq_SUG0}, since $\Xi\lesssim \Upsilon$ by Lemma \ref{lem_Xi_basic}, it follows that \hyperref[eq_KJ_tail]{$\text{T}^{(J)}(\Xi)$} also holds. Combining \hyperref[eq_PIJ]{$\text{PI}^{(J)}(\Xi)$} and \hyperlink{eq_CE_strong}{$\text{CE}^{(J)}_{\text{strong}}(\Xi)$}, and applying Corollary \ref{cor_pform}, we conclude that $(\mathcal{E}^{(J,\Upsilon)},\mathcal{F}^{(J,\Upsilon)})$ is a regular $p$-energy. This completes the proof of Theorem \ref{thm_sub}.

\section{\texorpdfstring{Self-improvement of {$\text{CS}^{(J)}_{\text{weak}}$} and {$\text{CS}^{(J)}_{\text{cont}}$}}{Self-improvement of CSJweak and CSJcont}}\label{sec_CSJ_self}

In this section, we prove the self-improvement property of \hyperlink{eq_CSJ_weak}{$\text{CS}^{(J)}_{\text{weak}}$} and \hyperlink{eq_CSJ_cont}{$\text{CS}^{(J)}_{\text{cont}}$}; see also \cite[Lemma 6.2]{Mur24a} and \cite[Proposition 3.1]{Yan25c} for analogous results in the local setting, and \cite[Lemma 2.9]{GHH18} and \cite[PROPOSITION 2.4]{CKW21} in the non-local setting. Throughout this section, we always assume that $(\mathcal{E}^{(J)},\mathcal{F}^{(J)})$ is a non-local $p$-form given by a kernel $K^{(J)}$.

\begin{proposition}\label{prop_CSJ_self}
Assume \ref{eq_VD}, \ref{eq_KJ_tail}, \hyperlink{eq_CSJ_weak}{$\text{CS}^{(J)}_{\text{weak}}(\Upsilon)$}. Then for any $\varepsilon>0$, there exists $C_\varepsilon>0$ depending on $\varepsilon$, such that for any $x_0\in X$ , for any $R,r,R'>0$ with $R+r<R'$, there exists a cutoff function $\phi\in \mathcal{F}^{(J)}$ for $B(x_0,R)\subseteq B(x_0,R+r)$ such that for any $f\in \widehat{\mathcal{F}}^{(J)}\cap L^\infty(X;m)$, we have
\begin{align}
&\int_{B(x_0,R')}\lvert f\rvert^p \mathrm{d}\Gamma^{(J)}_{B(x_0,R')}(\phi)\nonumber\\
&\le \varepsilon \int_{B(x_0,R+r)\backslash \overline{B(x_0,R)}}\lvert \phi\rvert^p \mathrm{d}\Gamma^{(J)}_{B(x_0,R+r)\backslash \overline{B(x_0,R)}}(f)+\frac{C_\varepsilon}{\Upsilon(r)}\int_{B(x_0,R')}\lvert f\rvert^p \mathrm{d}m.\label{eq_CSJ_self}
\end{align}
Moreover, if \hyperlink{eq_CSJ_cont}{$\text{CS}^{(J)}_{\text{cont}}(\Upsilon)$} holds, then we can take $\phi\in \mathcal{F}^{(J)}\cap C_c(X)$.
\end{proposition}

\begin{proof}
Firstly, we prove that there exists $C>0$ such that for any $x_0\in X$, for any $R,r,R'>0$ with $R+r<R'$, there exists a cutoff function $\phi\in \mathcal{F}^{(J)}$ for $B(x_0,R)\subseteq B(x_0,R+r)$ such that for any $f\in \mathcal{F}^{(J)}$, we have
\begin{align}
&\int_{B(x_0,R')}\lvert f\rvert^p \mathrm{d}\Gamma^{(J)}_{B(x_0,R')}(\phi)\nonumber\\
&\le C\int_{B(x_0,R+r)\backslash \overline{B(x_0,R)}} \mathrm{d}\Gamma^{(J)}_{B(x_0,R+r)\backslash \overline{B(x_0,R)}}(f)+\frac{C}{\Upsilon(r)}\int_{B(x_0,R')}\lvert f\rvert^p \mathrm{d}m.\label{eq_CSJ_self1}
\end{align}

Let $A_1,A_2,C_1,C_2$ be the constants in \ref{eq_CSJ_energy}, $L=2(A_1+8A_2+5)+1$, and $V$ an $(\frac{1}{L}r)$-net. For any $v\in V$, by \hyperlink{eq_CSJ_weak}{$\text{CS}^{(J)}_{\text{weak}}(\Upsilon)$}, there exists a cutoff function $\phi_v\in \mathcal{F}^{(J)}$ for $B(v,\frac{1}{L}r)\subseteq B(v,\frac{A_1}{L}r)$ such that
$$\int_{B(v,\frac{A_2}{L}r)}\lvert f\rvert^p \mathrm{d}\Gamma^{(J)}_{B(v,\frac{A_2}{L}r)}(\phi_v)\le C_1\int_{B(v,\frac{A_2}{L}r)}\mathrm{d}\Gamma^{(J)}_{B(v,\frac{A_2}{L}r)}(f)+\frac{C_2}{\Upsilon(r)}\int_{B(v,\frac{A_2}{L}r)}\lvert f\rvert^p \mathrm{d}m.$$

We claim that there exists $C_3>0$ such that for any open set $\Omega\supseteq B(v,\frac{A_2}{L}r)$, we have
\begin{equation}\label{eq_phiv}
\int_\Omega \lvert f\rvert^p \mathrm{d}\Gamma^{(J)}_\Omega(\phi_v)\le C_3\int_{B(v,\frac{A_2}{L}r)}\mathrm{d}\Gamma^{(J)}_{B(v,\frac{A_2}{L}r)}(f)+\frac{C_3}{\Upsilon(r)}\int_{\Omega}\lvert f\rvert^p \mathrm{d}m.
\end{equation}
Indeed, we have
\begin{align*}
\int_\Omega \lvert f\rvert^p \mathrm{d}\Gamma^{(J)}_\Omega(\phi_v)=\left(\int_{B(v,\frac{A_2}{L}r)}+\int_{\Omega\backslash B(v,\frac{A_2}{L}r)} \right)\lvert f\rvert^p \mathrm{d}\Gamma^{(J)}_\Omega(\phi_v)=I_1+I_2,
\end{align*}
and $I_1=I_{11}+I_{12}$, where
\begin{align*}
&I_{11}=\int_{B(v,\frac{A_2}{L}r)}\int_{B(v,\frac{A_2}{L}r)}\lvert f(x)\rvert^p \lvert \phi_v(x)-\phi_v(y)\rvert^p K^{(J)}(x,y) m(\mathrm{d}x) m(\mathrm{d}y)\\
&=\int_{B(v,\frac{A_2}{L}r)}\lvert f\rvert^p \mathrm{d}\Gamma^{(J)}_{B(v,\frac{A_2}{L}r)}(\phi_v)\lesssim\int_{B(v,\frac{A_2}{L}r)}\mathrm{d}\Gamma^{(J)}_{B(v,\frac{A_2}{L}r)}(f)+\frac{1}{\Upsilon(r)}\int_{B(v,\frac{A_2}{L}r)}\lvert f\rvert^p \mathrm{d}m,
\end{align*}
and by \ref{eq_KJ_tail}, we have
\begin{align*}
&I_{12}=\int_{B(v,\frac{A_2}{L}r)}\int_{\Omega\backslash B(v,\frac{A_2}{L}r)}\lvert f(x)\rvert^p \lvert \phi_v(x)-\phi_v(y)\rvert^p K^{(J)}(x,y) m(\mathrm{d}x) m(\mathrm{d}y)\\
&=\int_{B(v,\frac{A_1}{L}r)}\int_{\Omega\backslash B(v,\frac{A_2}{L}r)}\lvert f(x)\rvert^p \phi_v(x)^p K^{(J)}(x,y) m(\mathrm{d}x) m(\mathrm{d}y)\\
&\le\int_{B(v,\frac{A_1}{L}r)} \lvert f(x)\rvert^p \left(\int_{X\backslash B(x,\frac{A_2-A_1}{L}r)}K^{(J)}(x,y)m(\mathrm{d}y)\right)m(\mathrm{d}x)\\
&\lesssim \frac{1}{\Upsilon(r)}\int_{B(v,\frac{A_1}{L}r)}\lvert f\rvert^p \mathrm{d}m,
\end{align*}
and
\begin{align*}
&I_2=\int_{\Omega\backslash B(v,\frac{A_2}{L}r)}\int_{\Omega}\lvert f(x)\rvert^p \lvert \phi_v(x)-\phi_v(y)\rvert^p K^{(J)}(x,y) m(\mathrm{d}x) m(\mathrm{d}y)\\
&=\int_{\Omega\backslash B(v,\frac{A_2}{L}r)}\int_{B(v,\frac{A_1}{L}r)}\lvert f(x)\rvert^p \phi_v(y)^p K^{(J)}(x,y) m(\mathrm{d}x) m(\mathrm{d}y)\\
&\le\int_{\Omega\backslash B(v,\frac{A_2}{L}r)} \lvert f(x)\rvert^p \left(\int_{X\backslash B(x,\frac{A_2-A_1}{L}r)}K^{(J)}(x,y)m(\mathrm{d}y)\right)m(\mathrm{d}x)\\
&\lesssim \frac{1}{\Upsilon(r)}\int_{\Omega\backslash B(v,\frac{A_2}{L}r)}\lvert f\rvert^p \mathrm{d}m,
\end{align*}
hence we have (\ref{eq_phiv}).

Let $V_1=V\cap B(x_0,R+\frac{1}{2}r)$, then by \ref{eq_VD}, we have $\# V_1<+\infty$. Let
\begin{equation}\label{eq_phi_self1}
\phi=\max_{v\in V_1}\phi_v,
\end{equation}
then $\phi\in \mathcal{F}^{(J)}$ satisfies that $0\le \phi\le 1$ in $X$, $\phi=1$ in $B(x_0,R+\frac{1}{2}r-\frac{1}{L}r)\supseteq \overline{B(x_0,R)}$, and $\mathrm{supp}(\phi)\subseteq B(x_0,R+\frac{1}{2}r+\frac{A_1}{L}r)\subseteq B(x_0,R+r)$, hence $\phi\in \mathcal{F}^{(J)}$ is a cutoff function for $B(x_0,R)\subseteq B(x_0,R+r)$. Moreover
\begin{align*}
&\int_{B(x_0,R')}\lvert f\rvert^p \mathrm{d}\Gamma^{(J)}_{B(x_0,R')}(\phi)=\iint\limits_{B(x_0,R')\times B(x_0,R')}\lvert f(x)\rvert^p \lvert \phi(x)-\phi(y)\rvert^p K^{(J)}(x,y)m(\mathrm{d}x) m(\mathrm{d}y)\\
&=\left(\iint\limits_{\substack{B(x_0,R')\times B(x_0,R')\\d(x,y)\ge \frac{1}{L}r}}+\iint\limits_{\substack{B(x_0,R')\times B(x_0,R')\\d(x,y)<\frac{1}{L}r}}\right)\ldots \mathrm{d}m \mathrm{d}m=I_3+I_4.
\end{align*}
By \ref{eq_KJ_tail}, we have
\begin{align*}
&I_3\le\iint\limits_{\substack{B(x_0,R')\times B(x_0,R')\\d(x,y)\ge \frac{1}{L}r}}\lvert f(x)\rvert^p  K^{(J)}(x,y)m(\mathrm{d}x) m(\mathrm{d}y)\\
&\le\int_{B(x_0,R')}\lvert f(x)\rvert^p \left(\int_{X\backslash B(x,\frac{1}{L}r)}K^{(J)}(x,y)m(\mathrm{d}y)\right)m(\mathrm{d}x)\\
&\lesssim \frac{1}{\Upsilon(r)}\int_{B(x_0,R')}\lvert f\rvert^p \mathrm{d}m.
\end{align*}
For any $x,y\in X$ with $d(x,y)<\frac{1}{L}r$, we have $\lvert \phi(x)-\phi(y)\rvert\ne0$ only if $x,y\in U$, where
$$U=B(x_0,R+\frac{1}{2}r+\frac{A_1+1}{L}r)\backslash \overline{B(x_0,R+\frac{1}{2}r-\frac{2}{L}r)}.$$
Let
\begin{align*}
V_2=V\cap \left(B(x_0,R+\frac{1}{2}r+\frac{A_1+2}{L}r)\backslash B(x_0,R+\frac{1}{2}r-\frac{3}{L}r)\right),
\end{align*}
then $ U\subseteq\cup_{v\in V_2}B(v,\frac{1}{L}r)$, hence
\begin{align*}
&I_4=\iint_{\substack{U\times U\\d(x,y)<\frac{1}{L}r}}\lvert f(x)\rvert^p \lvert \phi(x)-\phi(y)\rvert^p K^{(J)}(x,y)m(\mathrm{d}x) m(\mathrm{d}y)\\
&\le\sum_{v\in V_2}\iint_{\substack{B(v,\frac{1}{L}r)\times U\\d(x,y)<\frac{1}{L}r}}\ldots \mathrm{d}m \mathrm{d}m\le\sum_{v\in V_2}\iint_{B(v,\frac{1}{L}r)\times B(v,\frac{2}{L}r)}\ldots\mathrm{d}m \mathrm{d}m\\
&\le\sum_{v\in V_2}\iint_{B(v,\frac{2}{L}r)\times B(v,\frac{2}{L}r)}\lvert f(x)\rvert^p \lvert \phi(x)-\phi(y)\rvert^p K^{(J)}(x,y)m(\mathrm{d}x) m(\mathrm{d}y).
\end{align*}
By \ref{eq_VD}, there exists some positive integer $M$ depending only on $C_{VD}, A_1,L$ such that
$$\# \left\{w\in V:d(v,w)<\frac{A_1+2}{L}r\right\}\le M\text{ for any }v\in V.$$
For any $v\in V_2$, in $B(v,\frac{2}{L}r)$, we have
$$\phi=\max_{w\in V_1}\phi_w=\max \left\{\phi_w:w\in V_1,d(v,w)<\frac{A_1+2}{L}r\right\}.$$
By the elementary fact that $\lvert a_1\vee a_2-a_3\vee a_4\rvert\le \lvert a_1-a_3\rvert+\lvert a_2-a_4\rvert$ for any $a_1,\ldots,a_4\in \mathbb{R}$, for any $x,y\in B(v,\frac{2}{L}r)$, we have
\begin{align*}
\lvert \phi(x)-\phi(y)\rvert^p\le \left(\sum_{\substack{w\in V_1\\d(v,w)<\frac{A_1+2}{L}r}}\lvert \phi_w(x)-\phi_w(y)\rvert\right)^p\le M^{p-1}\sum_{\substack{w\in V_1\\d(v,w)<\frac{A_1+2}{L}r}}\lvert \phi_w(x)-\phi_w(y)\rvert^p.
\end{align*}
Hence
\begin{align*}
&I_4\le M^{p-1}\sum_{v\in V_2}\sum_{\substack{w\in V_1\\d(v,w)<\frac{A_1+2}{L}r}}\iint_{B(v,\frac{2}{L}r)\times B(v,\frac{2}{L}r)}\lvert f(x)\rvert^p \lvert \phi_w(x)-\phi_w(y)\rvert^p\\
&\hspace{200pt}\cdot K^{(J)}(x,y)m(\mathrm{d}x) m(\mathrm{d}y)\\
&\le M^{p-1}\sum_{\substack{v\in V_2,w\in V_1\\d(v,w)<\frac{A_1+2}{L}r}}\iint_{B(w,\frac{A_1+4}{L}r)\times B(w,\frac{A_1+4}{L}r)} \ldots\mathrm{d}m \mathrm{d}m\\
&\le M^{p-1}\sum_{\substack{v\in V_2,w\in V_1\\d(v,w)<\frac{A_1+2}{L}r}}\iint_{B(w,\frac{8A_2}{L}r)\times B(w,\frac{8A_2}{L}r)} \ldots\mathrm{d}m \mathrm{d}m\\
&= M^{p-1}\sum_{\substack{v\in V_2,w\in V_1\\d(v,w)<\frac{A_1+2}{L}r}}\int_{B(w,\frac{8A_2}{L}r)}\lvert f\rvert^p \mathrm{d}\Gamma^{(J)}_{B(w,\frac{8A_2}{L}r)}(\phi_w).
\end{align*}
Let
\begin{align*}
V_3=V_1\backslash B(x_0,R+\frac{1}{2}r-\frac{A_1+5}{L}r),
\end{align*}
then
\begin{align*}
&I_4\le M^{p-1}\sum_{\substack{v\in V_2,w\in V_3\\d(v,w)<\frac{A_1+2}{L}r}}\int_{B(w,\frac{8A_2}{L}r)}\lvert f\rvert^p \mathrm{d}\Gamma^{(J)}_{B(w,\frac{8A_2}{L}r)}(\phi_w)\\
&\le M^p\sum_{w\in V_3}\int_{B(w,\frac{8A_2}{L}r)}\lvert f\rvert^p \mathrm{d}\Gamma^{(J)}_{B(w,\frac{8A_2}{L}r)}(\phi_w).
\end{align*}
Taking $v=w$ and $\Omega=B(w,\frac{8A_2}{L}r)$ in (\ref{eq_phiv}), we have
\begin{align*}
&\int_{B(w,\frac{8A_2}{L}r)}\lvert f\rvert^p \mathrm{d}\Gamma^{(J)}_{B(w,\frac{8A_2}{L}r)}(\phi_w)\\
&\lesssim\int_{B(w,\frac{A_2}{L}r)}\mathrm{d}\Gamma^{(J)}_{B(w,\frac{A_2}{L}r)}(f)+\frac{1}{\Upsilon(r)}\int_{B(w,\frac{8A_2}{L}r)}\lvert f\rvert^p \mathrm{d}m\\
&\le\int_{B(w,\frac{A_2}{L}r)}\mathrm{d}\Gamma^{(J)}_{B(x_0,R+r)\backslash \overline{B(x_0,R)}}(f)+\frac{1}{\Upsilon(r)}\int_{B(w,\frac{8A_2}{L}r)}\lvert f\rvert^p \mathrm{d}m,
\end{align*}
hence
\begin{align*}
&I_4\lesssim\sum_{w\in V_3}\left(\int_{B(w,\frac{A_2}{L}r)}\mathrm{d}\Gamma^{(J)}_{B(x_0,R+r)\backslash \overline{B(x_0,R)}}(f)+\frac{1}{\Upsilon(r)}\int_{B(w,\frac{8A_2}{L}r)}\lvert f\rvert^p \mathrm{d}m\right)\\
&=\int_X\left(\sum_{w\in V_3}1_{B(w,\frac{A_2}{L}r)}\right)\mathrm{d}\Gamma^{(J)}_{B(x_0,R+r)\backslash \overline{B(x_0,R)}}(f)+\frac{1}{\Upsilon(r)}\int_X\left(\sum_{w\in V_3}1_{B(w,\frac{8A_2}{L}r)}\right)\lvert f\rvert^p \mathrm{d}m.
\end{align*}
By \ref{eq_VD}, there exists some positive integer $N$ depending only on $C_{VD}, A_2,L$ such that
$$\sum_{w\in V_3}1_{B(w,\frac{8A_2}{L}r)}\le N1_{\cup_{w\in V_3}B(w,\frac{8A_2}{L}r)}\le N1_{B(x_0,R+r)\backslash \overline{B(x_0,R)}},$$
hence
$$I_4\lesssim\int_{B(x_0,R+r)\backslash \overline{B(x_0,R)}}\mathrm{d}\Gamma^{(J)}_{B(x_0,R+r)\backslash \overline{B(x_0,R)}}(f)+\frac{1}{\Upsilon(r)}\int_{B(x_0,R+r)\backslash \overline{B(x_0,R)}}\lvert f\rvert^p \mathrm{d}m.$$
Therefore, we have
\begin{align*}
&\int_{B(x_0,R')}\lvert f\rvert^p \mathrm{d}\Gamma^{(J)}_{B(x_0,R')}(\phi)=I_3+I_4\\
&\lesssim\int_{B(x_0,R+r)\backslash \overline{B(x_0,R)}}\mathrm{d}\Gamma^{(J)}_{B(x_0,R+r)\backslash \overline{B(x_0,R)}}(f)+\frac{1}{\Upsilon(r)}\int_{B(x_0,R')}\lvert f\rvert^p \mathrm{d}m.
\end{align*}

Secondly, we prove (\ref{eq_CSJ_self}). Let $a\in(0,1)$ be chosen later. For any $n\ge1$, let $s_n=\frac{1-a}{a}a^nr$, then $\sum_{n=1}^{+\infty}s_n=r$. Let $r_0=0$ and $r_n=\sum_{k=1}^ns_k$ for any $n\ge1$, then $r_n\uparrow r$. For any $n\ge0$, let $B_n=B(x_0,R+r_n)$, then $B(x_0,R)=B_0\subseteq B_n\uparrow B(x_0,R+r)$. For any $n\ge0$, by (\ref{eq_CSJ_self1}), there exists a cutoff function $\phi_n\in \mathcal{F}^{(J)}$ for $B_n\subseteq B_{n+1}$ such that
\begin{align}
&\int_{B(x_0,R')}\lvert f\rvert^p \mathrm{d}\Gamma^{(J)}_{B(x_0,R')}(\phi_n)\nonumber\\
&\le C\int_{B_{n+1}\backslash \overline{B_n}}\mathrm{d}\Gamma^{(J)}_{B_{n+1}\backslash \overline{B_n}}(f)+\frac{C}{\Upsilon(s_{n+1})}\int_{B(x_0,R')}\lvert f\rvert^p \mathrm{d}m.\label{eq_CSJ_self2}
\end{align}
Let $b\in(0,1)$ be chosen later, $b_n=b^n-b^{n+1}$ for any $n\ge0$, and
\begin{equation}\label{eq_phi_self2}
\phi=\sum_{k=0}^{+\infty}b_k\phi_k=\sum_{k=0}^{+\infty}(b^k-b^{k+1})\phi_k,
\end{equation}
then $0\le\phi\le1$ in $X$, $\phi=1$ in $B(x_0,R)$, and $\phi=0$ on $X\backslash B(x_0,R+r)$. For any $n\ge0$, in $B_{n+1}$, we have $\phi=\sum_{k=0}^n(b^k-b^{k+1})\phi_k+b^{n+1}\ge b^{n+1}$; in $B_{n+1}\backslash \overline{B_n}$, we have $\phi=(b^n-b^{n+1})\phi_n+b^{n+1}$, hence $b^{n+1}\le \phi\le b^n$ in $B_{n+1}\backslash \overline{B_n}$.

For any $n\ge0$, let $\varphi_n=\sum_{k=0}^nb_k\phi_k$. By Fatou's lemma, we only need to show that for suitable choices of $a$ and $b$ depending on $\varepsilon$, we have
\begin{align*}
&\int_{B(x_0,R')}\lvert f\rvert^p \mathrm{d}\Gamma^{(J)}_{B(x_0,R')}(\varphi_n)\\
&\le \varepsilon \int_{B(x_0,R+r)\backslash \overline{B(x_0,R)}}\lvert \phi\rvert^p \mathrm{d}\Gamma^{(J)}_{B(x_0,R+r)\backslash \overline{B(x_0,R)}}(f)+\frac{C_\varepsilon}{\Upsilon(r)}\int_{B(x_0,R')}\lvert f\rvert^p \mathrm{d}m,
\end{align*}
for any $n>p$, where $C_\varepsilon>0$ is independent of $n$.

Let $p=p_0+\delta_0$, where $1\le p_0\in \mathbb{Z}$ and $\delta_0\in[0,1)$. We have the following three cases:
\begin{enumerate}[label=(\alph*),ref=(\alph*)]
\item\label{item_pd1} $p_0\ge2$ and $\delta_0\in(0,1)$,
\item\label{item_pd2} $p_0\ge2$ and $\delta_0=0$,
\item\label{item_pd3} $p_0=1$ and $\delta_0\in(0,1)$.
\end{enumerate}

In case \ref{item_pd1}, we have
\begin{align*}
\lvert \varphi_n(x)-\varphi_n(y)\rvert^p\le\left(\sum_{k=0}^n b_k \lvert \phi_k(x)-\phi_k(y)\rvert\right)^{p_0}\left(\sum_{k=0}^n b_k \lvert \phi_k(x)-\phi_k(y)\rvert\right)^{\delta_0},
\end{align*}
where
\begin{align*}
\left(\sum_{k=0}^n b_k \lvert \phi_k(x)-\phi_k(y)\rvert\right)^{\delta_0}\le\sum_{k=0}^n \left(b_k \lvert \phi_k(x)-\phi_k(y)\rvert\right)^{\delta_0},
\end{align*}
and
\begin{align*}
&\left(\sum_{k=0}^n b_k \lvert \phi_k(x)-\phi_k(y)\rvert\right)^{p_0}=\sum_{\substack{t_0,\ldots,t_n\in \{0,\ldots,p_0\}\\t_0+\ldots+t_n=p_0}}\prod_{k=0}^n\left(b_k\lvert \phi_k(x)-\phi_k(y)\rvert\right)^{{t_k}}.
\end{align*}
For any sequence $\{t_0,\ldots,t_n\}\in \{0,\ldots,p_0\}^{n+1}$ with $t_0+\ldots+t_n=p_0$, let
$$I(t_0,\ldots,t_n)=\left\{k:t_k\ge1\right\}\subseteq\left\{0,\ldots,n\right\},$$
and
$$i(t_0,\ldots,t_n)=\max\left\{k\in I(t_0,\ldots,t_n)\middle|
\begin{array}{l}
I(t_0,\ldots,t_n)\supseteq\{i\in \mathbb{Z}:l\le i\le k\}\\
\text{for any $l\in I(t_0,\ldots,t_n)$ with $l\le k$}
\end{array}
\right\},$$
then $1\le \# I(t_0,\ldots,t_n)\le p_0$. We say that $\{t_0,\ldots,t_n\}$ is connected if $I(t_0,\ldots,t_n)$ consists of consecutive integers, or equivalently, $i(t_0,\ldots,t_n)=\max\{k:k\in I(t_0,\ldots,t_n)\}$; otherwise we say that it is disconnected, see Table \ref{tab_Ii} for two examples with $n=4$ and $p_0=3$. Obviously, if $\{t_0,\ldots,t_n\}$ is disconnected, then $k=i(t_0,\ldots,t_n)\le n-2$ and $I(t_0,\ldots,t_n)\cap\{k+2,k+3,\ldots\}\ne\emptyset$.

\begin{table}[ht]
\centering
\renewcommand{\arraystretch}{1.5} 
\begin{tabular}{|c|c|c|}
\hline
$\{t_0,\ldots,t_n\}$& $\{0,1,1,1,0\}$ & $\{0,1,1,0,1\}$\\
\hline
$I(t_0,\ldots,t_n)$&$\{1,2,3\}$&$\{1,2,4\}$\\
\hline
$i(t_0,\ldots,t_n)$&3&2\\
\hline
connectedness& connected&disconnected\\
\hline
\end{tabular}
\caption{Two examples of $\{t_0,\ldots,t_n\}$ with $n=4$ and $p_0=3$}
\label{tab_Ii}
\end{table}

Then we have
\begin{align*}
&\int_{B(x_0,R')}\lvert f\rvert^p \mathrm{d}\Gamma^{(J)}_{B(x_0,R')}(\varphi_n)\\
&=\left(\sum_{{\{t_0,\ldots,t_n\}:\text{connected}}}+\sum_{{\{t_0,\ldots,t_n\}:\text{disconnected}}}\right)\\
&\hspace{20pt}\int_{B(x_0,R')}\int_{B(x_0,R')}\lvert f(x)\rvert^p \prod_{k=0}^n\left(b_k\lvert \phi_k(x)-\phi_k(y)\rvert\right)^{{t_k}} \\
&\hspace{100pt}\cdot\left(\sum_{k=0}^n b_k \lvert \phi_k(x)-\phi_k(y)\rvert\right)^{\delta_0}K^{(J)}(x,y) m(\mathrm{d}x) m(\mathrm{d}y)\\
&\le I_5+I_6,
\end{align*}
where
\begin{align*}
&I_5=\int_{B(x_0,R')}\int_{B(x_0,R')}\lvert f(x)\rvert^p\left(\sum_{{\{t_0,\ldots,t_n\}:\text{connected}}} \prod_{k=0}^n\left(b_k\lvert \phi_k(x)-\phi_k(y)\rvert\right)^{{t_k}}\right)\\
&\hspace{100pt}\cdot\left(\sum_{k=0}^n \left(b_k \lvert \phi_k(x)-\phi_k(y)\rvert\right)^{\delta_0}\right)K^{(J)}(x,y) m(\mathrm{d}x) m(\mathrm{d}y),
\end{align*}
and
\begin{align*}
&I_6=\sum_{{\{t_0,\ldots,t_n\}:\text{disconnected}}}\int_{B(x_0,R')}\int_{B(x_0,R')}\lvert f(x)\rvert^p \prod_{k=0}^n\left(b_k\lvert \phi_k(x)-\phi_k(y)\rvert\right)^{{t_k}}\\
&\hspace{100pt}\cdot K^{(J)}(x,y) m(\mathrm{d}x) m(\mathrm{d}y).
\end{align*}
By Young's inequality, we have
\begin{align*}
&\sum_{\{t_0,\ldots,t_n\}:\text{connected}}\prod_{k=0}^n\left(b_k\lvert \phi_k(x)-\phi_k(y)\rvert\right)^{{t_k}}\\
&\le\sum_{\{t_0,\ldots,t_n\}:\text{connected}}\sum_{k=0}^n \frac{t_k}{p_0}\left(b_k\lvert \phi_k(x)-\phi_k(y)\rvert\right)^{p_0}\\
&\le \sum_{\{t_0,\ldots,t_n\}:\text{connected}}\sum_{k\in I(t_0,\ldots,t_n)} \left(b_k\lvert \phi_k(x)-\phi_k(y)\rvert\right)^{p_0}.
\end{align*}
For any fixed $k=0,\ldots,n$, there exist at most $2^{p_0-1}$ connected sequences $\{t_0,\ldots,t_n\}$ such that $k\in I(t_0,\ldots,t_n)$, hence
\begin{align*}
\sum_{\{t_0,\ldots,t_n\}:\text{connected}}\sum_{k\in I(t_0,\ldots,t_n)} \left(b_k\lvert \phi_k(x)-\phi_k(y)\rvert\right)^{p_0}\le2^{p_0-1}\sum_{k=0}^n \left(b_k\lvert \phi_k(x)-\phi_k(y)\rvert\right)^{p_0}.
\end{align*}
Thus we have
\begin{align*}
&I_5\le2^{p_0-1}\int_{B(x_0,R')}\int_{B(x_0,R')}\lvert f(x)\rvert^p\left(\sum_{k=0}^n\left(b_k\lvert \phi_k(x)-\phi_k(y)\rvert\right)^{{p_0}}\right)\\
&\hspace{100pt}\cdot\left(\sum_{k=0}^n \left(b_k \lvert \phi_k(x)-\phi_k(y)\rvert\right)^{\delta_0}\right)K^{(J)}(x,y) m(\mathrm{d}x) m(\mathrm{d}y)\\
&=2^{p_0-1}(I_{51}+I_{52}+I_{53}),
\end{align*}
where
\begin{align*}
I_{51}&=\int_{B(x_0,R')}\int_{B(x_0,R')}\lvert f(x)\rvert^p\left(\sum_{k=0}^n\left(b_k\lvert \phi_k(x)-\phi_k(y)\rvert\right)^{{p}}\right)K^{(J)}(x,y) m(\mathrm{d}x) m(\mathrm{d}y),\\
I_{52}&=\int_{B(x_0,R')}\int_{B(x_0,R')}\lvert f(x)\rvert^p\\
&\hspace{15pt}\cdot\left(\sum_{\substack{k,l=0\\|k-l|=1}}^n\left(b_k\lvert \phi_k(x)-\phi_k(y)\rvert\right)^{{p_0}}\left(b_l \lvert \phi_l(x)-\phi_l(y)\rvert\right)^{\delta_0}\right)K^{(J)}(x,y) m(\mathrm{d}x) m(\mathrm{d}y),\\
I_{53}&=\int_{B(x_0,R')}\int_{B(x_0,R')}\lvert f(x)\rvert^p\\
&\hspace{15pt}\cdot\left(\sum_{\substack{k,l=0\\|k-l|\ge2}}^n\left(b_k\lvert \phi_k(x)-\phi_k(y)\rvert\right)^{{p_0}}\left(b_l \lvert \phi_l(x)-\phi_l(y)\rvert\right)^{\delta_0}\right)K^{(J)}(x,y) m(\mathrm{d}x) m(\mathrm{d}y).
\end{align*}

For $I_{51}$, by (\ref{eq_CSJ_self2}), we have
\begin{align*}
&I_{51}=\sum_{k=0}^n b_k^p\int_{B(x_0,R')}\lvert f\rvert^p \mathrm{d}\Gamma^{(J)}_{B(x_0,R')}(\phi_k)\\
&\le \sum_{k=0}^n b_k^p \left(C\int_{B_{k+1}\backslash \overline{B_k}}\mathrm{d}\Gamma^{(J)}_{B_{k+1}\backslash \overline{B_k}}(f)+\frac{C}{\Upsilon(s_{k+1})}\int_{B(x_0,R')}\lvert f\rvert^p \mathrm{d}m\right)\\
&\overset{(*)}{\scalebox{2}[1]{$\le$}} C \left(\frac{1-b}{b}\right)^p\sum_{k=0}^n\int_{B_{k+1}\backslash \overline{B_k}}\lvert \phi\rvert^p \mathrm{d}\Gamma^{(J)}_{B(x_0,R+r)\backslash \overline{B(x_0,R)}}(f)\\
&\hspace{10pt}+C\sum_{k=0}^n \frac{b_k^p}{\Upsilon(s_{k+1})}\int_{B(x_0,R')}\lvert f\rvert^p \mathrm{d}m\\
&\le C \left(\frac{1-b}{b}\right)^p\int_{B(x_0,R+r)\backslash \overline{B(x_0,R)}}\lvert \phi\rvert^p \mathrm{d}\Gamma^{(J)}_{B(x_0,R+r)\backslash \overline{B(x_0,R)}}(f)\\
&\hspace{10pt}+C_4\sum_{k=0}^{+\infty} \frac{b^{pk}}{\Upsilon(a^kr)}\int_{B(x_0,R')}\lvert f\rvert^p \mathrm{d}m,
\end{align*}
where $(*)$ follows from the fact that $\phi\ge b^{k+1}$ in $B_{k+1}$, and $C_4>0$ is some constant.

For $I_{52}$, by Young's inequality, we have
\begin{align*}
&I_{52}\le\int_{B(x_0,R')}\int_{B(x_0,R')}\lvert f(x)\rvert^p\\
&\hspace{15pt}\cdot\left(\sum_{\substack{k,l=0\\|k-l|=1}}^n\left(\frac{p_0}{p}\left(b_k\lvert \phi_k(x)-\phi_k(y)\rvert\right)^{{p}}+\frac{\delta_0}{p}\left(b_l \lvert \phi_l(x)-\phi_l(y)\rvert\right)^{p}\right)\right)K^{(J)}(x,y) m(\mathrm{d}x) m(\mathrm{d}y)\\
&\le 4\int_{B(x_0,R')}\int_{B(x_0,R')}\lvert f(x)\rvert^p\left(\sum_{k=0}^n\left(b_k\lvert \phi_k(x)-\phi_k(y)\rvert\right)^{{p}}\right)K^{(J)}(x,y) m(\mathrm{d}x) m(\mathrm{d}y)=4I_{51}.
\end{align*}

For $I_{53}$, for any $l\ge k+2$, by \ref{eq_KJ_tail}, we have
\begin{align*}
&\int_{B(x_0,R')}\int_{B(x_0,R')}\lvert f(x)\rvert^p\left(b_k\lvert \phi_k(x)-\phi_k(y)\rvert\right)^{{p_0}}\left(b_l \lvert \phi_l(x)-\phi_l(y)\rvert\right)^{\delta_0}K^{(J)}(x,y) m(\mathrm{d}x) m(\mathrm{d}y)\\
&=b_k^{p_0}b_l^{\delta_0}\int_{B_{k+1}}\int_{B(x_0,R')\backslash B_{k+2}}\lvert f(x)\rvert^p \phi_k(x)^p(1-\phi_l(y))^{\delta_0}K^{(J)}(x,y)m(\mathrm{d}x) m(\mathrm{d}y)\\
&\hspace{10pt}+b_k^{p_0}b_l^{\delta_0}\int_{B(x_0,R')\backslash B_{k+2}}\int_{B_{k+1}}\lvert f(x)\rvert^p \phi_k(y)^p(1-\phi_l(x))^{\delta_0}K^{(J)}(x,y)m(\mathrm{d}x) m(\mathrm{d}y)\\
&\le b_k^{p_0}b_l^{\delta_0}\int_{B(x_0,R')}\lvert f(x)\rvert^p \left(\int_{X\backslash B(x,s_{k+2})}K^{(J)}(x,y)m(\mathrm{d}y)\right) m(\mathrm{d}x)\\
&\lesssim\frac{ b_k^{p_0}b_l^{\delta_0}}{\Upsilon(s_{k+2})}\int_{B(x_0,R')}\lvert f\rvert^p \mathrm{d}m.
\end{align*}
Taking the sum over $k,l$ with $l\ge k+2$, we have
\begin{align*}
\sum_{\substack{k,l=0\\l\ge k+2}}^n\int_{B(x_0,R')}\int_{B(x_0,R')}\ldots \mathrm{d}m \mathrm{d}m\lesssim \sum_{k=0}^n \frac{ b_k^{p}}{\Upsilon(s_{k+2})}\int_{B(x_0,R')}\lvert f\rvert^p \mathrm{d}m.
\end{align*}
Similarly, we also have
\begin{align*}
\sum_{\substack{k,l=0\\k\ge l+2}}^n\int_{B(x_0,R')}\int_{B(x_0,R')}\ldots \mathrm{d}m \mathrm{d}m\lesssim \sum_{l=0}^n \frac{ b_l^{p}}{\Upsilon(s_{l+2})}\int_{B(x_0,R')}\lvert f\rvert^p \mathrm{d}m,
\end{align*}
which gives
\begin{align*}
I_{53}\lesssim \sum_{k=0}^n \frac{ b_k^{p}}{\Upsilon(s_{k+2})}\int_{B(x_0,R')}\lvert f\rvert^p \mathrm{d}m\lesssim \sum_{k=0}^{+\infty} \frac{ b^{pk}}{\Upsilon(a^kr)}\int_{B(x_0,R')}\lvert f\rvert^p \mathrm{d}m.
\end{align*}
Hence
\begin{align}
&I_5\le 2^{p_0-1}(I_{51}+I_{52}+I_{53})\le 2^{p_0-1}(I_{51}+4I_{51}+I_{53})\nonumber\\
&\le 5\cdot2^{p_0-1}C \left(\frac{1-b}{b}\right)^p\int_{B(x_0,R+r)\backslash \overline{B(x_0,R)}}\lvert \phi\rvert^p \mathrm{d}\Gamma^{(J)}_{B(x_0,R+r)\backslash \overline{B(x_0,R)}}(f)\nonumber\\
&\hspace{15pt}+C_5 \sum_{k=0}^{+\infty} \frac{ b^{pk}}{\Upsilon(a^kr)}\int_{B(x_0,R')}\lvert f\rvert^p \mathrm{d}m,\label{eq_CSJ_self3}
\end{align}
where $C_5>0$ is some constant.

We now consider $I_6$. For any fixed $k=0,\ldots,n-2$ and any fixed sequence of non-negative integers $\{s_0,\ldots,s_k\}$ satisfying that
\begin{equation}\label{eq_cond_sk}
\sum_{i=0}^k s_i\le p_0-1 \text{ and there exists }l\le k\text{ such that }s_i=0\text{ if and only if }i<l,
\end{equation}
we have
\begin{align*}
&\sum_{\substack{{\{t_0,\ldots,t_n\}:\text{disconnected}}\\i(t_0,\ldots,t_n)=k\\ \{t_0,\ldots,t_k\}=\{s_0,\ldots,s_k\}}}\prod_{i=0}^n\left(b_i\lvert \phi_i(x)-\phi_i(y)\rvert\right)^{{t_i}}\\
&=\prod_{i=0}^k\left(b_i\lvert \phi_i(x)-\phi_i(y)\rvert\right)^{{s_i}}\sum_{\substack{{\{t_0,\ldots,t_n\}:\text{disconnected}}\\i(t_0,\ldots,t_n)=k\\ \{t_0,\ldots,t_k\}=\{s_0,\ldots,s_k\}}}\prod_{i=k+2}^n\left(b_i\lvert \phi_i(x)-\phi_i(y)\rvert\right)^{{t_i}}\\
&=\prod_{i=0}^k\left(b_i\lvert \phi_i(x)-\phi_i(y)\rvert\right)^{{s_i}} \left(\sum_{i=k+2}^n \left(b_i \lvert \phi_i(x)-\phi_i(y)\rvert\right)\right)^{p_0-\sum_{i=0}^ks_i}.
\end{align*}
By \ref{eq_KJ_tail}, we have
\begin{align*}
&\sum_{\substack{{\{t_0,\ldots,t_n\}:\text{disconnected}}\\i(t_0,\ldots,t_n)=k\\ \{t_0,\ldots,t_k\}=\{s_0,\ldots,s_k\}}}\int_{B(x_0,R')}\int_{B(x_0,R')}\lvert f(x)\rvert^p \prod_{k=0}^n\left(b_k\lvert \phi_k(x)-\phi_k(y)\rvert\right)^{{t_k}}\\
&\hspace{100pt}\cdot K^{(J)}(x,y) m(\mathrm{d}x) m(\mathrm{d}y)\\
&=\int_{B(x_0,R')}\int_{B(x_0,R')}\lvert f(x)\rvert^p\prod_{i=0}^k\left(b_i\lvert \phi_i(x)-\phi_i(y)\rvert\right)^{{s_i}} \left(\sum_{i=k+2}^n \left(b_i \lvert \phi_i(x)-\phi_i(y)\rvert\right)\right)^{p_0-\sum_{i=0}^ks_i}\\
&\hspace{100pt}\cdot K^{(J)}(x,y) m(\mathrm{d}x) m(\mathrm{d}y)\\
&=\int_{B_{k+1}}\int_{B(x_0,R')\backslash B_{k+2}}\lvert f(x)\rvert^p\prod_{i=0}^k\left(b_i\phi_i(x)\right)^{{s_i}} \left(\sum_{i=k+2}^n \left(b_i (1-\phi_i(y))\right)\right)^{p_0-\sum_{i=0}^ks_i}\\
&\hspace{100pt}\cdot K^{(J)}(x,y) m(\mathrm{d}x) m(\mathrm{d}y)\\
&\hspace{15pt}+\int_{B(x_0,R')\backslash B_{k+2}}\int_{B_{k+1}}\lvert f(x)\rvert^p\prod_{i=0}^k\left(b_i\phi_i(y)\right)^{{s_i}} \left(\sum_{i=k+2}^n \left(b_i (1-\phi_i(x))\right)\right)^{p_0-\sum_{i=0}^ks_i}\\
&\hspace{100pt}\cdot K^{(J)}(x,y) m(\mathrm{d}x) m(\mathrm{d}y)\\
&\le\prod_{i=0}^kb_i^{{s_i}} \left(\sum_{i=k+2}^n b_i \right)^{p_0-\sum_{i=0}^ks_i}\int_{B(x_0,R')}\lvert f(x)\rvert^p \left(\int_{X\backslash B(x,s_{k+2})} K^{(J)}(x,y) m(\mathrm{d}y)\right) m(\mathrm{d}x)\\
&\lesssim b^{(k-\sum_{i=0}^ks_i)\sum_{i=0}^ks_i} b^{k(p_0-\sum_{i=0}^ks_i)}\frac{1}{\Upsilon(s_{k+2})}\int_{B(x_0,R')}\lvert f\rvert^p \mathrm{d}m\\
&=b^{p_0k-\left(\sum_{i=0}^ks_i\right)^2}\frac{1}{\Upsilon(s_{k+2})}\int_{B(x_0,R')}\lvert f\rvert^p \mathrm{d}m\lesssim \frac{b^{p_0k}}{\Upsilon(a^kr)}\int_{B(x_0,R')}\lvert f\rvert^p \mathrm{d}m.
\end{align*}
For fixed $k$, there exist at most $2^{p_0-2}$ sequences $\{s_0,\ldots,s_k\}$ satisfying (\ref{eq_cond_sk}). Hence, taking the sum over $k$ and $\{s_0,\ldots,s_k\}$, we have
\begin{align*}
&I_6\le C_6\sum_{k=0}^{n-2}\frac{b^{p_0k}}{\Upsilon(a^kr)}\int_{B(x_0,R')}\lvert f\rvert^p \mathrm{d}m\le C_6\sum_{k=0}^{+\infty}\frac{b^{p_0k}}{\Upsilon(a^kr)}\int_{B(x_0,R')}\lvert f\rvert^p \mathrm{d}m,
\end{align*}
where $C_6>0$ is some constant. Combining this with (\ref{eq_CSJ_self3}), we have
\begin{align*}
&\int_{B(x_0,R')}\lvert f\rvert^p \mathrm{d}\Gamma^{(J)}_{B(x_0,R')}(\varphi_n)\le I_5+I_6\\
&\le 5\cdot2^{p_0-1}C \left(\frac{1-b}{b}\right)^p\int_{B(x_0,R+r)\backslash \overline{B(x_0,R)}}\lvert \phi\rvert^p \mathrm{d}\Gamma^{(J)}_{B(x_0,R+r)\backslash \overline{B(x_0,R)}}(f)\\
&\hspace{15pt}+(C_5+C_6)\sum_{k=0}^{+\infty}\frac{b^{p_0k}}{\Upsilon(a^kr)}\int_{B(x_0,R')}\lvert f\rvert^p \mathrm{d}m,
\end{align*}
where by (\ref{eq_beta12}), we have
$$\sum_{k=0}^{+\infty}\frac{b^{p_0k}}{\Upsilon(a^kr)}\le \frac{C_\Upsilon}{\Upsilon(r)}\sum_{k=0}^{+\infty}\left(\frac{b^{p_0}}{a^{\beta^{(2)}_\Upsilon}}\right)^k.$$

For any $\varepsilon>0$, first choose $b\in(0,1)$ sufficiently close to 1 so that $5\cdot2^{p_0-1}C \left(\frac{1-b}{b}\right)^p<\varepsilon$; second, since
$$\frac{b^{p_0}}{a^{\beta^{(2)}_\Upsilon}}\to b^{p_0}<1\text{ as }a\uparrow1,$$
there exists $a\in(0,1)$, depending only on $b,p_0,\beta^{(2)}_\Upsilon$ such that $\frac{b^{p_0}}{a^{\beta^{(2)}_\Upsilon}}<1$, then $C_7=\sum_{k=0}^{+\infty}\left(\frac{b^{p_0}}{a^{\beta^{(2)}_\Upsilon}}\right)^k<+\infty$. In summary, in case \ref{item_pd1}, we have
\begin{align*}
&\int_{B(x_0,R')}\lvert f\rvert^p \mathrm{d}\Gamma^{(J)}_{B(x_0,R')}(\varphi_n)\\
&\le \varepsilon \int_{B(x_0,R+r)\backslash \overline{B(x_0,R)}}\lvert \phi\rvert^p \mathrm{d}\Gamma^{(J)}_{B(x_0,R+r)\backslash \overline{B(x_0,R)}}(f)+\frac{C_\varepsilon}{\Upsilon(r)}\int_{B(x_0,R')}\lvert f\rvert^p \mathrm{d}m,
\end{align*}
where $C_\varepsilon=C_\Upsilon(C_5+C_6)C_7$. Compared with case \ref{item_pd1}, in case \ref{item_pd2}, the terms $I_{52}$ and $I_{53}$ do not appear, and in case \ref{item_pd3}, the term $I_6$ does not appear; in both cases, the argument still applies.

Finally, if \hyperlink{eq_CSJ_cont}{$\text{CS}^{(J)}_{\text{cont}}(\Upsilon)$} holds, then by the construction of $\phi$ as in (\ref{eq_phi_self1}) and (\ref{eq_phi_self2}), we can take $\phi\in \mathcal{F}^{(J)}\cap C_c(X)$.
\end{proof}

\section{Proof of Proposition \ref{prop_regular}}\label{sec_regular}

In this section, we give the proof of Proposition \ref{prop_regular} following the ideas in \cite{CGHL25}. Throughout this section, we always assume that $(\mathcal{E}^{(J)},\mathcal{F}^{(J)})$ is a non-local $p$-form given by a kernel $K^{(J)}$. It suffices to prove the following result.

\begin{proposition}\label{prop_dense}
Assume \ref{eq_VD}, \ref{eq_KJ_tail}, \ref{eq_PIJ}, \hyperlink{eq_CSJ_cont}{$\text{CS}^{(J)}_{\text{cont}}(\Upsilon)$}. Then
\begin{enumerate}[label=(\arabic*),ref=(\arabic*)]
\item\label{item_dense_Cc} $\mathcal{F}^{(J)}\cap C_c(X)$ is uniformly dense in $C_c(X)$.
\item\label{item_dense_Lp} $\mathcal{F}^{(J)}$ is $L^p$-dense in $L^p(X;m)$.
\item\label{item_dense_FJ} $\mathcal{F}^{(J)}\cap C_c(X)$ is $\mathcal{E}^{(J)}_1$-dense in $\mathcal{F}^{(J)}$.
\end{enumerate}
\end{proposition}

The proofs of \ref{item_dense_Cc} and \ref{item_dense_Lp} are relatively direct as follows.

\begin{proof}[Proofs of Proposition \ref{prop_dense} \ref{item_dense_Cc} and \ref{item_dense_Lp}]
It is obvious that $\mathcal{F}^{(J)}\cap C_c(X)$ is a sub-algebra of $C_c(X)$. By the Stone-Weierstrass theorem, to show that $\mathcal{F}^{(J)}\cap C_c(X)$ is uniformly dense in $C_c(X)$, we only need to show that $\mathcal{F}^{(J)}\cap C_c(X)$ separates points and vanishes nowhere. Indeed, for any distinct $x,y\in X$, let $R=d(x,y)\in(0,+\infty)$, by Proposition \ref{prop_CSJ_self}, there exists a cutoff function $\phi\in \mathcal{F}^{(J)}\cap C_c(X)$ for $B(x,\frac{1}{4}R)\subseteq B(x,\frac{1}{2}R)$, in particular, we have $\phi(x)=1\ne0=\phi(y)$, thus we have \ref{item_dense_Cc}.

For any $u\in L^p(X;m)$ and any $\varepsilon>0$. First, there exists $v\in C_c(X)$ such that $\lVert u-v\rVert_{L^p(X;m)}<\varepsilon$. Second, take $x_0\in X$ and $R>0$ such that $\mathrm{supp}(v)\subseteq B(x_0,R)$. By \ref{item_dense_Cc}, there exists $w\in \mathcal{F}^{(J)}\cap C_c(X)$ such that $\sup_X \lvert v-w\rvert<\frac{\varepsilon}{V(x_0,2R)^{\frac{1}{p}}}$. By Proposition \ref{prop_CSJ_self}, there exists a cutoff function $\psi\in \mathcal{F}^{(J)}\cap C_c(X)$ for $B(x_0,R)\subseteq B(x_0,2R)$, then $\psi w\in \mathcal{F}^{(J)}\cap C_c(X)$ and
$$\int_X \lvert \psi w-v\rvert^p \mathrm{d}m\le\int_{B(x_0,R)}\lvert w-v\rvert^p \mathrm{d}m+\int_{B(x_0,2R)\backslash B(x_0,R)}\lvert w\rvert^p \mathrm{d}m\le \varepsilon^p,$$
hence $\lVert \psi w-u\rVert_{L^p(X;m)}\le\lVert \psi w-v\rVert_{L^p(X;m)}+\lVert u-v\rVert_{L^p(X;m)}<2\varepsilon$, which gives \ref{item_dense_Lp}.
\end{proof}

We make some preparations in order to prove \ref{item_dense_FJ}.

\begin{lemma}\label{lem_cptspt}
Assume \ref{eq_VD}, \ref{eq_KJ_tail}, \hyperlink{eq_CSJ_weak}{$\text{CS}^{(J)}_{\text{weak}}(\Upsilon)$}. Then
$$\left\{u\in \mathcal{F}^{(J)}\cap L^\infty(X;m):\mathrm{supp}(u)\text{ is compact}\right\}$$
is $\mathcal{E}^{(J)}_1$-dense in $\mathcal{F}^{(J)}$.
\end{lemma}

\begin{proof}
For any $u\in \mathcal{F}^{(J)}$ and any $M\ge1$, we have $(u\vee (-M))\wedge M\in \mathcal{F}^{(J)}$ and by the dominated convergence theorem, we have $\lim_{M\to+\infty}\mathcal{E}_1^{(J)}(u-(u\vee (-M))\wedge M)=0$. Hence we only need to show that any $u\in \mathcal{F}^{(J)}\cap L^\infty(X;m)$ can be $\mathcal{E}^{(J)}_1$-approximated by functions in $\mathcal{F}^{(J)}\cap L^\infty(X;m)$ with compact support.

Fix arbitrary $x_0\in X$. By Proposition \ref{prop_CSJ_self}, there exist $C_1,C_2>0$ such that for any $R>0$, there exists a cutoff function $\phi_R\in \mathcal{F}^{(J)}$ for $B(x_0,R)\subseteq B(x_0,2R)$ such that
\begin{align}
&\int_{B(x_0,3R)}\lvert u\rvert^p \mathrm{d}\Gamma^{(J)}_{B(x_0,3R)}(\phi_R)\nonumber\\
&\le C_1\int_{B(x_0,2R)\backslash \overline{B(x_0,R)}} \mathrm{d}\Gamma^{(J)}_{B(x_0,2R)\backslash \overline{B(x_0,R)}}(u)+\frac{C_2}{\Upsilon(R)}\int_{B(x_0,3R)}\lvert u\rvert^p \mathrm{d}m.\label{eq_cptspt1}
\end{align}
Then $\phi_Ru\in \mathcal{F}^{(J)}\cap L^\infty(X;m)$ has compact support. By the dominated convergence theorem, we have $\lim_{R\to+\infty}\lVert u-\phi_Ru\rVert_{L^p(X;m)}=0$. Moreover, we have
\begin{align*}
&\mathcal{E}^{(J)}(u-\phi_Ru)=\int_{X}\int_X \lvert (1-\phi_R(x))u(x)-(1-\phi_R(y))u(y)\rvert^p K^{(J)}(x,y) m(\mathrm{d}x) m(\mathrm{d}y)\\
&\le2^{p-1}\int_X\int_X \lvert u(x)\rvert^p \lvert \phi_R(x)-\phi_R(y)\rvert^p K^{(J)}(x,y) m(\mathrm{d}x)m(\mathrm{d}y)\\
&\hspace{15pt}+2^{p-1}\int_X\int_X (1-\phi_R(y))^p\lvert u(x)-u(y)\rvert^pK^{(J)}(x,y) m(\mathrm{d}x)m(\mathrm{d}y)\\
&=2^{p-1}\left(I_1+I_2\right).
\end{align*}
For $I_1$, we have
\begin{align*}
&I_1=\left(\int_{B(x_0,3R)}\int_{B(x_0,3R)}+\int_{B(x_0,3R)}\int_{X\backslash B(x_0,3R)}+\int_{X\backslash B(x_0,3R)}\int_{B(x_0,3R)}\right)\ldots m(\mathrm{d}x)m(\mathrm{d}y)\\
&=I_{11}+I_{12}+I_{13}.
\end{align*}
By (\ref{eq_cptspt1}) and the dominated convergence theorem, we have
\begin{align*}
&I_{11}=\int_{B(x_0,3R)}\lvert u\rvert^p \mathrm{d}\Gamma^{(J)}_{B(x_0,3R)}(\phi_R)\\
&\le C_1\int_{B(x_0,2R)\backslash \overline{B(x_0,R)}} \mathrm{d}\Gamma^{(J)}_{B(x_0,2R)\backslash \overline{B(x_0,R)}}(u)+\frac{C_2}{\Upsilon(R)}\int_{X}\lvert u\rvert^p \mathrm{d}m\to0
\end{align*}
as $R\to+\infty$. By \ref{eq_KJ_tail}, we have
\begin{align*}
&I_{12}=\int_{B(x_0,2R)}\int_{X\backslash B(x_0,3R)}\lvert u(x)\rvert^p\phi_R(x)^p K^{(J)}(x,y) m(\mathrm{d}x)m(\mathrm{d}y)\\
&\le\int_{B(x_0,2R)}\lvert u(x)\rvert^p\left(\int_{X\backslash B(x,R)} K^{(J)}(x,y) m(\mathrm{d}y)\right)m(\mathrm{d}x)\\
&\lesssim \frac{1}{\Upsilon(R)}\int_{B(x_0,2R)}\lvert u\rvert^p \mathrm{d}m\le \frac{1}{\Upsilon(R)}\int_X \lvert u\rvert^p \mathrm{d}m\to0
\end{align*}
as $R\to+\infty$, and
\begin{align*}
&I_{13}=\int_{X\backslash B(x_0,3R)}\int_{B(x_0,2R)}\lvert u(x)\rvert^p\phi_R(y)^p K^{(J)}(x,y)m(\mathrm{d}x)m(\mathrm{d}y)\\
&\le\int_{X\backslash B(x_0,3R)}\lvert u(x)\rvert^p\left(\int_{X\backslash B(x,R)} K^{(J)}(x,y)m(\mathrm{d}y)\right)m(\mathrm{d}x)\\
&\lesssim \frac{1}{\Upsilon(R)}\int_{X\backslash B(x_0,3R)}\lvert u\rvert^p \mathrm{d}m\le \frac{1}{\Upsilon(R)}\int_X \lvert u\rvert^p \mathrm{d}m\to0
\end{align*}
as $R\to+\infty$. Hence $\lim_{R\to+\infty}I_1=0$. By the dominated convergence theorem, we have
$$I_2\le\int_X\int_{X\backslash B(x_0,R)}\lvert u(x)-u(y)\rvert^pK^{(J)}(x,y) m(\mathrm{d}x)m(\mathrm{d}y)\to0$$
as $R\to+\infty$. In summary, we have $\lim_{R\to+\infty}\mathcal{E}^{(J)}_1(u-\phi_Ru)=0$.
\end{proof}

We have the following partition of unity.

\begin{lemma}\label{lem_partitionJ}
Assume \ref{eq_VD}, \ref{eq_KJ_tail}, \hyperlink{eq_CSJ_cont}{$\text{CS}^{(J)}_{\text{cont}}(\Upsilon)$}. For any $\varepsilon>0$, there exists a countable subset $V\subseteq X$ satisfying the following conditions:
\begin{enumerate}[label=(\alph*),ref=(\alph*)]
\item\label{item_cover} $\cup_{v\in V}B(v,\varepsilon)=X$.
\item\label{item_disjoint} For any distinct $v,w\in V$, we have $B(v,\frac{1}{2}\varepsilon)\cap B(w,\frac{1}{2}\varepsilon)=\emptyset$.
\end{enumerate}
Moreover, there exist $C_1,C_2>0$, independent of $\varepsilon$, and a family of functions $\{\phi_v\in \mathcal{F}^{(J)}\cap C_c(X):v\in V\}$ satisfying the following conditions:
\begin{enumerate}[label=(\arabic*),ref=(\arabic*)]
\item\label{item_phiv_spt} For any $v\in V$, we have $0\le \phi_v\le 1$ in $X$ and $\phi_v=0$ on $X\backslash B(v,\frac{5}{4}\varepsilon)$.
\item\label{item_phiv_unit} $\sum_{v\in V}\phi_v=1$.
\item\label{item_phiv_CS} For any $v\in V$, let
$$N_v=\{w\in V:B(w,3\varepsilon)\cap B(v,3\varepsilon)\ne\emptyset\},$$
then for any $f\in \widehat{\mathcal{F}}^{(J)}\cap L^\infty(X;m)$, we have
\begin{align*}
&\int_{B(v,3\varepsilon)}\lvert f\rvert^p \mathrm{d}\Gamma^{(J)}_{B(v,3\varepsilon)}(\phi_v)\le C_1\sum_{w\in N_v}\int_{B(w,3\varepsilon)}\mathrm{d}\Gamma^{(J)}_{B(w,3\varepsilon)}(f)+\frac{C_2}{\Upsilon(\varepsilon)}\int_{B(v,6\varepsilon)}\lvert f\rvert^p \mathrm{d}m.
\end{align*}
\end{enumerate}
\end{lemma}

\begin{proof}
By Zorn's lemma, there exists a subset $V\subseteq X$ satisfying \ref{item_cover} and \ref{item_disjoint}. Since $X$ is separable, we have $V$ is countable. By \ref{eq_VD}, there exists a positive integer $N$ depending only on $C_{VD}$ such that $\# N_v\le N$ for any $v\in V$. By Proposition \ref{prop_CSJ_self}, there exists $C>0$, independent of $\varepsilon$, such that for any $v\in V$, there exists a cutoff function $\psi_v\in \mathcal{F}^{(J)}\cap C_c(X)$ for $B(v,\varepsilon)\subseteq B(v,\frac{5}{4}\varepsilon)$ such that for any $f\in \widehat{\mathcal{F}}^{(J)}\cap L^\infty(X;m)$, we have
\begin{equation}\label{eq_psiv_CS}
\int_{B(v,3\varepsilon)}\lvert f\rvert^p \mathrm{d}\Gamma^{(J)}_{B(v,3\varepsilon)}(\psi_v)\le \int_{B(v,3\varepsilon)} \mathrm{d}\Gamma^{(J)}_{B(v,3\varepsilon)}(f)+\frac{C}{\Upsilon(\varepsilon)}\int_{B(v,3\varepsilon)}\lvert f\rvert^p \mathrm{d}m.
\end{equation}

In any bounded open set, the summation $\sum_{v\in V}\psi_v$ consists of finitely many non-zero terms and $1\le \sum_{v\in V}\psi_v\le N$. For any $v\in V$, let
$$\phi_v=\frac{\psi_v}{\sum_{w\in V}\psi_w},$$
then $\phi_v\in C_c(X)$, \ref{item_phiv_spt} and \ref{item_phiv_unit} follow directly, in particular, we have
$$\phi_v=\frac{\psi_v}{\sum_{w\in N_v}\psi_w}\text{ in }B(v,3\varepsilon),$$
and for any $x,y\in B(v,3\varepsilon)$, we have
\begin{align*}
&\lvert \phi_v(x)-\phi_v(y)\rvert=\lvert \frac{\psi_v(x)}{\sum_{w\in N_v}\psi_w(x)}-\frac{\psi_v(y)}{\sum_{w\in N_v}\psi_w(y)}\rvert\\
&\le \frac{\lvert \psi_v(x)-\psi_v(y)\rvert}{\sum_{w\in N_v}\psi_w(x)}+\psi_v(y)\frac{\sum_{w\in N_v}\lvert \psi_w(x)-\psi_w(y)\rvert}{\sum_{w\in N_v}\psi_w(x)\sum_{w\in N_v}\psi_w(y)}\\
&\le\lvert \psi_v(x)-\psi_v(y)\rvert+\sum_{w\in N_v}\lvert \psi_w(x)-\psi_w(y)\rvert,
\end{align*}
which gives
$$\lvert \phi_v(x)-\phi_v(y)\rvert^p \le (N+1)^{p-1}\left(\lvert \psi_v(x)-\psi_v(y)\rvert^p+\sum_{w\in N_v}\lvert \psi_w(x)-\psi_w(y)\rvert^p\right).$$
Hence
\begin{align*}
&\int_{B(v,3\varepsilon)}\int_{B(v,3\varepsilon)}\lvert f(x)\rvert^p \lvert \phi_v(x)-\phi_v(y)\rvert^p K^{(J)}(x,y)m(\mathrm{d}x)m(\mathrm{d}y)\\
&\lesssim\int_{B(v,3\varepsilon)}\int_{B(v,3\varepsilon)}\lvert f(x)\rvert^p \lvert \psi_v(x)-\psi_v(y)\rvert^p K^{(J)}(x,y)m(\mathrm{d}x)m(\mathrm{d}y)\\
&\hspace{15pt}+\sum_{w\in N_v}\int_{B(v,3\varepsilon)}\int_{B(v,3\varepsilon)}\lvert f(x)\rvert^p \lvert \psi_w(x)-\psi_w(y)\rvert^pK^{(J)}(x,y)m(\mathrm{d}x)m(\mathrm{d}y)\\
&=I+\sum_{w\in N_v}I_w.
\end{align*}
By (\ref{eq_psiv_CS}), we have
$$I\lesssim\int_{B(v,3\varepsilon)}\mathrm{d}\Gamma^{(J)}_{B(v,3\varepsilon)}(f)+\frac{1}{\Upsilon(\varepsilon)}\int_{B(v,3\varepsilon)}\lvert f\rvert^p \mathrm{d}m.$$
For any $w\in N_v$, we have
$$I_{w}=\left(\iint\limits_{\substack{B(v,3\varepsilon)\times B(v,3\varepsilon)\\d(x,y)<\frac{1}{4}\varepsilon}}+\iint\limits_{\substack{B(v,3\varepsilon)\times B(v,3\varepsilon)\\ d(x,y)\ge\frac{1}{4}\varepsilon}}\right)\ldots m(\mathrm{d}x)m(\mathrm{d}y)=I_{w1}+I_{w2}.$$
For any $x,y\in B(v,3\varepsilon)$ with $d(x,y)<\frac{1}{4}\varepsilon$, we have $\psi_w(x)-\psi_w(y)\ne0$ only if $x,y\in B(w,\frac{3}{2}\varepsilon)\subseteq B(w,3\varepsilon)$, hence by (\ref{eq_psiv_CS}), we have
\begin{align*}
&I_{w1}\le\int_{B(w,3\varepsilon)}\int_{B(w,3\varepsilon)}\lvert f(x)\rvert^p \lvert \psi_w(x)-\psi_w(y)\rvert^pK^{(J)}(x,y)m(\mathrm{d}x)m(\mathrm{d}y)\\
&\lesssim\int_{B(w,3\varepsilon)}\mathrm{d}\Gamma^{(J)}_{B(w,3\varepsilon)}(f)+\frac{1}{\Upsilon(\varepsilon)}\int_{B(w,3\varepsilon)}\lvert f\rvert^p \mathrm{d}m.
\end{align*}
By \ref{eq_KJ_tail}, we have
\begin{align*}
I_{w2}\le\int_{B(v,3\varepsilon)}\lvert f(x)\rvert^p \left(\int_{X\backslash B(x,\frac{1}{4}\varepsilon)}K^{(J)}(x,y)m(\mathrm{d}y)\right)m(\mathrm{d}x)\lesssim \frac{1}{\Upsilon(\varepsilon)}\int_{B(v,3\varepsilon)}\lvert f\rvert^p \mathrm{d}m.
\end{align*}
Since $\# N_v\le N$ and $B(w,3\varepsilon)\subseteq B(v,6\varepsilon)$ for any $w\in N_v$, we have
\begin{align}
&\int_{B(v,3\varepsilon)}\int_{B(v,3\varepsilon)}\lvert f(x)\rvert^p \lvert \phi_v(x)-\phi_v(y)\rvert^p K^{(J)}(x,y)m(\mathrm{d}x)m(\mathrm{d}y)\nonumber\\
&\lesssim I+\sum_{w\in N_v}I_w\lesssim\sum_{w\in N_v}\int_{B(w,3\varepsilon)}\mathrm{d}\Gamma^{(J)}_{B(w,3\varepsilon)}(f)+\frac{1}{\Upsilon(\varepsilon)}\int_{B(v,6\varepsilon)}\lvert f\rvert^p \mathrm{d}m.\label{eq_partition1}
\end{align}
Moreover, we have
\begin{align*}
&\int_X\int_X \lvert \phi_v(x)-\phi_v(y)\rvert^p K^{(J)}(x,y)m(\mathrm{d}x)m(\mathrm{d}y)\\
&=\left(\int_{B(v,3\varepsilon)}\int_{B(v,3\varepsilon)}+2\int_{B(v,3\varepsilon)}\int_{X\backslash B(v,3\varepsilon)}\right)\ldots m(\mathrm{d}x)m(\mathrm{d}y)=J_1+2J_2.
\end{align*}
By \ref{eq_KJ_tail}, we have
\begin{align*}
&J_2=\int_{B(v,\frac{5}{4}\varepsilon)}\int_{X\backslash B(v,3\varepsilon)}\phi_v(x)^p K^{(J)}(x,y) m(\mathrm{d}x) m(\mathrm{d}y)\\
&\le \int_{B(v,\frac{5}{4}\varepsilon)}\left(\int_{X\backslash B(x,\frac{7}{4}\varepsilon)}K^{(J)}(x,y) m(\mathrm{d}y) \right)m(\mathrm{d}x)<+\infty.
\end{align*}
By taking $f\equiv1$ in (\ref{eq_partition1}), we have $J_1<+\infty$. Hence $\phi_v\in \mathcal{F}^{(J)}$ and \ref{item_phiv_CS} follows from (\ref{eq_partition1}).
\end{proof}

\begin{proof}[Proof of Proposition \ref{prop_dense} \ref{item_dense_FJ}]
By Lemma \ref{lem_cptspt}, we only need to show that any $u\in \mathcal{F}^{(J)}\cap L^\infty(X;m)$ with compact support can be $\mathcal{E}^{(J)}_1$-approximated by functions in $\mathcal{F}^{(J)}\cap C_c(X)$.

For any $\varepsilon>0$, let $V$ and $\{\phi_v\in \mathcal{F}^{(J)}\cap C_c(X):v\in V\}$ be given as in Lemma \ref{lem_partitionJ}. Let
$$u^{(\varepsilon)}=\sum_{v\in V}u_{B(v,\varepsilon)}\phi_v,$$
then $u^{(\varepsilon)}\in \mathcal{F}^{(J)}\cap C_c(X)$ is well-defined. Moreover, we have
$$u-u^{(\varepsilon)}=\sum_{v\in V}(u-u_{B(v,\varepsilon)})\phi_v=\sum_{v\in V}u_v\phi_v,$$
where $u_v=u-u_{B(v,\varepsilon)}$. By \ref{eq_VD}, there exists a positive integer $N$ depending only on $C_{VD}$ such that
$$\sum_{v\in V}1_{B(v,12\varepsilon)}\le N1_{\cup_{v\in V}B(v,12\varepsilon)}\le N,$$
then $V$ can be written as a disjoint union $V=\sqcup_{i=1}^NV_i$ such that, for any $i=1,\ldots,N$, the family $\{B(v,12\varepsilon):v\in V_i\}$ is pairwise disjoint.

Firstly, we show that $\lim_{\varepsilon\downarrow0}\lVert u-u^{(\varepsilon)} \rVert_{L^p(X;m)}=0$. Indeed, for any $x\in X$, the summation $\sum_{v\in V}u_v(x)\phi_v(x)$ consists of at most $N$ non-zero terms, by H\"older's inequality, we have
$$\lvert u-u^{(\varepsilon)}\rvert^p=\lvert \sum_{v\in V}u_v\phi_v\rvert^p \le N^{p-1}\sum_{v\in V}\lvert u_v\rvert^p\phi_v^p,$$
which gives
$$\int_X \lvert u-u^{(\varepsilon)}\rvert^p \mathrm{d}m\le N^{p-1}\sum_{v\in V}\int_X\lvert u_v\rvert^p\phi_v^p \mathrm{d}m\le N^{p-1}\sum_{v\in V}\int_{B(v,2\varepsilon)}\lvert u_v\rvert^p \mathrm{d}m.$$
By \ref{eq_VD}, \ref{eq_PIJ} and H\"older's inequality, we have
\begin{align}
&\int_{B(v,2\varepsilon)}\lvert u_v\rvert^p \mathrm{d}m=\int_{B(v,2\varepsilon)}\lvert u-u_{B(v,\varepsilon)}\rvert^p \mathrm{d}m\nonumber\\
&\le \frac{1}{m(B(v,\varepsilon))}\int_{B(v,2\varepsilon)}\int_{B(v,2\varepsilon)}\lvert u(x)-u(y)\rvert^p m(\mathrm{d}x)m(\mathrm{d}y)\nonumber\\
&\le \frac{2^p m(B(v,2\varepsilon))}{m(B(v,\varepsilon))}\int_{B(v,2\varepsilon)}\lvert u-u_{B(v,2\varepsilon)}\rvert^p \mathrm{d}m\lesssim \Upsilon(\varepsilon)\int_{B(v,2\varepsilon)}\mathrm{d}\Gamma^{(J)}_{B(v,2\varepsilon)}(u).\label{eq_regular}
\end{align}
Hence
\begin{align*}
&\int_X \lvert u-u^{(\varepsilon)}\rvert^p \mathrm{d}m\lesssim \Upsilon(\varepsilon)\sum_{v\in V}\int_{B(v,2\varepsilon)}\mathrm{d}\Gamma^{(J)}_{B(v,2\varepsilon)}(u)\\
&\le \Upsilon(\varepsilon)\int_X\left(\sum_{v\in V}1_{B(v,2\varepsilon)}\right)\mathrm{d}\Gamma^{(J)}_{X}(u)\le N\Upsilon(\varepsilon)\mathcal{E}^{(J)}(u)\to0
\end{align*}
as $\varepsilon\downarrow0$.

Secondly, we show that $\lim_{\varepsilon\downarrow0}\mathcal{E}^{(J)}(u-u^{(\varepsilon)})=0$. By Minkowski's inequality, we have
\begin{align*}
&\mathcal{E}^{(J)}(u-u^{(\varepsilon)})^{\frac{1}{p}}=\left(\int_X\int_X \lvert \sum_{v\in V}\left(u_v(x)\phi_v(x)-u_v(y)\phi_v(y)\right)\rvert^p K^{(J)}(x,y)m(\mathrm{d}x)m(\mathrm{d}y)\right)^{\frac{1}{p}}\\
&\le\sum_{i=1}^N\left(\int_X\int_X \left( \sum_{v\in V_i}\lvert u_v(x)\phi_v(x)-u_v(y)\phi_v(y)\rvert\right)^p K^{(J)}(x,y)m(\mathrm{d}x)m(\mathrm{d}y)\right)^{\frac{1}{p}}.
\end{align*}
For any $x,y\in X$, we have $u_v(x)\phi_v(x)-u_v(y)\phi_v(y)\ne0$ only if $x\in B(v,2\varepsilon)$ or $y\in B(v,2\varepsilon)$; moreover, noting that $B(v,2\varepsilon)\cap B(w,2\varepsilon)=\emptyset$ for any distinct $v,w\in V_i$, we have
\begin{align*}
&\left( \sum_{v\in V_i}\lvert u_v(x)\phi_v(x)-u_v(y)\phi_v(y)\rvert\right)^p\le \left(\sum_{v\in V_i} \lvert \ldots\rvert 1_{B(v,2\varepsilon)}(x)+\sum_{{v\in V_i}} \lvert \ldots\rvert1_{B(v,2\varepsilon)}(y)\right)^p\\
&\le2^{p-1} \left(\sum_{v\in V_i} \lvert \ldots\rvert 1_{B(v,2\varepsilon)}(x)\right)^p+2^{p-1} \left(\sum_{v\in V_i} \lvert \ldots\rvert 1_{B(v,2\varepsilon)}(y)\right)^p\\
&=2^{p-1}\sum_{v\in V_i}\lvert\ldots\rvert^p1_{B(v,2\varepsilon)}(x)+2^{p-1}\sum_{v\in V_i}\lvert \ldots\rvert^p1_{B(v,2\varepsilon)}(y).
\end{align*}
By symmetry, we have
\begin{align*}
&\mathcal{E}^{(J)}(u-u^{(\varepsilon)})^{\frac{1}{p}}\\
&\le2\sum_{i=1}^N\left(\sum_{v\in V_i}\int_{B(v,2\varepsilon)}\int_X \lvert u_v(x)\phi_v(x)-u_v(y)\phi_v(y)\rvert^p K^{(J)}(x,y)m(\mathrm{d}x)m(\mathrm{d}y)\right)^{\frac{1}{p}}\\
&=2\sum_{i=1}^N \left(I_i^{(\varepsilon)}\right)^{\frac{1}{p}}.
\end{align*}
Recall that $N$ is independent of $\varepsilon$, it remains to show that $\lim_{\varepsilon\downarrow0}I_i^{(\varepsilon)}=0$ for any $i=1,\ldots,N$. Noting that
\begin{align*}
&I_i^{(\varepsilon)}\\
&\le2^{p-1}\sum_{v\in V_i}\int_{B(v,2\varepsilon)}\int_X \lvert u_v(x)\rvert^p\lvert\phi_v(x)-\phi_v(y)\rvert^p K^{(J)}(x,y)m(\mathrm{d}x)m(\mathrm{d}y)\\
&\hspace{15pt}+2^{p-1}\sum_{v\in V_i}\int_{B(v,2\varepsilon)}\int_X \lvert u_v(x)-u_v(y)\rvert^p \phi_v(y)^p K^{(J)}(x,y)m(\mathrm{d}x)m(\mathrm{d}y)\\
&=2^{p-1}\sum_{v\in V_i}I^{(\varepsilon)}_{i1v}+2^{p-1}\sum_{v\in V_i}I^{(\varepsilon)}_{i2v}.
\end{align*}
For $I^{(\varepsilon)}_{i1v}$, we have
\begin{align*}
I^{(\varepsilon)}_{i1v}=\left(\int_{B(v,2\varepsilon)}\int_{B(v,3\varepsilon)}+\int_{B(v,2\varepsilon)}\int_{X\backslash B(v,3\varepsilon)}\right)\ldots m(\mathrm{d}x)m(\mathrm{d}y)=J_1+J_2.
\end{align*}
By Lemma \ref{lem_partitionJ} \ref{item_phiv_CS}, we have
\begin{align*}
&J_1\le\int_{B(v,3\varepsilon)}\lvert u_v\rvert^p \mathrm{d}\Gamma^{(J)}_{B(v,3\varepsilon)}(\phi_v)\\
&\le C_1\sum_{w\in N_v}\int_{B(w,3\varepsilon)}\mathrm{d}\Gamma^{(J)}_{B(w,3\varepsilon)}(u_v)+\frac{C_2}{\Upsilon(\varepsilon)}\int_{B(v,6\varepsilon)}\lvert u_v\rvert^p \mathrm{d}m\\
&=C_1\sum_{w\in N_v}\int_{B(w,3\varepsilon)}\mathrm{d}\Gamma^{(J)}_{B(w,3\varepsilon)}(u)+\frac{C_2}{\Upsilon(\varepsilon)}\int_{B(v,6\varepsilon)}\lvert u_v\rvert^p \mathrm{d}m,
\end{align*}
where $N_v=\{w\in V:B(w,3\varepsilon)\cap B(v,3\varepsilon)\ne \emptyset\}$. For any $w\in N_v$, we have $B(w,3\varepsilon)\subseteq B(v,9\varepsilon)$, hence
\begin{align*}
&\sum_{w\in N_v}\int_{B(w,3\varepsilon)}\mathrm{d}\Gamma^{(J)}_{B(w,3\varepsilon)}(u)\le\int_X\left(\sum_{w\in N_v}1_{B(w,3\varepsilon)}\right)\mathrm{d}\Gamma^{(J)}_{B(v,9\varepsilon)}(u)\\
&\le \int_X N1_{\cup_{{w\in N_v}}B(w,3\varepsilon)}\mathrm{d}\Gamma^{(J)}_{B(v,9\varepsilon)}(u)\le N\int_{B(v,9\varepsilon)}\mathrm{d}\Gamma^{(J)}_{B(v,9\varepsilon)}(u).
\end{align*}
Similar to (\ref{eq_regular}), we also have
\begin{align*}
\int_{B(v,6\varepsilon)}\lvert u_v\rvert^p \mathrm{d}m\lesssim \Upsilon(\varepsilon)\int_{B(v,6\varepsilon)}\mathrm{d}\Gamma^{(J)}_{B(v,6\varepsilon)}(u).
\end{align*}
Hence
\begin{align*}
J_1\lesssim\int_{B(v,9\varepsilon)}\mathrm{d}\Gamma^{(J)}_{B(v,9\varepsilon)}(u).
\end{align*}
By \ref{eq_KJ_tail}, we have
\begin{align*}
&J_2=\int_{B(v,2\varepsilon)}\int_{X\backslash B(v,3\varepsilon)}\lvert u_v(x)\rvert^p\phi_v(x)^p K^{(J)}(x,y)m(\mathrm{d}y)m(\mathrm{d}x)\\
&\le\int_{B(v,2\varepsilon)}\lvert u_v(x)\rvert^p\left(\int_{X\backslash B(x,\varepsilon)} K^{(J)}(x,y)m(\mathrm{d}y)\right)m(\mathrm{d}x)\\
&\lesssim \frac{1}{\Upsilon(\varepsilon)}\int_{B(v,2\varepsilon)}\lvert u_v\rvert^p \mathrm{d}m\lesssim \int_{B(v,2\varepsilon)}\mathrm{d}\Gamma^{(J)}_{B(v,2\varepsilon)}(u),
\end{align*}
where the last inequality follows from (\ref{eq_regular}). Hence
$$I^{(\varepsilon)}_{i1v}=J_1+J_2\lesssim\int_{B(v,9\varepsilon)}\mathrm{d}\Gamma^{(J)}_{B(v,9\varepsilon)}(u).$$

For $I^{(\varepsilon)}_{i2v}$, we have
$$I^{(\varepsilon)}_{i2v}=\int_{B(v,2\varepsilon)}\int_{B(v,2\varepsilon)}\lvert u(x)-u(y)\rvert ^p\phi_v(y)^p K^{(J)}(x,y)m(\mathrm{d}x)m(\mathrm{d}y)\le \int_{B(v,2\varepsilon)}\mathrm{d}\Gamma^{(J)}_{B(v,2\varepsilon)}(u).$$

Therefore, we have
\begin{align*}
&I_i^{(\varepsilon)}\le 2^{p-1}\sum_{v\in V_i}\left(I^{(\varepsilon)}_{i1v}+I^{(\varepsilon)}_{i2v}\right)\lesssim \sum_{v\in V_i}\int_{B(v,9\varepsilon)}\mathrm{d}\Gamma^{(J)}_{B(v,9\varepsilon)}(u)\\
&\le\sum_{v\in V_i}\int_{B(v,9\varepsilon)}\left(\int_{B(x,18\varepsilon)}\lvert u(x)-u(y)\rvert^p K^{(J)}(x,y)m(\mathrm{d}y)\right)m(\mathrm{d}x)\\
&\overset{(\star)}{\scalebox{2}[1]{$\le$}}\int_X\left(\int_{B(x,18\varepsilon)}\lvert u(x)-u(y)\rvert^p K^{(J)}(x,y)m(\mathrm{d}y)\right)m(\mathrm{d}x)\overset{(\dagger)}{\scalebox{2}[1]{$\to$}}0
\end{align*}
as $\varepsilon\downarrow0$, where $(\star)$ follows from the fact that the family $\{B(v,9\varepsilon):v\in V_i\}$ is pairwise disjoint, and $(\dagger)$ follows from the dominated convergence theorem.
\end{proof}

\section{Proof of \texorpdfstring{``$\text{CE}^{(\bullet)}_\square\Rightarrow\text{CS}^{(\bullet)}_\square$"}{CE to CS} in Theorems \ref{thm_equiv_L} and \ref{thm_equiv_J}}\label{sec_CE2CS}

For notational convenience, we write $\mathcal{E}^{(\bullet)}$, $\mathcal{F}^{(\bullet)}$, $\Gamma^{(\bullet)}$, \ldots\ to denote $\mathcal{E}^{(L)}$, $\mathcal{F}^{(L)}$, $\Gamma^{(L)}$, \ldots\ in Theorem \ref{thm_equiv_L}, and $\mathcal{E}^{(J)}$, $\mathcal{F}^{(J)}$, $\Gamma^{(J)}$, \ldots\ in Theorem \ref{thm_equiv_J}. We write $\bullet=L$ to mean the setting of Theorem \ref{thm_equiv_L}, and $\bullet=J$ to mean the setting of Theorem \ref{thm_equiv_J}. We prove these results for the same doubling function $\Psi$.

Let $\delta\in(0,1)$, $A_{E}>1$ be the constants in \hyperlink{eq_CE_strong}{$\text{CE}^{(\bullet)}_\square(\Psi)$}. Let $A=8A_{E}A_{PI}$, where $A_{PI}\ge1$ is the constant in \ref{eq_PIL} for $\bullet=L$, and $A_{PI}=1$ as in \hyperref[eq_PIJ]{$\text{PI}^{(J)}(\Psi)$} for $\bullet=J$.

For any ball $B(x_0,R)$, let $\phi\in \mathcal{F}^{(\bullet)}$ be a cutoff function for $B(x_0,R)\subseteq B(x_0,A_{E}R)$ given by \hyperlink{eq_CE_strong}{$\text{CE}^{(\bullet)}_\square(\Psi)$}. For $\bullet=L$, take any $f\in {\mathcal{F}}^{(L)}\cap C_c(X)$; and for $\bullet=J$, take any $f\in \widehat{\mathcal{F}}^{(J)}\cap L^\infty(X;m)$. Write
\begin{align*}
&\int_{B(x_0,AR)}\lvert f\rvert^p \mathrm{d}\Gamma^{(\bullet)}_{X}(\phi)=\left(\int_{B(x_0,2A_{E}R)}+\int_{B(x_0,AR)\backslash B(x_0,2A_{E}R)}\right)\lvert f\rvert^p \mathrm{d}\Gamma^{(\bullet)}_{X}(\phi)=I_1+I_2.
\end{align*}
For $I_1$, we have
\begin{align*}
&I_1\lesssim\int_{B(x_0,2A_{E}R)}\lvert f-f_{B(x_0,2A_{E}R)}\rvert^p \mathrm{d}\Gamma^{(\bullet)}_{X}(\phi)+\lvert f_{B(x_0,2A_{E}R)}\rvert^p\int_{B(x_0,2A_{E}R)}\mathrm{d}\Gamma^{(\bullet)}_{X}(\phi)\\
&=I_{11}+I_{12}.
\end{align*}
By \hyperref[eq_CE_energy]{$\text{CE}^{(\bullet)}(\Psi)\text{-}2$} and H\"older's inequality, we have
\begin{align*}
I_{12}\lesssim \frac{V(x_0,2A_{E}R)}{\Psi(R)}\frac{1}{V(x_0,2A_{E}R)}\int_{B(x_0,2A_{E}R)} \lvert f\rvert^p \mathrm{d}m\le\frac{1}{\Psi(R)}\int_{B(x_0,AR)} \lvert f\rvert^p \mathrm{d}m.
\end{align*}
For any $n\ge0$, let $\varepsilon_n=\frac{1}{2^n}(2A_{E}R)$ and $V_n$ an $\varepsilon_n$-net of $X$ with $x_0\in V_n$. For any Lebesgue point $x\in B(x_0,2A_{E}R)$ of $f$ (with respect to the measure $m$), for any $n\ge0$ , there exists $y_n\in V_n$ such that $d(x,y_n)<\varepsilon_n$. Without loss of generality, we may assume that $y_0=x_0$. Since $d(y_n,y_{n+1})<\varepsilon_n+\varepsilon_{n+1}$, we have $B(y_{n+1},\varepsilon_{n+1})\subseteq B(y_n,2\varepsilon_n)$. By \ref{eq_VD}, $\text{PI}^{(\bullet)}(\Psi)$ and H\"older's inequality, we have
\begin{align*}
&\lvert f_{B(y_n,\varepsilon_n)}-f_{B(y_{n+1},\varepsilon_{n+1})}\rvert^p\\
&\lesssim \left(\frac{1}{V(y_n,2\varepsilon_n)^2}\int_{B(y_n,2\varepsilon_n)}\int_{B(y_n,2\varepsilon_n)}\lvert f(y)-f(z)\rvert m(\mathrm{d}y)m(\mathrm{d}z)\right)^p\\
&\le\frac{1}{V(y_n,2\varepsilon_n)^2}\int_{B(y_n,2\varepsilon_n)}\int_{B(y_n,2\varepsilon_n)}\lvert f(y)-f(z)\rvert^p m(\mathrm{d}y)m(\mathrm{d}z)\\
&\lesssim \frac{\Psi(\varepsilon_n)}{V(y_n,\varepsilon_n)}\int_{B(y_n,2A_{PI}\varepsilon_n)}\mathrm{d}\Gamma^{(\bullet)}_{B(y_n,2A_{PI}\varepsilon_n)}(f)\le \frac{\Psi(\varepsilon_n)}{V(y_n,\varepsilon_n)}\int_{B(y_n,2A_{PI}\varepsilon_n)}\mathrm{d}\Gamma^{(\bullet)}_{B(x_0,AR)}(f).
\end{align*}
Taking any fixed $\alpha\in(0,\frac{\delta}{p})$, by H\"older's inequality, we have
\begin{align*}
&\lvert f(x)-f_{B(x_0,2A_ER)}\rvert^p=\lim_{n\to+\infty}\lvert f_{B(y_n,\varepsilon_n)}-f_{B(y_0,\varepsilon_0)}\rvert^p\\
&\le\left(\sum_{n=0}^{+\infty}2^{\alpha n}\lvert f_{B(y_n,\varepsilon_n)}-f_{B(y_{n+1},\varepsilon_{n+1})}\rvert\cdot2^{-\alpha n}\right)^p\\
&\le\left(\sum_{n=0}^{+\infty}2^{p\alpha n}\lvert f_{B(y_n,\varepsilon_n)}-f_{B(y_{n+1},\varepsilon_{n+1})}\rvert^p\right)\left(\sum_{n=0}^{+\infty}2^{-\alpha \frac{p}{p-1} n}\right)^{p-1}\\
&\lesssim\sum_{n=0}^{+\infty}2^{p\alpha n}\frac{\Psi(\varepsilon_n)}{V(y_n,\varepsilon_n)}\int_{B(y_n,2A_{PI}\varepsilon_n)}\mathrm{d}\Gamma^{(\bullet)}_{B(x_0,AR)}(f)\\
&\le\sum_{n=0}^{+\infty}\sum_{y\in V_n\cap B(x_0,4A_{E}R)}2^{p\alpha n}\frac{\Psi(\varepsilon_n)}{V(y,\varepsilon_n)}\left(\int_{B(y,2A_{PI}\varepsilon_n)}\mathrm{d}\Gamma^{(\bullet)}_{B(x_0,AR)}(f)\right)1_{B(y,\varepsilon_n)}(x).
\end{align*}
For $\bullet=L$, since $f\in \mathcal{F}^{(L)}\cap C_c(X)$, it holds for \emph{any} $x\in B(x_0,2A_{E}R)$; for $\bullet=J$, it holds for \emph{$m$-a.e.} $x\in B(x_0,2A_{E}R)$, since $\Gamma^{(J)}_{X}(\phi)$ is absolutely continuous with respect to $m$, it also holds for $\Gamma^{(J)}_{X}(\phi)$-a.e. $x\in B(x_0,2A_{E}R)$. In summary, the above inequality holds for $\Gamma^{(\bullet)}_X(\phi)$-a.e. $x\in B(x_0,2A_{E}R)$.

By \hyperref[eq_CE_energy]{$\text{CE}^{(\bullet)}(\Psi)\text{-}2$}, we have
\begin{align*}
&I_{11}\lesssim\sum_{n=0}^{+\infty}\sum_{y\in V_n\cap B(x_0,4A_{E}R)}2^{p\alpha n}\frac{\Psi(\varepsilon_n)}{V(y,\varepsilon_n)}\left(\int_{B(y,2A_{PI}\varepsilon_n)}\mathrm{d}\Gamma^{(\bullet)}_{B(x_0,AR)}(f)\right)\int_{B(y,\varepsilon_n)}\mathrm{d}\Gamma^{(\bullet)}_{X}(\phi)\\
&\lesssim\sum_{n=0}^{+\infty}\sum_{y\in V_n\cap B(x_0,4A_{E}R)}2^{p\alpha n}\frac{\Psi(\varepsilon_n)}{V(y,\varepsilon_n)}\left(\int_{B(y,2A_{PI}\varepsilon_n)}\mathrm{d}\Gamma^{(\bullet)}_{B(x_0,AR)}(f)\right)\left(\frac{\varepsilon_n}{R}\wedge 1\right)^\delta \frac{V(y,\varepsilon_n)}{\Psi(\varepsilon_n\wedge R)}\\
&\lesssim\sum_{n=0}^{+\infty} 2^{(p\alpha-\delta)n}\int_X \left(\sum_{y\in V_n\cap B(x_0,4A_{E}R)}1_{B(y,2A_{PI}\varepsilon_n)}\right)\mathrm{d}\Gamma^{(\bullet)}_{B(x_0,AR)}(f),
\end{align*}
where the last inequality follows from the fact that $\varepsilon_n\wedge R\asymp \varepsilon_n$. By \ref{eq_VD}, there exists some positive integer $M$ depending only on $A_{E}, A_{PI}, C_{VD}$ such that
$$\sum_{y\in V_n\cap B(x_0,4A_{E}R)}1_{B(y,2A_{PI}\varepsilon_n)}\le M1_{\cup_{{y\in V_n\cap B(x_0,4A_{E}R)}}B(y,2A_{PI}\varepsilon_n)}\le M1_{B(x_0,AR)},$$
which gives
\begin{align*}
I_{11}\lesssim\sum_{n=0}^{+\infty} 2^{(p\alpha-\delta)n}\int_{B(x_0,AR)}\mathrm{d}\Gamma^{(\bullet)}_{B(x_0,AR)}(f)\lesssim\int_{B(x_0,AR)}\mathrm{d}\Gamma^{(\bullet)}_{B(x_0,AR)}(f),
\end{align*}
where the second inequality follows from the fact that $\alpha<\frac{\delta}{p}$. Hence
$$I_1\lesssim I_{11}+I_{12}\lesssim\int_{B(x_0,AR)}\mathrm{d}\Gamma^{(\bullet)}_{B(x_0,AR)}(f)+\frac{1}{\Psi(R)}\int_{B(x_0,AR)} \lvert f\rvert^p \mathrm{d}m.$$

For $I_2$, when $\bullet=L$, the strongly local property implies that $I_2=0$. When $\bullet=J$, by \hyperref[eq_KJ_tail]{$\text{T}^{(J)}(\Psi)$}, we have
\begin{align*}
&I_2=\int_{B(x_0,AR)\backslash B(x_0,2A_{E}R)} \lvert f(x)\rvert^p \left(\int_{X}\lvert \phi(x)-\phi(y)\rvert^p K^{(J)}(x,y)m(\mathrm{d}y)\right)m(\mathrm{d}x)\\
&=\int_{B(x_0,AR)\backslash B(x_0,2A_{E}R)} \lvert f(x)\rvert^p \left(\int_{B(x_0,A_{E}R)}\phi(y)^p K^{(J)}(x,y)m(\mathrm{d}y)\right)m(\mathrm{d}x)\\
&\le\int_{B(x_0,AR)\backslash B(x_0,2A_{E}R)} \lvert f(x)\rvert^p \left(\int_{X\backslash B(x,A_{E}R)} K^{(J)}(x,y)m(\mathrm{d}y)\right)m(\mathrm{d}x)\\
&\lesssim\frac{1}{\Psi(R)}\int_{B(x_0,AR)\backslash B(x_0,2A_{E}R)}\lvert f\rvert^p \mathrm{d}m\le \frac{1}{\Psi(R)}\int_{B(x_0,AR)}\lvert f\rvert^p \mathrm{d}m.
\end{align*}
In summary, we have
\begin{equation}\label{eq_CE2CS}
\int_{B(x_0,AR)}\lvert f\rvert^p \mathrm{d}\Gamma^{(\bullet)}_{X}(\phi)=I_1+I_2\lesssim\int_{B(x_0,AR)}\mathrm{d}\Gamma^{(\bullet)}_{B(x_0,AR)}(f)+\frac{1}{\Psi(R)}\int_{B(x_0,AR)} \lvert f\rvert^p \mathrm{d}m.
\end{equation}

For $\bullet=L$, (\ref{eq_CE2CS}) holds for any $f\in \mathcal{F}^{(L)}\cap C_c(X)$, then for any $f\in \mathcal{F}^{(L)}$, by \cite[Proposition 8.5, Proposition 8.12]{Yan25a}, there exists $\{f_n\}\subseteq \mathcal{F}^{(L)}\cap C_c(X)$ such that $\{f_n\}$ is $\mathcal{E}^{(L)}_1$-convergent to $f$ and $\{f_n\}$ converges to $\widetilde{f}$ q.e. on $X$, which is also $\Gamma^{(L)}(\phi)$-a.e. on $X$, then by Fatou's lemma, we have
\begin{align*}
&\int_{B(x_0,AR)}\lvert \widetilde{f}\rvert^p \mathrm{d}\Gamma^{(L)}(\phi)\le\varliminf_{n\to+\infty}\int_{B(x_0,AR)}\lvert f_n\rvert^p \mathrm{d}\Gamma^{(L)}(\phi)\\
&\lesssim\varliminf_{n\to+\infty}\left(\int_{B(x_0,AR)}\mathrm{d}\Gamma^{(L)}(f_n)+\frac{1}{\Psi(R)}\int_{B(x_0,AR)} \lvert f_n\rvert^p \mathrm{d}m\right)\\
&=\int_{B(x_0,AR)}\mathrm{d}\Gamma^{(L)}(f)+\frac{1}{\Psi(R)}\int_{B(x_0,AR)} \lvert f\rvert^p \mathrm{d}m.
\end{align*}
Noting that $\phi\in \mathcal{F}^{(L)}$ is also a cutoff function for $B(x_0,R)\subseteq B(x_0,AR)$, we conclude that \hyperlink{eq_CSL_strong}{$\text{CS}^{(L)}_\square(\Psi)$} holds with $A_S=A$ therein.

For $\bullet=J$, (\ref{eq_CE2CS}) holds for any $f\in \widehat{\mathcal{F}}^{(J)}\cap L^\infty(X;m)$, hence by Lemma \ref{lem_CSJ_energy}, we have \hyperlink{eq_CSJ_strong}{$\text{CS}^{(J)}_{\square}(\Psi)$} holds with $A_1=A_{E}$ and $A_2=A$ therein.\hfill$\square$

\section{\texorpdfstring{Proof of ``{$\text{CS}^{(L)}_{\text{weak}}\Rightarrow\text{CE}^{(L)}_{\text{strong}}$" in Theorem \ref{thm_equiv_L}}}{Proof of CSLweak to CELstrong in Theorem \ref{thm_equiv_L}}}\label{sec_Moser_L}

Throughout this section, we always assume that \ref{eq_VD} holds and that $(\mathcal{E}^{(L)},\mathcal{F}^{(L)})$ is a strongly local regular $p$-energy satisfying \ref{eq_PIL}, \hyperlink{eq_CSL_weak}{$\text{CS}^{(L)}_{\text{weak}}(\Psi)$}. The main results of this section are the following interior and boundary oscillation inequalities, from which we give the proof of ``\hyperlink{eq_CSL_weak}{$\text{CS}^{(L)}_{\text{weak}}(\Psi)$}$\Rightarrow$\hyperlink{eq_CE_strong}{$\text{CE}^{(L)}_{\text{strong}}(\Psi)$}" in Theorem \ref{thm_equiv_L}.

\begin{proposition}\label{prop_int_osc_L}
There exist $\delta\in(0,1)$, $C>0$ such that for any bounded open subset $\Omega\subseteq X$, let $f\in L^\infty(\Omega;m)$ satisfy $\lVert f\rVert_{L^\infty(\Omega;m)}\le1$, let $u\in \mathcal{F}^{(L)}$ be bounded in $\Omega$ and satisfy $-\Delta_p^{(L)}u=f$ in $\Omega$, then for any $x_0\in \Omega$ and any $R,r>0$ with $r\le R$ and $B(x_0,R)\subseteq\Omega$, we have
$$\eosc_{B(x_0,r)}u\le C \left(\frac{r}{R}\right)^\delta \left(\eosc_{B(x_0,R)}u+\Psi(R)^{\frac{1}{p-1}}\right).$$
In particular, we have $u\in C(\Omega)$.
\end{proposition}

\begin{proposition}\label{prop_bdy_osc_L}
Let $c\in(0,1)$. Then there exist $\delta\in(0,1)$, $C>0$ such that for any bounded open subset $\Omega\subseteq X$, let $f\in L^\infty(\Omega;m)$ satisfy $\lVert f\rVert_{L^\infty(\Omega;m)}\le1$, let $u\in \mathcal{F}^{(L)}(\Omega)$ be non-negative bounded in $\Omega$ and satisfy $-\Delta^{(L)}_pu=f$ in $\Omega$, let $x_0\in \partial \Omega$ satisfy $X\backslash \Omega$ has a $(c,r_0)$-corkscrew at $x_0$ for some $r_0>0$, then for any $R,r\in(0,r_0)$ with $r\le R$, we have
$$\esup_{B(x_0,r)\cap\Omega}u\le C \left(\frac{r}{R}\right)^\delta\left(\esup_{B(x_0,R)\cap\Omega}u+\Psi(R)^{\frac{1}{p-1}}\right).$$
In particular, we have $\lim_{r\downarrow0}\esup_{{B(x_0,r)\cap\Omega}}u=0$.
\end{proposition}

The proofs rely on the following interior and boundary weak Harnack inequalities. We will follow the classical Moser iteration technique as in \cite[Theorems 8.18 and 8.26]{GT01}, while \hyperlink{eq_CSL_weak}{$\text{CS}^{(L)}_{\text{weak}}(\Psi)$} will provide the necessary cutoff functions.

\begin{proposition}\label{prop_int_wEHI_L}
There exist $q>0$, $A>1$, $C>1$ such that the following holds. For any $x_0\in X$ and any $R>0$, let $u\in \mathcal{F}^{(L)}$ be non-negative bounded in $B(x_0,AR)$ and satisfy $-\Delta^{(L)}_pu\ge -1$ in $B(x_0,AR)$, then we have
$$\left(\dashint_{B(x_0,R)}u^q \mathrm{d}m\right)^{\frac{1}{q}}\le C \left(\einf_{B(x_0,\frac{1}{2}R)}u+\Psi(R)^{\frac{1}{p-1}}\right).$$
\end{proposition}

\begin{proposition}\label{prop_bdy_wEHI_L}
There exist $q>0$, $A>1$, $C>1$ such that the following holds. For any bounded open subset $\Omega\subseteq X$, any $x_0\in X$, and any $R>0$, let $u\in \mathcal{F}^{(L)}$ be non-negative bounded in $B(x_0,AR)$ and satisfy $\widetilde{u}=M$ q.e. on $B(x_0,AR)\backslash\Omega$, for some constant $M\in [0,+\infty)$, and $-\Delta^{(L)}_pu\ge -1$ in $B(x_0,AR)\cap\Omega$, then we have
$$\left(\dashint_{B(x_0,R)}\left(u\wedge M\right)^q \mathrm{d}m\right)^{\frac{1}{q}}\le C \left(\einf_{B(x_0,\frac{1}{2}R)}\left(u\wedge M\right)+\Psi(R)^{\frac{1}{p-1}}\right).$$
\end{proposition}

It suffices to prove Proposition \ref{prop_bdy_wEHI_L}. Indeed, assuming Proposition \ref{prop_bdy_wEHI_L}, by taking $\Omega=B(x_0,AR)$ and $M=\lVert u\rVert_{L^\infty(B(x_0,AR);m)}<+\infty$, and noting that $u\wedge M=u$ in $B(x_0,AR)$, we obtain Proposition \ref{prop_int_wEHI_L}.

\begin{lemma}\label{lem_Moser_L}
For any $q>0$, there exists $C>1$ such that the following holds. For any bounded open subset $\Omega\subseteq X$, any $x_0\in X$, and any $R_0>0$, let $u\in \mathcal{F}^{(L)}$ be non-negative bounded in $B(x_0,R_0)$ and satisfy $\widetilde{u}=M$ q.e. on $B(x_0,R_0)\backslash\Omega$, for some constant $M\in[0,+\infty)$, and $-\Delta^{(L)}_pu\ge-1$ in $B(x_0,R_0)\cap\Omega$, then for any $R\in(0,R_0)$, we have
$$\einf_{B(x_0,\frac{1}{2}R)}(u\wedge M)+\Psi(R)^{\frac{1}{p-1}}\ge \frac{1}{C} \left(\dashint_{B(x_0,R)}\left(u\wedge M+\Psi(R)^{\frac{1}{p-1}}\right)^{-q}\mathrm{d}m\right)^{-\frac{1}{q}}.$$
\end{lemma}

\begin{proof}
For notational convenience, we may assume that all functions in $\mathcal{F}^{(L)}$ are quasi-continuous. If $M=0$, then the result is trivial, hence we may assume that $M>0$. For any $R,r>0$ with $r\le R$ and $R+r<R_0$, for any $K\ge\Psi(r)^{\frac{1}{p-1}}$, let $\overline{u}=u\wedge M+K$, and $\phi\in \mathcal{F}^{(L)}$ a cutoff function for $B(x_0,R)\subseteq B(x_0,R+r)$ to be chosen later. For any $\alpha< -(p-1)$, let
$$v=\phi^p \left(\overline{u}^\alpha-(M+K)^\alpha\right)=\phi^p \left((u\wedge M+K)^\alpha-(M+K)^\alpha\right),$$
then $v\in \mathcal{F}^{(L)}(B(x_0,R+r)\cap\Omega)$ is non-negative bounded, hence
\begin{equation}\label{eq_wEHI_L1}
\mathcal{E}^{(L)}(u;v)\ge-\int_{X}v \mathrm{d}m\ge -\frac{1}{\Psi(r)}\int_{B(x_0,R+r)}\overline{u}^{\alpha+p-1}\mathrm{d}m.
\end{equation}
On the other hand
\begin{equation}\label{eq_wEHI_L2}
\mathcal{E}^{(L)}(u;v)=\alpha\int_{B(x_0,R+r)}\phi^p \overline{u}^{\alpha-1} \mathrm{d}\Gamma^{(L)}(\overline{u})+p\int_{B(x_0,R+r)}\phi^{p-1}\left(\overline{u}^\alpha-(M+K)^\alpha\right)\mathrm{d}\Gamma^{(L)}(\overline{u};\phi),
\end{equation}
where
\begin{align}
&p\lvert \int_{B(x_0,R+r)}\phi^{p-1}\left(\overline{u}^\alpha-(M+K)^\alpha\right)\mathrm{d}\Gamma^{(L)}(\overline{u};\phi)\nonumber\rvert\\
&=p\lvert \int_{B(x_0,R+r)}\phi^{p-1}\overline{u}^{(\alpha-1)\frac{p-1}{p}}\overline{u}^{-(\alpha-1)\frac{p-1}{p}} \left(\overline{u}^\alpha-(M+K)^\alpha\right)\mathrm{d}\Gamma^{(L)}(\overline{u};\phi)\rvert\nonumber\\
&\le p\left(\int_{B(x_0,R+r)}\phi^p \overline{u}^{\alpha-1}\mathrm{d}\Gamma^{(L)}(\overline{u})\right)^{\frac{p-1}{p}}\left(\int_{B(x_0,R+r)}\overline{u}^{\alpha+p-1}\mathrm{d}\Gamma^{(L)}(\phi)\right)^{\frac{1}{p}}\nonumber\\
&\le \frac{1}{4}\lvert \alpha\rvert\int_{B(x_0,R+r)}\phi^p \overline{u}^{\alpha-1}\mathrm{d}\Gamma^{(L)}(\overline{u})+\frac{C_1}{\lvert \alpha\rvert^{p-1}}\int_{B(x_0,R+r)}\overline{u}^{\alpha+p-1}\mathrm{d}\Gamma^{(L)}(\phi),\label{eq_wEHI_L3}
\end{align}
here the last inequality follows from Young's inequality, and $C_1>0$ depends only on $p$. By the self-improvement property of \hyperlink{eq_CSL_weak}{$\text{CS}^{(L)}_{\text{weak}}(\Psi)$} (see \cite[Proposition 3.1]{Yan25c}), for any $\varepsilon>0$, there exist $C_\varepsilon>0$ and a cutoff function $\phi\in \mathcal{F}^{(L)}$ for $B(x_0,R)\subseteq B(x_0,R+r)$ such that
\begin{align}
&\int_{B(x_0,R+r)}\overline{u}^{\alpha+p-1}\mathrm{d}\Gamma^{(L)}(\phi)\nonumber\\
&\le \varepsilon\int_{B(x_0,R+r)}\phi^p \mathrm{d}\Gamma^{(L)} \left(\overline{u}^{\frac{\alpha+p-1}{p}}\right)+\frac{C_\varepsilon}{\Psi(r)}\int_{B(x_0,R+r)}\overline{u}^{\alpha+p-1}\mathrm{d}m\nonumber\\
&=\varepsilon \left(\frac{\lvert\alpha+p-1\rvert}{p}\right)^p\int_{B(x_0,R+r)}\phi^p \overline{u}^{\alpha-1} \mathrm{d}\Gamma^{(L)}(\overline{u})+\frac{C_\varepsilon}{\Psi(r)}\int_{B(x_0,R+r)}\overline{u}^{\alpha+p-1}\mathrm{d}m.\label{eq_wEHI_L4}
\end{align}
Combining (\ref{eq_wEHI_L1})--(\ref{eq_wEHI_L4}), we have
\begin{align*}
&\lvert \alpha\rvert\int_{B(x_0,R+r)}\phi^p \overline{u}^{\alpha-1} \mathrm{d}\Gamma^{(L)}(\overline{u})\\
&\le \left(\frac{1}{4}+\varepsilon C_1 C_2\right)\lvert \alpha\rvert\int_{B(x_0,R+r)}\phi^p \overline{u}^{\alpha-1} \mathrm{d}\Gamma^{(L)}(\overline{u})\\
&\hspace{15pt}+\left(\frac{C_1C_\varepsilon}{\lvert \alpha\rvert^{p-1}}+1\right)\frac{1}{\Psi(r)}\int_{B(x_0,R+r)}\overline{u}^{\alpha+p-1}\mathrm{d}m,
\end{align*}
where $C_2=\sup_{\alpha<-(p-1)}\left(\frac{\lvert \alpha+p-1\rvert}{p\lvert \alpha\rvert}\right)^p\in(0,+\infty)$ depends only on $p$. Let $\varepsilon=\frac{1}{4C_1C_2}$, then
$$\int_{B(x_0,R+r)}\phi^p \overline{u}^{\alpha-1} \mathrm{d}\Gamma^{(L)}(\overline{u})\le  \frac{2}{\lvert \alpha\rvert^p} \left(\lvert \alpha\rvert^{p-1}+C_1C_\varepsilon\right)\frac{1}{\Psi(r)}\int_{B(x_0,R+r)}\overline{u}^{\alpha+p-1}\mathrm{d}m,$$
and
$$\int_{B(x_0,R+r)}\overline{u}^{\alpha+p-1}\mathrm{d}\Gamma^{(L)}(\phi)\le \left(2\varepsilon C_2 \left(\lvert \alpha\rvert^{p-1}+C_1C_\varepsilon\right)+C_\varepsilon\right)\frac{1}{\Psi(r)}\int_{B(x_0,R+r)}\overline{u}^{\alpha+p-1}\mathrm{d}m,$$
hence
\begin{align*}
&\mathcal{E}^{(L)}(\phi \overline{u}^{\frac{\alpha+p-1}{p}})\\
&\le 2^{p-1}\left(\left(\frac{\lvert\alpha+p-1\rvert}{p}\right)^p\int_{B(x_0,R+r)}\phi^p \overline{u}^{\alpha-1}\mathrm{d}\Gamma^{(L)}(\overline{u})+\int_{B(x_0,R+r)}\overline{u}^{{\alpha+p-1}}\mathrm{d}\Gamma^{(L)}(\phi)\right)\\
&\le \frac{C_3 \lvert \alpha\rvert^{p-1}}{\Psi(r)}\int_{B(x_0,R+r)}\overline{u}^{\alpha+p-1}\mathrm{d}m,
\end{align*}
for any $\alpha<-(p-1)$, where $C_3>0$ depends only on $p,\varepsilon,C_1,C_2,C_\varepsilon$. Since $\phi \overline{u}^{\frac{\alpha+p-1}{p}}\in \mathcal{F}^{(L)}(B(x_0,R+r))$, by Proposition \ref{prop_SobolevL}, there exist $\kappa>1$, $C_4>0$ such that
$$\mathcal{E}^{(L)}(\phi \overline{u}^{\frac{\alpha+p-1}{p}})\ge \frac{1}{C_4} \frac{V(x_0,R+r)^{\frac{\kappa-1}{\kappa}}}{\Psi(R+r)}\left(\int_{B(x_0,R+r)}(\phi \overline{u}^{\frac{\alpha+p-1}{p}})^{p\kappa}\mathrm{d}m\right)^{\frac{1}{\kappa}}.$$
Hence
\begin{align}
&\left(\dashint_{B(x_0,R)}({u}\wedge M+K)^{\kappa(\alpha+p-1)}\mathrm{d}m\right)^{\frac{1}{\kappa}}\nonumber\\
&\le C_3C_4C_{VD}^{\frac{1}{\kappa}} \lvert \alpha\rvert^{p-1}\frac{\Psi(R+r)}{\Psi(r)}\dashint_{B(x_0,R+r)}({u}\wedge M+K)^{\alpha+p-1} \mathrm{d}m,\label{eq_wEHI_L_iteration}
\end{align}
for any $\alpha<-(p-1)$, $r\le R$, $K\ge\Psi(r)^{\frac{1}{p-1}}$. For any $q>0$ and any $R\in(0,R_0)$. For any $k\ge0$, let $R_k=(\frac{1}{2}+\frac{1}{2^{k+1}})R$ and $q_k=-q\kappa^k$, then $R=R_0>R_1>\ldots>R_k\downarrow \frac{1}{2}R$, and $-q=q_0>q_1>\ldots>q_k\downarrow-\infty$, let
$$I_k=\left(\dashint_{B(x_0,R_k)}\left(u\wedge M+\Psi(R)^{\frac{1}{p-1}}\right)^{q_k}\mathrm{d}m\right)^{\frac{1}{q_k}},$$
then by (\ref{eq_wEHI_L_iteration}), we have
\begin{align*}
I_{k+1}^{-q}&\le \left(C_3C_4C_{VD}^{\frac{1}{\kappa}} \left(q\kappa^k+p-1\right)^{p-1}\frac{\Psi(R_k)}{\Psi(R_k-R_{k+1})}\right)^{\frac{1}{\kappa^k}}I_{k}^{-q}\le C_5^{\frac{1}{\kappa^k}}C_6^{\frac{k}{\kappa^k}}I_k^{-q},
\end{align*}
where $C_5\ge1$ depends only on $p,q,\kappa,C_\Psi,C_{VD},C_3,C_4$, and $C_6=\kappa^{p-1}C_\Psi$. By Lemma \ref{lem_ele1}, we have
$$I_{k}^{-q}\le C_5^{\frac{\kappa}{\kappa-1}}C_6^{\frac{\kappa}{(\kappa-1)^2}}I_0^{-q},$$
that is, $I_k\ge \frac{1}{C_7}I_0$, where $C_7=\left(C_5^{\frac{\kappa}{\kappa-1}}C_6^{\frac{\kappa}{(\kappa-1)^2}}\right)^{\frac{1}{q}}$. Letting $k\to+\infty$, we have
$$\einf_{B(x_0,\frac{1}{2}R)}(u\wedge M)+\Psi(R)^{\frac{1}{p-1}}\ge \frac{1}{C_7} \left(\dashint_{B(x_0,R)}\left(u\wedge M+\Psi(R)^{\frac{1}{p-1}}\right)^{-q}\mathrm{d}m\right)^{-\frac{1}{q}}.$$
\end{proof}

\begin{proof}[Proof of Proposition \ref{prop_bdy_wEHI_L}]
For notational convenience, we may assume that all functions in $\mathcal{F}^{(L)}$ are quasi-continuous. If $M=0$, then the result is trivial, hence we may assume that $M>0$. Let $A_{PI},A_{cap}$ be the constants in \ref{eq_PIL}, \ref{eq_ucapL}, respectively, and $A=48A_{PI}A_{cap}$, then Lemma \ref{lem_Moser_L} holds with $R_0=AR$ therein. Let $K=\Psi(R)^{\frac{1}{p-1}}$ and $\overline{u}=u\wedge M+K$. We claim that $\log \overline{u}\in \mathrm{BMO}(B(x_0,12R))$ and
$$\lVert \log \overline{u}\rVert_{\mathrm{BMO}(B(x_0,12R))}\le C_1,$$
where $C_1>0$ depends only on $p,C_\Psi, C_{VD}$ and the constants in \ref{eq_PIL}, \ref{eq_ucapL}. Indeed, for any ball $B\subseteq B(x_0,12R)$ with radius $r$, we have $r<24R$ and $A_{PI}A_{cap}B\subseteq B(x_0,AR)$. Since $\log \frac{\overline{u}}{K}\in \mathcal{F}^{(L)}$, by \ref{eq_PIL}, we have
\begin{align}
&\dashint_B\lvert\log \overline{u}-\left(\log \overline{u}\right)_B\rvert\mathrm{d} m\le\left(\dashint_B\lvert\log \frac{\overline{u}}{K}-\left(\log \frac{\overline{u}}{K}\right)_B\rvert^p\mathrm{d} m\right)^{\frac{1}{p}}\nonumber\\
&\lesssim \left(\frac{\Psi(r)}{m(B)}\int_{A_{PI}B}\mathrm{d}\Gamma^{(L)}\left(\log \frac{\overline{u}}{K}\right)\right)^{\frac{1}{p}}=\left(\frac{\Psi(r)}{m(B)}\int_{A_{PI}B}\mathrm{d}\Gamma^{(L)}(\log \overline{u})\right)^{\frac{1}{p}}.\label{eq_BMO_L1}
\end{align}
By \ref{eq_ucapL}, there exists a cutoff function $\phi\in\mathcal{F}^{(L)}$ for $A_{PI}B\subseteq A_{PI}A_{cap}B$ such that
\begin{equation}\label{eq_BMO_L2}
\mathcal{E}^{(L)}(\phi)\lesssim \frac{m(B)}{\Psi(r)},
\end{equation}
hence
\begin{align}
&\int_{A_{PI}B}\mathrm{d}\Gamma^{(L)}(\log \overline{u})\le\int_{X}\phi^p \overline{u}^{-p}\mathrm{d}\Gamma^{(L)}(\overline{u})=\int_{X}\phi^p \overline{u}^{-p}\mathrm{d}\Gamma^{(L)}(u;\overline{u})\nonumber\\
&=-\frac{1}{p-1}\int_X\phi^p\mathrm{d}\Gamma^{(L)}\left(u;\overline{u}^{1-p}-(M+K)^{1-p}\right)\nonumber\\
&=-\frac{1}{p-1}\left( \mathcal{E}^{(L)}\left(u;\phi^p\left(\overline{u}^{1-p}-(M+K)^{1-p}\right)\right)\right.\nonumber\\
&\hspace{50pt}\left.-p\int_X\phi^{p-1}\left(\overline{u}^{1-p}-(M+K)^{1-p}\right)\mathrm{d}\Gamma^{(L)}(\overline{u};\phi)\right).\label{eq_BMO_L3}
\end{align}
Since $\phi^p\left(\overline{u}^{1-p}-(M+K)^{1-p}\right)\in \mathcal{F}^{(L)}(A_{PI}A_{cap}B\cap \Omega)$ is non-negative, by assumption, we have
\begin{align}
& \mathcal{E}^{(L)}\left(u;\phi^p\left(\overline{u}^{1-p}-(M+K)^{1-p}\right)\right)\ge -\int_X \phi^p\left(\overline{u}^{1-p}-(M+K)^{1-p}\right)\mathrm{d}m\nonumber\\
&\ge -\frac{1}{\Psi(R)}\int_X\phi^p \mathrm{d}m\ge -\frac{m(A_{PI}A_{cap}B)}{\Psi(R)}\ge -C_\Psi^5\frac{m(A_{PI}A_{cap}B)}{\Psi(r)},\label{eq_BMO_L4}
\end{align}
where the last inequality follows from (\ref{eq_beta12}) and the fact that $r<24R$. Moreover
\begin{align*}
&\frac{p}{p-1}\lvert \int_X\phi^{p-1}\left(\overline{u}^{1-p}-(M+K)^{1-p}\right)\mathrm{d}\Gamma^{(L)}(\overline{u};\phi)\rvert\\
&\le \frac{p}{p-1} \left(\int_X\phi^{p}\left(\overline{u}^{1-p}-(M+K)^{1-p}\right)^{\frac{p}{p-1}}\mathrm{d}\Gamma^{(L)}(\overline{u})\right)^{\frac{p-1}{p}}\left(\int_X \mathrm{d}\Gamma^{(L)}(\phi)\right)^{\frac{1}{p}}\\
&\le\frac{p}{p-1}\left(\int_X\phi^{p}\overline{u}^{-p}\mathrm{d}\Gamma^{(L)}(\overline{u})\right)^{\frac{p-1}{p}}\mathcal{E}^{(L)}(\phi)^{\frac{1}{p}}\le \frac{1}{2}\int_X\phi^{p}\overline{u}^{-p}\mathrm{d}\Gamma^{(L)}(\overline{u})+C_2 \mathcal{E}^{(L)}(\phi),
\end{align*}
where the last inequality follows from Young's inequality, and $C_2>0$ depends only on $p$. Combining this with (\ref{eq_BMO_L3}) and (\ref{eq_BMO_L4}), we have
\begin{align*}
&\int_{X}\phi^p \overline{u}^{-p}\mathrm{d}\Gamma^{(L)}(\overline{u})\le\frac{2C_\Psi^5}{p-1}\frac{m(A_{PI}A_{cap}B)}{\Psi(r)}+2C_2 \mathcal{E}^{(L)}(\phi).
\end{align*}
Then by (\ref{eq_BMO_L2}), \ref{eq_VD}, we have
\begin{align*}
&\int_{A_{PI}B}\mathrm{d}\Gamma^{(L)}(\log \overline{u})\le\int_{X}\phi^p \overline{u}^{-p}\mathrm{d}\Gamma^{(L)}(\overline{u})\le\frac{2C_\Psi^5}{p-1}\frac{m(A_{PI}A_{cap}B)}{\Psi(r)}+2C_2 \mathcal{E}^{(L)}(\phi)\lesssim\frac{m(B)}{\Psi(r)}.
\end{align*}
By (\ref{eq_BMO_L1}), we have
\begin{align*}
&\dashint_B\lvert\log \overline{u}-\left(\log \overline{u}\right)_B\rvert\mathrm{d} m\le C_1,
\end{align*}
where $C_1>0$ is some constant. By Lemma \ref{lem_crossover}, there exist $C_{3},C_{4}>0$ such that
\begin{align*}
&\left(\dashint_{B(x_0,R)}\exp \left(\frac{C_{4}}{2C_1}\log (u\wedge M+\Psi(R)^{\frac{1}{p-1}})\right)\mathrm{d} m\right)\\
&\cdot\left(\dashint_{B(x_0,R)}\exp \left(-\frac{C_{4}}{2C_1}\log (u\wedge M+\Psi(R)^{\frac{1}{p-1}})\right)\mathrm{d} m\right)\\
&\le(C_{3}+1)^2.
\end{align*}
Let $q=\frac{C_{4}}{2C_1}$, then
$$\left(\dashint_{B(x_0,R)}(u\wedge M+\Psi(R)^{\frac{1}{p-1}})^q\mathrm{d} m\right)\left(\dashint_{B(x_0,R)}(u\wedge M+\Psi(R)^{\frac{1}{p-1}})^{-q}\mathrm{d} m\right)\le (C_{3}+1)^2.$$
Combining this with Lemma \ref{lem_Moser_L}, we have
\begin{align*}
&\einf_{B(x_0,\frac{1}{2}R)}(u\wedge M)+\Psi(R)^{\frac{1}{p-1}}\\
&\gtrsim \left(\dashint_{B(x_0,R)}(u\wedge M+\Psi(R)^{\frac{1}{p-1}})^{q}\mathrm{d} m\right)^{\frac{1}{q}}\ge \left(\dashint_{B(x_0,R)}(u\wedge M)^{q}\mathrm{d} m\right)^{\frac{1}{q}}.
\end{align*}
\end{proof}

We give the proofs of Propositions \ref{prop_int_osc_L} and \ref{prop_bdy_osc_L} as follows.

\begin{proof}[Proof of Proposition \ref{prop_int_osc_L}]
Let $q>0$, $A>1$, $C_1>1$ be the constants appearing in Proposition \ref{prop_int_wEHI_L}. Fix $R_0>0$ with $B(x_0,R_0)\subseteq \Omega$, for any $R\in(0,R_0)$, let
$$M_1=\esup_{B(x_0,\frac{1}{2A}R)}u,\quad m_1=\einf_{B(x_0,\frac{1}{2A}R)}u,$$
$$M_2=\esup_{B(x_0,R)}u,\quad m_2=\einf_{B(x_0,R)}u.$$
Since $M_2-u,u-m_2\in \mathcal{F}^{(L)}$ are non-negative bounded in $B(x_0,R)$, and $-\Delta_p^{(L)}(M_2-u)=-f\ge-1$ in $B(x_0,R)$, $-\Delta_p^{(L)}(u-m_2)=f\ge-1$ in $B(x_0,R)$, by Proposition \ref{prop_int_wEHI_L}, we have
\begin{align*}
\left(\dashint_{B(x_0,\frac{1}{A}R)}(M_2-u)^q \mathrm{d}m\right)^{\frac{1}{q}}&\le C_1 \left(\einf_{B(x_0,\frac{1}{2A}R)}(M_2-u)+\Psi\left(\frac{1}{A}R\right)^{\frac{1}{p-1}}\right),\\
\left(\dashint_{B(x_0,\frac{1}{A}R)}(u-m_2)^q \mathrm{d}m\right)^{\frac{1}{q}}&\le C_1 \left(\einf_{B(x_0,\frac{1}{2A}R)}(u-m_2)+\Psi\left(\frac{1}{A}R\right)^{\frac{1}{p-1}}\right),
\end{align*}
where
\begin{align*}
\einf_{B(x_0,\frac{1}{2A}R)}(M_2-u)&=M_2-M_1,\\
\einf_{B(x_0,\frac{1}{2A}R)}(u-m_2)&=m_1-m_2.
\end{align*}
By the norm and quasi-norm properties of $L^q$-spaces (see \cite[Equations (1.1.3) and (1.1.4)]{Gra14}), there exists $C_2\ge1$ depending only on $q$ such that
\begin{align*}
&M_2-m_2\le C_2\left(\left(\dashint_{B(x_0,\frac{1}{A}R)}(M_2-u)^q \mathrm{d}m\right)^{\frac{1}{q}}+\left(\dashint_{B(x_0,\frac{1}{A}R)}(u-m_2)^q \mathrm{d}m\right)^{\frac{1}{q}}\right)\\
&\le C_1C_2 \left((M_2-M_1)+(m_1-m_2)+2\Psi \left(\frac{1}{A}R\right)^{\frac{1}{p-1}}\right),
\end{align*}
hence
\begin{align*}
\eosc_{B(x_0,\frac{1}{2A}R)}u\le\frac{C_1C_2-1}{C_1C_2}\eosc_{B(x_0,R)}u+2\Psi \left(\frac{1}{A}R\right)^{\frac{1}{p-1}}.
\end{align*}
By Lemma \ref{lem_ele2}, for any $\mu\in(0,1)$ and any $r\le R$, we have
\begin{align*}
&\eosc_{B(x_0,r)}u\le \frac{C_1C_2}{C_1C_2-1} \left(\frac{r}{R}\right)^{(1-\mu)\log_{2A}\frac{C_1C_2}{C_1C_2-1}}\eosc_{B(x_0,R)}u+2C_1C_2\Psi \left(\frac{1}{A}\left(\frac{r}{R}\right)^\mu R\right)^{\frac{1}{p-1}}\\
&\le\frac{C_1C_2}{C_1C_2-1} \left(\frac{r}{R}\right)^{(1-\mu)\log_{2A}\frac{C_1C_2}{C_1C_2-1}}\eosc_{B(x_0,R)}u+2C_1C_2C_\Psi^{\frac{1}{p-1}} \left(\frac{r}{R}\right)^{\frac{p\mu}{p-1}}\Psi \left(R\right)^{\frac{1}{p-1}},
\end{align*}
where the second inequality follows from (\ref{eq_Psi}), then there exist $\delta,\mu\in(0,1)$ depending only $p,A,C_1C_2$ such that $(1-\mu)\log_{2A}\frac{C_1C_2}{C_1C_2-1}\ge \delta$ and $\frac{p\mu}{p-1}\ge\delta$, hence
$$\eosc_{B(x_0,r)}u\le C_3 \left(\frac{r}{R}\right)^\delta \left(\eosc_{B(x_0,R)}u+\Psi(R)^{\frac{1}{p-1}}\right),$$
where $C_3=\max\{\frac{C_1C_2}{C_1C_2-1},2C_1C_2C_\Psi^{\frac{1}{p-1}}\}$.
\end{proof}

\begin{proof}[Proof of Proposition \ref{prop_bdy_osc_L}]
Let $q>0$, $A>1$, $C_1>1$ be the constants appearing in Proposition \ref{prop_bdy_wEHI_L}. Since $u\in \mathcal{F}^{(L)}(\Omega)$ is non-negative bounded in $\Omega$, for any $R\in(0,r_0)$, let
$$M_1=\esup_{B(x_0,\frac{1}{2A}R)\cap \Omega}u=\esup_{B(x_0,\frac{1}{2A}R)}u,\quad M_2=\esup_{B(x_0,R)\cap \Omega}u=\esup_{B(x_0,R)}u,$$
then $0\le M_1\le M_2<+\infty$, $M_2-u\in \mathcal{F}^{(L)}$ is non-negative bounded in $B(x_0,R)$, $M_2-\widetilde{u}=M_2$ q.e. on $B(x_0,R)\backslash\Omega$, $(M_2-u)\wedge M_2=M_2-u$ in $B(x_0,R)$, and $-\Delta^{(L)}_p(M_2-u)=-f\ge-1$ in $B(x_0,R)\cap\Omega$, by Proposition \ref{prop_bdy_wEHI_L}, we have
\begin{align*}
\left(\dashint_{B(x_0,\frac{1}{A}R)}\left(M_2-u\right)^q \mathrm{d}m\right)^{\frac{1}{q}}\le C_1 \left(\einf_{B(x_0,\frac{1}{2A}R)}(M_2-u)+\Psi \left(\frac{1}{A}R\right)^{\frac{1}{p-1}}\right),
\end{align*}
where
$$\einf_{B(x_0,\frac{1}{2A}R)}(M_2-u)=M_2-M_1.$$
Since $X\backslash\Omega$ has a $(c,r_0)$-corkscrew at $x_0$ and $R\in(0,r_0)$, we have $B(x_0,\frac{1}{A}R)\backslash\Omega$ contains a ball $B$ with radius $\frac{c}{A}R$. Since $M_2-\widetilde{u}=M_2$ q.e. on $B(x_0,\frac{1}{A}R)\backslash\Omega\supseteq B$, by \ref{eq_VD}, we have
$$\left(\dashint_{B(x_0,\frac{1}{A}R)}\left(M_2-u\right)^q \mathrm{d}m\right)^{\frac{1}{q}}\ge \left(\frac{m(B)}{V(x_0,\frac{1}{A}R)}\right)^{\frac{1}{q}}M_2\ge \frac{1}{C_2}M_2,$$
where $C_2=C_{VD}^{\frac{2-\log_2c}{q}}\ge1$. Hence
$$\esup_{B(x_0,\frac{1}{2A}R)}u\le \frac{C_1C_2-1}{C_1C_2}\esup_{B(x_0,R)}u+\Psi \left(\frac{1}{A}R\right)^{\frac{1}{p-1}}.$$
By Lemma \ref{lem_ele2}, for any $\mu\in(0,1)$ and any $R,r\in(0,r_0)$ with $r\le R$, we have
\begin{align*}
&\esup_{B(x_0,r)}u\le \frac{C_1C_2}{C_1C_2-1}\left(\frac{r}{R}\right)^{(1-\mu)\log_{2A}\frac{C_1C_2}{C_1C_2-1}}\esup_{B(x_0,R)}u+C_1C_2\Psi \left(\frac{1}{A}\left(\frac{r}{R}\right)^\mu R\right)^{\frac{1}{p-1}}\\
&\le \frac{C_1C_2}{C_1C_2-1}\left(\frac{r}{R}\right)^{(1-\mu)\log_{2A}\frac{C_1C_2}{C_1C_2-1}}\esup_{B(x_0,R)}u+C_1C_2C_\Psi^{\frac{1}{p-1}}\left(\frac{r}{R}\right)^{\frac{p\mu}{p-1}}\Psi \left(R\right)^{\frac{1}{p-1}},
\end{align*}
where the second inequality follows from (\ref{eq_Psi}), then there exist $\delta,\mu\in(0,1)$ depending only $p,A,C_1C_2$ such that $(1-\mu)\log_{2A}\frac{C_1C_2}{C_1C_2-1}\ge \delta$ and $\frac{p\mu}{p-1}\ge\delta$, hence
$$\esup_{B(x_0,r)\cap \Omega}u\le C_3 \left(\frac{r}{R}\right)^\delta \left(\esup_{B(x_0,R)\cap \Omega}u+\Psi(R)^{\frac{1}{p-1}}\right),$$
where $C_3=\max\{\frac{C_1C_2}{C_1C_2-1},C_1C_2C_\Psi^{\frac{1}{p-1}}\}$.
\end{proof}

We give the proof of ``\hyperlink{eq_CSL_weak}{$\text{CS}^{(L)}_{\text{weak}}(\Psi)$}$\Rightarrow$\hyperlink{eq_CE_strong}{$\text{CE}^{(L)}_{\text{strong}}(\Psi)$}" in Theorem \ref{thm_equiv_L} as follows.

\begin{proof}[Proof of ``\hyperlink{eq_CSL_weak}{$\text{CS}^{(L)}_{\text{weak}}(\Psi)$}$\Rightarrow$\hyperlink{eq_CE_strong}{$\text{CE}^{(L)}_{\text{strong}}(\Psi)$}" in Theorem \ref{thm_equiv_L}]
We refine the proof of \cite[Proposition 3.1]{Yan25d}. For any ball $B=B(x_0,r)$, by Lemma \ref{lem_corkscrew}, there exist $c\in(0,1)$, which depends only on $C_{cc}$, and an open subset $\Omega\subseteq X$ with $4B\subseteq \Omega\subseteq 5B$ such that for any $x\in \partial \Omega$, $X\backslash \Omega$ has a $(c,32r)$-corkscrew at $x$.

By Proposition \ref{prop_exist}, there exists a unique $u_\Omega\in \mathcal{F}^{(L)}(\Omega)$ such that
$$-\Delta^{(L)}_p u_\Omega+\frac{1}{\Psi(r)}\lvert u_\Omega\rvert^{p-2}u_\Omega=1\text{ in }\Omega.$$
By Proposition \ref{prop_comparisonL}, we have $0\le u_\Omega\le {\Psi(r)}^{\frac{1}{p-1}}$ in $X$. Similarly, there exists a unique $v\in \mathcal{F}^{(L)}(4B)$ such that $-\Delta_p^{(L)}v+\frac{1}{\Psi(r)}\lvert v\rvert^{p-2}v=1$ in $4B$. By \cite[Proposition 3.17 and Theorem 2.3]{Yan25d}, there exist $C_1,C_2>0$ such that
$$\einf_{2B}v\ge \frac{C_1}{\left(C_2+\frac{1}{\Psi(r)}\Psi(4r)\right)^{\frac{1}{p-1}}}\Psi(4r)^{\frac{1}{p-1}}\ge C_3\Psi(r)^{\frac{1}{p-1}},$$
where $C_3>0$ depends only on $p,C_\Psi,C_1,C_2$. Since $4B\subseteq \Omega$, by Proposition \ref{prop_comparisonL}, we have $u_\Omega\ge v$ in $4B$. Since $-\Delta^{(L)}_p u_\Omega=1-\frac{1}{\Psi(r)}\lvert u_\Omega\rvert^{p-2}u_\Omega$ in $\Omega$, where $0\le1-\frac{1}{\Psi(r)}\lvert u_\Omega\rvert^{p-2}u_\Omega\le1$ in $\Omega$, an initial application of Propositions \ref{prop_int_osc_L} and \ref{prop_bdy_osc_L} yields that $u_\Omega\in C \left(\overline{\Omega}\right)$ and $u_\Omega=0$ on $\partial \Omega$, and hence $u_\Omega\in C(X)$ with $u_\Omega=0$ on $X\backslash \Omega$. Hence
\begin{equation}\label{eq_u_lbd_L}
\inf_{B}u_\Omega\ge \inf_{2B}u_\Omega=\einf_{2B}u_\Omega\ge\einf_{2B}v\ge C_3\Psi(r)^{\frac{1}{p-1}}.
\end{equation}

Let $\delta_1,\delta_2\in(0,1)$ be the constants appearing in Propositions \ref{prop_int_osc_L}, \ref{prop_bdy_osc_L}, respectively, and $\delta=\min\{\delta_1,\delta_2\}\in(0,1)$.

We claim that there exists $C_4>0$ such that
\begin{equation}\label{eq_u_Holder_L}
\lvert u_\Omega(x)-u_\Omega(y)\rvert\le C_4 \left(\frac{d(x,y)}{r}\wedge 1\right)^\delta \Psi(r)^{\frac{1}{p-1}}\text{ for any }x,y\in X.
\end{equation}
Indeed, if $x,y\in X\backslash \Omega$, then this result is trivial, hence without loss of generality, we may assume that $x\in \Omega$, then $D=\mathrm{dist}(x,X\backslash \Omega)\in(0,10r)$, and there exists $y_0\in\partial\Omega$ such that $D\le d(y_0,x)<2D$. If $d(x,y)\ge r$, then this result follows directly from the fact that $0\le u_\Omega\le\Psi(r)^{\frac{1}{p-1}}$ in $X$, hence we may assume that $d(x,y)<r$.

If $y\in \Omega$ and $d(x,y)<D$, then by Proposition \ref{prop_int_osc_L}, we have
\begin{align*}
&\lvert u_\Omega(x)-u_\Omega(y)\rvert\lesssim\left(\frac{d(x,y)}{D}\right)^{\delta_1}\left(\mathrm{osc}_{B(x,D)}u_{\Omega}+\Psi(D)^{\frac{1}{p-1}}\right)\\
&\le \left(\frac{d(x,y)}{D}\right)^{\delta}\left(\sup_{B(y_0,3D)\cap\Omega}u_{\Omega}+\Psi(D)^{\frac{1}{p-1}}\right).
\end{align*}
Since $3D<32r$, by Proposition \ref{prop_bdy_osc_L}, we have
\begin{align*}
\sup_{B(y_0,3D)\cap\Omega}u_\Omega\lesssim \left(\frac{D}{r}\right)^{\delta_2}\left(\sup_Xu_{\Omega}+\Psi(r)^{\frac{1}{p-1}}\right)\lesssim \left(\frac{D}{r}\right)^{\delta}\Psi(r)^{\frac{1}{p-1}}.
\end{align*}
By (\ref{eq_Psi}), we have
$$\Psi(D)^{\frac{1}{p-1}}\lesssim \left(\frac{D}{r}\right)^{\frac{p}{p-1}}\Psi(r)^{\frac{1}{p-1}}\le\left(\frac{D}{r}\right)^{\delta}\Psi(r)^{\frac{1}{p-1}}.$$
Hence
$$\lvert u_\Omega(x)-u_{\Omega}(y)\rvert\lesssim \left(\frac{d(x,y)}{r}\right)^{\delta}\Psi(r)^{\frac{1}{p-1}}.$$

If $y\in\Omega$ and $d(x,y)\ge D$, then $x,y\in B(y_0,3d(x,y))\subseteq B(y_0,3r)$, hence by Proposition \ref{prop_bdy_osc_L}, we have
\begin{align*}
&\lvert u_\Omega(x)-u_\Omega(y)\rvert\le 2\sup_{B(y_0,3d(x,y))\cap\Omega}u_\Omega\\
&\lesssim  \left(\frac{d(x,y)}{r}\right)^{\delta_2}\left(\sup_Xu_\Omega+\Psi(r)^{\frac{1}{p-1}}\right)\lesssim \left(\frac{d(x,y)}{r}\right)^{\delta}\Psi(r)^{\frac{1}{p-1}}.
\end{align*}

If $y\in X\backslash\Omega$, then $D\le d(x,y)$, hence by Proposition \ref{prop_bdy_osc_L}, we have
\begin{align*}
&\lvert u_\Omega(x)-u_\Omega(y)\rvert=u_\Omega(x)\le \sup_{B(y_0,2D)\cap\Omega}u_\Omega\\
&\lesssim \left(\frac{D}{r}\right)^{\delta_2}\left(\sup_Xu_\Omega+\Psi(r)^{\frac{1}{p-1}}\right)\lesssim \left(\frac{d(x,y)}{r}\right)^{\delta}\Psi(r)^{\frac{1}{p-1}}.
\end{align*}
Now we finish the proof of (\ref{eq_u_Holder_L}).

We claim that there exists $C_5>0$ such that
\begin{equation}\label{eq_u_energy_L}
\int_{B(x,s)}\mathrm{d}\Gamma^{(L)}(u_\Omega)\le C_5 \left(\frac{s}{r}\wedge 1\right)^\delta\Psi(r)^{\frac{p}{p-1}}\frac{V(x,s)}{\Psi(s\wedge r)}\text{ for any }x\in X,s>0.
\end{equation}
Indeed, by \ref{eq_VD}, \ref{eq_ucapL}, there exists a cutoff function $\psi\in \mathcal{F}^{(L)}$ for $B(x,s)\subseteq B(x,2s)$ such that $\mathcal{E}^{(L)}(\psi)\lesssim \frac{V(x,s)}{\Psi(s)}$, then
\begin{align}
&\int_{B(x,s)}\mathrm{d}\Gamma^{(L)}(u_\Omega)\le \int_X\psi^p \mathrm{d}\Gamma^{(L)}(u_\Omega)=\int_X\psi^p \mathrm{d}\Gamma^{(L)}\left(u_\Omega;u_\Omega-\inf_{B(x,2s)}u_\Omega\right)\nonumber\\
&=\mathcal{E}^{(L)}\left(u_\Omega;\psi^p \left(u_\Omega-\inf_{B(x,2s)}u_\Omega\right)\right)-p\int_X\psi^{p-1}\left(u_\Omega-\inf_{B(x,2s)}u_\Omega\right)\mathrm{d}\Gamma^{(L)}(u_\Omega;\psi).\label{eq_CS2CE_L1}
\end{align}
By (\ref{eq_u_Holder_L}), we have
\begin{equation}\label{eq_u_osc_L}
0\le u_\Omega-\inf_{B(x,2s)}u_\Omega\lesssim \left(\frac{s}{r}\wedge1\right)^\delta \Psi(r)^{\frac{1}{p-1}}\text{ in }B(x,2s).
\end{equation}
Since $-\Delta_p^{(L)}u_\Omega+\frac{1}{\Psi(r)}\lvert u_\Omega\rvert^{p-2}u_\Omega=1$ in $\Omega$, and $\psi^p \left(u_\Omega-\inf_{B(x,2s)}u_\Omega\right)\in \mathcal{F}^{(L)}(\Omega)$ is non-negative in $X$, we have
\begin{align}
&\mathcal{E}^{(L)}\left(u_\Omega;\psi^p \left(u_\Omega-\inf_{B(x,2s)}u_\Omega\right)\right)\nonumber\\
&\le \mathcal{E}^{(L)}\left(u_\Omega;\psi^p \left(u_\Omega-\inf_{B(x,2s)}u_\Omega\right)\right)+\frac{1}{\Psi(r)}\int_X \lvert u_\Omega\rvert^{p-2}u_\Omega\psi^p \left(u_\Omega-\inf_{B(x,2s)}u_\Omega\right)\mathrm{d}m\nonumber\\
&=\int_X\psi^p \left(u_\Omega-\inf_{B(x,2s)}u_\Omega\right)\mathrm{d}m\lesssim \left(\frac{s}{r}\wedge 1\right)^\delta \Psi(r)^{\frac{1}{p-1}}V(x,s),\label{eq_CS2CE_L2}
\end{align}
where the last inequality follows from (\ref{eq_u_osc_L}) and \ref{eq_VD}. Moreover, we have
\begin{align*}
&\lvert \int_X\psi^{p-1}\left(u_\Omega-\inf_{B(x,2s)}u_\Omega\right)\mathrm{d}\Gamma^{(L)}(u_\Omega;\psi)\rvert\\
&\le \left(\int_X\psi^p \mathrm{d}\Gamma^{(L)}(u_\Omega)\right)^{\frac{p-1}{p}}\left(\int_X \left(u_\Omega-\inf_{B(x,2s)}u_\Omega\right)^p \mathrm{d}\Gamma^{(L)}(\psi)\right)^{\frac{1}{p}}\\
&\overset{(\star)}{\scalebox{2}[1]{$\lesssim$}} \left(\frac{s}{r}\wedge1\right)^\delta \Psi(r)^{\frac{1}{p-1}}\left(\int_X\psi^p \mathrm{d}\Gamma^{(L)}(u_\Omega)\right)^{\frac{p-1}{p}}\mathcal{E}^{(L)}(\psi)^{\frac{1}{p}}\\
&\lesssim\left(\frac{s}{r}\wedge1\right)^\delta \Psi(r)^{\frac{1}{p-1}}\left(\frac{V(x,s)}{\Psi(s)}\right)^{\frac{1}{p}}\left(\int_X\psi^p \mathrm{d}\Gamma^{(L)}(u_\Omega)\right)^{\frac{p-1}{p}},
\end{align*}
where $(\star)$ also follows from (\ref{eq_u_osc_L}). Combining this with (\ref{eq_CS2CE_L1}) and (\ref{eq_CS2CE_L2}), there exists $C_6>0$ such that
\begin{align*}
&\int_X\psi^p \mathrm{d}\Gamma^{(L)}(u_\Omega)\\
&\le C_6\left(\frac{s}{r}\wedge 1\right)^\delta \Psi(r)^{\frac{1}{p-1}}V(x,s)+C_6\left(\frac{s}{r}\wedge1\right)^\delta \Psi(r)^{\frac{1}{p-1}}\left(\frac{V(x,s)}{\Psi(s)}\right)^{\frac{1}{p}}\left(\int_X\psi^p \mathrm{d}\Gamma^{(L)}(u_\Omega)\right)^{\frac{p-1}{p}}\\
&\le C_6\left(\frac{s}{r}\wedge 1\right)^\delta \Psi(r)^{\frac{1}{p-1}}V(x,s)+\frac{1}{2}\int_X\psi^p \mathrm{d}\Gamma^{(L)}(u_\Omega)+C_7\left(\frac{s}{r}\wedge1\right)^{p\delta} \Psi(r)^{\frac{p}{p-1}}\frac{V(x,s)}{\Psi(s)},
\end{align*}
where the last inequality follows from Young's inequality, and $C_7>0$ depends only on $p,C_6$, hence
\begin{align*}
&\int_{B(x,s)}\mathrm{d}\Gamma^{(L)}(u_\Omega)\le \int_X\psi^p \mathrm{d}\Gamma^{(L)}(u_\Omega)\\
&\le 2C_6\left(\frac{s}{r}\wedge 1\right)^\delta \Psi(r)^{\frac{1}{p-1}}V(x,s)+2C_7\left(\frac{s}{r}\wedge1\right)^{p\delta} \Psi(r)^{\frac{p}{p-1}}\frac{V(x,s)}{\Psi(s)}\\
&\le 2(C_6+C_7)\left(\frac{s}{r}\wedge 1\right)^\delta\Psi(r)^{\frac{p}{p-1}}\left(\frac{1}{\Psi(r)}+\frac{1}{\Psi(s)}\right)V(x,s)\\
&\le 4(C_6+C_7)\left(\frac{s}{r}\wedge 1\right)^\delta\Psi(r)^{\frac{p}{p-1}}\frac{V(x,s)}{\Psi(s\wedge r)}.
\end{align*}
Now we finish the proof of (\ref{eq_u_energy_L}).

Finally, let $\phi=\left(\frac{1}{C_3\Psi(r)^{\frac{1}{p-1}}}u_\Omega\right)\wedge 1$, then $\phi\in \mathcal{F}^{(L)}(\Omega)\subseteq \mathcal{F}^{(L)}(8B)$. By (\ref{eq_u_lbd_L}), we have $\phi=1$ in $B$, hence $\phi$ is a cutoff function for $B\subseteq 8B$. By (\ref{eq_u_Holder_L}), we have
$$\lvert \phi(x)-\phi(y)\rvert\le \frac{C_4}{C_3}\left(\frac{d(x,y)}{r}\wedge 1\right)^\delta\text{ for any }x,y\in X,$$
that is, \hyperref[eq_CE_Holder]{$\text{CE}^{(L)}(\Psi)\text{-}1$} holds. By (\ref{eq_u_energy_L}), for any $x\in X$, $s>0$, we have
$$\int_{B(x,s)}\mathrm{d}\Gamma^{(L)}(\phi)\le \frac{1}{C_3^p\Psi(r)^{\frac{p}{p-1}}}\int_{B(x,s)}\mathrm{d}\Gamma^{(L)}(u_\Omega)\le \frac{C_5}{C_3^p}\left(\frac{s}{r}\wedge1\right)^\delta \frac{V(x,s)}{\Psi(s\wedge r)},$$
that is, \hyperref[eq_CE_energy]{$\text{CE}^{(L)}(\Psi)\text{-}2$} holds.
\end{proof}

\section{Two oscillation inequalities}\label{sec_Moser_J}

Throughout this section, we always assume that \ref{eq_VD} holds and that $(\mathcal{E}^{(J)},\mathcal{F}^{(J)})$ is a non-local regular $p$-energy given by a kernel $K^{(J)}$ satisfying \ref{eq_UPR}, \ref{eq_KJ_tail}, \ref{eq_PIJ}, \hyperlink{eq_CSJ_weak}{$\text{CS}^{(J)}_{\text{weak}}(\Upsilon)$}. The main results of this section are the following interior and boundary oscillation inequalities, which will play an important role in Section \ref{sec_CS2CE_J}.

\begin{proposition}\label{prop_int_osc_J}
There exist $\delta\in(0,1)$, $C>0$ such that for any bounded open subset $\Omega\subseteq X$, let $f\in L^\infty(\Omega;m)$ satisfy $\lVert f\rVert_{L^\infty(\Omega;m)}\le1$, let $u\in \widehat{\mathcal{F}}^{(J)}\cap L^\infty(X;m)$ satisfy $-\Delta_p^{(J)}u=f$ in $\Omega$, then for any $x_0\in \Omega$ and any $R,r>0$ with $r\le R$ and $B(x_0,R)\subseteq\Omega$, we have
$$\eosc_{B(x_0,r)}u\le C \left(\frac{r}{R}\right)^\delta \left(\lVert u\rVert_{L^\infty(B(x_0,R);m)}+\Upsilon(R)^{\frac{1}{p-1}}\mathrm{Tail}^{(J)}(u_-;x_0,R)+\Upsilon(R)^{\frac{1}{p-1}}\right).$$
In particular, we have $u\in C(\Omega)$.
\end{proposition}

\begin{proposition}\label{prop_bdy_osc_J}
Let $c\in(0,1)$. Then there exist $\delta\in(0,1)$, $C>0$ such that for any bounded open subset $\Omega\subseteq X$, let $f\in L^\infty(\Omega;m)$ satisfy $\lVert f\rVert_{L^\infty(\Omega;m)}\le 1$, let $u\in {\mathcal{F}}^{(J)}(\Omega)\cap L^\infty(X;m)$ be non-negative in $\Omega$ and satisfy $-\Delta_p^{(J)}u=f$ in $\Omega$, let $x_0\in \partial\Omega$ satisfy $X\backslash \Omega$ has a $(c,r_0)$-corkscrew at $x_0$ for some $r_0>0$, then for any $R,r\in(0,r_0)$ with $r\le R$, we have
$$\esup_{B(x_0,r)\cap \Omega}u\le C \left(\frac{r}{R}\right)^\delta \left(\esup_X u+\Upsilon(R)^{\frac{1}{p-1}}\right).$$
In particular, we have $\lim_{r\downarrow0}\esup_{B(x_0,r)\cap \Omega}u=0$.
\end{proposition}

The proofs rely on the following interior and boundary weak Harnack inequalities. In comparison with the local case, namely Propositions \ref{prop_int_wEHI_L} and \ref{prop_bdy_wEHI_L}, special attention must be paid to the tail terms. We refer to \cite[Theorems 3.6 and 3.7]{KLL23} for weak Harnack inequalities for superharmonic functions in $\mathbb{R}^n$.

\begin{proposition}\label{prop_int_wEHI_J}
There exist $q>0$, $A>1$, $C>1$ such that for any $x_0\in X$ and any $R>0$, let $u\in \widehat{\mathcal{F}}^{(J)}\cap L^\infty(X;m)$ be non-negative in $B(x_0,AR)$, and satisfy $-\Delta_p^{(J)}u\ge-1$ in $B(x_0,AR)$, then we have
$$\left(\dashint_{B(x_0,R)}u^q \mathrm{d}m\right)^{\frac{1}{q}}\le C \left(\einf_{B(x_0,\frac{1}{2}R)}u+\Upsilon(R)^{\frac{1}{p-1}}\mathrm{Tail}^{(J)}\left(u_-;x_0,AR\right)+\Upsilon(R)^{\frac{1}{p-1}}\right).$$
\end{proposition}

\begin{proposition}\label{prop_bdy_wEHI_J}
There exist $q>0$, $A>1$, $C>1$ such that the following holds. For any bounded open subset $\Omega\subseteq X$, any $x_0\in X$, and any $R>0$, let $u\in \widehat{\mathcal{F}}^{(J)}\cap L^\infty(X;m)$ be non-negative in $B(x_0,AR)$ and satisfy $\widetilde{u}=M$ q.e. on $B(x_0,AR)\backslash \Omega$, for some constant $M\in [0,+\infty)$, and $-\Delta_p^{(J)}u\ge-1$ in $B(x_0,AR)\cap \Omega$, then we have
\begin{align*}
&\left(\dashint_{B(x_0,R)}(u\wedge M)^q \mathrm{d}m\right)^{\frac{1}{q}}\\
&\le C \left(\einf_{B(x_0,\frac{1}{2}R)}(u\wedge M)+\Upsilon(R)^{\frac{1}{p-1}}\mathrm{Tail}^{(J)}\left((u\wedge M)_-;x_0,AR\right)+\Upsilon(R)^{\frac{1}{p-1}}\right).
\end{align*}
\end{proposition}

It suffices to prove Proposition \ref{prop_bdy_wEHI_J}. Indeed, assuming Proposition \ref{prop_bdy_wEHI_J}, by taking $\Omega=B(x_0,AR)$ and $M=\lVert u\rVert_{L^\infty(B(x_0,AR);m)}<+\infty$, and noting that $u\wedge M=u$ in $B(x_0,AR)$ and $(u\wedge M)_-=u_-$ in $X$, we obtain Proposition \ref{prop_int_wEHI_J}.

\begin{lemma}\label{lem_Moser_J}
For any $q>0$, there exists $C>1$ such that the following holds. For any bounded open subset $\Omega\subseteq X$, any $x_0\in X$, and any $R_0>0$, let $u\in \widehat{\mathcal{F}}^{(J)}\cap L^\infty(X;m)$ be non-negative in $B(x_0,R_0)$ and satisfy $\widetilde{u}=M$ q.e. on $B(x_0,R_0)\backslash\Omega$ for some constant $M\in[0,+\infty)$, and $-\Delta_p^{(J)}u\ge-1$ in $B(x_0,R_0)\cap\Omega$, then for any $R\in(0,R_0)$, we have
\begin{align*}
&\einf_{B(x_0,\frac{1}{2}R)}(u\wedge M)+\Upsilon(R)^{\frac{1}{p-1}}\mathrm{Tail}^{(J)}\left((u\wedge M)_-;x_0,R\right)+\Upsilon(R)^{\frac{1}{p-1}}\\
&\ge \frac{1}{C}\left(\dashint_{B(x_0,R)}\left(u\wedge M+\Upsilon(R)^{\frac{1}{p-1}}\mathrm{Tail}^{(J)}\left((u\wedge M)_-;x_0,R\right)+\Upsilon(R)^{\frac{1}{p-1}}\right)^{-q}\mathrm{d}m\right)^{-\frac{1}{q}}.
\end{align*}
\end{lemma}

\begin{proof}
For notational convenience, we may assume that all functions in $\mathcal{F}^{(J)}$ are quasi-continuous. If $M=0$, then the result is trivial, hence we may assume that $M>0$. For any $R,r>0$ with $r\le R$ and $R+r<R_0$, for any
$$K\ge \Upsilon(R)^{\frac{1}{p-1}}\mathrm{Tail}^{(J)}\left((u\wedge M)_-;x_0,R\right)+\Upsilon(r)^{\frac{1}{p-1}},$$
let $\overline{u}=u\wedge M+K$, then by Lemma \ref{lem_cutoff}, we have $-\Delta_p^{(J)}\overline{u}\ge-1$ in $B(x_0,R_0)\cap \Omega$. Let $\phi\in \mathcal{F}^{(J)}$ be a cutoff function for $B(x_0,R)\subseteq B(x_0,R+\frac{1}{2}r)$ to be chosen later. For any $\alpha< -(p-1)$, let
$$v=\phi^p \left(\overline{u}^\alpha-(M+K)^\alpha\right)=\phi^p \left((u\wedge M+K)^\alpha-(M+K)^\alpha\right),$$
then $v\in \mathcal{F}^{(J)}(B(x_0,R+\frac{1}{2}r)\cap\Omega)$ is non-negative, hence
\begin{equation}\label{eq_wEHI_J1}
\mathcal{E}^{(J)}(\overline{u};v)\ge-\int_X v \mathrm{d}m\ge -\frac{1}{\Upsilon(r)}\int_{B(x_0,R+r)}\overline{u}^{\alpha+p-1}\mathrm{d}m,
\end{equation}
where
\begin{align}
&\mathcal{E}^{(J)}(\overline{u};v)=\left(\int_{B(x_0,R+r)}\int_{B(x_0,R+r)}+2\int_{B(x_0,R+r)}\int_{X\backslash B(x_0,R+r)}\right)\ldots \mathrm{d}m \mathrm{d}m=I_1+2I_2.\label{eq_wEHI_J4}
\end{align}
For $I_2$, we have
\begin{align*}
&I_2=\int_{B(x_0,R+\frac{1}{2}r)}\int_{X\backslash B(x_0,R+r)}\lvert \overline{u}(x)-\overline{u}(y)\rvert^{p-2}(\overline{u}(x)-\overline{u}(y))\\
&\hspace{120pt}\cdot\phi(x)^p \left(\overline{u}(x)^\alpha -(M+K)^\alpha\right)K^{(J)}(x,y)m(\mathrm{d}x)m(\mathrm{d}y)\\
&\le\int_{B(x_0,R+\frac{1}{2}r)}\int_{X\backslash B(x_0,R+r)}\left( \overline{u}(x)-\overline{u}(y)\right)^{p-1} \overline{u}(x)^\alpha K^{(J)}(x,y)1_{\overline{u}(x)\ge \overline{u}(y)}m(\mathrm{d}x)m(\mathrm{d}y)\\
&\overset{(\star)}{\scalebox{2}[1]{$\lesssim$}}\int_{B(x_0,R+\frac{1}{2}r)}\int_{X\backslash B(x_0,R+r)}\overline{u}(x)^{\alpha+p-1}  K^{(J)}(x,y)m(\mathrm{d}x)m(\mathrm{d}y)\\
&\hspace{10pt}+\int_{B(x_0,R+\frac{1}{2}r)}\int_{X\backslash B(x_0,R+r)}(u(y)\wedge M)_-^{p-1}\overline{u}(x)^{\alpha}K^{(J)}(x,y)m(\mathrm{d}x)m(\mathrm{d}y)\\
&=I_{21}+I_{22},
\end{align*}
where $(\star)$ follows from the fact that
\begin{align*}
&\left(\overline{u}(x)-\overline{u}(y)\right)^{p-1}1_{\overline{u}(x)\ge \overline{u}(y)}=\left({u}(x)\wedge M-{u}(y)\wedge M\right)^{p-1}1_{{u}(x)\wedge M\ge {u}(y)\wedge M}\\
&\le 2^{p-1}\left(\left(u(x)\wedge M\right)^{p-1}+\left(u(y)\wedge M\right)_-^{p-1}\right)\le2^{p-1}\left(\overline{u}(x)^{p-1}+\left(u(y)\wedge M\right)_-^{p-1}\right),
\end{align*}
due to $u\ge0$ in $B(x_0,R+\frac{1}{2}r)$. By \ref{eq_KJ_tail}, we have
\begin{align*}
&I_{21}\le\int_{B(x_0,R+\frac{1}{2}r)}\overline{u}(x)^{\alpha+p-1}\left(\int_{X\backslash B(x,\frac{1}{2}r)}  K^{(J)}(x,y)m(\mathrm{d}y)\right)m(\mathrm{d}x)\\
&\lesssim \frac{1}{\Upsilon(r)}\int_{B(x_0,R+\frac{1}{2}r)}\overline{u}^{\alpha+p-1}\mathrm{d}m.
\end{align*}
Since $r\le R$, by \ref{eq_UPR}, for any $x\in B(x_0,R+\frac{1}{2}r)$ and $y\in X\backslash B(x_0,R+r)$, we have
\begin{align*}
&K^{(J)}(x,y)\lesssim \left(\frac{d(x_0,y)}{d(x,y)}\vee 1\right)^\theta K^{(J)}(x_0,y)\\
&\le \left(\frac{d(x,y)+d(x_0,x)}{d(x,y)}\right)^\theta K^{(J)}(x_0,y)\lesssim \left(\frac{R}{r}\right)^{\theta}K^{(J)}(x_0,y),
\end{align*}
which gives
\begin{align*}
&I_{22}\lesssim \left(\frac{R}{r}\right)^{\theta}\mathrm{Tail}^{(J)}({(u\wedge M)}_-;x_0,R)^{p-1}\int_{B(x_0,R+\frac{1}{2}r)}\overline{u}^{\alpha}\mathrm{d}m\\
&\le\left(\frac{R}{r}\right)^{\theta}\frac{\Upsilon(R)}{\Upsilon(r)} \mathrm{Tail}^{(J)}({(u\wedge M)}_-;x_0,R)^{p-1}\int_{B(x_0,R+\frac{1}{2}r)}\overline{u}^{\alpha}\mathrm{d}m\\
&\le\left(\frac{R}{r}\right)^{\theta}\frac{1}{\Upsilon(r)} \int_{B(x_0,R+\frac{1}{2}r)}\overline{u}^{\alpha+p-1}\mathrm{d}m,
\end{align*}
where the last inequality follows from the fact that
$$\Upsilon(R)^{\frac{1}{p-1}}\mathrm{Tail}^{(J)}({(u\wedge M)}_-;x_0,R)\le K\le \overline{u}\text{ in }B(x_0,R+\frac{1}{2}r).$$
Thus, we have
\begin{align}
&I_2\lesssim I_{21}+I_{22}\lesssim \left(\frac{R}{r}\right)^{\theta}\frac{1}{\Upsilon(r)}\int_{B(x_0,R+\frac{1}{2}r)}\overline{u}^{\alpha+p-1}\mathrm{d}m.\label{eq_wEHI_J5}
\end{align}
By Lemma \ref{lem_ele4}, since in $B(x_0,R+r)$, $v=\alpha\phi^p F(\overline{u}-K)$ and $G(\overline{u}-K)=\frac{p}{\alpha+p-1}\overline{u}^{\frac{\alpha+p-1}{p}}$, where $F,G$ are the functions defined therein, we have for any $x,y\in B(x_0,R+r)$,
\begin{align*}
&\lvert \overline{u}(x)^{\frac{\alpha+p-1}{p}}-\overline{u}(y)^{\frac{\alpha+p-1}{p}}\rvert^p\max \{\phi(x)^p,\phi(y)^p\}\\
&\le 2C\left(1+\left(\frac{\lvert \alpha+p-1\rvert}{\lvert \alpha\rvert}\right)^p\right)\max\{\overline{u}(x)^{\alpha+p-1},\overline{u}(y)^{\alpha+p-1}\}\lvert \phi(x)-\phi(y)\rvert^p\\
&\hspace{15pt}-\frac{2\lvert \alpha+p-1\rvert^p}{p^p\lvert \alpha\rvert}\lvert \overline{u}(x)-\overline{u}(y)\rvert^{p-2}(\overline{u}(x)-\overline{u}(y))(v(x)-v(y)),
\end{align*}
where $C>0$ depends only on $p$. Hence
\begin{align*}
&I_3=\int\limits_{B(x_0,R+r)}\int\limits_{B(x_0,R+r)}\lvert \overline{u}(x)^{\frac{\alpha+p-1}{p}}-\overline{u}(y)^{\frac{\alpha+p-1}{p}}\rvert^p\max \{\phi(x)^p,\phi(y)^p\}K^{(J)}(x,y)m(\mathrm{d}x) m(\mathrm{d}y)\\
&\le4C\left(1+\left(\frac{\lvert \alpha+p-1\rvert}{\lvert \alpha\rvert}\right)^p\right)\\
&\hspace{50pt}\cdot\int_{B(x_0,R+r)}\int_{B(x_0,R+r)}\overline{u}(x)^{\alpha+p-1}\lvert \phi(x)-\phi(y)\rvert^p K^{(J)}(x,y) m(\mathrm{d}x)m(\mathrm{d}y)\\
&\hspace{15pt}-\frac{2\lvert \alpha+p-1\rvert^p}{p^p\lvert \alpha\rvert}I_1\\
&\le C_1\int_{B(x_0,R+r)}\overline{u}^{\alpha+p-1}\mathrm{d}\Gamma^{(J)}_{B(x_0,R+r)}(\phi)-\frac{2\lvert \alpha+p-1\rvert^p}{p^p\lvert \alpha\rvert}I_1,
\end{align*}
where $C_1=4C\left(1+\sup_{\alpha<-(p-1)}\left(\frac{\lvert \alpha+p-1\rvert}{\lvert \alpha\rvert}\right)^p\right)\in(0,+\infty)$ depends only on $p$. By Proposition \ref{prop_CSJ_self}, there exist $C_2>0$ depending on $C_1$, and a cutoff function $\phi\in \mathcal{F}^{(J)}$ for $B(x_0,R)\subseteq B(x_0,R+\frac{1}{2}r)$, such that
\begin{align}
&\int_{B(x_0,R+r)}\overline{u}^{\alpha+p-1}\mathrm{d}\Gamma^{(J)}_{B(x_0,R+r)}(\phi)\nonumber\\
&\le \frac{1}{2C_1}\int_{B(x_0,R+r)}\lvert \phi\rvert^p \mathrm{d}\Gamma_{B(x_0,R+r)}^{(J)}\left(\overline{u}^{\frac{\alpha+p-1}{p}}\right)+\frac{C_2}{\Upsilon(r)}\int_{B(x_0,R+r)}\overline{u}^{\alpha+p-1} \mathrm{d}m\nonumber\\
&\le \frac{1}{2C_1}I_3+\frac{C_2}{\Upsilon(r)}\int_{B(x_0,R+r)}\overline{u}^{\alpha+p-1} \mathrm{d}m,\label{eq_wEHI_J6}
\end{align}
which gives
$$I_3\le \frac{2C_1C_2}{\Upsilon(r)}\int_{B(x_0,R+r)}\overline{u}^{\alpha+p-1} \mathrm{d}m-\frac{4\lvert \alpha+p-1\rvert^p}{p^p\lvert \alpha\rvert}I_1.$$
Combining this with (\ref{eq_wEHI_J1})--(\ref{eq_wEHI_J6}), we have
\begin{align*}
&\int_{B(x_0,R+r)}\mathrm{d}\Gamma_{B(x_0,R+r)}^{(J)}\left(\phi \overline{u}^{\frac{\alpha+p-1}{p}}\right)\\
&\lesssim \int_{B(x_0,R+r)}\lvert \phi\rvert^p\mathrm{d}\Gamma_{B(x_0,R+r)}^{(J)}\left(\overline{u}^{\frac{\alpha+p-1}{p}}\right)+\int_{B(x_0,R+r)}\overline{u}^{{\alpha+p-1}}\mathrm{d}\Gamma_{B(x_0,R+r)}^{(J)}\left(\phi \right)\\
&\lesssim I_3+\frac{1}{\Upsilon(r)}\int_{B(x_0,R+r)}\overline{u}^{\alpha+p-1} \mathrm{d}m\\
&\lesssim \frac{1}{\Upsilon(r)}\int_{B(x_0,R+r)}\overline{u}^{\alpha+p-1} \mathrm{d}m-\frac{\lvert \alpha+p-1\rvert^p}{\lvert \alpha\rvert}I_1\\
&=\frac{1}{\Upsilon(r)}\int_{B(x_0,R+r)}\overline{u}^{\alpha+p-1} \mathrm{d}m+\frac{\lvert \alpha+p-1\rvert^p}{\lvert \alpha\rvert}\left(2I_2-\mathcal{E}^{(J)}(\overline{u};v)\right)\\
&\lesssim \lvert \alpha\rvert^{p-1}\left(\frac{R}{r}\right)^{\theta}\frac{1}{\Upsilon(r)}\int_{B(x_0,R+r)}\overline{u}^{\alpha+p-1} \mathrm{d}m.
\end{align*}
Noting that
\begin{align*}
&\mathcal{E}^{(J)}\left(\phi \overline{u}^{\frac{\alpha+p-1}{p}}\right)\\
&=\int_{B(x_0,R+r)}\mathrm{d}\Gamma_{B(x_0,R+r)}^{(J)}\left(\phi \overline{u}^{\frac{\alpha+p-1}{p}}\right)+2\int_{B(x_0,R+r)}\mathrm{d}\Gamma_{X\backslash B(x_0,R+r)}^{(J)}\left(\phi \overline{u}^{\frac{\alpha+p-1}{p}}\right),
\end{align*}
where by \ref{eq_KJ_tail}, we have
\begin{align*}
&\int_{B(x_0,R+r)}\mathrm{d}\Gamma_{X\backslash B(x_0,R+r)}^{(J)}\left(\phi \overline{u}^{\frac{\alpha+p-1}{p}}\right)\\
&=\int_{B(x_0,R+\frac{1}{2}r)}\int_{X\backslash B(x_0,R+r)}\phi(x)^p \overline{u}(x)^{\alpha+p-1}K^{(J)}(x,y)m(\mathrm{d}x)m(\mathrm{d}y)\\
&\le\int_{B(x_0,R+\frac{1}{2}r)}\overline{u}(x)^{\alpha+p-1}\left(\int_{X\backslash B(x,\frac{1}{2}r)} K^{(J)}(x,y)m(\mathrm{d}y)\right)m(\mathrm{d}x)\\
&\lesssim\frac{1}{\Upsilon(r)}\int_{B(x_0,R+\frac{1}{2}r)}\overline{u}^{\alpha+p-1}\mathrm{d}m\le \frac{1}{\Upsilon(r)}\int_{B(x_0,R+r)}\overline{u}^{\alpha+p-1}\mathrm{d}m,
\end{align*}
hence
$$\mathcal{E}^{(J)}\left(\phi \overline{u}^{\frac{\alpha+p-1}{p}}\right)\lesssim \lvert \alpha\rvert^{p-1}\left(\frac{R}{r}\right)^{\theta}\frac{1}{\Upsilon(r)}\int_{B(x_0,R+r)}\overline{u}^{\alpha+p-1} \mathrm{d}m.$$
Since $\phi \overline{u}^{\frac{\alpha+p-1}{p}}\in \mathcal{F}^{(J)}(B(x_0,R+r))$, by Proposition \ref{prop_SobolevJ}, we have
\begin{align*}
\left(\int_{B(x_0,R+r)}\lvert \phi \overline{u}^{\frac{\alpha+p-1}{p}}\rvert^{p\kappa}\mathrm{d}m\right)^{\frac{1}{\kappa}}\lesssim \frac{\Upsilon(R+r)}{V(x_0,R+r)^{\frac{\kappa-1}{\kappa}}}\mathcal{E}^{(J)}\left(\phi \overline{u}^{\frac{\alpha+p-1}{p}}\right).
\end{align*}
Hence
\begin{align}
&\left(\dashint_{B(x_0,R)}({u}\wedge M+K)^{\kappa(\alpha+p-1)}\mathrm{d}m\right)^{\frac{1}{\kappa}}\nonumber\\
&\le C_3 \lvert \alpha\rvert^{p-1} \left(\frac{R}{r}\right)^{\theta}\frac{\Upsilon(R+r)}{\Upsilon(r)}\dashint_{B(x_0,R+r)}({u}\wedge M+K)^{\alpha+p-1}\mathrm{d}m,\label{eq_wEHI_J_iteration}
\end{align}
for any $\alpha<-(p-1)$, $r\le R$, $K\ge \Upsilon(R)^{\frac{1}{p-1}}\mathrm{Tail}^{(J)}\left((u\wedge M)_-;x_0,R\right)+\Upsilon(r)^{\frac{1}{p-1}}$, where $C_3>0$ is some constant. For any $q>0$ and any $R\in(0,R_0)$. For any $k\ge0$, let $R_k=(\frac{1}{2}+\frac{1}{2^{k+1}})R$ and $q_k=-q\kappa^k$, then $R=R_0>R_1>\ldots>R_k\downarrow \frac{1}{2}R$, and $-q=q_0>q_1>\ldots>q_k\downarrow-\infty$, let
$$J_k=\left(\dashint_{B(x_0,R_k)}\left(u\wedge M+\Upsilon(R)^{\frac{1}{p-1}}\mathrm{Tail}^{(J)}\left((u\wedge M)_-;x_0,R\right)+\Upsilon(R)^{\frac{1}{p-1}}\right)^{q_k}\mathrm{d}m\right)^{\frac{1}{q_k}}.$$
Since $u\ge0$ in $B(x_0,R_0)$, we have
$$\mathrm{Tail}^{(J)}\left((u\wedge M)_-;x_0,R_k\right)=\mathrm{Tail}^{(J)}\left((u\wedge M)_-;x_0,R\right).$$
By (\ref{eq_wEHI_J_iteration}), we have
\begin{align*}
&J_{k+1}^{-q}\le \left(C_3 \left(q\kappa^k+p-1\right)^{p-1}\left(\frac{R_{k+1}}{R_k-R_{k+1}}\right)^{\theta}\frac{\Upsilon(R_k)}{\Upsilon(R_k-R_{k+1})}\right)^{\frac{1}{\kappa^k}}J_k^{-q}\\
&\le C_4^{\frac{1}{\kappa^k}}C_5^{\frac{k}{\kappa^k}}J_k^{-q},
\end{align*}
where $C_4\ge1$ depends only on $p,q,\kappa,C_\Upsilon,C_{VD},C_3$, and $C_5=2^\theta\kappa^{p-1}C_\Upsilon$. By Lemma \ref{lem_ele1}, we have
$$J_k^{-q}\le C_4^{\frac{\kappa}{\kappa-1}}C_5^{\frac{\kappa}{(\kappa-1)^2}}J_0^{-q},$$
that is, $J_k\ge \frac{1}{C_6}J_0$, where $C_6=\left(C_4^{\frac{\kappa}{\kappa-1}}C_5^{\frac{\kappa}{(\kappa-1)^2}}\right)^{\frac{1}{q}}$. Letting $k\to+\infty$, we have
\begin{align*}
&\einf_{B(x_0,\frac{1}{2}R)}(u\wedge M)+\Upsilon(R)^{\frac{1}{p-1}}\mathrm{Tail}^{(J)}\left((u\wedge M)_-;x_0,R\right)+\Upsilon(R)^{\frac{1}{p-1}}\\
&\ge \frac{1}{C_6}\left(\dashint_{B(x_0,R)}\left(u\wedge M+\Upsilon(R)^{\frac{1}{p-1}}\mathrm{Tail}^{(J)}\left((u\wedge M)_-;x_0,R\right)+\Upsilon(R)^{\frac{1}{p-1}}\right)^{-q}\mathrm{d}m\right)^{-\frac{1}{q}}.
\end{align*}
\end{proof}

\begin{proof}[Proof of Proposition \ref{prop_bdy_wEHI_J}]
For notational convenience, we may assume that all functions in $\mathcal{F}^{(J)}$ are quasi-continuous. If $M=0$, then the result is trivial, hence we may assume that $M>0$. Let $A_1,A_2$ be the constants in \hyperlink{eq_CSJ_weak}{$\text{CS}^{(J)}_{\text{weak}}(\Upsilon)$}, and $A=48A_2$, then Lemma \ref{lem_Moser_J} holds with $R_0=AR$ therein. Let
$$K=\Upsilon(R)^{\frac{1}{p-1}}\mathrm{Tail}^{(J)}\left((u\wedge M)_-;x_0,R\right)+\Upsilon(R)^{\frac{1}{p-1}},$$
and $\overline{u}=u\wedge M+K$. We claim that $\log \overline{u}\in \mathrm{BMO}(B(x_0,12R))$ and
$$\lVert \log \overline{u}\rVert_{\mathrm{BMO}(B(x_0,12R))}\le C_1,$$
where $C_1>0$ depends only on $p,C_\Upsilon,C_{VD}$ and the constants in \ref{eq_UPR}, \ref{eq_KJ_tail}, \ref{eq_PIJ}, \hyperlink{eq_CSJ_weak}{$\text{CS}^{(J)}_{\text{weak}}(\Upsilon)$}. Indeed, for any ball $B\subseteq B(x_0,12R)$ with radius $r$, we have $r<24R$ and $A_1B\subseteq A_2B\subseteq B(x_0,AR)$. By \ref{eq_PIJ}, we have
\begin{align}
&\dashint_{B}\lvert \log \overline{u}-(\log \overline{u})_B\rvert \mathrm{d}m\le \left(\dashint_{B}\lvert \log {\overline{u}}-\left(\log {\overline{u}}\right)_B\rvert^p \mathrm{d}m\right)^{\frac{1}{p}}\nonumber\\
&\lesssim\left(\frac{\Upsilon(r)}{m(B)}\int_{B}\int_{B}\lvert \log {\overline{u}(x)}-\log {\overline{u}(y)}\rvert^p K^{(J)}(x,y) m(\mathrm{d}x) m(\mathrm{d}y)\right)^{\frac{1}{p}}.\label{eq_BMO_J1}
\end{align}
By \hyperlink{eq_CSJ_weak}{$\text{CS}^{(J)}_{\text{weak}}(\Upsilon)$}, there exists a cutoff function $\phi\in \mathcal{F}^{(J)}$ for $B\subseteq A_1B$ satisfying the condition therein. Taking $f\equiv1$ in \ref{eq_CSJ_energy}, we have
\begin{equation}\label{eq_BMO_J2}
\int_{A_2B}\int_{A_2B}\lvert \phi(x)-\phi(y)\rvert^p K^{(J)}(x,y)m (\mathrm{d}x)m(\mathrm{d}y)\lesssim \frac{m(A_2B)}{\Upsilon(r)}\lesssim \frac{m(B)}{\Upsilon(r)}.
\end{equation}
By Lemma \ref{lem_ele5}, there exist $C_2,C_3>0$ depending only on $p$ such that
\begin{align}
&\int_{B}\int_{B}\lvert \log {\overline{u}(x)}-\log {\overline{u}(y)}\rvert^p K^{(J)}(x,y) m(\mathrm{d}x) m(\mathrm{d}y)\nonumber\\
&\le\int_{A_2B}\int_{A_2B}\lvert \log {\overline{u}(x)}-\log {\overline{u}(y)}\rvert^p\min\{\phi(x)^p,\phi(y)^p\} K^{(J)}(x,y) m(\mathrm{d}x) m(\mathrm{d}y)\nonumber\\
&\le \frac{C_3}{C_2}\int_{A_2B}\int_{A_2B}\lvert \phi(x)-\phi(y)\rvert^p K^{(J)}(x,y) m(\mathrm{d}x)m(\mathrm{d}y)\nonumber\\
&\hspace{15pt}-\frac{1}{C_2}\int_{A_2B}\int_{A_2B}\lvert \overline{u}(x)-\overline{u}(y)\rvert^{p-2}(\overline{u}(x)-\overline{u}(y))\nonumber\\
&\hspace{50pt}\cdot\left(\phi(x)^p\left(\overline{u}(x)^{1-p}-(M+K)^{1-p}\right)-\phi(y)^p\left(\overline{u}(y)^{1-p}-(M+K)^{1-p}\right)\right)\nonumber\\
&\hspace{50pt}\cdot K^{(J)}(x,y) m(\mathrm{d}x) m(\mathrm{d}y).\label{eq_BMO_J3}
\end{align}
Since $-\Delta_p^{(J)}u\ge-1$ in $B(x_0,AR)\cap\Omega$, by Lemma \ref{lem_cutoff}, we have $-\Delta_p^{(J)}\overline{u}\ge-1$ in $B(x_0,AR)\cap\Omega$. Let
$$v=\phi^p \left(\overline{u}^{1-p}-(M+K)^{1-p}\right)=\phi^p \left(({u}\wedge M+K)^{1-p}-(M+K)^{1-p}\right),$$
then $v\in \mathcal{F}^{(J)}(A_1B\cap\Omega)$ is non-negative, hence
\begin{equation}\label{eq_BMO_J4}
\mathcal{E}^{(J)}(\overline{u};v)\ge -\int_X v\mathrm{d}m\ge-\int_{A_1B}\overline{u}^{1-p}\mathrm{d}m\ge- \frac{m(A_1B)}{\Upsilon(R)}\gtrsim -\frac{m(B)}{\Upsilon(r)}.
\end{equation}
On the other hand, we have
\begin{align}
&\mathcal{E}^{(J)}(\overline{u};v)\nonumber\\
&=\left(\int_{A_2B}\int_{A_2B}+2\int_{A_2B}\int_{X\backslash A_2B}\right)\lvert \overline{u}(x)-\overline{u}(y)\rvert^{p-2}(\overline{u}(x)-\overline{u}(y))\nonumber\\
&\hspace{50pt}\cdot \left(\phi(x)^p\left(\overline{u}(x)^{1-p}-(M+K)^{1-p}\right)-\phi(y)^p\left(\overline{u}(y)^{1-p}-(M+K)^{1-p}\right)\right)\nonumber\\
&\hspace{50pt}\cdot K^{(J)}(x,y) m(\mathrm{d}x) m(\mathrm{d}y)\nonumber\\
&=I_1+2I_2.\label{eq_BMO_J6}
\end{align}
For $I_2$, we have
\begin{align*}
&I_2=\int_{A_1B}\int_{X\backslash A_2B}\lvert \overline{u}(x)-\overline{u}(y)\rvert^{p-2}(\overline{u}(x)-\overline{u}(y))\\
&\hspace{50pt}\cdot\phi(x)^p\left(\overline{u}(x)^{1-p}-(M+K)^{1-p}\right)K^{(J)}(x,y) m(\mathrm{d}x) m(\mathrm{d}y)\\
&\le\int_{A_1B}\int_{X\backslash A_2B}\left( \overline{u}(x)-\overline{u}(y)\right)^{p-1}\overline{u}(x)^{1-p}K^{(J)}(x,y)1_{\overline{u}(x)\ge \overline{u}(y)} m(\mathrm{d}x) m(\mathrm{d}y)\\
&\overset{(\star)}{\scalebox{2}[1]{$\lesssim$}}\int_{A_1B}\int_{X\backslash A_2B}K^{(J)}(x,y) m(\mathrm{d}x) m(\mathrm{d}y)\\
&\hspace{15pt}+\int_{A_1B}\int_{X\backslash A_2B}(u(y)\wedge M)_-^{p-1}\overline{u}(x)^{1-p}K^{(J)}(x,y) m(\mathrm{d}x) m(\mathrm{d}y)\\
&=I_{21}+I_{22}.
\end{align*}
where $(\star)$ follows from the fact that
\begin{align*}
&\left( \overline{u}(x)-\overline{u}(y)\right)^{p-1}1_{\overline{u}(x)\ge \overline{u}(y)}=\left( {u}(x)\wedge M-{u}(y)\wedge M\right)^{p-1}1_{{u}(x)\wedge M\ge {u}(y)\wedge M}\\
&\le2^{p-1} \left(\left({u}(x)\wedge M\right)^{p-1}+\left({u}(y)\wedge M\right)_-^{p-1}\right)\le 2^{p-1} \left(\overline{u}(x)^{p-1}+\left({u}(y)\wedge M\right)_-^{p-1}\right),
\end{align*}
due to $u\ge0$ in $A_1B$. By \ref{eq_KJ_tail}, we have
\begin{align*}
&I_{21}\le\int_{A_1B} \left(\int_{X\backslash B(x,(A_2-A_1)r)}K^{(J)}(x,y)m(\mathrm{d}y)\right)m(\mathrm{d}x)\lesssim \frac{m(A_1B)}{\Upsilon(r)}\lesssim \frac{m(B)}{\Upsilon(r)}.
\end{align*}
Since $u\ge0$ in $B(x_0,AR)$, we have
\begin{align*}
I_{22}=\int_{A_1B}\int_{X\backslash B(x_0,AR)}(u(y)\wedge M)_-^{p-1}\overline{u}(x)^{1-p}K^{(J)}(x,y) m(\mathrm{d}x) m(\mathrm{d}y).
\end{align*}
For any $x\in A_1B\subseteq B(x_0,36A_2R)$ and $y\in X\backslash B(x_0,AR)=X\backslash B(x_0,48A_2R)$, by the triangle inequality, we have $\frac{4}{7}d(x,y)\le d(x_0,y)\le 4d(x,y)$, then by \ref{eq_UPR}, we have
$$K^{(J)}(x,y)\lesssim \left(\frac{d(x_0,y)}{d(x,y)}\vee1\right)^\theta K^{(J)}(x_0,y)\lesssim K^{(J)}(x_0,y),$$
hence
\begin{align*}
&I_{22}\lesssim \left(\int_{A_1B}\overline{u}^{1-p}\mathrm{d}m\right)\left(\int_{X\backslash B(x_0,AR)}\left(u(y)\wedge M\right)_-^{p-1}K^{(J)}(x_0,y)m(\mathrm{d}y)\right)\\
&\overset{(\dagger)}{\scalebox{2}[1]{$=$}}\left(\int_{A_1B}\overline{u}^{1-p}\mathrm{d}m\right)\left(\int_{X\backslash B(x_0,R)}\left(u(y)\wedge M\right)_-^{p-1}K^{(J)}(x_0,y)m(\mathrm{d}y)\right)\\
&=\mathrm{Tail}^{(J)}((u\wedge M)_-;x_0,R)^{p-1}\int_{A_1B}\overline{u}^{1-p}\mathrm{d}m\overset{(\ast)}{\scalebox{2}[1]{$\le$}}\frac{m(A_1B)}{\Upsilon(R)}\lesssim \frac{m(B)}{\Upsilon(r)},
\end{align*}
where $(\dagger)$ follows from the fact that $u\ge0$ in $B(x_0,AR)$, and $(\ast)$ follows from the fact that
\begin{align*}
\Upsilon(R)^{\frac{1}{p-1}}\mathrm{Tail}^{(J)}\left((u\wedge M)_-;x_0,R\right)\le K\le \overline{u}\text{ in }A_1B.
\end{align*}
Therefore, we have
\begin{align*}
I_2\lesssim I_{21}+I_{22}\lesssim \frac{m(B)}{\Upsilon(r)}.
\end{align*}
Combining this with (\ref{eq_BMO_J2})--(\ref{eq_BMO_J6}), we have
\begin{align*}
&\int_{B}\int_{B}\lvert \log {\overline{u}(x)}-\log {\overline{u}(y)}\rvert^p K^{(J)}(x,y) m(\mathrm{d}x) m(\mathrm{d}y)\\
&\lesssim \frac{m(B)}{\Upsilon(r)}-I_1=\frac{m(B)}{\Upsilon(r)}+2I_2-\mathcal{E}^{(J)}(\overline{u};v)\lesssim \frac{m(B)}{\Upsilon(r)}.
\end{align*}
By (\ref{eq_BMO_J1}), we have
\begin{align*}
\dashint_{B}\lvert \log \overline{u}-(\log \overline{u})_B\rvert \mathrm{d}m\le C_{1},
\end{align*}
where $C_1>0$ is some constant. By Lemma \ref{lem_crossover}, there exist $C_{4},C_{5}>0$ such that
\begin{align*}
&\left(\dashint_{B(x_0,R)}\exp \left(\frac{C_{5}}{2C_1}\log \overline{u}\right)\mathrm{d}m\right)\left(\dashint_{B(x_0,R)}\exp \left(-\frac{C_{5}}{2C_1}\log \overline{u}\right)\mathrm{d}m\right)\le \left(C_{4}+1\right)^2.
\end{align*}
Let $q=\frac{C_{5}}{2C_1}$, then
\begin{align}
&\left(\dashint_{B(x_0,R)}\left({u}\wedge M+\Upsilon(R)^{\frac{1}{p-1}}\mathrm{Tail}^{(J)}\left((u\wedge M)_-;x_0,R\right)+\Upsilon(R)^{\frac{1}{p-1}}\right)^q\mathrm{d}m\right)\nonumber\\
&\cdot\left(\dashint_{B(x_0,R)}\left({u}\wedge M+\Upsilon(R)^{\frac{1}{p-1}}\mathrm{Tail}^{(J)}\left((u\wedge M)_-;x_0,R\right)+\Upsilon(R)^{\frac{1}{p-1}}\right)^{-q}\mathrm{d}m\right)\nonumber\\
&\le \left(C_{4}+1\right)^2.\label{eq_BMO_J7}
\end{align}
Since $u\ge0$ in $B(x_0,AR)$, we have
\begin{align*}
\mathrm{Tail}^{(J)}\left((u\wedge M)_-;x_0,R\right)=\mathrm{Tail}^{(J)}\left((u\wedge M)_-;x_0,AR\right).
\end{align*}
Combining this with (\ref{eq_BMO_J7}) and Lemma \ref{lem_Moser_J}, we have
\begin{align*}
&\einf_{B(x_0,\frac{1}{2}R)}(u\wedge M)+\Upsilon(R)^{\frac{1}{p-1}}\mathrm{Tail}^{(J)}\left((u\wedge M)_-;x_0,AR\right)+\Upsilon(R)^{\frac{1}{p-1}}\\
&\ge \einf_{B(x_0,\frac{1}{2}R)}(u\wedge M)+\Upsilon(R)^{\frac{1}{p-1}}\mathrm{Tail}^{(J)}\left((u\wedge M)_-;x_0,R\right)+\Upsilon(R)^{\frac{1}{p-1}}\\
&\gtrsim\left(\dashint_{B(x_0,R)}\left(u\wedge M+\Upsilon(R)^{\frac{1}{p-1}}\mathrm{Tail}^{(J)}\left((u\wedge M)_-;x_0,R\right)+\Upsilon(R)^{\frac{1}{p-1}}\right)^q\mathrm{d}m\right)^{\frac{1}{q}}\\
&\ge\left(\dashint_{B(x_0,R)}\left(u\wedge M\right)^q\mathrm{d}m\right)^{\frac{1}{q}}.
\end{align*}
\end{proof}

Using the two weak Harnack inequalities established above, we prove the oscillation inequalities. In comparison with the local case, namely Propositions \ref{prop_int_osc_L} and \ref{prop_bdy_osc_L}, special care must be taken with the tail terms. Our proof follows the induction argument of \cite[Lemma 5.1]{DiKP16}, in particular the reasoning leading to Equation (5.6) therein.

We need the following result for preparation; see also \cite[Lemma 4.2]{DiKP14}, \cite[Lemma 2.6]{CKW19}.

\begin{lemma}\label{lem_tail_pm}
There exists $C>0$ such that for any ball $B=B(x_0,R)$, let $u\in \widehat{\mathcal{F}}^{(J)}\cap L^\infty(X;m)$ satisfy $-\Delta_p^{(J)}u\ge-1$ in $B$, then we have
$$\Upsilon(R)^{\frac{1}{p-1}}\mathrm{Tail}^{(J)}(u_+;x_0,R)\le C \left(\lVert u\rVert_{L^\infty(B;m)}+\Upsilon(R)^{\frac{1}{p-1}}\mathrm{Tail}^{(J)}(u_-;x_0,R)+\Upsilon(R)^{\frac{1}{p-1}}\right).$$
\end{lemma}

\begin{proof}
Firstly, we assume that $k=\lVert u\rVert_{L^\infty(B;m)}>0$. Let $w=u-2k$, then $-3k\le w\le0$ in $B$. By \ref{eq_VD}, \ref{eq_ucapJ}, there exists a cutoff function $\phi\in \mathcal{F}^{(J)}$ for $\frac{1}{2}B\subseteq \frac{3}{4}B$ such that $\mathcal{E}^{(J)}(\phi)\lesssim \frac{m(B)}{\Upsilon(R)}$, then $\phi^p w\in \mathcal{F}^{(J)}(\frac{3}{4}B)\subseteq \mathcal{F}^{(J)}(B)$ is \emph{non-positive}. By assumption, we have
\begin{equation}\label{eq_tail_pm1}
\mathcal{E}^{(J)}(w;\phi^p w)=\mathcal{E}^{(J)}(u;\phi^p w)\le-\int_{X}\phi^p w \mathrm{d}m\le 3k m(B),
\end{equation}
where
\begin{align}
&\mathcal{E}^{(J)}(w;\phi^p w)\nonumber\\
&=\left(\int_{B}\int_{B}+2\int_{B}\int_{X\backslash B}\right)\lvert w(x)-w(y)\rvert^{p-2}(w(x)-w(y))\left(\phi(x)^pw(x)-\phi(y)^pw(y)\right)\nonumber\\
&\hspace{100pt}\cdot K^{(J)}(x,y)m(\mathrm{d}x)m(\mathrm{d}y)\nonumber\\
&=I_1+2I_2.\label{eq_tail_pm2}
\end{align}
For $I_1$, for $m$-a.e. $x,y\in B$, we have $\lvert w(x)\rvert,\lvert w(y)\rvert\le 3k$, then by Lemma \ref{lem_ele6}, there exists $C>0$ depending only on $p$, such that
\begin{align*}
&\lvert w(x)-w(y)\rvert^{p-2}(w(x)-w(y))\left(\phi(x)^pw(x)-\phi(y)^pw(y)\right)\\
&\ge \frac{1}{4}\lvert w(x)-w(y)\rvert^p\left(\phi(x)^p+\phi(y)^p\right)-{C}\left(\lvert w(x)\rvert^p+\lvert w(y)\rvert^p\right)\lvert \phi(x)-\phi(y)\rvert^p\\
&\ge-{C}\cdot2\cdot(3k)^p \lvert \phi(x)-\phi(y)\rvert^p=-2\cdot 3^pCk^p\lvert \phi(x)-\phi(y)\rvert^p,
\end{align*}
hence
\begin{align}
I_1\gtrsim -k^p\int_B\int_{B}\lvert \phi(x)-\phi(y)\rvert^pK^{(J)}(x,y)m(\mathrm{d}x)m(\mathrm{d}y)\ge-k^p \mathcal{E}^{(J)}(\phi)\gtrsim-k^p \frac{m(B)}{\Upsilon(R)}.\label{eq_tail_pm3}
\end{align}
For $I_2$, we have
\begin{align*}
I_2=\int_{\frac{3}{4}B}\int_{X\backslash B}\lvert w(x)-w(y)\rvert^{p-2}(w(x)-w(y))\phi(x)^pw(x)K^{(J)}(x,y)m(\mathrm{d}x)m(\mathrm{d}y).
\end{align*}
For $m$-a.e. $x\in \frac{3}{4}B$, we have $-k\le u(x)\le k$, hence
\begin{align*}
&\lvert w(x)-w(y)\rvert^{p-2}(w(x)-w(y))w(x)=\lvert u(y)-u(x)\rvert^{p-2}(u(y)-u(x))(2k-u(x))\\
&=\lvert u(y)-u(x)\rvert^{p-2}(u(y)-u(x))_+(2k-u(x))-\lvert u(y)-u(x)\rvert^{p-2}(u(y)-u(x))_-(2k-u(x))\\
&\ge k (u(y)-u(x))_+^{p-1}-3k(u(y)-u(x))_-^{p-1}\ge k (u(y)-k)_+^{p-1}-3k(u(y)-k)_-^{p-1},
\end{align*}
which gives
\begin{align}
&I_2\ge k\int_{\frac{3}{4}B}\int_{X\backslash B}(u(y)-k)_+^{p-1}\phi(x)^pK^{(J)}(x,y)m(\mathrm{d}x)m(\mathrm{d}y)\nonumber\\
&\hspace{15pt}-3k\int_{\frac{3}{4}B}\int_{X\backslash B}(u(y)-k)_-^{p-1}\phi(x)^pK^{(J)}(x,y)m(\mathrm{d}x)m(\mathrm{d}y)\nonumber\\
&=I_{21}-3I_{22}.\label{eq_tail_pm4}
\end{align}
For $I_{21}$, since $u(y)_+^{p-1}\le 2^{p-1}\left((u(y)-k)_+^{p-1}+k^{p-1}\right)$, we have
\begin{align}
&I_{21}\ge k\int_{\frac{1}{2}B}\int_{X\backslash B}(u(y)-k)_+^{p-1}K^{(J)}(x,y)m(\mathrm{d}x)m(\mathrm{d}y)\nonumber\\
&\ge \frac{k}{2^{p-1}}\int_{\frac{1}{2}B}\int_{X\backslash B}u(y)_+^{p-1}K^{(J)}(x,y)m(\mathrm{d}x)m(\mathrm{d}y)-k^p\int_{\frac{1}{2}B}\int_{X\backslash B}K^{(J)}(x,y)m(\mathrm{d}x)m(\mathrm{d}y)\nonumber\\
&=I_{211}-I_{212}.\label{eq_tail_pm5}
\end{align}
For $I_{22}$, since $(u(y)-k)_-^{p-1}\le 2^{p-1}\left(u(y)_-^{p-1}+k^{p-1}\right)$, we have
\begin{align*}
&I_{22}\lesssim k\int_{\frac{3}{4}B}\int_{X\backslash B}u(y)_-^{p-1}K^{(J)}(x,y)m(\mathrm{d}x)m(\mathrm{d}y)+k^p\int_{\frac{3}{4}B}\int_{X\backslash B}K^{(J)}(x,y)m(\mathrm{d}x)m(\mathrm{d}y).
\end{align*}
For any $x\in \frac{3}{4}B$ and $y\in X\backslash B$, by the triangle inequality, we have $\frac{4}{7}d(x,y)\le d(x_0,y)\le 4d(x,y)$, then by \ref{eq_UPR}, we have
$$K^{(J)}(x,y)\asymp K^{(J)}(x_0,y).$$
Hence
\begin{align}
I_{211}&\asymp k\int_{\frac{1}{2}B}\int_{X\backslash B}u(y)_+^{p-1}K^{(J)}(x_0,y)m(\mathrm{d}x)m(\mathrm{d}y)\asymp k{m(B)}\mathrm{Tail}^{(J)}(u_+;x_0,R)^{p-1},\label{eq_tail_pm6}\\
I_{212}&\asymp k^p\int_{\frac{1}{2}B}\int_{X\backslash B}K^{(J)}(x_0,y)m(\mathrm{d}x)m(\mathrm{d}y)\lesssim k^p\frac{m(B)}{\Upsilon(R)},\label{eq_tail_pm7}
\end{align}
and
\begin{align}
&I_{22}\lesssim k\int_{\frac{3}{4}B}\int_{X\backslash B}u(y)_-^{p-1}K^{(J)}(x_0,y)m(\mathrm{d}x)m(\mathrm{d}y)+k^p\int_{\frac{3}{4}B}\int_{X\backslash B}K^{(J)}(x_0,y)m(\mathrm{d}x)m(\mathrm{d}y)\nonumber\\
&\lesssim k {m(B)}\mathrm{Tail}^{(J)}(u_-;x_0,R)^{p-1}+k^p \frac{m(B)}{\Upsilon(R)}.\label{eq_tail_pm8}
\end{align}
Combining (\ref{eq_tail_pm1})--(\ref{eq_tail_pm8}), we have
\begin{align*}
k{m(B)}\mathrm{Tail}^{(J)}(u_+;x_0,R)^{p-1}\lesssim k {m(B)}\mathrm{Tail}^{(J)}(u_-;x_0,R)^{p-1}+k^p \frac{m(B)}{\Upsilon(R)}+km(B),
\end{align*}
that is,
\begin{align*}
k {\Upsilon(R)}\mathrm{Tail}^{(J)}(u_+;x_0,R)^{p-1}\lesssim k^p+k{\Upsilon(R)}\mathrm{Tail}^{(J)}(u_-;x_0,R)^{p-1}+k\Upsilon(R).
\end{align*}
Since $k=\lVert u\rVert_{L^\infty(B;m)}>0$, we have
$${\Upsilon(R)^{\frac{1}{p-1}}}\mathrm{Tail}^{(J)}(u_+;x_0,R)\lesssim\lVert u\rVert_{L^\infty(B;m)}+{\Upsilon(R)^{\frac{1}{p-1}}}\mathrm{Tail}^{(J)}(u_-;x_0,R)+\Upsilon(R)^{\frac{1}{p-1}}.$$

Secondly, we assume that $\lVert u\rVert_{L^\infty(B;m)}=0$. For any $\varepsilon>0$, we have $u+\varepsilon\in \widehat{\mathcal{F}}^{(J)}\cap L^\infty(X;m)$ satisfies $-\Delta_p^{(J)}(u+\varepsilon)\ge-1$ in $B$, and $u+\varepsilon=\varepsilon$ in $B$. By above, we have
\begin{align*}
&{\Upsilon(R)^{\frac{1}{p-1}}}\mathrm{Tail}^{(J)}((u+\varepsilon)_+;x_0,R)\lesssim \varepsilon+{\Upsilon(R)^{\frac{1}{p-1}}}\mathrm{Tail}^{(J)}((u+\varepsilon)_-;x_0,R)+\Upsilon(R)^{\frac{1}{p-1}}\\
&\le \varepsilon+{\Upsilon(R)^{\frac{1}{p-1}}}\mathrm{Tail}^{(J)}(u_-;x_0,R)+\Upsilon(R)^{\frac{1}{p-1}}.
\end{align*}
Letting $\varepsilon\downarrow0$, by Fatou's lemma, we have
$${\Upsilon(R)^{\frac{1}{p-1}}}\mathrm{Tail}^{(J)}(u_+;x_0,R)\lesssim{\Upsilon(R)^{\frac{1}{p-1}}}\mathrm{Tail}^{(J)}(u_-;x_0,R)+\Upsilon(R)^{\frac{1}{p-1}}.$$
\end{proof}

We give the proofs of Propositions \ref{prop_int_osc_J} and \ref{prop_bdy_osc_J} as follows.

\begin{proof}[Proof of Proposition \ref{prop_int_osc_J}]
Let $q>0$, $A_1>1$ be the constants appearing in Proposition \ref{prop_int_wEHI_J}. Let $a>1$ be chosen sufficiently large (to be specified later) such that $A=2aA_1$ satisfies $A^{\beta_\Upsilon^{(1)}}\ge 2^p$, where $\beta_\Upsilon^{(1)}$ is a constant appearing in (\ref{eq_beta12}). For any $k\ge0$, let $R_k=\frac{1}{A^k}R$ and $B_k=B(x_0,R_k)$. Let $\alpha\in[\frac{1}{2},1)$ be chosen later and
$$K=2\lVert u\rVert_{L^\infty(B(x_0,R);m)}+{\Upsilon(R)^{\frac{1}{p-1}}}\mathrm{Tail}^{(J)}(u_-;x_0,R)+\Upsilon(R)^{\frac{1}{p-1}}.$$
The proof proceeds by choosing suitable $a$ and $\alpha$, and applying induction to show that, for any $k\ge0$, we have
\begin{equation}\label{eq_int_osc_J}
\eosc_{B_k}u\le \alpha^k K=\alpha^k \left(2\lVert u\rVert_{L^\infty(B_0;m)}+{\Upsilon(R)^{\frac{1}{p-1}}}\mathrm{Tail}^{(J)}(u_-;x_0,R)+\Upsilon(R)^{\frac{1}{p-1}}\right).
\end{equation}

To begin with, using the a priori bounds on $a$ and $\alpha$, we have
\begin{equation}\label{eq_aalpha1}
\alpha^{p-1}A^{\beta_\Upsilon^{(1)}}\ge \frac{1}{2^{p-1}}2^p=2.
\end{equation}
We emphasize that, in the following proof, the implicit constants in $\lesssim$ and $\asymp$ are \emph{independent} of $\alpha$ and $A$ (or, equivalently, $a$). For $k=0$, the result is obvious. Assume that the result holds for $j=0,\ldots,k$. For $k+1$, let $v=\esup_{B_k}u-u$, then $v\in \widehat{\mathcal{F}}^{(J)}\cap L^\infty(X;m)$ is non-negative in $B_k$ and satisfies $-\Delta_p^{(J)}v=-f\ge-1$ in $B_k$. Since $B_k=2aA_1B_{k+1}\supseteq 2A_1B_{k+1}$, by Proposition \ref{prop_int_wEHI_J}, we have
\begin{align*}
\left(\dashint_{2B_{k+1}}v^q \mathrm{d}m\right)^{\frac{1}{q}}\lesssim \einf_{B_{k+1}}v+{\Upsilon(2R_{k+1})^{\frac{1}{p-1}}}\mathrm{Tail}^{(J)}\left(v_-;x_0,2A_1R_{k+1}\right)+\Upsilon(2R_{k+1})^{\frac{1}{p-1}},
\end{align*}
where
$$\einf_{B_{k+1}}v=\esup_{B_{k}}u-\esup_{B_{k+1}}u.$$
By (\ref{eq_beta12}) and (\ref{eq_aalpha1}), we have
\begin{align}
&\Upsilon(2R_{k+1})\lesssim \left(\frac{2R_{k+1}}{R}\right)^{\beta^{(1)}_\Upsilon}\Upsilon(R)\asymp \frac{1}{A^{\beta^{(1)}_\Upsilon(k+1)}}\Upsilon(R)=\frac{1}{\left(\alpha^{p-1}A^{\beta^{(1)}_\Upsilon}\right)^{k+1}}\alpha^{(p-1)(k+1)}\Upsilon(R)\nonumber\\
&\le\frac{1}{\alpha^{p-1}A^{\beta^{(1)}_\Upsilon}}\alpha^{(p-1)(k+1)}\Upsilon(R)\le\frac{1}{A^{\beta^{(1)}_\Upsilon}}\left(\alpha^{k}K\right)^{p-1}.\label{eq_UpRk1}
\end{align}
Since $v\ge0$ in $B_k$, we have
\begin{align*}
&\mathrm{Tail}^{(J)}\left(v_-;x_0,2A_1R_{k+1}\right)^{p-1}=\int_{X\backslash 2A_1B_{k+1}}v(y)_-^{p-1}K^{(J)}(x_0,y) m(\mathrm{d}y)\\
&=\int_{X\backslash B_{k}}v(y)_-^{p-1}K^{(J)}(x_0,y) m(\mathrm{d}y)=\left(\sum_{j=0}^{k-1}\int_{B_j\backslash B_{j+1}}+\int_{X\backslash B_0}\right)v(y)_-^{p-1}K^{(J)}(x_0,y) m(\mathrm{d}y).
\end{align*}
By the induction assumption, for any $j=0,\ldots,k-1$, we have
$$v_-=\left(\esup_{B_k}u-u\right)_-\le  \eosc_{B_j}u\le \alpha^jK\text{ in }B_j.$$
By \ref{eq_KJ_tail}, we have
$$\int_{B_j\backslash B_{j+1}}v(y)_-^{p-1}K^{(J)}(x_0,y) m(\mathrm{d}y)\lesssim \frac{1}{\Upsilon(R_{j+1})}\left(\alpha^jK\right)^{p-1}.$$
Since $v_-=(u-\esup_{B_k}u)_+\le u_++\lvert \esup_{B_k}u\rvert\le u_++K$ in $X$, by \ref{eq_KJ_tail}, we have
\begin{align*}
\int_{X\backslash B_0}v(y)_-^{p-1}K^{(J)}(x_0,y) m(\mathrm{d}y)\lesssim \mathrm{Tail}^{(J)}(u_+;x_0,R)^{p-1}+\frac{1}{\Upsilon(R)}K^{p-1}.
\end{align*}
Since $-\Delta_p^{(J)}u=f\ge-1$ in $B(x_0,R)$, by Lemma \ref{lem_tail_pm}, we have
$$\Upsilon(R)^{\frac{1}{p-1}}\mathrm{Tail}^{(J)}(u_+;x_0,R)\lesssim \lVert u\rVert_{L^\infty(B(x_0,R);m)}+\Upsilon(R)^{\frac{1}{p-1}}\mathrm{Tail}^{(J)}(u_-;x_0,R)+\Upsilon(R)^{\frac{1}{p-1}}\le K,$$
which gives
$$\int_{X\backslash B_0}v(y)_-^{p-1}K^{(J)}(x_0,y) m(\mathrm{d}y)\lesssim \frac{1}{\Upsilon(R)}K^{p-1}.$$
Hence
\begin{align}
&\Upsilon(2R_{k+1})\mathrm{Tail}^{(J)}\left(v_-;x_0,2A_1R_{k+1}\right)^{p-1}\nonumber\\
&\lesssim\Upsilon(R_{k+1}) \left(\sum_{j=0}^{k-1} \frac{1}{\Upsilon(R_{j+1})}(\alpha^j K)^{p-1}+\frac{1}{\Upsilon(R)}K^{p-1}\right)\lesssim K^{p-1}\sum_{j=0}^{k}\alpha^{(p-1)j}\frac{\Upsilon(R_{k+1})}{\Upsilon(R_j)}\nonumber\\
&\overset{(\diamond)}{\scalebox{2}[1]{$\lesssim$}}K^{p-1}\sum_{j=0}^{k}\alpha^{(p-1)j}\left(\frac{R_{k+1}}{R_j}\right)^{\beta_\Upsilon^{(1)}}\asymp K^{p-1}\frac{1}{A^{\beta^{(1)}_\Upsilon(k+1)}}\sum_{j=0}^{k}\left(\alpha^{p-1}A^{\beta_\Upsilon^{(1)}}\right)^j\nonumber\\
&\overset{(\star)}{\scalebox{2}[1]{$\le$}}K^{p-1}\frac{1}{A^{\beta^{(1)}_\Upsilon(k+1)}}\frac{\left(\alpha^{p-1}A^{\beta_\Upsilon^{(1)}}\right)^{k+1}}{\alpha^{p-1}A^{\beta_\Upsilon^{(1)}}-1}\overset{(\dagger)}{\scalebox{2}[1]{$\le$}}K^{p-1}\frac{\alpha^{(p-1)(k+1)}}{\frac{1}{2}\alpha^{p-1}A^{\beta^{(1)}_\Upsilon}}=\frac{2}{A^{\beta^{(1)}_\Upsilon}}\left(\alpha^k K\right)^{p-1},\label{eq_osc_Tail1}
\end{align}
where $(\diamond)$ follows from (\ref{eq_beta12}), and $(\star)$ and $(\dagger)$ follow from (\ref{eq_aalpha1}). Hence
\begin{align*}
&\left(\dashint_{2B_{k+1}}\left(\esup_{B_k}u-u\right)^q \mathrm{d}m\right)^{\frac{1}{q}}\\
&\lesssim\esup_{B_{k}}u-\esup_{B_{k+1}}u\\
&\hspace{15pt}+\frac{1}{A^{\frac{\beta_\Upsilon^{(1)}}{p-1}}}\alpha^{k}\left(\lVert u\rVert_{L^\infty(B(x_0,R);m)}+\Upsilon(R)^{\frac{1}{p-1}}\mathrm{Tail}^{(J)}(u_-;x_0,R)+\Upsilon(R)^{\frac{1}{p-1}}\right).
\end{align*}
Similarly, by considering $u-\einf_{B_k}u\in \widehat{\mathcal{F}}^{(J)}\cap L^\infty(X;m)$, which is non-negative in $B_k$ and satisfies $-\Delta_p^{(J)}(u-\einf_{B_k}u)=f\ge -1$ in $B_k$, noting that $\left(u-\einf_{B_k}u\right)_-\le u_-+\lvert \einf_{B_k}u\rvert\le u_-+K$ in $X$, we also have
\begin{align*}
&\left(\dashint_{2B_{k+1}}\left(u-\einf_{B_k}u\right)^q \mathrm{d}m\right)^{\frac{1}{q}}\\
&\lesssim \einf_{B_{k+1}}u-\einf_{B_{k}}u\\
&\hspace{15pt}+\frac{1}{A^{\frac{\beta_\Upsilon^{(1)}}{p-1}}}\alpha^{k}\left(\lVert u\rVert_{L^\infty(B(x_0,R);m)}+\Upsilon(R)^{\frac{1}{p-1}}\mathrm{Tail}^{(J)}(u_-;x_0,R)+\Upsilon(R)^{\frac{1}{p-1}}\right).
\end{align*}
Hence there exists $C_1>1$ such that
\begin{align*}
\left(\dashint_{2B_{k+1}}\left(\esup_{B_k}u-u\right)^q \mathrm{d}m\right)^{\frac{1}{q}}&\le C_1 \left(\esup_{B_{k}}u-\esup_{B_{k+1}}u+\frac{1}{A^{\frac{\beta_\Upsilon^{(1)}}{p-1}}}\alpha^{k}K\right),\\
\left(\dashint_{2B_{k+1}}\left(u-\einf_{B_k}u\right)^q \mathrm{d}m\right)^{\frac{1}{q}}&\le C_1 \left(\einf_{B_{k+1}}u-\einf_{B_k}u+\frac{1}{A^{\frac{\beta_\Upsilon^{(1) }}{p-1}}}\alpha^{k}K\right).\\
\end{align*}
By the norm and quasi-norm properties of $L^q$-spaces (see \cite[Equations (1.1.3) and (1.1.4)]{Gra14}), there exists $C_2\ge1$ depending only on $q$ such that
\begin{align*}
&\esup_{B_k}u-\einf_{B_k}u\\
&\le C_2 \left(\left(\dashint_{2B_{k+1}}\left(\esup_{B_k}u-u\right)^q \mathrm{d}m\right)^{\frac{1}{q}}+\left(\dashint_{2B_{k+1}}\left(u-\einf_{B_k}u\right)^q \mathrm{d}m\right)^{\frac{1}{q}}\right)\\
&\le C_1C_2\left(\left(\sup_{B_{k}}u-\einf_{B_{k}}u\right)-\left(\esup_{B_{k+1}}u-\einf_{B_{k+1}}u\right)+\frac{2}{A^{\frac{\beta_\Upsilon^{(1)}}{p-1}}}\alpha^{k}K\right),
\end{align*}
hence
\begin{align*}
\eosc_{B_{k+1}}u\le \frac{C_1C_2-1}{C_1C_2}\eosc_{B_k}u+\frac{2}{A^{\frac{\beta_\Upsilon^{(1)}}{p-1}}}\alpha^{k}K\le\left(\frac{C_1C_2-1}{C_1C_2}+\frac{2}{{A^{\frac{\beta_\Upsilon^{(1)}}{p-1}}}}\right)\alpha^{k}K,
\end{align*}
where the second inequality follows from the induction assumption. First, recall that $A=2aA_1$, choose $a>1$, depending only on $p,\beta_\Upsilon^{(1)},C_1,C_2$, sufficiently large such that $\frac{C_1C_2-1}{C_1C_2}+\frac{2}{{A^{\frac{\beta_\Upsilon^{(1)}}{p-1}}}}<1$. Second, choose $\alpha=\max\{\frac{1}{2},\frac{C_1C_2-1}{C_1C_2}+\frac{2}{{A^{\frac{\beta_\Upsilon^{(1)}}{p-1}}}}\}\in[\frac{1}{2},1)$. Then, $\eosc_{B_{k+1}}u\le \alpha^{k+1}K$, that is, (\ref{eq_int_osc_J}) holds for $k+1$. By induction, (\ref{eq_int_osc_J}) holds for any $k\ge0$.

For any $r\in(0,R]$, there exists a unique integer $k\ge0$ such that $R_{k+1}<r\le R_k$, then by (\ref{eq_int_osc_J}), we have
\begin{align*}
&\eosc_{B(x_0,r)}u\le \eosc_{B_k}u \le \alpha^k K\\
&\le \frac{2}{\alpha}\left(\frac{r}{R}\right)^{\delta}\left(\lVert u\rVert_{L^\infty(B(x_0,R);m)}+\Upsilon(R)^{\frac{1}{p-1}}\mathrm{Tail}^{(J)}(u_-;x_0,R)+\Upsilon(R)^{\frac{1}{p-1}}\right),
\end{align*}
where $\delta=-\log_A {\alpha}>0$.
\end{proof}

\begin{proof}[Proof of Proposition \ref{prop_bdy_osc_J}]
We combine the arguments in the proofs of Proposition \ref{prop_int_osc_J} and Proposition \ref{prop_bdy_osc_L}. Let $q>0$, $A_1>1$ be the constants appearing in Proposition \ref{prop_bdy_wEHI_J}. Let $a>1$ be chosen sufficiently large (to be specified later) such that $A=2aA_1$ satisfies $A^{\beta_\Upsilon^{(1)}}\ge 2^p$, where $\beta^{(1)}_\Upsilon$ is a constant appearing in (\ref{eq_beta12}). For any $k\ge0$, let $R_k=\frac{1}{A^k}R$ and $B_k=B(x_0,R_k)$. Let $\alpha\in[\frac{1}{2},1)$ be chosen later and
$$K=\esup_{X}u+\Upsilon(R)^{\frac{1}{p-1}}.$$
The proof also proceeds by choosing suitable $a$ and $\alpha$, and applying induction to show that, for any $k\ge0$, we have
\begin{equation}\label{eq_bdy_osc_J}
\esup_{B_k}u\le \alpha^k K=\alpha^k \left(\esup_{X}u+\Upsilon(R)^{\frac{1}{p-1}}\right).
\end{equation}

By the a priori bounds on $a$ and $\alpha$, we also have
\begin{equation}\label{eq_aalpha2}
\alpha^{p-1}A^{\beta_\Upsilon^{(1)}}\ge \frac{1}{2^{p-1}}2^p=2.
\end{equation}
We also emphasize that, in the following proof, the implicit constants in $\lesssim$ and $\asymp$ are \emph{independent} of $\alpha$ and $A$ (or, equivalently, $a$). For $k=0$, the result is obvious. Assume that the result holds for $j=0,\ldots,k$. For $k+1$, let $M=\esup_{B_k}u$, then $M\in[0,+\infty)$. Let $v=M-u$, then $v\in \widehat{\mathcal{F}}^{(J)}\cap L^\infty(X;m)$ is non-negative in $B_k$, $\widetilde{v}=M$ q.e. on $B_k\backslash \Omega$, $v\wedge M=v=M-u$ in $X$, $(v\wedge M)_-=(u-M)_+$ in $X$, and satisfies $-\Delta_p^{(J)}v=-f\ge-1$ in $B_k\cap \Omega$. Since $B_k=2aA_1B_{k+1}\supseteq 2A_1B_{k+1}$, by Proposition \ref{prop_bdy_wEHI_J}, we have
\begin{align*}
&\left(\dashint_{2B_{k+1}}(v\wedge M)^q \mathrm{d}m\right)^{\frac{1}{q}}\\
&\lesssim \einf_{B_{k+1}}(v\wedge M)+\Upsilon(2R_{k+1})^{\frac{1}{p-1}}\mathrm{Tail}^{(J)}\left((v\wedge M)_-;x_0,2A_1R_{k+1}\right)+\Upsilon(2R_{k+1})^{\frac{1}{p-1}},
\end{align*}
that is,
\begin{align*}
&\left(\dashint_{2B_{k+1}}\left(M-u\right)^q \mathrm{d}m\right)^{\frac{1}{q}}\\
&\lesssim \einf_{B_{k+1}}\left(M-u\right)+\Upsilon(2R_{k+1})^{\frac{1}{p-1}}\mathrm{Tail}^{(J)}\left((u-M)_+;x_0,2A_1R_{k+1}\right)+\Upsilon(2R_{k+1})^{\frac{1}{p-1}},
\end{align*}
where
$$\einf_{B_{k+1}}\left(M-u\right)=\esup_{B_{k}}u-\esup_{B_{k+1}}u.$$
By (\ref{eq_beta12}) and (\ref{eq_aalpha2}), following the same argument as in (\ref{eq_UpRk1}), we also have
\begin{equation*}
\Upsilon(2R_{k+1})\lesssim \frac{1}{A^{\beta^{(1)}_\Upsilon}}\alpha^{(p-1)k}\Upsilon(R)\le\frac{1}{A^{\beta^{(1)}_\Upsilon}}\left(\alpha^{k}K\right)^{p-1}.
\end{equation*}
Since $u\le M$ in $B_k$, we have
\begin{align*}
&\mathrm{Tail}^{(J)}\left((u-M)_+;x_0,2A_1R_{k+1}\right)^{p-1}\\
&=\int_{X\backslash 2A_1B_{k+1}}(u(y)-M)_+^{p-1} K^{(J)}(x_0,y) m(\mathrm{d}y)=\int_{X\backslash B_{k}}(u(y)-M)_+^{p-1} K^{(J)}(x_0,y) m(\mathrm{d}y)\\
&=\left(\sum_{j=0}^{k-1}\int_{B_j\backslash B_{j+1}}+\int_{X\backslash B_0}\right)(u(y)-M)_+^{p-1} K^{(J)}(x_0,y) m(\mathrm{d}y).
\end{align*}
By the induction assumption, for any $j=0,\ldots,k-1$, we have
$$(u-M)_+\le \esup_{B_j}u\le \alpha^jK\text{ in }B_j.$$
By \ref{eq_KJ_tail}, we have
$$\int_{B_j\backslash B_{j+1}}(u(y)-M)_+^{p-1} K^{(J)}(x_0,y) m(\mathrm{d}y)\lesssim \frac{1}{\Upsilon(R_{j+1})}\left(\alpha^jK\right)^{p-1}.$$
Since $(u-M)_+\le 2\esup_X u\le 2K$ in $X$, by \ref{eq_KJ_tail}, we have
\begin{align*}
&\int_{X\backslash B_0}(u(y)-M)_+^{p-1} K^{(J)}(x_0,y) m(\mathrm{d}y)\lesssim \frac{1}{\Upsilon(R)}K^{p-1}.
\end{align*}
Hence
\begin{align*}
&\Upsilon(2R_{k+1})\mathrm{Tail}^{(J)}\left((u-M)_+;x_0,2A_1R_{k+1}\right)^{p-1}\nonumber\\
&\lesssim\Upsilon(R_{k+1}) \left(\sum_{j=0}^{k-1}\frac{1}{\Upsilon(R_{j+1})}(\alpha^j K)^{p-1} +\frac{1}{\Upsilon(R)}K^{p-1}\right)\nonumber\\
&\overset{(\star)}{\scalebox{2}[1]{$\lesssim$}}\frac{1}{A^{\beta^{(1)}_\Upsilon}}\left(\alpha^kK\right)^{p-1},
\end{align*}
where $(\star)$ follows the same argument as in the proof, starting from the second line and continuing to the end of (\ref{eq_osc_Tail1}). Hence there exists $C_1>1$ such that
$$\left(\dashint_{2B_{k+1}}\left(\esup_{B_k}u-u\right)^q \mathrm{d}m\right)^{\frac{1}{q}}\le C_1 \left(\esup_{B_k}u-\esup_{B_{k+1}}u+\frac{1}{A^{\frac{\beta_\Upsilon^{(1)}}{p-1}}}\alpha^kK\right).$$
Since $X\backslash \Omega$ has a $(c,r_0)$-corkscrew at $x_0$, we have $2B_{k+1}\backslash\Omega$ contains a ball $B$ with radius $2cR_{k+1}$. Since $\esup_{B_k}u-\widetilde{u}=\esup_{B_k}u$ q.e. on $2B_{k+1}\backslash \Omega\supseteq B$, by \ref{eq_VD}, we have
$$\left(\dashint_{2B_{k+1}}\left(\esup_{B_k}u-u\right)^q \mathrm{d}m\right)^{\frac{1}{q}}\ge \left(\frac{m(B)}{m \left(2B_{k+1}\right)}\right)^{\frac{1}{q}}\esup_{B_k}u\ge\frac{1}{C_2}\esup_{B_k}u,$$
where $C_2=C_{VD}^{\frac{2-\log_2c}{q}}\ge1$. Hence
$$\esup_{B_k}\le C_1C_2 \left(\esup_{B_k}u-\esup_{B_{k+1}}u+\frac{1}{A^{\frac{\beta_\Upsilon^{(1)}}{p-1}}}\alpha^kK\right),$$
which gives
$$\esup_{B_{k+1}}u\le \frac{C_1C_2-1}{C_1C_2}\esup_{B_k}u+\frac{1}{A^{\frac{\beta_\Upsilon^{(1)}}{p-1}}}\alpha^kK\le \left(\frac{C_1C_2-1}{C_1C_2}+\frac{1}{A^{\frac{\beta_\Upsilon^{(1)}}{p-1}}}\right)\alpha^kK,$$
where the second inequality follows from the induction assumption. First, recall that $A=2aA_1$, choose $a>1$, depending only on $p,\beta_\Upsilon^{(1)},C_1,C_2$, sufficiently large such that $\frac{C_1C_2-1}{C_1C_2}+\frac{1}{{A^{\frac{\beta_\Upsilon^{(1)}}{p-1}}}}<1$. Second, choose $\alpha=\max\{\frac{1}{2},\frac{C_1C_2-1}{C_1C_2}+\frac{1}{{A^{\frac{\beta_\Upsilon^{(1)}}{p-1}}}}\}\in[\frac{1}{2},1)$. Then, $\esup_{B_{k+1}}u\le \alpha^{k+1}K$, that is, (\ref{eq_bdy_osc_J}) holds for $k+1$. By induction, (\ref{eq_bdy_osc_J}) holds for any $k\ge0$.

For any $r\in(0,R]$, there exists a unique integer $k\ge0$ such that $R_{k+1}<r\le R_k$, then by (\ref{eq_bdy_osc_J}), we have
\begin{align*}
\esup_{B(x_0,r)}u\le \esup_{B_k}u \le \alpha^k K\le \frac{1}{\alpha}\left(\frac{r}{R}\right)^{\delta}\left(\esup_Xu+\Upsilon(R)^{\frac{1}{p-1}}\right),
\end{align*}
where $\delta=-\log_A {\alpha}>0$.
\end{proof}

\section{Proof of \texorpdfstring{``{$\text{CS}^{(J)}_{\text{weak}}\Rightarrow\text{CE}^{(J)}_{\text{strong}}$}"}{CSJweak to CEJstrong} in Theorem \ref{thm_equiv_J}}\label{sec_CS2CE_J}

Throughout this section, we always assume that \ref{eq_VD} holds and that $(\mathcal{E}^{(J)},\mathcal{F}^{(J)})$ is a non-local regular $p$-energy given by a kernel $K^{(J)}$ satisfying \ref{eq_UPR}, \ref{eq_KJ_tail}, \ref{eq_PIJ}, \hyperlink{eq_CSJ_weak}{$\text{CS}^{(J)}_{\text{weak}}(\Upsilon)$}.

We have the following result in the non-local setting, which parallels \cite[Proposition 3.15]{Yan25d} in the local setting. In \cite{Yan25d}, this result was proved via a Wolff potential estimate under the assumption of a suitable elliptic Harnack inequality. In our setting, since \hyperlink{eq_CSJ_weak}{$\text{CS}^{(J)}_{\text{weak}}(\Upsilon)$} holds, certain arguments from the theory of Dirichlet forms (with $p=2$) can be employed.

\begin{proposition}\label{prop_EFJ}
There exists $C>0$ such that for any ball $B=B(x_0,R)$, for any $u\in \mathcal{F}^{(J)}(B)$ satisfying that $-\Delta_p^{(J)}u=1$ in $B$, we have
\begin{align*}
\esup_{B}u&\le C \Upsilon(R)^{\frac{1}{p-1}},\label{eq_EFJ1}\tag*{$\text{E}^{(J)}(\Upsilon)_\le$}\\
\einf_{\frac{1}{4}B}u&\ge \frac{1}{C}\Upsilon(R)^{\frac{1}{p-1}}.\label{eq_EFJ2}\tag*{$\text{E}^{(J)}(\Upsilon)_\ge$}
\end{align*}
\end{proposition}

The proof proceeds along the same ideas as that of \cite[Corollary 12.5]{GHH24a}, as follows.

We say that the non-local Faber-Krahn inequality \ref{eq_FKJ} holds if there exist $\nu>0$, $C>0$ such that for any non-empty open subset $\Omega\subseteq B$, for any $u\in \mathcal{F}^{(J)}(\Omega)$, we have
\begin{equation*}\label{eq_FKJ}\tag*{$\text{FK}^{(J)}(\Upsilon)$}
\lVert u\rVert_{L^p(\Omega;m)}^p\le C\Upsilon(R) \left(\frac{m(\Omega)}{m(B)}\right)^\nu \mathcal{E}^{(J)}(u).
\end{equation*}
By \cite[Proposition 3.5]{CY25}, we have
\begin{center}
\ref{eq_CC} + \ref{eq_VD} + \ref{eq_PIJ} $\Rightarrow$ \ref{eq_FKJ},
\end{center}
where \ref{eq_CC} and \ref{eq_VD} together imply condition (RVD) therein. The same argument as in \cite[THEOREM 9.4]{GHL15} (see also \cite[Lemma 12.2]{GHH24a}) yields
\begin{center}
\ref{eq_FKJ} $\Rightarrow$ \ref{eq_EFJ1}.
\end{center}
Hence, we obtain \ref{eq_EFJ1} in Proposition \ref{prop_EFJ}.

We say that the condition \hypertarget{eq_LG}{$\text{LG}$} (Lemma of growth) holds if there exist $\varepsilon,\tau\in(0,1)$ such that for any ball $B=B(x_0,R)$, for any $u\in \widehat{\mathcal{F}}^{(J)}\cap L^\infty(X;m)$ which is superharmonic in $2B$ and non-negative in $X$, for any $a>0$, if
$$\frac{m(B\cap\{u<a\})}{m(B)}\le \varepsilon,$$
then
$$\einf_{\frac{1}{2}B}u\ge \tau a.$$
See also \cite[Lemma 4.4]{CY25} and \cite[Definition 11.1]{GHH24a} for related conditions. Following \cite{GHH24a}, we require that $u$ be non-negative in the whole space $X$. Under this assumption, the tail term in \cite{CY25} does not appear, which is sufficient for our purposes.

By \cite[Lemma 4.4]{CY25} (see also \cite[Lemma 11.2]{GHH24a}), we have
\begin{center}
\ref{eq_VD} + \hyperlink{eq_CSJ_weak}{$\text{CS}^{(J)}_{\text{weak}}(\Upsilon)$} + \ref{eq_FKJ} + \ref{eq_KJ_tail} $\Rightarrow$ \hyperlink{eq_LG}{\text{LG}}.
\end{center}
In the pure non-local setting of \cite{CY25}, condition (SI) therein is equivalent to \ref{eq_FKJ} (see \cite[Theorem 3.8]{CY25}), while condition (TJ) corresponds to \ref{eq_KJ_tail}. Moreover, condition (CS) therein is implied by \hyperlink{eq_CSJ_weak}{$\text{CS}^{(J)}_{\text{weak}}(\Upsilon)$}, which follows from the self-improvement property of \hyperlink{eq_CSJ_weak}{$\text{CS}^{(J)}_{\text{weak}}(\Upsilon)$} (Lemma \ref{prop_CSJ_self}).

The same argument as in \cite[Lemma 12.4]{GHH24a} yields
\begin{center}
\ref{eq_VD} + \hyperlink{eq_LG}{$\text{LG}$} + \ref{eq_ucapJ} $\Rightarrow$ \ref{eq_EFJ2}.
\end{center}
In the argument, the technical result \cite[Proposition 12.3]{GHH24a} (or \cite[Lemma 3.7]{GHH18}) is replaced by the following result (see also \cite[Lemma 4.5]{CY25} for a slightly different statement). Consequently, we obtain \ref{eq_EFJ2} in Proposition \ref{prop_EFJ}.

\begin{lemma}
Let $\Omega\subseteq X$ be a bounded open subset, and let $u\in \widehat{\mathcal{F}}^{(J)}\cap L^\infty(X;m)$ and $\phi\in \mathcal{F}^{(J)}(\Omega)\cap L^\infty(X;m)$ be non-negative in $X$. For any $\lambda>0$, let $u_\lambda=u+\lambda$, then we have $\frac{\phi^p}{u_\lambda^{p-1}}\in \mathcal{F}^{(J)}(\Omega)$ and
$$\mathcal{E}^{(J)}\left(u;\frac{\phi^p}{u_\lambda^{p-1}}\right)\le C \mathcal{E}^{(J)}(\phi),$$
where $C>0$ is some constant depending only on $p$.
\end{lemma}

\begin{proof}
By the Markovian property (see also \cite[Proposition A.3]{CY25}), we have $\frac{\phi^p}{u_\lambda^{p-1}}\in \mathcal{F}^{(J)}(\Omega)$. Moreover, we have
\begin{align*}
&\mathcal{E}^{(J)}\left(u;\frac{\phi^p}{u_\lambda^{p-1}}\right)\\
&=\left(\int_{\Omega}\int_{\Omega}+2\int_\Omega\int_{X\backslash \Omega}\right)\lvert u(x)-u(y)\rvert^{p-2} \left(u(x)-u(y)\right) \left(\frac{\phi(x)^p}{u_\lambda(x)^{p-1}}-\frac{\phi(y)^p}{u_\lambda(y)^{p-1}}\right)\\
&\hspace{100pt}\cdot K^{(J)}(x,y) m(\mathrm{d}x)m(\mathrm{d}y)\\
&=I_1+2I_2.
\end{align*}
For $I_1$, since both $u$ and $\phi$ are non-negative in $X$, by Lemma \ref{lem_ele5}, we have
\begin{align*}
&I_1\le -C_1 \int_\Omega\int_\Omega\lvert \log u_\lambda(x)-\log u_\lambda(y)\rvert^p\min\{\phi(x)^p,\phi(y)^p\}K^{(J)}(x,y)m(\mathrm{d}x)m(\mathrm{d}y)\\
&\hspace{20pt}+C_2 \int_\Omega\int_\Omega\lvert \phi(x)-\phi(y)\rvert^pK^{(J)}(x,y)m(\mathrm{d}x)m(\mathrm{d}y)\\
&\le C_2 \int_\Omega\int_\Omega\lvert \phi(x)-\phi(y)\rvert^pK^{(J)}(x,y)m(\mathrm{d}x)m(\mathrm{d}y),
\end{align*}
where $C_1,C_2>0$ depend only on $p$. For $I_2$, we have
\begin{align*}
&I_2=\int_\Omega\int_{X\backslash \Omega}\lvert u(x)-u(y)\rvert^{p-2} \left(u(x)-u(y)\right) \frac{\phi(x)^p}{u_\lambda(x)^{p-1}}K^{(J)}(x,y) m(\mathrm{d}x)m(\mathrm{d}y)\\
&\le\int_\Omega\int_{X\backslash \Omega}\left(u(x)-u(y)\right)^{p-1} \frac{\phi(x)^p}{u_\lambda(x)^{p-1}}K^{(J)}(x,y) 1_{u(x)\ge u(y)}m(\mathrm{d}x)m(\mathrm{d}y)\\
&\overset{(\diamond)}{\scalebox{2}[1]{$\le$}}\int_\Omega\int_{X\backslash \Omega}u_\lambda(x)^{p-1} \frac{\phi(x)^p}{u_\lambda(x)^{p-1}}K^{(J)}(x,y) m(\mathrm{d}x)m(\mathrm{d}y)\\
&=\int_\Omega\int_{X\backslash \Omega} \lvert \phi(x)-\phi(y)\rvert^p K^{(J)}(x,y)m(\mathrm{d}x)m(\mathrm{d}y),
\end{align*}
where $(\diamond)$ follows from the fact that
\begin{align*}
\left(u(x)-u(y)\right)^{p-1}1_{u(x)\ge u(y)}\le u(x)^{p-1}1_{u(x)\ge u(y)}\le u_\lambda(x)^{p-1},
\end{align*}
due to $u\ge0$ in $X$. Hence
\begin{align*}
&\mathcal{E}^{(J)}\left(u;\frac{\phi^p}{u_\lambda^{p-1}}\right)\\
&\le \left(C_2\int_\Omega\int_\Omega+2\int_\Omega\int_{X\backslash \Omega}\right)\lvert \phi(x)-\phi(y)\rvert^p K^{(J)}(x,y)m(\mathrm{d}x)m(\mathrm{d}y)\le C \mathcal{E}^{(J)}(\phi),
\end{align*}
where $C=\max\{C_2,1\}$.
\end{proof}

The following result in the non-local setting parallels \cite[Proposition 3.17]{Yan25d} in the local setting. The proof follows the same line of argument as therein, by applying Propositions \ref{prop_comparisonJ} and \ref{prop_EFJ}, and is therefore omitted.

\begin{proposition}\label{prop_resolvent}
There exist $C_1,C_2>0$ such that for any $\lambda>0$ and any ball $B=B(x_0,R)$, for any $u\in \mathcal{F}^{(J)}(B)$ satisfying that $-\Delta_p^{(J)}u+\lambda \lvert u\rvert^{p-2}u=1$ in $B$, we have
\begin{align*}
\esup_{B}u&\le \frac{1}{\lambda^{\frac{1}{p-1}}},\\
\einf_{\frac{1}{4}B}u&\ge \frac{C_1}{\left(C_2+\lambda \Upsilon(R)\right)^{\frac{1}{p-1}}}\Upsilon(R)^{\frac{1}{p-1}}.
\end{align*}
\end{proposition}

We give the proof of ``\hyperlink{eq_CSJ_weak}{$\text{CS}^{(J)}_{\text{weak}}(\Upsilon)$}$\Rightarrow${\hyperlink{eq_CE_strong}{$\text{CE}^{(J)}_{\text{strong}}(\Upsilon)$}}" in Theorem \ref{thm_equiv_J} as follows.

\begin{proof}[Proof of ``\hyperlink{eq_CSJ_weak}{$\text{CS}^{(J)}_{\text{weak}}(\Upsilon)$}$\Rightarrow${\hyperlink{eq_CE_strong}{$\text{CE}^{(J)}_{\text{strong}}(\Upsilon)$}}" in Theorem \ref{thm_equiv_J}]
For any ball $B=B(x_0,r)$, by Lemma \ref{lem_corkscrew}, there exist $c\in(0,1)$, which depends only on $C_{cc}$, and an open subset $\Omega\subseteq X$ with $4B\subseteq \Omega\subseteq 5B$ such that for any $x\in \partial \Omega$, $X\backslash \Omega$ has a $(c,32r)$-corkscrew at $x$.

By Proposition \ref{prop_exist}, there exists a unique $u_\Omega\in \mathcal{F}^{(J)}(\Omega)$ such that
$$-\Delta_p^{(J)}u_\Omega+\frac{1}{\Upsilon(r)}\lvert u_\Omega\rvert^{p-2}u_\Omega=1\text{ in }\Omega.$$
By Proposition \ref{prop_comparisonJ}, we have $0\le u_\Omega\le \Upsilon(r)^{\frac{1}{p-1}}$ in $X$. Similarly, there exists a unique $v\in \mathcal{F}^{(J)}(4B)$ such that $-\Delta_p^{(J)}v+\frac{1}{\Upsilon(r)}\lvert v\rvert^{p-2}v=1$ in $4B$. By Proposition \ref{prop_resolvent}, there exist $C_1,C_2>0$ such that
$$\inf_{B}v\ge \frac{C_1}{\left(C_2+\frac{1}{\Upsilon(r)}\Upsilon(4r)\right)^{\frac{1}{p-1}}}\Upsilon(4r)^{\frac{1}{p-1}}\ge C_3\Upsilon(r)^{\frac{1}{p-1}},$$
where $C_3>0$ depends only on $p,C_\Upsilon,C_1,C_2$. Since $4B\subseteq \Omega$, by Proposition \ref{prop_comparisonJ}, we have $u_\Omega\ge v$ in $X$. Since $-\Delta_p^{(J)}u_\Omega=1-\frac{1}{\Upsilon(r)}\lvert u_\Omega\rvert^{p-2}u_\Omega$ in $\Omega$, where $0\le1-\frac{1}{\Upsilon(r)}\lvert u_\Omega\rvert^{p-2}u_\Omega\le1$ in $\Omega$, an initial application of Propositions \ref{prop_int_osc_J} and \ref{prop_bdy_osc_J} yields that $u_\Omega\in C(\overline{\Omega})$ and $u_\Omega=0$ on $\partial\Omega$, and hence $u_\Omega\in C(X)$ with $u_\Omega=0$ on $X\backslash\Omega$. Hence
\begin{equation}\label{eq_u_lbd_J}
\inf_{B}u_\Omega=\einf_{B}u_\Omega\ge\einf_{B}v\ge C_3\Upsilon(r)^{\frac{1}{p-1}}.
\end{equation}

Let $\delta_1,\delta_2\in(0,1)$ be the constants appearing in Propositions \ref{prop_int_osc_J} and \ref{prop_bdy_osc_J}, respectively, and $\delta=\min\{\delta_1,\delta_2,\frac{\beta_\Upsilon^{(1)}}{p-1},\frac{\beta^{(1)}_\Upsilon}{2p}\}\in(0,1)$.

We claim that there exists $C_4>0$ such that
\begin{equation}\label{eq_u_Holder_J}
\lvert u_\Omega(x)-u_\Omega(y)\rvert\le C_4 \left(\frac{d(x,y)}{r}\wedge 1\right)^\delta\Upsilon(r)^{\frac{1}{p-1}}\text{ for any }x,y\in X.
\end{equation}
Indeed, if $x,y\in X\backslash \Omega$, then this result is trivial, hence without loss of generality, we may assume that $x\in \Omega$, then $D=\mathrm{dist}(x,X\backslash \Omega)\in(0,10r)$, and there exists $y_0\in\partial\Omega$ such that $D\le d(y_0,x)<2D$. If $d(x,y)\ge r$, then this result follows directly from the fact that $0\le u_\Omega\le\Upsilon(r)^{\frac{1}{p-1}}$ in $X$, hence we may assume that $d(x,y)<r$.

If $y\in \Omega$ and $d(x,y)<D$, then by Proposition \ref{prop_int_osc_J}, we have
\begin{align*}
&\lvert u_\Omega(x)-u_\Omega(y)\rvert\lesssim\left(\frac{d(x,y)}{D}\right)^{\delta_1}\left(\sup_{B(x,D)}u_{\Omega}+\Upsilon(D)^{\frac{1}{p-1}}\right)\\
&\le \left(\frac{d(x,y)}{D}\right)^{\delta}\left(\sup_{B(y_0,3D)\cap\Omega}u_{\Omega}+\Upsilon(D)^{\frac{1}{p-1}}\right).
\end{align*}
Since $3D<32r$, by Proposition \ref{prop_bdy_osc_J}, we have
\begin{align*}
\sup_{B(y_0,3D)\cap\Omega}u_\Omega\lesssim \left(\frac{D}{r}\right)^{\delta_2}\left(\sup_{X}u_{\Omega}+\Upsilon(r)^{\frac{1}{p-1}}\right)\lesssim \left(\frac{D}{r}\right)^{\delta}\Upsilon(r)^{\frac{1}{p-1}}.
\end{align*}
By (\ref{eq_beta12}), we have
$$\Upsilon(D)^{\frac{1}{p-1}}\lesssim \left(\frac{D}{r}\right)^{\frac{\beta_\Upsilon^{(1)}}{p-1}}\Upsilon(r)^{\frac{1}{p-1}}\lesssim\left(\frac{D}{r}\right)^{\delta}\Upsilon(r)^{\frac{1}{p-1}}.$$
Hence
$$\lvert u_\Omega(x)-u_{\Omega}(y)\rvert\lesssim \left(\frac{d(x,y)}{r}\right)^{\delta}\Upsilon(r)^{\frac{1}{p-1}}.$$

If $y\in\Omega$ and $d(x,y)\ge D$, then $x,y\in B(y_0,3d(x,y))\subseteq B(y_0,3r)$, hence by Proposition \ref{prop_bdy_osc_J}, we have
\begin{align*}
&\lvert u_\Omega(x)-u_\Omega(y)\rvert\le 2\sup_{B(y_0,3d(x,y))\cap\Omega}u_\Omega\\
&\lesssim  \left(\frac{d(x,y)}{r}\right)^{\delta_2}\left(\sup_{X}u_\Omega+\Upsilon(r)^{\frac{1}{p-1}}\right)\lesssim \left(\frac{d(x,y)}{r}\right)^{\delta}\Upsilon(r)^{\frac{1}{p-1}}.
\end{align*}

If $y\in X\backslash\Omega$, then $D\le d(x,y)$, hence by Proposition \ref{prop_bdy_osc_J}, we have
\begin{align*}
&\lvert u_\Omega(x)-u_\Omega(y)\rvert=u_\Omega(x)\le \sup_{B(y_0,2D)\cap\Omega}u_\Omega\\
&\lesssim \left(\frac{D}{r}\right)^{\delta_2}\left(\sup_{X}u_\Omega+\Upsilon(r)^{\frac{1}{p-1}}\right)\lesssim \left(\frac{d(x,y)}{r}\right)^{\delta}\Upsilon(r)^{\frac{1}{p-1}}.
\end{align*}
Now we finish the proof of (\ref{eq_u_Holder_J}).

We claim that there exists $C_5>0$ such that
\begin{equation}\label{eq_u_energy_J}
\int_{B(x,s)}\mathrm{d}\Gamma^{(J)}_X(u_\Omega)\le C_5\left(\frac{s}{r}\wedge 1\right)^\delta \Upsilon(r)^{\frac{p}{p-1}}\frac{V(x,s)}{\Upsilon(s\wedge r)}\text{ for any }x\in X,s>0.
\end{equation}
Indeed, by \ref{eq_VD}, \ref{eq_ucapJ}, there exists a cutoff function $\psi\in \mathcal{F}^{(J)}$ for $B(x,s)\subseteq B(x,2s)$ such that $\mathcal{E}^{(J)}(\psi)\lesssim \frac{V(x,s)}{\Upsilon(s)}$, then
\begin{align}
&\int_{B(x,s)}\mathrm{d}\Gamma^{(J)}_X(u_\Omega)\le \int_{B(x,2s)}\psi^p\mathrm{d}\Gamma^{(J)}_X(u_\Omega)\nonumber\\
&=\left(\int_{B(x,2s)}\int_{B(x,4s)}+\int_{B(x,2s)}\int_{X\backslash B(x,4s)}\right)\psi(y)^p \lvert u_\Omega(y)-u_{\Omega}(z)\rvert^{p}K^{(J)}(y,z) m(\mathrm{d}y)m(\mathrm{d}z)\nonumber\\
&=I_1+I_2.\label{eq_CS2CE_J1}
\end{align}
Let $c=\inf_{B(x,4s)}u_\Omega$, then by (\ref{eq_u_Holder_J}), we have
\begin{equation}\label{eq_u_osc_J}
0\le u_\Omega-c\lesssim\left(\frac{s}{r}\wedge 1\right)^\delta\Upsilon(r)^{\frac{1}{p-1}}\text{ in }B(x,4s).
\end{equation}
By Lemma \ref{lem_ele6}, we have
\begin{align}
&I_1\le\int_{B(x,4s)}\int_{B(x,4s)}\left(\psi(y)^p+\psi(z)^p\right) \lvert (u_\Omega(y)-c)-(u_{\Omega}(z)-c)\rvert^{p}K^{(J)}(y,z) m(\mathrm{d}y)m(\mathrm{d}z)\nonumber\\
&\le 4\int_{B(x,4s)}\int_{B(x,4s)}\lvert u_\Omega(y)-u_{\Omega}(z)\rvert^{p-2}(u_\Omega(y)-u_\Omega(z))\nonumber\\
&\hspace{80pt}\cdot\left(\psi(y)^p(u_\Omega(y)-c)-\psi(z)^p(u_\Omega(z)-c)\right)K^{(J)}(y,z) m(\mathrm{d}y)m(\mathrm{d}z)\nonumber\\
&\hspace{15pt}+4C\int_{B(x,4s)}\int_{B(x,4s)} \left((u(y)-c)^p+(u(z)-c)^p\right)\lvert \psi(y)-\psi(z)\rvert^pK^{(J)}(y,z)m(\mathrm{d}y)m(\mathrm{d}z)\nonumber\\
&=4I_{11}+I_{12},\label{eq_CS2CE_J2}
\end{align}
where $C>0$ depends only on $p$. By (\ref{eq_u_osc_J}), we have
\begin{align}
&I_{12}\lesssim\left(\frac{s}{r}\wedge 1\right)^{p\delta}\Upsilon(r)^{\frac{p}{p-1}}\int_{B(x,4s)}\int_{B(x,4s)} \lvert \psi(y)-\psi(z)\rvert^pK^{(J)}(y,z)m(\mathrm{d}y)m(\mathrm{d}z)\nonumber\\
&\le\left(\frac{s}{r}\wedge 1\right)^{p\delta}\Upsilon(r)^{\frac{p}{p-1}}\mathcal{E}^{(J)}(\psi)\lesssim\left(\frac{s}{r}\wedge 1\right)^{p\delta}\Upsilon(r)^{\frac{p}{p-1}}\frac{V(x,s)}{\Upsilon(s)}\le\left(\frac{s}{r}\wedge 1\right)^{\delta}\Upsilon(r)^{\frac{p}{p-1}}\frac{V(x,s)}{\Upsilon(s)}.\label{eq_CS2CE_J3}
\end{align}
Since $-\Delta_p^{(J)}u_\Omega+\frac{1}{\Upsilon(r)}\lvert u_\Omega\rvert^{p-2}u_\Omega=1$ in $\Omega$, and $\psi^p(u_\Omega-c)\in \mathcal{F}^{(J)}(\Omega)$ is non-negative in $X$, we have
\begin{align}
&\mathcal{E}^{(J)}(u_\Omega;\psi^p(u_\Omega-c))\nonumber\\
&\le\mathcal{E}^{(J)}(u_\Omega;\psi^p(u_\Omega-c))+{\frac{1}{\Upsilon(r)}}\int_X \lvert u_\Omega\rvert^{p-2}u_\Omega\psi^p(u_\Omega-c)\mathrm{d}m\nonumber\\
&=\int_X\psi^p(u_\Omega-c)\mathrm{d}m=\int_{B(x,2s)}\psi^p(u_\Omega-c)\mathrm{d}m\lesssim\left(\frac{s}{r}\wedge 1\right)^\delta\Upsilon(r)^{\frac{1}{p-1}} V(x,s),\label{eq_CS2CE_J4}
\end{align}
where the last inequality follows from (\ref{eq_u_osc_J}) and \ref{eq_VD}. On the other hand, we have
\begin{equation}\label{eq_CS2CE_J5}
\mathcal{E}^{(J)}(u_\Omega;\psi^p(u_\Omega-c))=I_{11}+2I_{13},
\end{equation}
where
\begin{align*}
&I_{13}=\int_{B(x,4s)}\int_{X\backslash B(x,4s)} \lvert u_\Omega(y)-u_\Omega(z)\rvert^{p-2}(u_\Omega(y)-u_\Omega(z))\\
&\hspace{50pt}\cdot \left(\psi(y)^p(u_\Omega(y)-c)-\psi(z)^p(u_\Omega(z)-c)\right)K^{(J)}(y,z)m(\mathrm{d}y)m(\mathrm{d}z)\\
&=\int_{B(x,2s)}\int_{X\backslash B(x,4s)} \lvert u_\Omega(y)-u_\Omega(z)\rvert^{p-2}(u_\Omega(y)-u_\Omega(z))\\
&\hspace{50pt}\cdot\psi(y)^p(u_\Omega(y)-c)K^{(J)}(y,z)m(\mathrm{d}y)m(\mathrm{d}z).
\end{align*}
By (\ref{eq_u_osc_J}) and the fact that $0\le u_\Omega\le {\Upsilon(r)^{\frac{1}{p-1}}}$ in $X$, we have
\begin{align}
&\lvert I_{13}\rvert\le \int_{B(x,2s)}\int_{X\backslash B(x,4s)} \lvert u_\Omega(y)-u_\Omega(z)\rvert^{p-1}(u_\Omega(y)-c)K^{(J)}(y,z)m(\mathrm{d}y)m(\mathrm{d}z)\nonumber\\
&\lesssim\left(\frac{s}{r}\wedge 1\right)^\delta\Upsilon(r)^{\frac{p}{p-1}}\int_{B(x,2s)}\int_{X\backslash B(x,4s)}K^{(J)}(y,z)m(\mathrm{d}y)m(\mathrm{d}z)\nonumber\\
&\lesssim \left(\frac{s}{r}\wedge 1\right)^\delta\Upsilon(r)^{\frac{p}{p-1}}\frac{V(x,s)}{\Upsilon(s)},\label{eq_CS2CE_J6}
\end{align}
where the last inequality follows from \ref{eq_KJ_tail} and \ref{eq_VD}. Combining (\ref{eq_CS2CE_J2})--(\ref{eq_CS2CE_J6}), we have
\begin{align}
&I_1\le 4I_{11}+I_{12}=4 \left(\mathcal{E}^{(J)}(u_\Omega;\psi^p(u_\Omega-c))-2I_{13}\right)+I_{12}\nonumber\\
&\lesssim \left(\frac{s}{r}\wedge 1\right)^\delta \Upsilon(r)^{\frac{p}{p-1}}\left(\frac{1}{\Upsilon(r)}+\frac{1}{\Upsilon(s)}\right)V(x,s).\label{eq_CS2CE_J7}
\end{align}
For $I_2$, by (\ref{eq_u_Holder_J}) and \ref{eq_KJ_tail}, we have
\begin{align*}
&I_2\lesssim\int_{B(x,2s)}\left(\int_{X\backslash B(y,2s)}\left(\frac{d(y,z)}{r}\wedge 1\right)^{p\delta}\Upsilon(r)^{\frac{p}{p-1}}K^{(J)}(y,z)m(\mathrm{d}z)\right)m(\mathrm{d}y)\\
&=\int_{B(x,2s)}\left(\sum_{n=1}^{+\infty}\int_{B(y,2^{n+1}s)\backslash B(y,2^ns)}\ldots m(\mathrm{d}z)\right)m(\mathrm{d}y)\\
&\le \Upsilon(r)^{\frac{p}{p-1}}\int_{B(x,2s)}\sum_{n=1}^{+\infty} \left(\frac{2^{n+1}s}{r}\wedge1\right)^{p\delta} \left(\int_{X\backslash B(y,2^ns)}K^{(J)}(y,z)m(\mathrm{d}z)\right)m(\mathrm{d}y)\\
&\lesssim \Upsilon(r)^{\frac{p}{p-1}}V(x,2s)\sum_{n=1}^{+\infty} \left(\frac{2^{n+1}s}{r}\wedge1\right)^{p\delta} \frac{1}{\Upsilon(2^ns)}\\
&\lesssim \Upsilon(r)^{\frac{p}{p-1}} \frac{V(x,s)}{\Upsilon(s)}\sum_{n=1}^{+\infty}\left(\frac{2^ns}{r}\wedge 1\right)^{p\delta}\frac{1}{2^{\beta^{(1)}_\Upsilon n}},
\end{align*}
where the last inequality follows from \ref{eq_VD} and (\ref{eq_beta12}). Recall the following elementary result: for any $\alpha,\beta>0$ with $\alpha<\beta$, and any $t>0$, we have
\begin{align*}
\sum_{n=1}^{+\infty}\left(2^nt\wedge1\right)^\alpha \frac{1}{2^{\beta n}}\lesssim \left(t\wedge 1\right)^\alpha,
\end{align*}
where the implicit constant in $\lesssim$ depends only on $\alpha,\beta$. Since $p\delta<\beta^{(1)}_\Upsilon$, we have
$$I_2\lesssim \left(\frac{s}{r}\wedge1\right)^{p\delta}\Upsilon(r)^{\frac{p}{p-1}} \frac{V(x,s)}{\Upsilon(s)}\le \left(\frac{s}{r}\wedge1\right)^{\delta}\Upsilon(r)^{\frac{p}{p-1}} \frac{V(x,s)}{\Upsilon(s)}.$$
Combining this with (\ref{eq_CS2CE_J1}) and (\ref{eq_CS2CE_J7}), we have
\begin{align*}
&\int_{B(x,s)}\mathrm{d}\Gamma^{(J)}_X(u_\Omega)\le I_1+I_2\\
&\lesssim\left(\frac{s}{r}\wedge 1\right)^\delta \Upsilon(r)^{\frac{p}{p-1}}\left(\frac{1}{\Upsilon(r)}+\frac{1}{\Upsilon(s)}\right)V(x,s)\asymp\left(\frac{s}{r}\wedge 1\right)^\delta \Upsilon(r)^{\frac{p}{p-1}}\frac{V(x,s)}{\Upsilon(s\wedge r)}.
\end{align*}
Now we finish the proof of (\ref{eq_u_energy_J}).

Finally, let $\phi=\left(\frac{1}{C_3\Upsilon(r)^{\frac{1}{p-1}}}u_\Omega\right)\wedge1$, then $\phi\in \mathcal{F}^{(J)}(\Omega)\subseteq \mathcal{F}^{(J)}(8B)$. By (\ref{eq_u_lbd_J}), we have $\phi=1$ in $B$, hence $\phi$ is a cutoff function for $B\subseteq 8B$. By (\ref{eq_u_Holder_J}), we have
$$\lvert \phi(x)-\phi(y)\rvert\le \frac{C_4}{C_3}\left(\frac{d(x,y)}{r}\wedge1\right)^\delta\text{ for any }x,y\in X,$$
that is, \hyperref[eq_CE_Holder]{$\text{CE}^{(J)}(\Upsilon)\text{-}1$} holds. By (\ref{eq_u_energy_J}), for any $x\in X$, $s>0$, we have
\begin{align*}
\int_{B(x,s)}\mathrm{d}\Gamma^{(J)}_X(\phi)\le \frac{1}{C_3^p\Upsilon(r)^{\frac{p}{p-1}}}\int_{B(x,s)}\mathrm{d}\Gamma^{(J)}_X(u_\Omega)\le\frac{C_5}{C_3^p}\left(\frac{s}{r}\wedge 1\right)^\delta \frac{V(x,s)}{\Upsilon(s\wedge r)},
\end{align*}
that is, \hyperref[eq_CE_energy]{$\text{CE}^{(J)}(\Upsilon)\text{-}2$} holds.
\end{proof}

\section{Some elementary results}\label{sec_ele}

In this section, we give some elementary results used in the previous sections.

\begin{lemma}\label{lem_ele1}
Let $A,B\ge1$, $\kappa>1$ be constants. Let $\{I_k\}_{k\ge0}$ be a sequence of non-negative real-numbers satisfying that
$$I_{k+1}\le A^{\frac{1}{\kappa^k}}B^{\frac{k}{\kappa^k}}I_k\text{ for any }k\ge0.$$
Then
$$I_k\le A^{\frac{\kappa}{\kappa-1}}B^{\frac{\kappa}{(\kappa-1)^2}}I_0\text{ for any }k\ge0.$$
\end{lemma}

The proof is a routine iteration argument and is therefore omitted.

\begin{lemma}[{\cite[Lemma 8.23]{GT01}}]\label{lem_ele2}
Let $R_0\in(0,+\infty]$ and $\varphi,\psi:(0,R_0)\to[0,+\infty)$ two increasing functions. Assume that there exist $\tau,\gamma\in(0,1)$ such that
$$\varphi(\tau R)\le \gamma \varphi(R)+\psi(R)\text{ for any }R\in(0,R_0).$$
Then for any $\mu\in(0,1)$ and any $R,r\in(0,R_0)$ with $r\le R$, we have
$$\varphi(r)\le \frac{1}{\gamma}\left(\frac{r}{R}\right)^{(1-\mu)\log_\tau\gamma}\varphi(R)+\frac{1}{1-\gamma}\psi \left(\left(\frac{r}{R}\right)^{\mu}R\right).$$
\end{lemma}

The following result is a variant of \cite[Lemma A.1]{KLL23}.

\begin{lemma}\label{lem_ele4}
Let $K>0$, $M\ge0$, $\alpha<-(p-1)$ be constants and let $F,G:[0,M]\to \mathbb{R}$ be given by
\begin{align*}
F(x)&=\frac{1}{\alpha}\left((x+K)^\alpha-(M+K)^\alpha\right),\\
G(x)&=\frac{p}{\alpha+p-1}\left(x+K\right)^{\frac{\alpha+p-1}{p}}.
\end{align*}
Then there exists $C>0$ depending only on $p$ such that for any $x,y\in[0,M]$ and any $a,b\ge0$, we have
\begin{align*}
&\lvert x-y\rvert^{p-2}(x-y)\left(F(x)a^p-F(y)b^p\right)\\
&\ge \frac{1}{2}\lvert G(x)-G(y)\rvert^p \max\{a^p,b^p\}-C \left(1+\left(\frac{\lvert \alpha+p-1\rvert}{\lvert \alpha\rvert}\right)^p\right)\max\{\lvert G(x)\rvert^p,\lvert G(y)\rvert^p\}\lvert a-b\rvert^p.
\end{align*}
\end{lemma}

We need the following result.

\begin{lemma}[{\cite[Lemma 3.1]{DiKP16}}]
Let $p\in[1,+\infty)$ and $\varepsilon\in(0,1]$. Then for any $a,b\in \mathbb{R}$, we have
\begin{align*}
\lvert a\rvert^p\le \lvert b\rvert^p+c_p\varepsilon \lvert b\rvert^p+\left(1+c_p\varepsilon\right)\varepsilon^{1-p}\lvert a-b\rvert^p,
\end{align*}
where $c_p=(p-1)\Gamma \left(\max\{1,p-2\}\right)$ and $\Gamma$ stands for the standard Gamma function, hence
\begin{equation}\label{eq_ele_ab}
\max\{\lvert a\rvert^p,\lvert b\rvert^p\}\le \left(1+c_p\varepsilon\right)\min\{\lvert a\rvert^p,\lvert b\rvert^p\}+\left(1+c_p\varepsilon\right)\varepsilon^{1-p}\lvert a-b\rvert^p.
\end{equation}
\end{lemma}

\begin{proof}[Proof of Lemma \ref{lem_ele4}]
It is easy to see that both $F'$ and $G'$ are positive and monotone decreasing, and satisfy that for any $x\in[0,M]$, we have $G'(x)^p=F'(x)$ and
\begin{equation}\label{eq_ele_FG}
\frac{\lvert F(x)\rvert}{G'(x)^{p-1}}\le \frac{\lvert \alpha+p-1\rvert}{p \lvert \alpha\rvert}\lvert G(x)\rvert.
\end{equation}
Without loss of generality, we may assume that $x>y$, then
\begin{align*}
&\text{LHS}=(x-y)^{p-1}\left(F(x)a^p-F(y)b^p\right)\\
&=(x-y)^{p-1}\left(F(x)-F(y)\right)a^p+(x-y)^{p-1}F(y)\left(a^p-b^p\right).
\end{align*}
By H\"older's inequality, we have
\begin{align}
&F(x)-F(y)=\int_{y}^{x}F'(z)\mathrm{d}z=\int_{y}^{x}G'(z)^p \mathrm{d}z\nonumber\\
&\ge \frac{1}{(x-y)^{p-1}}\left(\int_y^x G'(z)\mathrm{d}z\right)^p=\frac{1}{(x-y)^{p-1}}\lvert G(x)-G(y)\rvert^p.\label{eq_ele4A}
\end{align}
Since $G'$ is positive and monotone decreasing, we have
\begin{equation}\label{eq_ele4B}
\lvert G(x)-G(y)\rvert=\int_{y}^{x}G'(z)\mathrm{d}z\ge(x-y)G'(y).
\end{equation}
Moreover, we have
\begin{equation}\label{eq_ele4C}
\lvert a^p-b^p\rvert=\lvert p\int_{a\wedge b}^{a\vee b}z^{p-1}\mathrm{d}z \rvert\le p \max\{a^{p-1},b^{p-1}\}\lvert a-b\rvert.
\end{equation}
Hence for any $\varepsilon\in(0,1]$, we have
\begin{align*}
&\text{LHS}\overset{(\diamond)}{\scalebox{2}[1]{$\ge$}}\lvert G(x)-G(y)\rvert^p a^p-(x-y)^{p-1}\lvert F(y)\rvert\cdot \lvert a^p-b^p\rvert\\
&\overset{(\star)}{\scalebox{2}[1]{$\ge$}}\lvert G(x)-G(y)\rvert^p\min\{a^p,b^p\}\\
&\hspace{10pt}-(x-y)^{p-1}\left(\frac{\lvert \alpha+p-1\rvert}{p\lvert \alpha\rvert}G'(y)^{p-1}\lvert G(y)\rvert\right)\left(p\max\{a^{p-1},b^{p-1}\}\lvert a-b\rvert\right)\\
&\overset{(\circledast)}{\scalebox{2}[1]{$\ge$}}\lvert G(x)-G(y)\rvert^p\left(\frac{1}{1+c_p\varepsilon}\max\{a^p,b^p\}-\varepsilon^{1-p}\lvert a-b\rvert^p\right)\\
&\hspace{10pt}-\frac{\lvert \alpha+p-1\rvert}{\lvert \alpha\rvert}\lvert G(x)-G(y)\rvert^{p-1}\lvert G(y)\rvert\max\{a^{p-1},b^{p-1}\}\lvert a-b\rvert\\
&\overset{(\dagger)}{\scalebox{2}[1]{$\ge$}}\frac{1}{1+c_p\varepsilon}\lvert G(x)-G(y)\rvert^p\max\{a^p,b^p\}-\varepsilon^{1-p}\lvert G(x)-G(y)\rvert^p \lvert a-b\rvert^p\\
&\hspace{10pt}-\frac{1}{4}\lvert G(x)-G(y)\rvert^{p}\max\{a^{p},b^{p}\}-C_1 \left(\frac{\lvert \alpha+p-1\rvert}{\lvert \alpha\rvert}\right)^p \lvert G(y)\rvert^p \lvert a-b\rvert^p\\
&\ge \left(\frac{1}{1+c_p\varepsilon}-\frac{1}{4}\right)\lvert G(x)-G(y)\rvert^{p}\max\{a^{p},b^{p}\}\\
&\hspace{10pt}-\left(2^p\varepsilon^{1-p}+C_1 \left(\frac{\lvert \alpha+p-1\rvert}{\lvert \alpha\rvert}\right)^p\right)\max\{\lvert G(x)\rvert^p,\lvert G(y)\rvert^p\}\lvert a-b\rvert^p,
\end{align*}
where $(\diamond)$ follows from (\ref{eq_ele4A}), $(\star)$ follows from (\ref{eq_ele_FG}) and (\ref{eq_ele4C}), $(\circledast)$ follows from (\ref{eq_ele_ab}) and (\ref{eq_ele4B}), $(\dagger)$ follows from Young's inequality, and $C_1>0$ depends only on $p$. Taking $\varepsilon=\min\{1,\frac{1}{3c_p}\}\in(0,1]$, which depends only on $p$, we have
\begin{align*}
&\lvert x-y\rvert^{p-2}(x-y)\left(F(x)a^p-F(y)b^p\right)\\
&\ge \frac{1}{2}\lvert G(x)-G(y)\rvert^p \max\{a^p,b^p\}-C \left(1+\left(\frac{\lvert \alpha+p-1\rvert}{\lvert \alpha\rvert}\right)^p\right)\max\{\lvert G(x)\rvert^p,\lvert G(y)\rvert^p\}\lvert a-b\rvert^p,
\end{align*}
where $C=\max\{2^p\varepsilon^{1-p},C_1\}$.
\end{proof}

\begin{lemma}[{\cite[Lemma A.4]{KLL23}}]\label{lem_ele5}
There exist $C_1,C_2>0$ depending only on $p$ such that for any $K>0$ and any $M\ge0$, for any $x,y\in[0,M]$ and any $a,b\ge0$, we have
\begin{align*}
&\lvert x-y\rvert^{p-2}(x-y)\left(\left((x+K)^{1-p}-(M+K)^{1-p}\right)a^p-\left((y+K)^{1-p}-(M+K)^{1-p}\right)b^p\right)\\
&\le -C_1 \lvert \log(x+K)-\log(y+K)\rvert^p\min\{a^p,b^p\}+C_2 \lvert a-b\rvert^p.
\end{align*}
Moreover, letting $M\to+\infty$, we obtain that for any $K>0$ and any $x,y,a,b\ge0$, we have
\begin{align*}
&\lvert x-y\rvert^{p-2}(x-y)\left((x+K)^{1-p}a^p-(y+K)^{1-p}b^p\right)\\
&\le -C_1 \lvert \log(x+K)-\log(y+K)\rvert^p\min\{a^p,b^p\}+C_2 \lvert a-b\rvert^p.
\end{align*}
\end{lemma}

\begin{lemma}\label{lem_ele6}
There exists $C>0$ depending only on $p$ such that for any $x,y\in \mathbb{R}$ and any $a,b\ge0$, we have
\begin{align*}
&\lvert x-y\rvert^{p-2}(x-y)(a^px-b^py)\ge \frac{1}{4}\lvert x-y\rvert^p(a^p+b^p)-C \left(\lvert x\rvert^p+\lvert y\rvert^p\right)\lvert a-b\rvert^p.
\end{align*}
\end{lemma}

\begin{proof}
Without loss of generality, we may assume that $x>y$, then
\begin{align*}
\text{LHS}=(x-y)^{p-1}\left(a^px-b^px+b^px-b^py\right)=(x-y)^pb^p+(x-y)^{p-1}x(a^p-b^p).
\end{align*}
For any $\varepsilon\in(0,1]$, by (\ref{eq_ele_ab}), we have
\begin{align*}
&b^p\ge \min\{a^p,b^p\}\ge \frac{1}{1+c_p\varepsilon}\max\{a^p,b^p\}-\varepsilon^{1-p}\lvert a-b\rvert^p\\
&\ge\frac{1}{2(1+c_p\varepsilon)}\left(a^p+b^p\right)-\varepsilon^{1-p}\lvert a-b\rvert^p.
\end{align*}
Together with (\ref{eq_ele4C}), we have
\begin{align*}
&\text{LHS}\ge\frac{1}{2(1+c_p\varepsilon)}(x-y)^p\left(a^p+b^p\right)-\varepsilon^{1-p}(x-y)^p\lvert a-b\rvert^p\\
&\hspace{15pt}-(x-y)^{p-1}\lvert x\rvert\left(p\max\{a^{p-1},b^{p-1}\}\lvert a-b\rvert\right)\\
&\overset{(\star)}{\scalebox{2}[1]{$\ge$}}\frac{1}{2(1+c_p\varepsilon)}(x-y)^p\left(a^p+b^p\right)-2^{p-1}\varepsilon^{1-p}\left(\lvert x\rvert^p+\lvert y\rvert^p\right)\lvert a-b\rvert^p\\
&\hspace{15pt}-\frac{1}{8}(x-y)^{p}\max\{a^{p},b^{p}\}-C_1\lvert x\rvert^p\lvert a-b\rvert^p\\
&\ge \left(\frac{1}{2(1+c_p\varepsilon)}-\frac{1}{8}\right)(x-y)^p\left(a^p+b^p\right)- \left(2^{p-1}\varepsilon^{1-p}+C_1\right) \left(\lvert x\rvert^p+\lvert y\rvert^p\right)\lvert a-b\rvert^p,
\end{align*}
where $(\star)$ follows from Young's inequality and $C_1>0$ depends only on $p$. Taking $\varepsilon=\min\{1,\frac{1}{3c_p}\}\in(0,1]$, which depends only on $p$, we have
\begin{align*}
\lvert x-y\rvert^{p-2}(x-y)(a^px-b^py)\ge \frac{1}{4}\lvert x-y\rvert^p (a^p+b^p)-C \left(\lvert x\rvert^p+\lvert y\rvert^p\right)\lvert a-b\rvert^p,
\end{align*}
where $C=2^{p-1}\varepsilon^{1-p}+C_1$.
\end{proof}

\begin{lemma}\label{lem_KJ_ele}
Assume \ref{eq_VD}, \ref{eq_KJ}. Then \ref{eq_UPR}, \ref{eq_KJ_tail}, \ref{eq_PIJ} hold.
\end{lemma}

\begin{proof}
First, by \ref{eq_VD}, \ref{eq_KJ}, for any $x_0,x,y\in X$ with $x_0\ne x$ and $x_0\ne y$, we have
$$\frac{K^{(J)}(x_0,x)}{K^{(J)}(x_0,y)}\le \frac{C_2V(x_0,d(x_0,y))\Upsilon(d(x_0,y))}{C_1V(x_0,d(x_0,x))\Upsilon(d(x_0,x))}\le \frac{C_2C_\Upsilon C_{VD}}{C_1} \left(\frac{d(x_0,y)}{d(x_0,x)}\vee1\right)^{\log_2(C_\Upsilon C_{VD})}.$$

Second, \ref{eq_KJ_tail} follows directly from \ref{eq_VD}, \hyperlink{eq_uKJ}{$\text{K}^{(J)}_\le(\Upsilon)$} by noting that for any $x\in X$ and any $r>0$, we have
\begin{align*}
&\int_{X\backslash B(x,r)}\frac{1}{V(x,d(x,y))\Upsilon(d(x,y))}m(\mathrm{d}y)\\
&=\sum_{n=0}^{+\infty}\int_{B(x,2^{n+1}r)\backslash B(x,2^nr)}\frac{1}{V(x,d(x,y))\Upsilon(d(x,y))}m(\mathrm{d}y)\\
&\le\sum_{n=0}^{+\infty} \frac{V(x,2^{n+1}r)}{V(x,2^nr)\Upsilon(2^nr)}\le \left(C_\Upsilon C_{VD}\sum_{n=0}^{+\infty}\frac{1}{2^{\beta_\Upsilon^{(1)}n}}\right)\frac{1}{\Upsilon(r)}.
\end{align*}

Third, by H\"older's inequality, we have
\begin{align*}
&\int_{B(x,r)}\lvert f-f_{B(x,r)}\rvert^p \mathrm{d}m\le \frac{1}{V(x,r)}\int_{B(x,r)}\int_{B(x,r)}\lvert f(y)-f(z)\rvert^p m(\mathrm{d}y) m(\mathrm{d}z)\\
&=\int_{B(x,r)}\int_{B(x,r)}\frac{\lvert f(y)-f(z)\rvert^p}{V(y,d(y,z))\Upsilon(d(y,z))} \frac{V(y,d(y,z))}{V(x,r)}\Upsilon(d(y,z))m(\mathrm{d}y) m(\mathrm{d}z),
\end{align*}
where
$$\Upsilon(d(y,z))\le \Upsilon(2r)\le C_\Upsilon \Upsilon(r),$$
and by \ref{eq_VD}, we have
\begin{align*}
\frac{V(y,d(y,z))}{V(x,r)}\le C_{VD}\left(\frac{d(x,y)+2r}{r}\right)^{\log_2C_{VD}}\le C_{VD}^3.
\end{align*}
Hence
$$\int_{B(x,r)}\lvert f-f_{B(x,r)}\rvert^p \mathrm{d}m\le C_\Upsilon C_{VD}^3\Upsilon(r)\int_{B(x,r)}\int_{B(x,r)}\frac{\lvert f(y)-f(z)\rvert^p}{V(y,d(y,z))\Upsilon(d(y,z))}m(\mathrm{d}y)m(\mathrm{d}z),$$
then \ref{eq_PIJ} follows from \hyperlink{eq_lKJ}{$\text{K}^{(J)}_\ge(\Upsilon)$}.
\end{proof}

\begin{lemma}\label{lem_CSJ_energy}
Assume \ref{eq_KJ_tail}. Then \ref{eq_CSJ_energy} is equivalent to
$$\int_{B(x_0,A_2r)}\lvert f\rvert^p \mathrm{d}\Gamma^{(J)}_{X}(\phi)\le C_1\int_{B(x_0,A_2r)}\mathrm{d}\Gamma^{(J)}_{B(x_0,A_2r)}(f)+\frac{C_2}{\Upsilon(r)}\int_{B(x_0,A_2r)}\lvert f\rvert^p \mathrm{d}m,$$
up to suitable modifications of the constants $C_1,C_2$.
\end{lemma}

\begin{proof}
``$\Leftarrow$": Obvious. ``$\Rightarrow$": It follows easily by noting that
$$\int_{B(x_0,A_2r)}\lvert f\rvert^p \mathrm{d}\Gamma^{(J)}_{X}(\phi)=\int_{B(x_0,A_2r)}\lvert f\rvert^p \mathrm{d}\Gamma^{(J)}_{B(x_0,A_2r)}(\phi)+I,$$
where
\begin{align*}
&I=\int_{B(x_0,A_2r)}\int_{X\backslash B(x_0,A_2r)}\lvert f(x)\rvert^p \lvert \phi(x)-\phi(y)\rvert^p K^{(J)}(x,y) m(\mathrm{d}x)m(\mathrm{d}y)\\
&=\int_{B(x_0,A_1r)}\int_{X\backslash B(x_0,A_2r)}\lvert f(x)\rvert^p  \phi(x)^p K^{(J)}(x,y) m(\mathrm{d}x)m(\mathrm{d}y)\\
&\le\int_{B(x_0,A_1r)}\lvert f(x)\rvert^p \left(\int_{X\backslash B(x,(A_2-A_1)r)}K^{(J)}(x,y)m(\mathrm{d}y)\right)m(\mathrm{d}x)\\
&\le \frac{C_T}{\Upsilon((A_2-A_1)r)}\int_{B(x,A_1r)}\lvert f\rvert^p \mathrm{d}m\le \frac{C}{\Upsilon(r)}\int_{B(x,A_2r)}\lvert f\rvert^p \mathrm{d}m,
\end{align*}
where $C>0$ depends only on $A_1,A_2,C_\Upsilon,C_T$.
\end{proof}

\begin{lemma}\label{lem_CSJ2ucapJ}
Assume \ref{eq_VD}, \ref{eq_KJ_tail}. Then \hyperlink{eq_CSJ_weak}{$\text{CS}^{(J)}_{\text{weak}}(\Upsilon)$} implies \ref{eq_ucapJ}.
\end{lemma}

\begin{proof}
By taking $f\equiv1\in \widehat{\mathcal{F}}^{(J)}\cap L^\infty(X;m)$ in Lemma \ref{lem_CSJ_energy}, we have
$$\int_{B(x_0,A_2r)} \mathrm{d}\Gamma^{(J)}_{X}(\phi)\le C_2\frac{V(x_0,A_2r)}{\Upsilon(r)}.$$
By \ref{eq_KJ_tail}, we have
\begin{align*}
&\int_{X\backslash B(x_0,A_2r)} \mathrm{d}\Gamma^{(J)}_{X}(\phi)=\int_{X\backslash B(x_0,A_2r)}\int_X \lvert \phi(x)-\phi(y)\rvert^p K^{(J)}(x,y) m(\mathrm{d}x)m(\mathrm{d}y)\\
&=\int_{X\backslash B(x_0,A_2r)}\int_{B(x_0,A_1r)} \phi(y)^p K^{(J)}(x,y) m(\mathrm{d}x)m(\mathrm{d}y)\\
&\le\int_{B(x_0,A_1r)}\left(\int_{X\backslash B(y,(A_2-A_1)r)}K^{(J)}(y,x)m(\mathrm{d}x)\right)m(\mathrm{d}y)\le C_T\frac{V(x_0,A_1r)}{\Upsilon((A_2-A_1)r)}.
\end{align*}
Then by \ref{eq_VD}, we have
$$\mathcal{E}^{(J)}(\phi)=\int_{B(x_0,A_2r)} \mathrm{d}\Gamma^{(J)}_{X}(\phi)+\int_{X\backslash B(x_0,A_2r)} \mathrm{d}\Gamma^{(J)}_{X}(\phi)\le C \frac{V(x_0,r)}{\Upsilon(r)},$$
where $C>0$ depends only on $A_1,A_2,C_2,C_\Upsilon,C_T$, which gives
$$\mathrm{cap}^{(J)}(B(x_0,r),X\backslash B(x_0,A_1r))\le C \frac{V(x_0,r)}{\Upsilon(r)},$$
hence \ref{eq_ucapJ} holds.
\end{proof}

\begin{lemma}\label{lem_cutoff}
Let $(\mathcal{E}^{(J)},\mathcal{F}^{(J)})$ be a non-local regular $p$-energy. Let $\Omega\subseteq X$ be a bounded open subset and let $u\in \widehat{\mathcal{F}}^{(J)}\cap L^\infty(X;m)$ satisfy $-\Delta_p^{(J)}u\ge-1$ in $\Omega$, that is, $\mathcal{E}^{(J)}(u;\varphi)\ge-\int_{X}\varphi \mathrm{d}m$ for any non-negative $\varphi\in \mathcal{F}^{(J)}(\Omega)$. Then for any $M\in \mathbb{R}$, we have $u\wedge M\in \widehat{\mathcal{F}}^{(J)}\cap L^\infty(X;m)$ also satisfies that $-\Delta_p^{(J)}(u\wedge M)\ge -1$ in $\Omega$.
\end{lemma}

\begin{proof}
Let $f\in C^2(\mathbb{R})$ satisfy $0\le f'\le1$ and $f^{\prime\prime}\le0$ in $\mathbb{R}$. Then the following elementary result holds: for any $X,Y\in \mathbb{R}$ and any $a,b\ge0$, we have
$$\lvert f(X)-f(Y)\rvert^{p-2}\left(f(X)-f(Y)\right)(a-b)\ge \lvert X-Y\rvert^{p-2}(X-Y)\left(f'(X)^{p-1}a-f'(Y)^{p-1}b\right).$$
Indeed, without loss of generality, we may assume that $X>Y$, then there exists $Z\in [Y,X]$ such that $f'(Z)=\frac{f(X)-f(Y)}{X-Y}$. Since $0\le f'(X)\le f'(Z)\le f'(Y)$ and $a,b\ge0$, we have
\begin{align*}
&\text{RHS}\le(X-Y)^{p-1}\left(f'(Z)^{p-1}a-f'(Z)^{p-1}b\right)\\
&=(X-Y)^{p-1} \left(\frac{f(X)-f(Y)}{X-Y}\right)^{p-1}(a-b)=\text{LHS}.
\end{align*}

Since $0\le f'\le 1$ in $\mathbb{R}$, by the Markovian property, we have $f(u),f'(u)^{p-1}\in \widehat{\mathcal{F}}^{(J)}\cap L^\infty(X;m)$. For any non-negative $\varphi\in \mathcal{F}^{(J)}\cap C_c(\Omega)$, we have
\begin{align*}
&\mathcal{E}^{(J)}(f(u);\varphi)\\
&=\int_X\int_X \lvert f(u(x))-f(u(y))\rvert^{p-2}\left(f(u(x))-f(u(y))\right)(\varphi(x)-\varphi(y))K^{(J)}(x,y) m(\mathrm{d}x)m(\mathrm{d}y)\\
&\ge\int_X\int_X \lvert u(x)-u(y)\rvert^{p-2}(u(x)-u(y)) \left(f'(u(x))^{p-1}\varphi(x)-f'(u(y))^{p-1}\varphi(y)\right)\\
&\hspace{50pt}\cdot K^{(J)}(x,y)m(\mathrm{d}x)m(\mathrm{d}y)\\
&=\mathcal{E}^{(J)}(u;f'(u)^{p-1} \varphi).
\end{align*}
By the Markovian property again, we have $f'(u)^{p-1} \varphi\in \mathcal{F}^{(J)}(\Omega)$ is non-negative. By assumption, we have
$$\mathcal{E}^{(J)}(u;f'(u)^{p-1} \varphi)\ge-\int_{X}f'(u)^{p-1} \varphi \mathrm{d}m\ge-\int_X \varphi \mathrm{d}m,$$
where the second inequality follows from that facts that $\varphi$ is non-negative and that $0\le f'\le1$ in $\mathbb{R}$. Hence $\mathcal{E}^{(J)}(f(u);\varphi)\ge-\int_X\varphi \mathrm{d}m$ for any non-negative $\varphi\in \mathcal{F}^{(J)}\cap C_c(\Omega)$. By Lemma \ref{lem_FOmega}, this inequality holds for any non-negative $\varphi\in \mathcal{F}^{(J)}(\Omega)$.

For any $M\in \mathbb{R}$, by the Markovian property, we have $u\wedge M\in \widehat{\mathcal{F}}^{(J)}\cap L^\infty(X;m)$. Let $g:\mathbb{R}\to \mathbb{R}$, $x\mapsto x\wedge M$, then there exists $\{f_n\}_{n}\subseteq C^2(\mathbb{R})$ with $0\le f'_n\le1$ and $f^{\prime\prime}_n\le0$ in $\mathbb{R}$ for any $n$, such that $f_n\uparrow g$ uniformly in $\mathbb{R}$. By above, for any non-negative $\varphi\in \mathcal{F}^{(J)}(\Omega)$, we have $\mathcal{E}^{(J)}(f_n(u);\varphi)\ge-\int_X \varphi \mathrm{d}m$. Letting $n\to+\infty$, by the dominated convergence theorem, we have $\mathcal{E}^{(J)}(u\wedge M;\varphi)\ge-\int_X \varphi \mathrm{d}m$, that is, $-\Delta_p^{(J)}(u\wedge M)\ge-1$ in $\Omega$.
\end{proof}

\begin{lemma}[{\cite[Corollary 2.5]{LM23}, \cite[Lemma 4.1]{BKKL25}}]\label{lem_Xi_basic}
Let $\Upsilon,\Psi$ be doubling functions satisfying (\ref{eq_beta12}) and \ref{eq_SUG0}. Let $\Xi$ be given by (\ref{eq_Xi}).
\begin{enumerate}[label=(\arabic*),ref=(\arabic*)]
\item\label{item_Xi1} The function $\Xi$ is well-defined and doubling. Moreover, it satisfies (\ref{eq_beta12}) and, for any $r>0$,
$$\Xi(r)\asymp\frac{\Psi(r)}{\sum_{n=0}^{+\infty}\frac{\Psi(\frac{1}{2^n}r)}{\Upsilon(\frac{1}{2^n}r)}}.$$
\item\label{item_Xi2} The functions $\Xi,\Upsilon,\Psi$ satisfy the following estimates:
$$\Xi(r)<\Upsilon(r)\text{ for any }r>0,$$
$$\frac{\Xi(R)}{\Xi(r)}\le \frac{\Psi(R)}{\Psi(r)}\text{ for any }r\le R,$$
\begin{equation}\label{eq_Xi4}
\Xi(r)
\begin{cases}
\lesssim \Psi(r),&\text{for any }r>1,\\
\gtrsim \Psi(r),&\text{for any }r\le1.
\end{cases}
\end{equation}
\end{enumerate}
\end{lemma}

\begin{proof}
Since $\Upsilon,\Psi$ are doubling functions satisfying (\ref{eq_beta12}), for any $r>0$,
$$\Xi(r)=\frac{\Psi(r)}{\int_0^r \frac{\mathrm{d}\Psi(t)}{\Upsilon(t)}}\asymp \frac{\Psi(r)}{\sum_{n=0}^{+\infty} \frac{\Psi(\frac{1}{2^n}r)-\Psi(\frac{1}{2^{n+1}}r)}{\Upsilon(\frac{1}{2^n}r)}}\asymp \frac{\Psi(r)}{\sum_{n=0}^{+\infty} \frac{\Psi(\frac{1}{2^n}r)}{\Upsilon(\frac{1}{2^n}r)}}.$$
By \ref{eq_SUG0}, we have $\Xi$ is well-defined. Moreover,
$$\Xi(r)=\frac{\Psi(r)}{\int_0^r \frac{\mathrm{d}\Psi(t)}{\Upsilon(t)}}<\frac{\Psi(r)}{\int_0^r \frac{\mathrm{d}\Psi(t)}{\Upsilon(r)}}=\Upsilon(r),$$
and
$$\mathrm{d}\Xi(r)=\frac{1}{\left(\int_0^r \frac{\mathrm{d}\Psi(t)}{\Upsilon(r)}\right)^2}\left(\frac{1}{\Xi(r)}-\frac{1}{\Upsilon(r)}\right)\Psi(r)\mathrm{d}\Psi(r),$$
hence $\Xi:[0,+\infty)\to[0,+\infty)$ is strictly increasing. For any $r>0$, by \ref{eq_SUG0},
\begin{align*}
\Xi(r)=\frac{\Psi(r)}{\int_0^r \frac{\mathrm{d}\Psi(t)}{\Upsilon(t)}}
\begin{cases}
\le\frac{\Psi(r)}{\int_0^1 \frac{\mathrm{d}\Psi(t)}{\Upsilon(t)}},&\text{for }r>1,\\
\ge\frac{\Psi(r)}{\int_0^1 \frac{\mathrm{d}\Psi(t)}{\Upsilon(t)}},&\text{for }r\le1,
\end{cases}
\end{align*}
hence (\ref{eq_Xi4}) holds. For any $r\le R$, we have
$$\frac{\Xi(R)}{\Xi(r)}=\frac{\Psi(R)}{\Psi(r)}\frac{\int_0^r \frac{\mathrm{d}\Psi(t)}{\Upsilon(t)}}{\int_0^R \frac{\mathrm{d}\Psi(t)}{\Upsilon(t)}}\le \frac{\Psi(R)}{\Psi(r)}.$$
Combining this with the proof of \cite[Lemma 4.1]{BKKL25}, especially with the argument proving Equation (4.5) therein, we have $\Xi$ is a doubling function satisfying (\ref{eq_beta12}).
\end{proof}

\bibliographystyle{plain}

\end{document}